\documentclass[11pt,fullpage]{article}
\usepackage{amsmath}
\usepackage{amsfonts,amssymb}
\usepackage{color}	
\usepackage{amsthm}
\usepackage{graphics,graphicx}
\usepackage{hyperref}
\usepackage{titlesec}
\usepackage[right = 2.5cm, left=2.5cm, top = 2.5cm, bottom =2.5cm]{geometry}
\numberwithin{equation}{section}

\titleformat{\paragraph}
{\normalfont\normalsize\bfseries}{\theparagraph}{1em}{}
\titlespacing*{\paragraph}
{0pt}{3.25ex plus 1ex minus .2ex}{1.5ex plus .2ex}

\titleformat{\subparagraph}
{\normalfont\normalsize\bfseries}{\theparagraph}{1em}{}
\titlespacing*{\paragraph}
{0pt}{3.25ex plus 1ex minus .2ex}{1.5ex plus .2ex}

\newcommand{\Real}{\mathbb R}

\newcommand{\Torus}{\mathbb T}
\newcommand{\Integers}{\mathbb Z}
\newcommand{\Integer}{\mathbb Z}
\newcommand{\norm}[1]{\left\lVert #1 \right\rVert}
\newcommand{\abs}[1]{\left\vert#1\right\vert}
\newcommand{\set}[1]{\left\{#1\right\}}

\newcommand{\grad}{\nabla}

\newcommand{\G}{\mathcal{G}}

\newcommand{\I}{\mathbf{I}}

\newcommand{\jap}[1]{\left\langle #1 \right\rangle} 

\newtheorem{theorem}{Theorem}
\theoremstyle{definition}
\newtheorem{remark}{Remark}

\theoremstyle{lemma}

\newtheorem{proposition}{Proposition}[section]

\theoremstyle{definition}

\theoremstyle{lemma}
\newtheorem{lemma}{Lemma}[section]

\numberwithin{remark}{section}

\begin{document}

\title{Dynamics near the subcritical transition of the 3D Couette flow II: Above threshold case}
\author{Jacob Bedrossian\footnote{\textit{jacob@cscamm.umd.edu}, University of Maryland, College Park} \, and Pierre Germain\footnote{\textit{pgermain@cims.nyu.edu}, Courant Institute of Mathematical Sciences} \, and Nader Masmoudi\footnote{\textit{masmoudi@cims.nyu.edu}, Courant Institute of Mathematical Sciences}}
\date{\today}
\maketitle

\begin{abstract} 
This is the second in a pair of works which study small disturbances to the plane, periodic 3D Couette flow in the incompressible Navier-Stokes equations at high Reynolds number \textbf{Re}. 
In this work, we show that there is constant $0 < c_0 \ll 1$, independent of $\textbf{Re}$, such that sufficiently regular disturbances of size $\epsilon \lesssim \textbf{Re}^{-2/3-\delta}$ for any $\delta > 0$ exist at least until $t = c_0\epsilon^{-1}$ and in general evolve to be $O(c_0)$ due to the lift-up effect.  
Further, after times $t \gtrsim \textbf{Re}^{1/3}$, the streamwise dependence of the solution is rapidly diminished by a mixing-enhanced dissipation effect and the solution is attracted back to the class of ``2.5 dimensional'' streamwise-independent solutions (sometimes referred to as ``streaks'').
The largest of these streaks are expected to eventually undergo a secondary instability at $t \approx \epsilon^{-1}$. Hence, our work strongly suggests, for \emph{all} (sufficiently regular) initial data, the genericity of the ``lift-up effect $\Rightarrow$ streak growth $\Rightarrow$ streak breakdown'' scenario for turbulent transition of the 3D Couette flow near the threshold of stability forwarded in the applied mathematics and physics literature. 
\end{abstract}

\setcounter{tocdepth}{1}
{\small\tableofcontents}

\section{Introduction}

This work is the second paper in our study of the 3D Navier-Stokes equation near the (plane, periodic) Couette flow, following our work \cite{BGM15I}.
In these works, we study the 3D Navier-Stokes equations near the Couette flow in the idealized domain $(x,y,z) \in \Torus \times \Real \times \Torus$: if $u + \left(y, 0, 0 \right)^T$ solves the Navier-Stokes equation, $u$ solves 
\begin{subequations}\label{def:3DNSE}
\begin{align} 
\partial_t u + y \partial_x u + u\cdot \grad u + \grad p^{NL} & = \begin{pmatrix} - u^2 \\ 0 \\ 0 \end{pmatrix} - \grad p^L  + \nu\Delta u \\ 
\Delta p^{NL} & = -\partial_i u^j \partial_j u^i \\ 
\Delta p^L & = -2\partial_x u^2 \\ 
\grad \cdot u & = 0,
\end{align}
\end{subequations}
where $\nu = \textbf{Re}^{-1}$ denotes the inverse Reynolds number, $p^{NL}$ is the nonlinear contribution to the pressure due to the disturbance 
and $p^L$ is the linear contribution to the pressure due to the interaction between the disturbance and the Couette flow.
The purpose of this work, along with \cite{BGM15I}, is to further the mathematically rigorous understanding of the qualitative  behavior of \eqref{def:3DNSE} for small perturbations and small $\nu$. 
This second work is focused on characterizing the dynamics of solutions above the stability threshold (but still not too large). 

A major focus of the theory of hydrodynamic stability is the study of laminar flow configurations and understanding when they are stable or when they may transition to a turbulent state (or a nonlinear intermediate state). 
The terminology \emph{subcritical transition} refers to a situation when the linear theory predicts stability below some critical Reynolds number or at all Reynolds number (the latter is the case here) 
but spontaneous transition to a turbulent state is observed in laboratory or computer experiments at a much lower Reynolds number than what this theory predicts.  
To our knowledge, the first quantitative study of this process in fluid mechanics was performed by Reynolds in 1883 \cite{Reynolds83}, and since then, subcritical transition has been observed to be a ubiquitous phenomenon in 3D hydrodynamics, repeated by countless physical experiments (see e.g. \cite{Nishioka1975,klebanoff1962,Tillmark92,Daviaud92,Elofsson1999,bottin98,HofJuelMullin2003,Mullin2011,LemoultEtAl2012}) and computer simulation (e.g. \cite{Orszag80,HLJ93,ReddySchmidEtAl98,DuguetEtAl2010} and the references therein) on subcritical transition phenomena have been performed in many different settings. 
See the texts \cite{DrazinReid81,Yaglom12,SchmidHenningson2001} and part I of our work \cite{BGM15I} for further discussion and references.

As discussed in \cite{BGM15I}, a natural expectation is that while the flow may be stable for all finite Reynolds number, the basin of stability is shrinking as $\nu \rightarrow 0$.  
Hence, it becomes of interest to, given a norm, determine the threshold of stability, sometimes called the ``transition threshold'', as a function of $\nu$. For example, one would like, given a norm $\norm{\cdot}_N$, to find a $\gamma = \gamma(N)$ such that $\norm{u_{in}}_N \ll \nu^\gamma$ implies stability and $\norm{u_{in}}_N \gg \nu^\gamma$ in general permits instability. Further, one would like to identify the possible pathways the solution can take towards transition. A great deal of work has been dedicated to identifying $\gamma$ and estimates from experiments, computer simulations, and formal analysis suggest a threshold somewhere between $1 \leq \gamma \leq 7/4$ for a variety of different initial data and configurations similar to the set-up in \eqref{def:3DNSE}  (see \cite{BGM15I} for more references and some of the representative works \cite{TTRD93,Waleffe95,BT97,LHRS94,ReddySchmidEtAl98,Chapman02,Mullin2011} or the text \cite{SchmidHenningson2001} and the references therein). 
In \cite{BGM15I}, we proved that, for sufficiently regular initial data, $\gamma = 1$ for \eqref{def:3DNSE} (that is, for a sufficiently strong norm $N$, $\gamma(N) = 1$). 

In this work our goal is to characterize the instabilities of above threshold solutions. 
We prove that there is a universal constant $c_0$ with  $0 < c_0 \ll 1$ such that for sufficiently regular initial data (in the same sense as \cite{BGM15I}) of size $\epsilon$, if $\epsilon \lesssim \nu^{2/3+\delta}$ for $\delta> 0$, then the solution to \eqref{def:3DNSE} exists until at least time $t = c_0\epsilon^{-1}$ and is rapidly attracted to the class of $x$-independent solutions known as \emph{streaks} for times $t \gtrsim \nu^{-1/3}$.
Due to a non-modal instability known as the lift-up effect, the streaks (and hence all solutions) will in general grow linearly as $O(\epsilon t)$ and by the final time can be $O(c_0)$ (which is independent of $\nu$).  
In our companion work \cite{BGM15I}, we studied solutions below the $\epsilon \ll \nu$ threshold and proved that these solutions are global, return to Couette flow, and also converge to the set of streak solutions. 
While our previous analysis did include solutions which get $O(c_0)$ from the Couette flow, all solutions never deviate farther from the Couette flow and are demonstrably not involved in any transition processes. 

The foremost interest of this work is that the threshold solutions we study can converge to streaks that, due to the lift-up effect, eventually become as large as the Couette flow itself (although we cannot follow our solutions to this point). 
These large streaks induce an unstable shear flow and are expected to become linearly unstable, sometimes referred to as a \emph{secondary instability} \cite{ReddySchmidEtAl98,Chapman02,SchmidHenningson2001}. 
The instability is observed to involve the rapid growth of $x$-dependent modes. 
The process by which large streaks exhibit instabilities and drive $x$-dependent flows is sometimes referred to as \emph{streak breakdown} and  is well-documented as one of the primary routes towards turbulent transition observed experimentally \cite{klebanoff1962,bottin98,Elofsson1999} and in computer simulations \cite{ReddySchmidEtAl98,DuguetEtAl2010}, in agreement with a variety of formal asymptotic calculations \cite{ReddySchmidEtAl98,Chapman02,SchmidHenningson2001}.  
That is, it is an expectation that a general route towards transition is the multi-step process ``lift-up effect $\Rightarrow$ streak growth $\Rightarrow$ streak breakdown $\Rightarrow$ transition''.   
Moreover, the general process of streak breakdown is thought to play an important role in sustaining turbulence near the transition threshold and in both the creation and decay of ``turbulent spots'' (see \cite{SchmidHenningson2001} and the references therein). 
While we cannot take our solutions through the secondary instability, we prove that solutions above the threshold (but not too far above) can in general converge to unstable streaks and that this is the only instability that possible, which is suggestive of the genericity of the above multi-step process as the first step towards turbulent transition near the threshold (for sufficiently regular data -- see Remark \ref{ref:lowreg} below for more discussion on rougher data).  

Unlike in \cite{BGM15I}, the solutions we are concerned with are unstable in the sense that they might transition for $t \gg \epsilon^{-1}$, and we are identifying that the streamwise vortex/streak instability associated with the lift-up effect is dominant whereas all other dynamics are suppressed. 
At the linear level, another important effect is the vortex stretching, which in particular, causes a direct cascade of energy to high frequencies in the $u^1$ and $u^3$ components and creates growth which is difficult to control.   
The stabilizing mechanisms suppressing the more complicated nonlinear effects are the mixing-enhanced dissipation and the inviscid damping, both due to the mixing from the background shear flow. 
Enhanced dissipation was first observed in \eqref{def:3DNSE} by Lord Kelvin \cite{Kelvin87} and has been observed in many contexts in fluid mechanics (see e.g. \cite{DrazinReid81,RhinesYoung83,DubrulleNazarenko94,LatiniBernoff01,BernoffLingevitch94,ReddySchmidEtAl98} and the mathematically rigorous works \cite{BeckWayne11,ConstantinEtAl08,BMV14}). 
In \eqref{def:3DNSE}, the mixing due to Couette drives information to high frequencies, enhancing the dissipation of $x$-dependent modes such that they decay on a time-scale like $\tau_{ED} \sim \nu^{-1/3}$, far faster than the natural ``heat equation'' time-scale $O(\nu^{-1})$. 
The idea that the enhanced dissipation effect has an important role to play in \eqref{def:3DNSE} dates back at least to \cite{DubrulleNazarenko94}. 
Indeed, in \cite{DubrulleNazarenko94}, an idea similar to the heuristic  \eqref{ineq:heur23} below appears. However, the expectation that a large mean shear should suppress certain kinds of instabilities has been suggested at varying levels of precision in many contexts (see e.g. \cite{DrazinReid81,Yaglom12,ReddySchmidEtAl98,Chapman02} and the references therein).  
Inviscid damping in fluid mechanics was first observed by Orr \cite{Orr07} in 1907 and turned out to be the hydrodynamic analogue of Landau damping in plasma physics; see \cite{BM13,BGM15I} for more discussion. 
Here, inviscid damping will suppress the $x$-dependence of $u^2$, key to controlling certain components of the nonlinearity that would otherwise be uncontrollable.

The fact that we prove results for initial data as large as $\nu^{2/3+\delta}$  shows that the streak growth scenario
is generic even for initial data which is far larger than the $O(\nu)$ threshold, at least for data which is sufficiently regular. 
Moreover, we are not aware of this exponent appearing anywhere in the applied mathematics or physics literature previously despite being a threshold of natural interest.
The $2/3$ threshold can be explained from heuristics. Formal analysis of the weakly nonlinear resonances, described in \S\ref{sec:Toy}, suggests that the natural time-scale before a general $x$-dependent solution could potentially become fully nonlinear, $\tau_{NL}$, is \emph{at least} $\tau_{NL} \gtrsim \epsilon^{-1/2}$. 
On the other hand, the enhanced dissipation occurs on time-scales like $\tau_{ED} \sim \nu^{-1/3}$. 
  Hence, if the enhanced dissipation is to dominate the 3D effects and relax the solution to the manifold of streaks, then we need the latter time scale to be shorter than the former, or rather: 
\begin{align} 
\tau_{ED} \sim \nu^{-1/3} \ll \epsilon^{-1/2} \lesssim \tau_{NL}. \label{ineq:heur23}
\end{align} 
This is the origin of the requirement $\epsilon \lesssim \nu^{2/3+\delta}$; the small $\delta> 0$ is to provide a little technical room to work with in the estimates (although we do not know if it can be removed).
We emphasize that getting a convincing estimate on $\tau_{NL}$ is challenging, which may explain why this threshold does not appear in the literature (moreover, the heuristics of \S\ref{sec:Toy} are likely only convincing when backed by Theorem \ref{thm:SRS} and its proof). 
After $t \gg \tau_{ED}$ the solution is very close to a streak and, due to the lift-up effect, in general $u^1_0(t)$ is growing like $\epsilon \jap{t}$ until times $t \sim \epsilon^{-1}$, at which point the streak will become fully nonlinear (see \cite{BGM15I,ReddySchmidEtAl98,Chapman02} and the references therein).  Below we discuss several other ways to derive the $\epsilon \sim \nu^{2/3}$ cut-off which are in some ways more straightforward but also a bit more ad-hoc (see \S\ref{sec:NonlinHeuristics} and \S\ref{sec:Toy}).  

As discussed in \cite{BGM15I,BM13}, if there is decay-via-mixing then, since mixing is time-reversible (at infinite Reynolds number),  necessarily there is growth-via-unmixing. This non-modal effect was first pointed out by Orr \cite{Orr07} and is now known as the \emph{Orr mechanism}. 
Some of the more subtle and problematic nonlinear effects here are 3D variants of the nonlinear manifestation of the Orr mechanism referred to as an ``echo''. These are resonances (perhaps more accurately ``pseudo-resonances'' \cite{Trefethen2005}) involving the excitation of unmixing modes (see \cite{Craik1971,VMW98,Vanneste02,BGM15I,BM13} and the references therein for discussion about this effect in the context of fluid mechanics and \cite{YuDriscoll02,YuDriscollONeil} for physical experiments isolating them in 2D Euler).
A similar resonance is also observed in plasmas, known there as a ``plasma echo'' \cite{MalmbergWharton68}.  
A key facet of the proof in \cite{BGM15I} was the use of careful weakly nonlinear analysis to estimate the possible effects of resonances of this general type (and also others).  

Relative to our previous work \cite{BGM15I}, this work will need more precision in the weakly nonlinear analysis and uses more detailed structure of the nonlinearity.  
In \cite{BGM15I}, a toy model was derived to model the ``worst possible'' behaviors due to the lift-up effect, the ``resonances'' associated with the Orr mechanism (e.g. echo-like), and the vortex stretching, accounting also for the stabilizing mechanisms of enhanced dissipation and inviscid damping (see \S\ref{sec:Toy}).
An approximate super-solution of this toy model was used to derive a set of good norms with which to measure the solution. 
The super-solution used in \cite{BGM15I} required $\epsilon \lesssim \nu$; here we will derive a super-solution which only requires $\epsilon \lesssim \nu^{2/3}$ but (A) it is more subtle than that of \cite{BGM15I} and (B) is only valid for $t \lesssim \epsilon^{-1}$. This latter point is not surprising: at around this point, the solution is expected to suffer streak breakdown and transition to turbulence (or at least escape a weakly nonlinear regime).   
One of the new complexities that the super-solution will introduce is that the norm used to measure $u^3$ will need to unbalance the regularity of different frequencies in the $x$-dependent modes of $u^3$ in a subtle and precise way. This turns out to be similar to a technique applied to the scalar vorticity in 2D \cite{BM13,BMV14}, however, 
here it is not so much the imbalances within $u^3$ itself which are important, but rather the imbalances between $u^3$ and the \emph{other} components. 
Together with the much smaller dissipation, the additional precision in the norm will noticeably complicate the proof of Theorem \ref{thm:SRS} below (relative to \cite{BGM15I}). 
Many terms here will require a more detailed treatment than that used in \cite{BGM15I}, either because of the more complicated norms or because the dissipation is weaker.
The additional precision will require some new techniques and better technical tools, including more precise multiplier inequalities relating time and frequency (see \S\ref{sec:nrmuse}) and several new elliptic estimates (see Appendix \ref{sec:PEL}). Another adjustment we will make here is a nonlinear coordinate transform which is more precise than the one employed in \cite{BGM15I}; in particular, we will need to account for mixing caused by $(0,0,u^3_0)^T$ as well as $(y + u^1_0,0,0)^T$ and hence treat the entire streak in an essentially Lagrangian fashion.     
In order to carry out this line of attack we will need to use more structure in the nonlinearity than \cite{BGM15I} and understand better certain ``null'' or ``non-resonant'' structures, in particular, detailed information about how certain frequencies interact.

\subsection{Linear behavior and streaks} \label{sec:LinStreak} 
Recall the following notation from \cite{BGM15I}: the projections of a function
$f$ onto zero and non-zero frequencies in $x$ are denoted, respectively, by 
\begin{align*}
f_0(y,z) & = \frac{1}{2\pi}\int f(x,y,z) dx \\ 
f_{\neq} & = f - f_0.  
\end{align*}
Next, we recall from \cite{BGM15I} the following Proposition, which regards the behavior of the linearized Navier-Stokes equations. 
There is a corresponding result also for the linearized Euler equations; see \cite{BGM15I} for more details.  
Without making any attempt to be optimal in terms of regularity, this proposition emphasizes the stabilizing mechanisms of enhanced dissipation and inviscid damping, and the destabilizing mechanisms of the lift-up effect and the vortex stretching due to the Couette flow. 
The lift up effect is seen in the transient growth in \eqref{eq:liftup_vs}, the enhanced dissipation in the exponentials $e^{-c\nu t^3}$, the inviscid damping in the $\jap{t}^{-2}$ decay in \eqref{ineq:U2LinearID_vs} which is uniform in $\nu$, and the vortex stretching in the lack of inviscid damping in \eqref{ineq:U1LinearID_vs} and \eqref{ineq:U3LinearID_vs} (which is sharp).     

\begin{proposition} \label{prop:linear} 
Consider the linearized Navier-Stokes equations 
\begin{subequations}\label{def:3DNSE_Linear}
\begin{align} 
\partial_t u + y \partial_x u  & = \begin{pmatrix} - u^2 \\ 0 \\ 0 \end{pmatrix} - \grad p^L + \nu \Delta u \\ 
\Delta p^L & = -2\partial_x u^2 \\ 
\grad \cdot u & = 0. 
\end{align}
\end{subequations}
Let $u_{in}$ be a divergence free vector field with $u_{in} \in H^7$. Then the solution to the linearized Navier-Stokes equations $u(t)$ with initial data $u_{in}$ satisfies the following for some $c \in (0,1/3)$, 
\begin{subequations}  
\begin{align} 
\norm{u^{2}_{\neq}(t)}_{2} + \norm{u^{2}_{\neq}(t,x+ty,y,z)}_{H^3}  & \lesssim \jap{t}^{-2} e^{-c\nu t^3} \norm{u_{in}^2}_{H^7} \label{ineq:U2LinearID_vs} \\ 
\norm{u^1_{\neq}(t,x+ty,y,z)}_{H^1} & \lesssim e^{-c\nu t^3} \norm{u_{in}}_{H^7} \label{ineq:U1LinearID_vs} \\ 
\norm{u^3_{\neq}(t,x+ty,y,z)}_{H^1} & \lesssim e^{-c\nu t^3}\norm{u_{in}}_{H^7}, \label{ineq:U3LinearID_vs} 
\end{align} 
\end{subequations} 
and the formulas
\begin{subequations} \label{def:NSEstreak}
\begin{align}
u^1_0(t,y,z) & = e^{\nu t \Delta} \left(u^1_{in \; 0} - tu_{in \; 0}^2\right)  \label{eq:liftup_vs} \\ 
u^2_0(t,y,z) & = e^{\nu t \Delta} u^2_{in \; 0} \\
u^3_0(t,y,z) & = e^{\nu t \Delta} u^3_{in \; 0}. 
\end{align}
\end{subequations}
\end{proposition}
Associated with the linear problem is the Laplacian expressed in the coordinates $X = x-ty$: 
\begin{align}
\Delta_L := \partial_{XX} + (\partial_Y - t\partial_X)^2 + \partial_{ZZ}. \label{def:DeltaL}
\end{align} 
The power of $t$ in this operator is responsible both for the inviscid damping of $u^2$ and the enhanced dissipation; see \cite{BGM15I} for more information. 

The next Proposition from \cite{BGM15I} recalls the nature of the streak solutions: 
\begin{proposition}[Streak solutions] \label{prop:Streak} 
Let $\nu \in [0,\infty)$, $u_{in} \in H^{5/2+}$ be divergence free and independent of $x$, that is, $u_{in}(x,y,z) = u_{in}(y,z)$, and denote by $u(t)$ the corresponding unique strong solution to \eqref{def:3DNSE} with initial data $u_{in}$. Then $u(t)$ is global in time and for all $T > 0$, $u(t) \in L^\infty( (0,T);H^{5/2+}(\Real^3))$. 
Moreover, the pair $(u^2(t),u^3(t))$ solves the 2D Navier-Stokes/Euler equations on $(y,z) \in \Real \times \Torus$:
\begin{subequations} \label{def:2DNSEStreak} 
\begin{align} 
\partial_t u^i + (u^2,u^3)\cdot \grad u^i &  = -\partial_i p + \nu \Delta u^i \\ 
\partial_y u^2 + \partial_z u^3 & = 0,  
\end{align} 
\end{subequations} 
and $u^1$ solves the (linear) forced advection-diffusion equation
\begin{align} 
\partial_t u^1 + (u^2,u^3)\cdot \grad u^1 = -u^2 + \nu \Delta u^1. \label{eq:u1streak}
\end{align}  
\end{proposition}
Suppose the streak is initially of size $\epsilon \gg \nu$. 
From \eqref{def:2DNSEStreak}, we see that the dissipation does not completely dominate the streak until $t \gtrsim \nu^{-1}$, before which it could be behaving like fully nonlinear 2D Navier-Stokes.  
Due to the lift-up effect in \eqref{eq:u1streak}, in general $u^1(t)$ is growing like $\epsilon \jap{t}$ until times $t \gtrsim \epsilon^{-1}$, at which point the streak will be on the same order as the Couette flow itself.
As discussed above, it is expected that sufficiently large streaks should suffer a secondary instability and break down into more complicated $x$-dependent flows (see e.g. \cite{ReddySchmidEtAl98,Chapman02,SchmidHenningson2001} and the references therein). 

\subsection{Statement of main results}
As in \cite{BGM15I}, our theorem requires the use of Gevrey regularity class \cite{Gevrey18}, defined on the Fourier side for $\lambda > 0$ and $s \in (0,1]$ as: 
\begin{align} 
\norm{f}_{\G^{\lambda;s}}^2 = \sum_{k,l}\int \abs{\widehat{f_k}(\eta,l)}^2e^{2\lambda\abs{k,\eta,l}^s} d\eta.  
\end{align}
For $s = 1$ the class coincides with real analytic, however, for $s < 1$ it is less restrictive, for example, compactly supported functions can still be Gevrey class with $s < 1$.
As discussed in \cite{BGM15I}, this regularity class arises in nearly all mathematically rigorous studies involving inviscid damping \cite{BM13,BMV14,BGM15I} or Landau damping \cite{CagliotiMaffei98,HwangVelazquez09,MouhotVillani11,BMM13,Young14} in nonlinear PDE. 
In these previous works, the Gevrey regularity arises naturally when studying echo resonances, and like \cite{BGM15I}, it arises here as well when controlling related weakly nonlinear resonances.    
 
\begin{theorem}[Above threshold dynamics] \label{thm:SRS}
For all $s \in (1/2,1)$, all $\lambda_0 > \lambda^\prime > 0$, all integers $\alpha \geq 10$ and all $\delta>0$,
there exists a constant $c_{00} = c_{00}(s,\lambda_0,\lambda^\prime,\alpha,\delta)$, a constant $K_0 = K_0(s,\lambda_0,\lambda^\prime)$, and a constant $\nu_0 = \nu_0(s,\lambda_0,\lambda^\prime,\alpha,\delta)$ such that for all $\delta_1 >0 $ sufficiently small relative to $\delta$, all $\nu \leq \nu_0$, $c_{0} \leq c_{00}$, and $\epsilon < \nu^{2/3+\delta}$, if $u_{in} \in L^2$ is a divergence-free vector field that can be written $u_{in} = u_S + u_R$ (both also divergence-free) with
\begin{align} 
\norm{u_S}_{\mathcal{G}^{\lambda_0;s}} + e^{K_0\nu^{-\frac{3s}{2(1-s)}}}\norm{u_R}_{H^{3}} & \leq \epsilon, \label{ineq:QuantGev2}
\end{align} 
then the unique, classical solution to \eqref{def:3DNSE} with initial data $u_{in}$ exists at least until time $T_F = c_0 \epsilon^{-1}$ 
and the following estimates hold with all implicit constants independent of $\nu$, $\epsilon$, $c_0$ and $t$:
\begin{itemize} 
\item[(i)] Transient growth of the streak for $t < T_F$: 
\begin{align}
\norm{u^1_0(t) -  e^{\nu t\Delta}\left(u_{in \; 0}^1 - tu_{in \; 0}^2\right) }_{\G^{\lambda^\prime;s}} & \lesssim c_{0}^2 \label{ineq:u01grwth} \\ 
\norm{u^2_0(t) -  e^{\nu t\Delta} u_{in \; 0}^2}_{\G^{\lambda^\prime;s}} + \norm{u^3_0(t) -  e^{\nu t\Delta} u_{in \; 0}^3}_{\G^{\lambda^\prime;s}} & \lesssim c_{0} \epsilon; \label{ineq:u023Duhamel}
\end{align}
\item[(ii)] uniform control of the background streak for $t < T_F$:
\begin{subequations}  
\begin{align} 
\norm{u^1_0(t)}_{\G^{\lambda^\prime;s}} & \lesssim \epsilon \jap{t} \\ 
\norm{u^2_0(t)}_{\G^{\lambda^\prime;s}} + \norm{u^3_0(t)}_{\G^{\lambda^\prime;s}} & \lesssim \epsilon; 
\end{align}
\end{subequations}  
\item[(iii)] the rapid convergence to a streak by the mixing-enhanced dissipation and inviscid damping of $x$-dependent modes:
\begin{subequations} 
\begin{align} 
\norm{u_{\neq}^{1}(t,x + ty + t\psi(t,y,z),y,z)}_{\G^{\lambda^\prime;s}} & \lesssim \frac{\epsilon t^{\delta_1}}{\jap{\nu t^3}^\alpha} \\ 
\norm{u_{\neq}^{3}(t,x + ty + t\psi(t,y,z),y,z)}_{\G^{\lambda^\prime;s}} & \lesssim \frac{\epsilon}{\jap{\nu t^3}^\alpha} \\ 
\norm{u^2_{\neq}(t,x + ty + t\psi(t,y,z),y,z)}_{\G^{\lambda^\prime;s}} & \lesssim \frac{\epsilon}{\jap{t} \jap{\nu t^3}^\alpha}, \label{ineq:u2damping}
\end{align}  
\end{subequations}
where $\psi(t,y,z)$ is an $O(\epsilon t)$ correction to the mixing which depends on the disturbance 
(defined below to satisfy the PDE \eqref{def:psi2}) and satisfies the estimate:   
\begin{align} 
\norm{\psi(t) - u_0^1(t)}_{\G^{\lambda^\prime;s}} \lesssim \epsilon \jap{t}^{-1}.  \label{ineq:psiest}
\end{align}
\end{itemize}
\end{theorem}

\begin{remark} \label{ineq:nulesseps} 
Without loss of generality we will assume for the remainder of the paper that $\nu \lesssim \epsilon$ as otherwise, Theorem \ref{thm:SRS} is covered by our previous work \cite{BGM15I}. 
\end{remark} 

\begin{remark}
If $u_{in \; 0}^2$ is such that $\norm{u_{in \; 0}^2}_{\G^{\lambda^\prime;s}} \geq \frac{1}{4}\epsilon \geq \frac{1}{16}\nu^{2/3+\delta}$ then \eqref{ineq:u01grwth} shows that for $c_{0}$ small (but independent of $\epsilon$ and $\nu$) and $\epsilon$ small, the streak $u_0^1(t)$ reaches the maximal amplitude of $\norm{u_0^1(t_{m})}_2 \gtrsim c_{0}$ at times $t_m \sim T_F = c_0\epsilon^{-1}$. Hence, the solution has grown from $O(\epsilon)$ to $O(c_0)$ over this time interval. 
Moreover, this time-scale is far shorter than the $\nu^{-1}$ time-scale over which $u_0$ will decay by viscous dissipation (at least the low frequencies) and so 
in general the solution will become fully nonlinear for $t \gtrsim T_F$.      
\end{remark}

\begin{remark} 
Notice that linear theory in Proposition \ref{prop:linear} suggests the $O(t^{-2})$ inviscid damping of $u^2$, whereas we only have $t^{-1}$ in \eqref{ineq:u2damping}. 
This discrepancy arises from a 3D nonlinear pressure effect and is explained in \S\ref{sec:Toy} (this discrepancy did not occur in \cite{BGM15I}). 
\end{remark} 

\begin{remark}
Note that the solutions in Theorem \ref{thm:SRS} are not only large solutions to 3D NSE, but also in general they are very far from equilibrium (relative to $\nu$). 
Using naive methods, one would only be able to prove existence until $T_F \sim \log\epsilon^{-1}$ or perhaps some polynomial such as $T_F \sim \epsilon^{-\beta}$ for $\beta \ll 1$ since the Couette flow is rapidly driving large gradients in the solution as well as amplifying the solution via the lift-up effect.  
It is the inviscid damping and enhanced dissipation, together with the 
precise structure of the nonlinearity, which allow us to prove existence all the way until $T_F \sim \epsilon^{-1}$ for these large, far-from-equilibrium, solutions. 
\end{remark} 

\begin{remark} 
As in \cite{BGM15I}, the solutions described in Theorem \ref{thm:SRS} can exhibit a roughly linear-in-time transfer of kinetic energy to high frequencies where it is ultimately dissipated at $\tau_{ED} \sim \nu^{-1/3}$. 
\end{remark}

\begin{remark} \label{ref:lowreg}
In experiments and computer simulations, ``lift-up effect $\Rightarrow$ streak growth $\Rightarrow$ streak breakdown'' is commonly observed, however there are a number of pathways to transition that have also been observed (see \cite{SchmidHenningson2001} and the references therein).  
 Further, it has been observed that the transition threshold in general can depend on the kind of perturbation being made (see e.g. \cite{ReddySchmidEtAl98,SchmidHenningson2001,FaisstEckhardt2004,Mullin2011} and the references therein -- in fact, this was even observed by Reynolds \cite{Reynolds83}). 
Theorem \ref{thm:SRS} and \cite{BGM15I} are not in contradiction with experimental observations, but instead suggest that this is partly related to the regularity of the perturbations.
Indeed, authors conducting computer simulations have explicitly related the transition threshold with the regularity of the initial data and determined different answers \cite{ReddySchmidEtAl98}. 
It may also be illuminating to note that while the works \cite{BM13,BMV14} rule out subcritical transition of Couette flow in 2D for sufficiently regular perturbations, the works of \cite{LinZeng11,LiLin11} suggest it is likely that for sufficiently rough disturbances (about $H^{5/2}$) one can observe subcritical transition even in 2D via a roll-up instability (and hence \eqref{def:3DNSE} should, in principle, admit a pathway to transition which is purely 2D at low enough regularities).     
\end{remark} 

\subsection{Notations and conventions} \label{sec:Notation}
We use superscripts to denote vector components and subscripts such as $\partial_i$ to denote derivatives with respect to the components $x,y,z$ (or $X,Y,Z$) with the obvious identification $\partial_1 = \partial_X$, $\partial_2 = \partial_Y$, and $\partial_3 = \partial_Z$. Summation notation is assumed: in a product, repeated vector and differentiation indices are always summed over all possible values.   

See Appendix \ref{apx:Gev} for the Fourier analysis conventions we are taking.
A convention we generally use is to denote the discrete $x$ (or $X$) frequencies as subscripts.   
By convention we always use Greek letters such as $\eta$ and $\xi$ to denote frequencies in the $y$ or $Y$ direction, frequencies in the $x$ or $X$ direction as $k$ or $k^\prime$ etc, and frequencies in the $z$ or $Z$ direction as $l$ or $l^\prime$ etc.
Another convention we use is to denote dyadic integers by $M,N \in 2^{\Integer}$ where  
\begin{align*} 
2^\Integer & = \set{...,2^{-j},...,\frac{1}{4},\frac{1}{2},1,2,...,2^j,...}, \\ 
\end{align*}
This will be useful when defining Littlewood-Paley projections and paraproduct decompositions. See \S\ref{sec:paranote} for more information on the paraproduct decomposition and the associated short-hand notations we employ. 
Given a function $m \in L^\infty_{loc}$, we define the Fourier multiplier $m(\grad) f$ by 
\begin{align*} 
(\widehat{m(\grad)f})_k(\eta) =  m( (ik,i\eta,il) ) \widehat{f_k}(\eta,l). 
\end{align*}   
We use the notation $f \lesssim g$ when there exists a constant $C > 0$ independent of the parameters of interest 
such that $f \leq Cg$ (we analogously define $f \gtrsim g$). 
Similarly, we use the notation $f \approx g$ when there exists $C > 0$ such that $C^{-1}g \leq f \leq Cg$. 
We sometimes use the notation $f \lesssim_{\alpha} g$ if we want to emphasize that the implicit constant depends on some parameter $\alpha$.
We also employ the shorthand $t^{\alpha+}$ when we mean that there is some small parameter $\gamma >0$ such that $t^{\alpha+\gamma}$ and that we can choose $\gamma$ as small as we want at the price of a constant (e.g. $\norm{f}_{L^\infty} \lesssim \norm{f}_{H^{3/2+}}$). 
We will denote the $\ell^1$ vector norm $\abs{k,\eta,l} = \abs{k} + \abs{\eta} + \abs{l}$, which by convention is the norm taken in our work. 
Similarly, given a scalar or vector in $\Real^n$ we denote
\begin{align*} 
\jap{v} = \left( 1 + \abs{v}^2 \right)^{1/2}. 
\end{align*} 
We denote the standard $L^p$ norms by $\norm{f}_{p}$ and Sobolev norms $\norm{f}_{H^\sigma} := \norm{\jap{\grad}^\sigma f}_2$.  
We make common use of the Gevrey-$\frac{1}{s}$ norm with Sobolev correction defined by 
\begin{align*} 
\norm{f}_{\G^{\lambda,\sigma;s}}^2 = \sum_{k,l}\int \abs{\widehat{f_k}(\eta,l)}^2 e^{2\lambda\abs{k,\eta,l}^s}\jap{k,\eta,l}^{2\sigma} d\eta. 
\end{align*} 
Since in most of the paper we are taking $s$ as a fixed constant, it is normally omitted. Also, if 
$\sigma =0$, it is omitted. 
We refer to this norm as the $\mathcal{G}^{\lambda,\sigma;s}$ norm and occasionally refer to the space of functions
\begin{align*} 
\mathcal{G}^{\lambda,\sigma;s} = \set{f \in L^2 :\norm{f}_{\G^{\lambda,\sigma;s}}<\infty}. 
\end{align*}
See Appendix \ref{apx:Gev} for a discussion of the basic properties of this norm and some related useful inequalities.

For $\eta \geq 0$, we define $E(\eta)\in \Integer$ to be the integer part.  
We define for $\eta \in \Real$ and $1 \leq \abs{k} \leq E(\sqrt{\abs{\eta}})$ with $\eta k > 0$, 
$t_{k,\eta} = \abs{\frac{\eta}{k}} - \frac{\abs{\eta}}{2\abs{k}(\abs{k}+1)} =  \frac{\abs{\eta}}{\abs{k}+1} +  \frac{\abs{\eta}}{2\abs{k}(\abs{k}+1)}$ and $t_{0,\eta} = 2 \abs{\eta}$ and 
the critical intervals  
\begin{align*} 
I_{k,\eta} = \left\{
\begin{array}{lr}
[t_{\abs{k},\eta},t_{\abs{k}-1,\eta}] & \textup{ if } \eta k \geq 0 \textup{ and } 1 \leq \abs{k} \leq E(\sqrt{\abs{\eta}}), \\ 
\emptyset & otherwise.
\end{array} 
\right. 
\end{align*} 
For minor technical reasons, we define a slightly restricted subset as the \emph{resonant intervals}
\begin{align*} 
\mathbf I_{k,\eta} = \left\{
\begin{array}{lr} 
I_{k,\eta} & 2\sqrt{\abs{\eta}} \leq t_{k,\eta}, \\ 
\emptyset & otherwise.
\end{array} 
\right. 
\end{align*} 
Note this is the same as putting a slightly more stringent requirement on $k$: $k \leq \frac{1}{2}\sqrt{\abs{\eta}}$.

\section{Outline of the proof} \label{sec:proof}
In this section we give an outline of the main steps of the proof of Theorem \ref{thm:SRS} and set up the main energy estimates, focusing on exposition, intuition, and organization. 
We will try to give specific emphasis to what is new relative to \cite{BGM15I}, and discuss fewer details on issues that are common to both works for the sake of brevity.   
After \S\ref{sec:proof}, the remainder of the paper is dedicated to the proof of the major energy estimates required and the analysis of the various norms and Fourier analysis tools being employed. 

\subsection{Summary and weakly nonlinear heuristics}
\subsubsection{New dependent variables}
As in \cite{BGM15I}, we find it natural to define the full set of auxiliary unknowns $q^i = \Delta u^i$ for $i = 1,2,3$. A computation shows that $(q^i)$ solves
\begin{equation} \label{def:qi}
\left\{ \begin{array}{l} 
\partial_t q^1 + y \partial_x q^1 +  2\partial_{xy} u^1 + u \cdot \grad q^1  = -q^2 + 2\partial_{xx} u^2 - q^j \partial_j u^1 + \partial_x\left(\partial_i u^j \partial_j u^i\right) - 2\partial_{i} u^j \partial_{ij}u^1 + \nu \Delta q^1  \\
\partial_t q^2 + y \partial_x q^2 + u \cdot \grad q^2  = -q^j \partial_j u^2 + \partial_y\left(\partial_i u^j \partial_j u^i\right) - 2\partial_{i} u^j \partial_{ij}u^2 + \nu \Delta q^2  \\
\partial_t q^3 + y \partial_x q^3 +  2\partial_{xy} u^3 + u \cdot \grad q^3  =  2\partial_{zx} u^2 -q^j \partial_j u^3 + \partial_z\left(\partial_i u^j \partial_j u^i\right) - 2\partial_{i} u^j \partial_{ij}u^3 + \nu \Delta q^3 .
\end{array}\right.
\end{equation}
Note that the linear terms have disappeared in the PDE for $q^2$ but not $q^1$ and $q^3$.  

\subsubsection{New independent variables} \label{sec:indepC} 
As in \cite{BGM15I}, the need for a change of independent variables can be understood by considering the convection term $y\partial_x q^i + u \cdot \nabla q^i$ which appears in \eqref{def:qi} above. 
Due to the mixing, any classical energy estimates on $q$ in (say) Sobolev spaces will rapidly grow. 
Via the lift-up effect, $u_0^1$ will be very large, whereas even the other contributions of the streak, $u_0^{2,3}$, will not be decaying and cannot be balanced by the dissipation as they are far larger than $\nu$. More specifically, the \emph{growth of gradients} caused by mixing due to the streak cannot be balanced. 
In \cite{BGM15I}, the coordinate system was modified to account for the mixing action of $u_0^1$ (and $u_0^2$ as a by-product); here we will go further and also account for $u_0^3$, effectively treating the entire streak in a sort of Lagrangian fashion so that norm growth due to these velocities is not seen in our coordinate system.  

A full study of the coordinate transformation is done in \S\ref{sec:coordinates} below, but let us just make a quick summary here. 
We start with the ansatz 
\begin{align*}
\left\{ \begin{array}{l} X  = x - ty - t \psi(t,y,z) \\ Y  = y + \psi(t,y,z) \\ Z  = z + \phi(t,y,z), \end{array} \right. 
\end{align*} 
The shift $\psi$ is chosen as in \cite{BGM15I}, however $\phi$ is chosen to eliminate the contributions of $u_0^3$ from the transport term. 
Indeed, consider the simple convection diffusion equation on a passive scalar $f(t,x,y,z)$
$$
\partial_t f + y \partial_x f + u \cdot \nabla f = \nu \Delta f.
$$
Denoting $F(t,X,Y,Z) = f(t,x,y,z)$ and $U(t,X,Y,Z) = u(t,x,y,z)$, and $\Delta_t$ and $\nabla^t$ for the expressions for $\Delta$ and $\nabla$ in the new coordinates, this simple equation becomes
\begin{align} 
\partial_t F + \begin{pmatrix}u^1 - t (1+\partial_y\psi) u^2 - t \partial_z\psi u^3 - \frac{d}{dt}(t\psi) + \nu t \Delta \psi \\ (1+\partial_y\psi) u^2 + \partial_z \psi u^3 + \partial_t \psi - \nu\Delta \psi  \\ (1 + \partial_z \phi)u^3 + \partial_y \phi u^2 + \partial_t \phi - \nu \Delta \phi \end{pmatrix} \cdot \grad_{X,Y,Z} F = \nu \tilde{\Delta_t} F,  \label{ineq:transf}
\end{align}
where $\tilde{\Delta_t}$ is a variant of $\Delta_t$ without lower order terms; it is given below in \eqref{def:tildeDel1}. 
To eliminate the zero frequency contribution of the first component of the velocity field, as in \cite{BGM15I}, we will choose $u^1_0 - t (1+\partial_y\psi) u^2_0 - t \partial_z\psi u^3_0 - \frac{d}{dt}(t\psi) + \nu t\Delta \psi = 0$. 
To eliminate the zero frequency contribution of the third component, we further choose $(1 + \partial_z \phi)u_0^3 + \partial_y \phi u^2_0 + \partial_t \phi = \nu \Delta \phi$. As in \cite{BGM15I}, we now recast the equations on $\psi,\phi$ in terms of $C^1(t,Y,Z)=\psi(t,y,z)$, $C^2(t,Y,Z)=\phi(t,y,z)$ and the auxiliary unknown $g = \frac{1}{t}(U_0^1 - C)$. 
After cancellations are carefully accounted for we have 
\begin{equation} \label{def:Cgintro}
\left\{
\begin{array}{l}
\partial_t C^1 + \tilde U_0 \cdot \grad_{Y,Z} C^1 = g - U_0^2 + \nu \tilde{\Delta_t} C^1, \\
\partial_t C^2 + \tilde U_0 \cdot \grad_{Y,Z} C^2 = - U_0^3 + \nu \tilde{\Delta_t} C^2, \\
\partial_t g  + \tilde U_0 \cdot \grad_{Y,Z}g = -\frac{2}{t}g -\frac{1}{t} \left(U_{\neq} \cdot \grad^t U^1_{\neq}\right)_0 + \nu \tilde{\Delta_t}  g,
\end{array}
\right.
\end{equation}
and 
\begin{align} \label{def:Qiintro}
\left\{
\begin{array}{l}
Q^1_t + \tilde U \cdot \grad_{X,Y,Z} Q^1 = -Q^2 - 2\partial_{XY}^t U^1 + 2\partial_{XX} U^2 - Q^j \partial_j^t U^1 - 2\partial_i^t U^j \partial_{ij}^t U^1 + \partial_X(\partial_i^t U^j \partial_j^t U^i) + \nu \tilde{\Delta_t} Q^1 \\ 
Q^2_t + \tilde U \cdot \grad_{X,Y,Z} Q^2 = -Q^j \partial_j^t U^2 - 2\partial_i^t U^j \partial_{ij}^t U^2 + \partial_Y^t(\partial_i^t U^j \partial_j^t U^i) + \nu \tilde{\Delta_t} Q^2 \\ 
Q^3_t + \tilde U \cdot \grad_{X,Y,Z} Q^3 = -2\partial_{XY}^t U^3 + 2\partial_{XZ}^t U^2 - Q^j \partial_j^t U^3 - 2\partial_i^t U^j \partial_{ij}^t U^3 + \partial_Z^t(\partial_i^t U^j \partial_j^t U^i) + \nu \tilde{\Delta_t} Q^3, 
\end{array}
\right.
\end{align} 
where $\partial_i^t$ denote derivatives including the Jacobian factors $\partial_z \psi,\partial_y \psi, \partial_y\phi,\partial_z\phi$ (see \S\ref{sec:coordinates} below) and 
$$\tilde U = \begin{pmatrix} U_{\neq}^1 - t(1 + \partial_y\psi ) U^2_{\neq} - t \partial_z\psi U^3_{\neq}  \\ (1 + \partial_y\psi)U^2_{\neq} + \partial_z\psi U^3_{\neq}+ g \\ (1 + \partial_z\phi)U^3_{\neq} + \partial_y\phi U_{\neq}^2 \end{pmatrix}.$$ 
Notice that this transformation almost completely eliminates the zero frequency contribution of $\tilde U_0$, so we are treating the advection by the evolving streak $u_0^1(t,y,z),u_0^2(t,y,z),u_0^3(t,y,z)$ in a nearly Lagrangian way (as in \cite{BGM15I}, $g$ is rapidly decaying independently of $\nu$). 

\subsection{Choice of the norms} 
The highest norms we use are of the general type $\norm{ A^i(t,\grad) Q^i(t)}_2$, where the $A^i$ are specially designed Fourier multipliers.  
See \eqref{def:A} below for the definitions of $A^i$. 
For $i = 1,2$ the norms are similar to \cite{BGM15I}, however, here they need to be adjusted at high frequencies in $Z$. 
For $i = 3$ the difference is more pronounced as the $w$ multiplier is replaced with a specially adjusted $w^3$.  
Recall that these factors are estimates on the ``worst-possible'' growth of high frequencies due to weakly nonlinear effects. 
Roughly speaking, here they are taken to satisfy the following for $\abs{k}^2 \lesssim \abs{\eta}$ (and hence $\sqrt{\abs{\eta}} \lesssim t \lesssim \abs{\eta}$), 
\begin{subequations} \label{def:approxw}  
\begin{align}
\frac{\partial_t {w(t,\eta)}}{w(t,\eta)}&  \sim \frac{1}{1 + |t-\frac{\eta}{k}|}, \qquad \mbox{when $\left| t-\frac{\eta}{k} \right| \lesssim \frac{\eta}{k^2}$} \quad \mbox{and} \quad w(1,\eta) = 1 \\ 
w^3_k(t,\eta) & \sim w(t,\eta), \qquad \mbox{when $\left| t-\frac{\eta}{k} \right| \lesssim \frac{\eta}{k^2}$} \\ 
w^3_{k^\prime}(t,\eta) &  \sim \frac{t}{\abs{k} + \abs{\eta - kt}} w(t,\eta), \qquad \mbox{when $\left| t-\frac{\eta}{k} \right| \lesssim \frac{\eta}{k^2}$} \quad \mbox{and} \quad k \neq k^\prime; 
\end{align}
\end{subequations}
see Appendix \ref{sec:def_nrm} for the full definition and \S\ref{sec:Toy} for the heuristic derivation.  
We see that $w^3$ unbalances the regularity in a way that enforces more control over frequencies near the critical times than away from the critical times. 
This is closely matched by the loss of ellipticity in $\Delta_L$ and allows to trade ellipticity and regularity back and forth in a specific way. 

Finally, as pointed out in \cite{BGM15I}, one can read off the requirement $s > 1/2$ from \eqref{def:approxw}. Indeed, integration over each critical time gives for some $C >0$,  
\begin{align}
\frac{w(2\eta,\eta)}{w(\sqrt{\eta},\eta)} \approx \left(\frac{\eta^{\sqrt{\eta}}}{(\sqrt{\eta}!)^2}\right)^C, \label{ineq:wloss}
\end{align} 
which predicts a growth like $O(e^{2C\sqrt{\eta}})$ up to a polynomial correction by Stirling's formula. 

\subsubsection{Weakly nonlinear heuristics} \label{sec:NonlinHeuristics}
First, let us point out another heuristic for deriving the requirement $\epsilon \lesssim \nu^{2/3}$. Many nonlinear terms in the proof are naturally estimated in the following general manner: 
\begin{align}
NL & \lesssim \frac{\epsilon}{\jap{\nu t^3}^\alpha}\norm{\sqrt{-\Delta_L} A^i Q^i}_2\norm{A^j Q^j}_2 \lesssim \nu\norm{\sqrt{-\Delta_L} A^i Q^i}_2^2 + \frac{\epsilon^2}{\nu \jap{\nu t^3}^{2\alpha}}\norm{A^j Q^j}_2^2 \label{ineq:32heurs} 
\end{align}
where recall from \S\ref{sec:LinStreak} that $\Delta_L = \partial_{XX} + (\partial_Y - t\partial_X)^2 + \partial_{ZZ}$, the leading order dissipation that comes from the linearized problem. The $\jap{\nu t^3}^{-\alpha}$ comes from a `low-frequency' factor that was estimated via the enhanced dissipation. 
Since $\int_0^\infty \frac{1}{\jap{\nu t^3}^{\alpha}} dt \approx \nu^{-1/3}$, it is apparent that $\nu \sim \epsilon^{3/2}$ is the smallest choice of $\nu$ such that \eqref{ineq:32heurs} can be integrated uniformly in $\nu \rightarrow 0$. 

Now, let us quickly recall some terminology from \cite{BGM15I} and some discussion on the weakly nonlinear effects. 
The behavior in Theorem \ref{thm:SRS} comes in essentially two phases. During early times $t \lesssim \tau_{NL} \sim \epsilon^{-1/2}$, the solution has strong 3D effects 
and the dissipation cannot control the leading order nonlinear terms. On this time scale, the regularity unbalancing in $w^3$ and insight from the toy model of \S\ref{sec:Toy} is crucial. After times $t \gtrsim \tau_{ED} \sim \nu^{-1/3}$ the enhanced dissipation begins to dominate and the solution converges to a streak; the main growth from then on is due to the lift-up effect. The assumption of $\epsilon \leq \nu^{2/3+\delta}$ is what ensures the two regimes overlap since then $\tau_{NL} \gtrsim \tau_{ED}$; moreover since $\delta > 0$, by picking $\nu$ small we can make sure that the overlap regime is large (that is, we can ensure $\tau_{NL} \gg \tau_{ED}$ so the dissipation dominates comfortably before the nonlinear time-scale). 
  
As in \cite{BGM15I}, we classify the nonlinear terms by the zero, or nonzero, $x$ frequency of the interacting functions: denote for instance $0 \cdot \neq \,\to\, \neq$ for the interaction of a zero mode (in $x$) and a non-zero mode (in $x$) giving a non-zero mode (in $x$), and similarly, with obvious notations, $0 \cdot 0 \to 0$, $\neq \cdot \neq \, \to \,\neq$, and $\neq \cdot \neq \,\to 0$.

\begin{itemize} 
\item[\textbf{(2.5NS)}]  ($0 \cdot 0 \to 0$)  For  \emph{2.5D Navier-Stokes}, this corresponds to self-interactions of the streak. We will see that there are new complexities to these terms here: due to the regularity imbalancing in $w^3$, the regularity of $u_0^3$ and $u_0^2$ are not the same and terms that were straightforward in \cite{BGM15I} are not so here. 
\item[\textbf{(SI)}] ($0 \cdot \neq \,\to\, \neq$) For \emph{secondary instability}, this effect is the transfer of momentum from the large $u_0^1$ mode to other modes. 
Actually, even more here than in \cite{BGM15I}, $u_0^2$ and $u_0^3$ will matter; especially the latter due to the regularity unbalances in $w^3$.  
These interactions are those that would arise when linearizing an $x$-dependent perturbation of a streak and so are what ultimately give rise to the secondary instabilities observed in larger streaks  (hence the terminology) \cite{ReddySchmidEtAl98,Chapman02}. 

\item[\textbf{(3DE)}] ($\neq \cdot \neq \,\to \,\neq$) For \emph{three dimensional echoes}, these effects are 3D variants of the 2D hydrodynamic echo phenomenon as observed in \cite{YuDriscoll02,YuDriscollONeil}. These are understood as weakly nonlinear interactions of $x$-dependent modes forcing unmixing modes \cite{VMW98,Vanneste02,BM13}. 
We will see in \S\ref{sec:Toy} that these are the primary reason for the regularity imbalances in $w^3$ and hence are the source of most of the additional difficulties in the proof of Theorem \ref{thm:SRS}. 
This involves two non-zero frequencies $k_1$, $k_2$ interacting to force mode $k_1 + k_2$ with $k_{1},k_2,k_1 + k_2 \neq 0$. 

\item[\textbf{(F)}] ($\neq \cdot \neq \,\to 0$) For \emph{nonlinear forcing}, this is the effect of the forcing from $x$-dependent modes back into $x$-independent modes. This involves two non-zero frequencies $k$ and $-k$ interacting to force a zero frequency (and as usual, in general this could involve a variety of the components). Similar to \textbf{(3DE)}, it is $u_0^3$ that is most strongly affected by these terms, and it is these that are responsible for altering the regularity of $u_0^3$ relative to $u_0^2$. 
 \end{itemize} 
As in \cite{BGM15I}, these nonlinear interactions are coupled to one another and can precipitate nonlinear cascades. 
The need to consider possible nonlinear bootstraps both precipitates the Gevrey-$2$ regularity requirement as in \cite{BGM15I} and the regularity imbalances in $u^3$, as we will derive formally in \S\ref{sec:Toy}. 

We will now begin a detailed outline of the proof of Theorem \ref{thm:SRS} and set up the main energy estimates that will comprise the majority of the paper. 

\subsection{Instantaneous regularization and continuation of solutions}
The first step is to see that our initial data becomes small in $\G^{\frac{3\lambda_0}{4} + \frac{\lambda^{\prime}}{4}}$ after a short time. 
We state without proof the appropriate lemma, see \cite{BMV14,BGM15I} for analogous lemmas. 

\begin{lemma}[Local existence and instanteous regularization] \label{lem:Loc} 
Let $u_{in} \in L^2$ satisfy \eqref{ineq:QuantGev2}. 
Then for all $\nu\in (0,1]$, $c_{0}$ sufficiently small, $K_0$ sufficiently large, and all $\lambda_0 > \lambda^\prime > 0$, if $u_{in}$ satisfies \eqref{ineq:QuantGev2}, then there exists a time $t_\star = t_\star(s,K_0,\lambda_0,\lambda^\prime) > 0$ and a unique classical solution to \eqref{def:3DNSE} with initial data $u_{in}$ on $[0,t_\star]$ which is real analytic on $(0,t_\star]$, and satisfies
\begin{align} 
\sup_{t \in [t_\star/2,t_\star]} \norm{u(t)}_{\G^{\bar{\lambda}}} \leq 2\epsilon, 
\end{align} 
where $\bar{\lambda} = \frac{3\lambda_0}{4} + \frac{\lambda^\prime}{4}$. 
\end{lemma}  

Once we have a solution we want to be able to continue it and ensure that it propagates analyticity based on low norm controls. 
This will allow us to rigorously justify our a priori estimates and that these a priori estimates allow us to propagate the solutions. 
See \cite{BGM15I} for more discussion. 
We remark here that analyticity itself is not important, we only need a regularity class which is a few derivatives stronger than the regularities we work in below, so that we may easily justify that the norms applied to the solution take values continuously in time. 

\begin{lemma}[Continuation and propagation of analyticity] \label{lem:Cont} 
Let $T > 0$ be such that the classical solution $u(t)$ to \eqref{def:3DNSE} constructed in Lemma \ref{lem:Loc} exists on $[0,T]$ and is real analytic for $t \in (0,T]$. 
Then there exists a maximal time of existence $T_0$ with $T < T_0 \leq \infty$ such that the solution $u(t)$ remains unique and real analytic on $(0,T_0)$. 
Moreover, if for some $\tau \leq T_0$ and $\sigma > 5/2$ we have $\limsup_{t \nearrow \tau} \norm{u(t)}_{H^{\sigma}} < \infty$, then $\tau < T_0$.  
\end{lemma}

\subsection{$Q^i$ formulation, the coordinate transformation, and some key cancellations}  \label{sec:coordinates}
As in \cite{BGM15I}, we remove the fast mixing action of \emph{both} the Couette flow \emph{and} $u_0^1(t)$. 
However, we go further and essentially treat the entire streak in a Lagrangian way so that we do not see the large gradient  growth due to the zero frequencies in the velocity field.  
In this work we need:
\begin{enumerate}  
\item to control the regularity loss due to transport effects in our special set of of norms until $t \sim \epsilon^{-1}$;
\item to be able to treat the Laplacian in the new coordinates as a perturbation from $\Delta_L$, so that we can take advantage of the inviscid damping and enhanced dissipation effects;
\item to be able to make practical estimates on the behavior of the coordinate system and the coordinate transformation needs to treat the dissipation in a natural way, instead of losing derivatives. 
\end{enumerate}
The latter two are the same as \cite{BGM15I} but the first one is potentially far more difficult since the streak is far larger than $\nu$ and so cannot be balanced by viscous effects. 
The middle requirement suggests the form 
\begin{subequations} \label{def:XYZ}
\begin{align} 
X & = x - ty - t \psi(t,y,z) \\ 
Y & = y + \psi(t,y,z) \\ 
Z & = z + \phi(t,y,z),  
\end{align}
\end{subequations} 
however, unlike \cite{BGM15I}, we will not take $\phi = 0$. 
Provided $\psi$ and $\phi$ is sufficiently small in a suitable sense, one can invert \eqref{def:XYZ} for $x,y,z$ as functions of $X,Y,Z$ (see \S\ref{sec:RegCont} and \cite{BGM15I} for more information).  
In keeping with the notation in \cite{BGM15I} , denote the Jacobian factors (by abuse of notation),
\begin{align*} 
\psi_t(t,Y,Z) & = \partial_t \psi(t,y(t,Y,Z),z(t,Y,Z)) \\ 
\psi_y(t,Y,Z) & = \partial_y \psi(t,y(t,Y,Z),z(t,Y,Z)) \\ 
\psi_z(t,Y,Z) & = \partial_z \psi(t,y(t,Y,Z),z(t,Y,Z)) \\
\phi_t(t,Y,Z) & = \partial_t \phi(t,y(t,Y,Z),z(t,Y,Z)) \\ 
\phi_y(t,Y,Z) & = \partial_y \phi(t,y(t,Y,Z),z(t,Y,Z)) \\ 
\phi_z(t,Y,Z) & = \partial_z \phi(t,y(t,Y,Z),z(t,Y,Z)). 
\end{align*}
In what follows we will usually omit the arguments of $y(t,Y,Z)$ and $z(t,Y,Z)$ and use a more informal notation, such as $\psi_t(t,Y,Z) = \partial_t \psi(t,y,z)$.  

Define the following notation for the $(x,y,z)$ derivatives in the new coordinate systems
\begin{subequations} 
\begin{align} 
\partial_X^t & = \partial_X \\ 
\partial_Y^t & = (1 + \psi_y)(\partial_Y - t\partial_X) + \phi_y \partial_Z \\ 
\partial_Z^t & = (1 + \phi_z) \partial_Z  + \psi_z(\partial_Y - t\partial_X) \\ 
\grad^t & = (\partial_X, \partial_Y^t, \partial_Z^t)^{T}.  
\end{align}
\end{subequations}
Note that these necessarily commute. 
Consider the transport of a passive scalar by a perturbation of the Couette flow: 
\begin{align}
\partial_t f + y \partial_x f + u\cdot \grad f = \nu \Delta f. \label{def:ftrans}
\end{align}
Denoting $F(t,X,Y,Z) = f(t,x,y,z)$, the transport equation \eqref{def:ftrans} in the new coordinate system is given by
\begin{align} 
\partial_t F + \begin{pmatrix}u^1 - t (1+\partial_y\psi) u^2 - t \partial_z\psi u^3 - \frac{d}{dt}(t\psi) + t\nu \Delta \psi \\ (1+\partial_y\psi) u^2 + \partial_z \psi u^3 + \partial_t \psi - \nu\Delta \psi  \\ (1+\partial_z \phi)u^3 + \partial_y\phi u^2 + \partial_t \phi - \nu \Delta \phi \end{pmatrix} \cdot \grad_{X,Y,Z} F = \nu \tilde{\Delta_t} F,  \label{def:transf2}
\end{align}
where the upper-case letters are evaluated at $(X,Y,Z)$ and the lower case letters are evaluated at $(x,y,z)$ and we are denoting
\begin{subequations} \label{def:Deltat}
\begin{align} 
\Delta_t F & = \partial_{XX} + \partial_Y^t\left(\partial_Y^t F\right) + \partial_Z^t\left(\partial_Z^t F\right) \\ 
\tilde{\Delta_t}F & = \Delta_t F - \Delta \psi (\partial_Y - t\partial_X)F - \Delta \phi \partial_Z F. \label{def:tildeDel1}
\end{align}
\end{subequations} 
Eliminating the zero frequency of the first component of the  velocity field in \eqref{def:transf2} provides the requirement on $\psi$ (the same as in \cite{BGM15I}), 
\begin{align} 
u_0^1 - t\left( 1 + \partial_y \psi \right)u_0^2 - t \partial_z \psi u_0^3 - \frac{d}{dt}(t\psi) & = -\nu t\Delta \psi. \label{def:psi2} 
\end{align}

In \cite{BGM15I}, $\phi$ was chosen to be zero for simplicity and the transport due $u_0^3$ was absorbed by the dissipation. 
Even with no dissipation at all, in standard regularity classes one could attempt to deal with $u_0^3$ until $t \sim \epsilon^{-1}$ by using the commutator trick employed in e.g. \cite{LevermoreOliver97,KukavicaVicol09}, however, armed with our complicated norms, which in particular, have a non-trivial angular dependence in frequency, this could become hard (see \cite{BM13} for what kind of issues could arise).  
Instead, we will shift our coordinate system along with $u_0^3$ by eliminating the third component of the velocity field in \eqref{def:transf2} via: 
\begin{align} 
(1+\partial_z \phi)u^3_0 + \partial_y\phi u^2_0 + \partial_t \phi = \nu \Delta \phi,  \label{def:phi}
\end{align} 
which, as mentioned above, is effectively a Lagrangian treatment of the background streak. 
Below we denote
\begin{align*} 
C^1(t,Y,Z) & = \psi(t,y,z)\\
C^2(t,Y,Z) & = \phi(t,y,z) \\ 
C(t,Y,Z) & = (C^1(t,Y,Z), C^2(t,Y,Z))^{T}. 
\end{align*} 
From the chain rule we derive:
\begin{subequations} \label{eq:psiyzt}
\begin{align} 
\psi_y & = \partial_Y^tC^1  = \left(1 + \psi_y\right) \partial_Y C^1 + \phi_y \partial_Z C^1 \\  
\psi_z & = \partial_Z^t C^1 = (1 + \phi_z) \partial_Z C^1 + \psi_z \partial_YC^1 \\ 
\phi_y & = \partial_Y^tC^2  = \left(1 + \psi_y\right) \partial_Y C^2 + \phi_y \partial_Z C^2 \\  
\phi_z & = \partial_Z^t C^2 = (1 + \phi_z) \partial_Z C^2 + \psi_z \partial_YC^2 \\ 
\psi_t & = \partial_t C^1 + \psi_t \partial_YC^1 + \phi_t \partial_ZC^1 \label{ineq:psit} \\ 
\phi_t & = \partial_t C^2 + \psi_t \partial_YC^2 + \phi_t \partial_ZC^2. \label{ineq:phit} 
\end{align} 
\end{subequations}
Analogous to \cite{BGM15I}, we will get estimates on $C^i$ and use them to deduce estimates on $\psi$ and $\phi$. 
This necessitates solving \eqref{eq:psiyzt} for $\psi_y,\psi_z,\phi_y,\phi_z$ -- note that these form a $4 \times 4$ linear system: 
\begin{align*}
\begin{pmatrix} 
1-\partial_YC^1 & 0 & -\partial_Z C^1 & 0 \\ 
0 & 1-\partial_YC^1 & 0 & -\partial_Z C^1  \\ 
-\partial_YC^2 & 0 & 1-\partial_ZC^2 & 0 \\ 
0 & -\partial_Y C^2 & 0 & 1-\partial_ZC^2
\end{pmatrix} 
\begin{pmatrix} 
\psi_y \\ 
\psi_z \\ 
\phi_y \\
\phi_z
\end{pmatrix}  
= 
\begin{pmatrix}
\partial_Y C^1 \\ 
\partial_Z C^1 \\ 
\partial_Y C^2 \\ 
\partial_Z C^2
\end{pmatrix}.   
\end{align*} 
For $\grad C^i$ sufficiently small we can solve the linear system and derive
\begin{subequations} \label{def:psi2sqrBrack} 
\begin{align} 
\phi_z & = \frac{\left(\partial_Z C^2 + \frac{\partial_Y C^2 \partial_Z C^1}{1 - \partial_Y C^1} \right)}{1 - \left(\partial_Z C^2 + \frac{\partial_Y C^2 \partial_Z C^1}{1 - \partial_Y C^1} \right) } = \sum_{n=1}^\infty \left(\partial_Z C^2 + \frac{\partial_Y C^2 \partial_Z C^1}{1 - \partial_Y C^1} \right)^n \\ 
\phi_y & = \frac{\partial_Y C^2}{\left(1 - \partial_Y C^1\right)\left(1 - \left(\partial_Z C^2 + \frac{\partial_Y C^2 \partial_Z C^1}{1 - \partial_Y C^1}\right) \right)} \\  
\psi_z & = \frac{(1 + \phi_z) \partial_Z C^1}{1 - \partial_Y C^1} \\  
\psi_y & = \frac{\partial_Y C^1 + \partial_Z^1 C \phi_y}{1 - \partial_Y C^1}; 
\end{align} 
\end{subequations} 
The precise form of \eqref{def:psi2sqrBrack} is not interesting and it is straightforward to recover estimates on the Jacobian factors from estimates on $C^i$ using \eqref{def:psi2sqrBrack} and the appropriate product rules.
Note that $\Delta_t C^1 = \Delta \psi$ and $\Delta_t C^2 = \Delta \phi$, and hence 
\begin{align}
\Delta_t f & = \tilde{\Delta_t}f + \Delta_t C^1 (\partial_Y - t\partial_X)f + \Delta_tC^2 \partial_Zf. \label{eq:tildeDeltaf}
\end{align} 
From the chain rule together with \eqref{def:psi2}, \eqref{def:phi}, and \eqref{eq:tildeDeltaf}, we derive 
\begin{subequations} 
\begin{align} 
\partial_t C^1 + \begin{pmatrix} \left(1 + \psi_y\right)U_0^2 + \psi_zU^3_0 + \psi_t - \nu\Delta_tC^1 \\ (1 + \phi_z)U_0^3 + \phi_y U_0^2 + \phi_t - \nu\Delta_t C^2 \end{pmatrix} \cdot \grad C^1  & = \frac{1}{t}\left(U_0^1 - tU_0^2 - C^1\right) + \nu \tilde{\Delta_t} C^1 \label{def:C} \\ 
\partial_t C^2 + \begin{pmatrix} \left(1 + \psi_y\right)U_0^2 + \psi_zU^3_0 + \psi_t - \nu \Delta_t C^1 \\ (1 + \phi_z)U_0^3 + \phi_y U_0^2 + \phi_t - \nu\Delta_tC^2  \end{pmatrix} \cdot \grad C^2  & =  -U_0^3  + \nu \tilde{\Delta_t} C^2. \label{def:C2}
\end{align} 
\end{subequations}
As in \cite{BGM15I}, we will define another auxiliary unknown $g$, 
\begin{align} 
g = \frac{1}{t}\left(U_0^1 - C^1\right), \label{def:g}
\end{align}
which, as in \cite{BGM15I}, roughly speaking, measures the time oscillations of $U_0^1$ and satisfies
\begin{align}
\partial_t g  + \begin{pmatrix} (1  + \psi_y)U_0^2 + \psi_z U_0^3 + \psi_t - \nu \Delta_t C^1 \\ (1 + \phi_z)U_0^3 + \phi_y U_0^2 + \phi_t - \nu\Delta_tC^2 \end{pmatrix} \cdot \grad_{Y,Z}g = -\frac{2g}{t} -\frac{1}{t} \left(\tilde{U}_{\neq} \cdot \grad U^1_{\neq}\right)_0 + \nu \tilde{\Delta_t} g.  \label{def:gPDE1}
\end{align} 
Next, from \eqref{def:C}, \eqref{def:g}, \eqref{ineq:psit}, and \eqref{ineq:phit}, we derive 
\begin{subequations}
\begin{align} 
\psi_t & = g - U_0^2 - \begin{pmatrix} \left(1 + \psi_y\right)U_0^2 + \psi_z U^3_0  \\ (1 + \phi_z)U_0^3 + \phi_y U_0^2 \end{pmatrix} \cdot \grad C^1 + \nu\Delta_t C^1 \label{def:dtpsi} \\ 
\phi_t & = -U_0^3 - \begin{pmatrix} \left(1 + \psi_y\right)U_0^2 + \psi_z U^3_0  \\ (1 + \phi_z)U_0^3 + \phi_y U_0^2 \end{pmatrix} \cdot \grad C^2 + \nu \Delta_t C^2.   \label{def:dtphi}
\end{align}  
\end{subequations} 
and equivalently, from \eqref{eq:psiyzt},  
\begin{subequations}
\begin{align} 
\psi_t + (1 + \psi_y) U_0^2 + \phi_z U_0^3 & = g + \nu \Delta_t C^1 \\ 
\phi_t + (1 + \phi_z)U_0^3 + \phi_y U_0^2 & = \nu \Delta_t C^2. 
\end{align}
\end{subequations}
Deriving the resulting cancellations as in \cite{BGM15I}, we have that the following velocity field will ultimately govern our equations: 
\begin{align} 
\tilde U = \tilde U_0 + \tilde U_{\neq } = \begin{pmatrix} 0 \\ g \\ 0 \end{pmatrix}  + \begin{pmatrix} U_{\neq}^1 - t(1 + \psi_y ) U^2_{\neq} - t \psi_z U^3_{\neq}  \\ (1 + \psi_y)U^2_{\neq} + \psi_z U^3_{\neq}\\ (1 + \phi_z)U^3_{\neq} + \phi_y U^2_{\neq} \end{pmatrix}. \label{def:tildeU2} 
\end{align} 
We also derive the governing equations
\begin{subequations} \label{def:CReal}
\begin{align} 
\partial_t C^1 +  g\partial_Y C^1  & = g - U_0^2  + \nu \tilde{\Delta_t} C^1  \label{def:C1Real} \\ 
\partial_t C^2 +  g\partial_Y C^2  & = -U_0^3  + \nu \tilde{\Delta_t} C^2,   \label{def:C2Real}  
\end{align} 
\end{subequations} 
and 
\begin{align}
\partial_t g  + g \partial_Y g = -\frac{2g}{t} -\frac{1}{t}\left(\tilde{U}_{\neq} \cdot \grad U^1_{\neq}\right)_0 + \nu \tilde{\Delta_t}  g.   \label{def:gPDE2}
\end{align} 
Further notice that the forcing term from non-zero frequencies can be written as 
\begin{align}
\left(\tilde{U}_{\neq} \cdot \grad U^1_{\neq}\right)_0 = \left(U_{\neq} \cdot \grad^t U^1_{\neq}\right)_0.  
\end{align} 
Furthermore, as in \cite{BGM15I} we have, denoting $Q^i(t,X,Y,Z) = q^i(t,x,y,z)$: 
\begin{align} \label{def:MainSys}
\left\{
\begin{array}{l}
Q^1_t + \tilde U \cdot \grad Q^1 = -Q^2 - 2\partial_{XY}^t U^1 + 2\partial_{XX} U^2 - Q^j \partial_j^t U^1 - 2\partial_i^t U^j \partial_{ij}^t U^1 + \partial_X(\partial_i^t U^j \partial_j^t U^i) + \nu \tilde{\Delta_t} Q^1 \\  
Q^2_t + \tilde U \cdot \grad Q^2 = -Q^j \partial_j^t U^2 - 2\partial_i^t U^j \partial_{ij}^t U^2 + \partial_Y^t(\partial_i^t U^j \partial_j^t U^i) + \nu \tilde{\Delta_t} Q^2 \\ 
Q^3_t + \tilde U \cdot \grad Q^3 = -2\partial_{XY}^t U^3 + 2\partial_{XZ}^t U^2 - Q^j \partial_j^t U^3 - 2\partial_i^t U^j \partial_{ij}^t U^3 + \partial_Z^t(\partial_i^t U^j \partial_j^t U^i) + \nu \tilde{\Delta_t} Q^3,
\end{array}
\right.
\end{align} 
where we use the following to recover the velocity fields: 
\begin{subequations} \label{eq:Ui}
\begin{align} 
U^i & = \Delta_t^{-1} Q^i \label{eq:UiFromQi} \\ 
\partial_i^t U^i & = 0. 
\end{align} 
\end{subequations} 
For the majority of the remainder of the proof, \eqref{def:MainSys}, together with \eqref{def:CReal}, \eqref{def:gPDE2} and \eqref{eq:Ui}, will be the main governing equations. 
The one exception will be in the treatment of the low frequencies of $X$ independent modes, where the use of \eqref{eq:UiFromQi} can be problematic. 
For these we use $X$ averages of the momentum equation. 

As in \cite{BGM15I}, from now on we will use the following vocabulary and shorthands 
\begin{subequations} 
\begin{align} 
\tilde U \cdot \grad Q^{\alpha} & = \textup{``transport nonlinearity''} & \mathcal{T} \\  
-Q^j \partial_j^t U^\alpha - 2\partial_i^t U^j \partial_{ij}^t U^\alpha & = \textup{``nonlinear stretching''} & NLS\\  
\partial_\alpha^t(\partial_i^t U^j \partial_j^t U^i)  & = \textup{``nonlinear pressure''} & NLP\\
-2\partial_{XY}^t U^\alpha & = \textup{``linear stretching''} & LS  \\
2\partial_{X\alpha}^t U^2 & = \textup{``linear pressure''} & LP \\ 
\left(\tilde{\Delta_t} - \Delta_L\right)Q^\alpha & =  \textup{``dissipation error''} & \mathcal{D}_E; 
\end{align} 
\end{subequations}
see \cite{BGM15I} for an explanation of the terminologies. 
As in \cite{BGM15I}, each of the nonlinear terms will be further sub-divided into as many as four pieces in accordance with the different types of nonlinear effects described in \S\ref{sec:NonlinHeuristics}. 
Furthermore, each of the three components of the solution are qualitatively different and measured with different norms, which means certain combinations of $i$ and $j$ need to be treated specially.  

As in \cite{BGM15I}, we need to take advantage of a special structure in the equations which reduces the potential strength of interactions of type \textbf{(F)}.  
By considering the interaction of two non-zero frequencies, $k$ and $-k$, and putting together the contributions from transport, stretching, and nonlinear pressure we get the terms  which we refer to as \emph{forcing}, corresponding to the nonlinear interactions of type \textbf{(F)}, 
\begin{align} 
\mathcal{F}^\alpha & := -\Delta_t \left(U^j_{\neq} \partial_j^t U^\alpha_{\neq}\right)_0 + \partial_\alpha^t \left(\partial_i^t U^j_{\neq} \partial_j^t U^i_{\neq}\right)_0  = -\partial_i^t \partial_i^t \partial_j^t \left(U^j_{\neq} U^\alpha_{\neq}\right)_0 + \partial_{\alpha}^t \partial_j^t \partial_i^t \left(U^i_{\neq} U^j_{\neq}\right)_0,  \label{eq:XavgCanc}
\end{align}
the advantage being that the $X$ averages remove the $-t\partial_X$ from the derivatives. 

\subsection{The toy model and design of the norms} \label{sec:Toy}
Following up on the approach discussed in \cite{BGM15I}, in this section we 
want to perform a weakly nonlinear analysis and determine both $\tau_{NL}$, the characteristic time-scale associated with fully 3D nonlinear effects, and the norms with which we want to measure the solution. 

Denote the Fourier dual variables of $(X,Y,Z)$ as $(k,\eta,l)$. 
As in \cite{BGM15I}, a time which satisfies $kt = \eta$ is called a \emph{critical time} (Orr's original terminology \cite{Orr07}) or \emph{resonant time} (after modern terminology \cite{Craik1971,YuDriscoll02,YuDriscollONeil,SchmidHenningson2001}). 
Notice that these are precisely the points in time/frequency where $\Delta_L$ loses ellipticity in $Y$ (recall \eqref{def:DeltaL}). 
Recall the definition of $\I_{k,\eta}$ from \S\ref{sec:Notation}, which denotes the resonant intervals $t \approx \frac{\eta}{k}$ with $k^2 \lesssim \abs{\eta}$.
This latter restriction is possible due to the uniform ellipticity of $\Delta_L$ with respect to $\partial_X$ which implies the larger the $k$, the weaker the effect of the resonance. 

From \cite{BGM15I}, we recall the toy model for the behavior of \eqref{def:MainSys} near critical times for $Q^2$ and $Q^3$ at frequency $(k,\eta,l)$ and $(k^\prime,\eta,l)$ for $kt \approx \eta$ and $k \neq k^\prime \approx k$:  
\begin{subequations} \label{def:Q2Q3_ToyFinal}
\begin{align} 
\partial_t \widehat{Q^2_{k}}(t,\eta,l) & = \max(\epsilon t, c_0) \frac{k}{k +\abs{\eta-kt}} \widehat{Q^3_k} - \nu\left(k^2 + \abs{\eta-kt}^2\right)\widehat{Q^2_k}  \label{eq:Q2k} \\   
\partial_t \widehat{Q^2_{k^\prime}}(t,\eta,l) & = \max(\epsilon t, c_0) \frac{k^\prime}{\jap{k^\prime, t}} \widehat{Q^3_{k^\prime}} - \nu\left(k^2 + \abs{\eta-kt}^2\right)\widehat{Q^2_{k^\prime}} \label{eq:Q2kp}  \\  
\partial_t \widehat{Q^3_{k^\prime}}(t,\eta,l) & = \frac{\epsilon t^3}{\jap{\nu t^3}^\alpha} \frac{\widehat{Q^{2}_{k}} }{k^2 +\abs{\eta-kt}^2} - \nu\left(k^2 + \abs{\eta-kt}^2\right)\widehat{Q^3_{k^\prime}}   \label{eq:Q3Toykprime}   \\  
\partial_t \widehat{Q^3_{k}}(t,\eta,l) & = \frac{k}{k +\abs{\eta-kt}}\widehat{Q^3_k} + \frac{k}{k +\abs{\eta-kt}} \widehat{Q^2_k} - \nu\left(k^2 + \abs{\eta-kt}^2\right)\widehat{Q^3_k} \\ 
\partial_t\widehat{Q^2_0}(t,\eta,l) & = \epsilon \widehat{Q^3_0} +  \frac{\epsilon t^2}{\jap{\nu t^3}^{\alpha}}\frac{\widehat{Q^2_k}}{k^2 +\abs{\eta-kt}^2} - \nu \eta^2 \widehat{Q^2_0} \\ 
 \partial_t\widehat{Q^3_0}(t,\eta,l) & = \epsilon \widehat{Q^3_0} + \frac{\epsilon t^3}{\jap{\nu t^3}^{\alpha}}\frac{\widehat{Q^2_k}}{k^2 +\abs{\eta-kt}^2} - \nu \eta^2 \widehat{Q^3_0}, \label{eq:Q3Toy0} 
\end{align}
\end{subequations} 
 where all unknowns are evaluated at frequency $(\eta,l)$. 

Let us first use \eqref{def:Q2Q3_ToyFinal} to get an estimate on $\tau_{NL}$. 
If we first consider the case $\nu = 0$, then we can estimate $\tau_{NL}$ from below if we can find an \emph{approximate} super-solution to \eqref{def:Q2Q3_ToyFinal} which will result in a reasonable regularity requirement (say analytic or weaker). 
Even with $\nu = 0$, we can verify that the following is a viable super-solution to \eqref{def:Q2Q3_ToyFinal} over $t \in \I_{k,\eta}$ provided $\epsilon t^2 \lesssim 1$:
\begin{subequations}  \label{def:badnrm}
\begin{align} 
\partial_t w(t,\eta)  \approx \frac{1}{1 + \abs{t - \frac{\eta}{k}}}w(t,\eta) \label{eq:wapprox} \\
Q^2_k \approx Q^2_{k^\prime}  \approx Q_0^2 \approx w(t,\eta) \\  
Q^3_k \approx Q^3_{k^\prime} \approx  Q^3_{0} \approx t w(t,\eta)
\end{align}
\end{subequations} 
Due to the fact that \emph{both} the resonant and non-resonant frequencies experience the same total growth $(\abs{\eta}\abs{k}^{-2})^c$, for some $c$,     
for all $\abs{k} \lesssim \sqrt{\abs{\eta}}$, the loss is multiplicatively amplified through \emph{each} critical time (to see this, take $k^\prime = k-1$ and consider the critical times $\eta/k,\eta/(k-1),\eta/(k-2),\ldots$). 
From this, one sees that this super solution predicts Gevrey-2 regularity loss (see \eqref{ineq:wloss} above or \cite{BGM15I,BM13} for more information). 
Therefore, even with no viscosity, according to the super-solution \eqref{def:badnrm}, a sufficiently regular solution could remain under control until at least $\tau_{NL} \gtrsim \epsilon^{-1/2}$.  
It would be more difficult to derive a good heuristic to estimate $\tau_{NL}$ from above; the toy model \eqref{def:Q2Q3_ToyFinal} is designed to give robust upper bounds on the dynamics, not necessarily to make a good model for any particular realization of the true dynamics, hence even if we explicitly solved \eqref{def:Q2Q3_ToyFinal} exactly, perhaps the toy model itself throws away too much information. 

In order to prove Theorem \ref{thm:SRS}, we will need a more accurate super-solution than \eqref{def:badnrm}. 
Notice further that the super-solution used in \cite{BGM15I} does not work here due to the terms in \eqref{eq:Q3Toykprime} and \eqref{eq:Q3Toy0} with $\epsilon t^3$ present. 
The idea is to take better advantage of the denominators in \eqref{def:Q2Q3_ToyFinal} to recover the extra $t$ in the numerators of these terms. 
Quite precisely, we will trade one power of the denominator for a power of $t$. 
To do this, one must permit the regularities to become unbalanced: \eqref{eq:Q3Toykprime} and \eqref{eq:Q3Toy0} both indicate that $Q^3_{k^\prime}$, for $k^\prime \neq k$ (e.g. \emph{non-critical} or \emph{non-resonant}) should be $t (k + \abs{\eta-kt})^{-1}$ larger than $Q^2_k$.  
Accordingly, we see that for $\epsilon \lesssim \nu^{2/3}$ and $\epsilon t \lesssim 1$, the following is an approximate super-solution for \eqref{def:Q2Q3_ToyFinal} 
over $\I_{k,\eta}$: 
\begin{subequations} \label{def:unstablesuper} 
\begin{align} 
\partial_t w(t,\eta)  \approx \frac{1}{1 + \abs{t - \frac{\eta}{k}}}w(t,\eta) \label{eq:wapprox} \\
w^{3}(t,k,\eta)  = w(t,\eta) \\ 
w^{3}(t,k^\prime,\eta)  = \frac{t}{\abs{k} + \abs{\eta - kt}}w(t,\eta) \\ 
w^{3}(t,0,\eta)  =  \frac{t}{\abs{k} + \abs{\eta - kt}} w(t,\eta) \\  
Q^2_k \approx Q^2_{k^\prime}  \approx Q^3_k \approx w(t,\eta) \\  
Q^3_{k^\prime} \approx  Q^3_{0} \approx w^3(t,k^\prime,\eta) \\ 
Q^1_k \approx Q^1_{k^\prime} \approx tQ^2_k. 
\end{align}
\end{subequations} 
The last line is not deduced directly from \eqref{def:Q2Q3_ToyFinal}, but is deduced (heuristically) in the derivation of \eqref{def:Q2Q3_ToyFinal} via the lift-up effect (see \cite{BGM15I}). 
Notice that when $Q^2_k$ forces $Q^3_{k^\prime}$ and $Q_0^3$ near the critical time, we will gain the factor of $t^{-1}(\abs{k} + \abs{\eta-kt})$, precisely what is needed to exchange the $\epsilon t^3$ in the leading terms in \eqref{eq:Q3Toykprime} and \eqref{eq:Q3Toy0} into $\epsilon t^2$. 
This suffices since $\epsilon t^2 \jap{\nu t^3} \lesssim 1$ when $\epsilon \lesssim \nu^{2/3}$ (another equivalent way of seeing the $2/3$ threshold).   
The regularity loss in \eqref{def:unstablesuper} is peaked near the critical times, and as in \cite{BGM15I}, we will further modify $w$ and $w^3$ to include additional steady, gradual losses of Gevrey-2 regularity over $1 \leq t \leq 2\abs{\eta}$ (see \eqref{def:wextraloss} in Appendix \ref{sec:Defw}). This will further unify the treatment of many estimates, and its potential usefulness is also suggested by the toy model (e.g. the first term in \eqref{eq:Q2kp}).   

As discussed in \cite{BGM15I}, the toy model \eqref{def:Q2Q3_ToyFinal} only provides an estimate on \eqref{def:MainSys} near the critical times. 
For $t \gg \abs{\eta,l}$ it does not apply. As in \cite{BGM15I}, we know from Proposition \ref{prop:linear} that $Q^3_{\neq}$ and $Q^1_{\neq}$ must grow quadratically at these `low' frequencies due to the vortex stretching inherent in the linear problem.  
On the other hand, Proposition \ref{prop:linear} predicts that $u^2$ decays like $\jap{t}^{-2}$, or equivalently, that $Q^2$ is uniformly bounded.  
This behavior was nearly preserved in the below threshold case \cite{BGM15I}, however, it turns out that the nonlinear effects here are strong enough to possibly cause a large growth in $Q^2$. 
The RHS of \eqref{eq:Q2k} originally came from the nonlinear pressure term in the $Q^2$ equation:  
\begin{align}
\partial_t Q^2_{\neq} & = \partial_Y^t\left(\partial_X U^3_{\neq} \partial_Z^t U^1_0\right) + ... \label{def:N}
\end{align}
For times/frequencies with $t \gg \abs{\grad_{Y,Z}}$, we can ignore any issues regarding the critical times and just estimate the size of this term
based on the predictions of Proposition \ref{prop:linear} and we have 
\begin{align*}
\norm{\partial_t Q^2} & \lesssim \frac{\epsilon^2 t^2}{\jap{\nu t^3}^\alpha} + ...
\end{align*} 
Therefore, if $\epsilon \sim \nu^{2/3}$ then we predict that $Q^2$ can be at best bounded by only $\approx \epsilon \jap{t} \jap{\nu t^3}^{-\alpha}$, which suggests a transient growth due to nonlinear effects, in contrast to \cite{BGM15I}. 
Further, this suggests the following inviscid damping/enhanced dissipation estimate: 
\begin{align} 
\norm{U^2_{\neq}} \lesssim \frac{\epsilon}{\jap{t} \jap{\nu t^3}^\alpha}, \label{ineq:u2growth}
\end{align} 
consistent with Theorem \ref{thm:SRS}. 
When considering the ubiquitous $U^j \partial^t_j$ and $\partial_i U^j \partial^t_j$ structure of the nonlinearity in \eqref{def:MainSys}, we see that \eqref{ineq:u2growth} is borderline in a certain sense. 
Indeed, we normally have factors like $U^2 (\partial_Y - t\partial_X)$ and so this will be just enough damping to ensure that 
(regularity issues aside) the $- t\partial_X$ derivatives do not completely dominate the nonlinearity and hence destroy the very special ``non-resonance'' structures available (indeed, this is the main role inviscid damping plays in the proof of Theorem \ref{thm:SRS}). This is also another way to derive the $2/3$ threshold. 

\subsection{Design of the norms based on the toy model} \label{def:designnorm}
The above heuristics suggests that we use a set of norms which is more complicated than the norms in \cite{BGM15I}. 
The high norms will be of the following form, for a time-varying $\lambda(t)$ defined below, $s > 1/2$, $0 < \delta_1 \ll \delta$, and corrector multipliers $w$, $w^3$, and $w_L$ (here $(t,k,\eta,l)$ are now arbitrary):
\begin{subequations} \label{def:A}
\begin{align} 
A^Q_k(t,\eta,l) & = e^{\lambda(t)\abs{k,\eta,l}^s}\jap{k,\eta,l}^\sigma\frac{1}{w_L(t,k,\eta,l)}\left(\frac{e^{\mu \abs{\eta}^{1/2}}}{w(t,\eta)} + e^{\mu\abs{l}^{1/2}}\right) \\ 
A^{1}_k(t,\eta,l) & = \frac{1}{\jap{t}}\left(\mathbf{1}_{k \neq 0} \min\left(1, \frac{\jap{\eta,l}^{1+\delta_1}}{\jap{t}^{1+\delta_1}}\right) + \mathbf{1}_{k = 0} \right) A^Q_k(t,\eta,l) \\ 
A^{2}_k(t,\eta,l) & = \left(\mathbf{1}_{k \neq 0} \min\left(1, \frac{\jap{\eta,l}}{t}\right) + \mathbf{1}_{k = 0} \right) A^Q_k(t,\eta,l) \\ 
A^{3}_k(t,\eta,l) & = \left(\mathbf{1}_{k \neq 0} \min\left(1, \frac{\jap{\eta,l}^2}{t^2}\right) + \mathbf{1}_{k = 0} \right) e^{\lambda(t)\abs{k,\eta,l}^s}\jap{k,\eta,l}^\sigma \nonumber \\ & \quad\quad \times \frac{1}{w_L(t,k,\eta,l)}\left(\frac{e^{\mu \abs{\eta}^{1/2}}}{w^3_k(t,\eta)} + e^{\mu\abs{l}^{1/2}}\right)   \\ 
A(t,\eta,l) & = \jap{\eta,l}^2 A^Q_0(t,\eta,l),
\end{align}
\end{subequations}
where $\mu$, $w$, and $w^3$ are defined precisely in Appendix \ref{sec:def_nrm} and $w_L$ is defined in Appendix \ref{sec:Nmult} ($w$ and $w^3$ are derived approximately in \eqref{def:unstablesuper} above).
As in \cite{BGM15I}, the multiplier $A$ is used to measure $C^i$ and $g$ whereas $A^i$ is used to measure $Q^i$. 
Here $\delta_1$ is chosen sufficiently small depending only on $\delta$. 
We choose the radius of Gevrey-$\frac{1}{s}$ regularity to satisfy 
\begin{align*} 
\dot{\lambda}(t) & =  - \frac{\delta_\lambda}{\jap{t}^{\min(2s,3/2)}} \\ 
\lambda(1) & = \frac{3 \lambda_0}{4} + \frac{\lambda^\prime}{4},   
\end{align*}
where we fix $\delta_\lambda \ll \min(1,\lambda_0 - \lambda^\prime)$ small such that $\lambda(t) > (\lambda_0 + \lambda^\prime)/2$. 

Let us briefly mention some implications of using $w^3$ in \eqref{def:A}. 
Note first of all from \eqref{def:unstablesuper} that $w^3$ is the same as $w$ except near the critical times, however, near the critical times, $w^3_k(t,\eta)$ for \emph{non-resonant} modes is larger, and hence \eqref{def:A} will assign them \emph{less} regularity (see \eqref{def:wNR3} in \S\ref{sec:Defw} for the precise definition). 
This will create a gain in energy estimates when $Q^2$ or $Q^1$ force $Q^3$ and will be a loss when the vice-versa occurs. 
It will also create a similar imbalance in nonlinear interactions between resonant and non-resonant modes in $Q^3$. 
The last detail to notice is that, due to the $+ e^{\mu \abs{l}^{1/2}}$, the effects of $w$ and $w^3$ are only visible in the subset of frequencies such that $\abs{\eta} \gtrsim \abs{l}$. This additional precision was not necessary in \cite{BGM15I}, however, it is necessary here due to problems with regularity imbalances at high frequencies in $Z$ (for example, in \S\ref{sec:Q3_TransNon}). Note it is natural that the resonances should not be relevant for high $Z$ frequencies, due to the uniform ellipticity in $Z$ of $\Delta_t$, however, this detail will make certain aspects of the proof more technical.   
We will need the following definition:
\begin{subequations} \label{def:Atilde}
\begin{align} 
\tilde{A}^Q_k(t,\eta,l) & = e^{\lambda(t)\abs{k,\eta,l}^s}\jap{k,\eta,l}^\sigma\frac{1}{w_L(t,k,\eta,l)} \frac{e^{\mu \abs{\eta}^{1/2}}}{w(t,\eta)}  \\ 
\tilde{A}^{1}_k(t,\eta,l) & = \frac{1}{\jap{t}}\left(\mathbf{1}_{k \neq 0} \min\left(1, \frac{\jap{\eta,l}^{1+\delta_1}}{\jap{t}^{1+\delta_1}}\right) + \mathbf{1}_{k = 0} \right) \tilde{A}^Q_k(t,\eta,l) \\ 
\tilde{A}^{2}_k(t,\eta,l) & = \left(\mathbf{1}_{k \neq 0} \min\left(1, \frac{\jap{\eta,l}}{t}\right) + \mathbf{1}_{k = 0} \right) \tilde{A}^Q_k(t,\eta,l) \\ 
\tilde{A}^{3}_k(t,\eta,l) & = \left(\mathbf{1}_{k \neq 0} \min\left(1, \frac{\jap{\eta,l}^2}{t^2}\right) + \mathbf{1}_{k = 0} \right) \tilde{A}^Q_k(t,\eta,l) \frac{w(t,\eta)}{w^3_k(t,\eta)} \\ 
\tilde{A}(t,\eta,l) & = \jap{\eta,l}^2 \tilde{A}^Q_0(t,\eta,l). 
\end{align}
\end{subequations}
Notice that $\tilde{A}^i \lesssim A^i$ and for $\abs{l} < \frac{1}{5}\abs{\eta}$ there holds $A^i \approx \tilde{A}^i$ (by Lemma \ref{lem:totalGrowthw}). 
Therefore, the difference between them is only visible if $\abs{l}$ is comparable to or larger than $\abs{\eta}$.    

To quantify the enhanced dissipation, we use a scheme similar to that used in \cite{BGM15I}, which itself was an expansion of the scheme of \cite{BMV14}, adjusted now to the larger expected growth of $Q^2$. 
Define $D$ as in \cite{BMV14}, 
\begin{align} 
D(t,\eta) & = \frac{1}{3\alpha}\nu \abs{\eta}^3 + \frac{1}{24 \alpha} \nu\left(t^3 - 8\abs{\eta}^3\right)_+. \label{def:D}
\end{align} 
Note this multiplier satisfies 
\begin{align}
D(t,\eta) \gtrsim \max(\nu \abs{\eta}^3, \nu t^3).\label{ineq:DLowB}
\end{align}
For some $\beta > 3\alpha+7$, we define the enhanced dissipation multipliers: 
\begin{subequations} \label{def:Anu} 
\begin{align}
A^{\nu}_k(t,\eta,l) & = e^{\lambda(t)\abs{k,\eta,l}^s}\jap{k,\eta,l}^{\beta} \jap{D(t,\eta)}^\alpha \frac{1}{w_L(t,k,\eta,l)} \mathbf{1}_{k \neq 0} \\ 
A^{\nu;1}_k(t,\eta,l) & = \jap{t}^{-1}\min\left(1, \frac{\jap{\eta,l}^{1+\delta_1}}{t^{1+\delta_1}}\right) A^{\nu}_k(t,\eta,l) \\  
A^{\nu;2}_k(t,\eta,l) & = \min\left(1, \frac{\jap{\eta,l}}{t}\right) A^{\nu}_k(t,\eta,l) \\
A^{\nu;3}_k(t,\eta,l) & = \min\left(1, \frac{\jap{\eta,l}^2}{t^2}\right) A^{\nu}_k(t,\eta,l). 
\end{align} 
\end{subequations} 
Fix $\gamma > \beta + 3\alpha + 12$ and $\sigma > \gamma + 6$. 
Note that we do not need $w$ or $w^3$ (or the associated regularity imbalances) in \eqref{def:Anu}.
Indeed, the Orr mechanism (and related nonlinear effects) does not play a major role in the enhanced dissipation estimates; they are instead mainly determined by careful estimates on how the vortex stretching manifests in the nonlinearity.  

\subsection{Main energy estimates} \label{sec:energy} 
In this section, we set up the main bootstrap argument to extend our estimates from $O(1)$ in time (from Lemma \ref{lem:Loc}) to $T_F = c_0 \epsilon^{-1}$.
Equipped with the norms defined in \eqref{def:Anu} and \eqref{def:A}, we will be able to propagate estimates via a bootstrap argument for as long as the solution to \eqref{def:3DNSE} exists and remains analytic; by un-doing the coordinate transformation (possible as long as it remains a small deformation in $yz$), this in turn allows us to continue the solution of \eqref{def:3DNSE} via Lemma \ref{lem:Cont}. The analyticity itself is not important, it only needs to be a regularity class slightly stronger than the norms defined in \S\ref{def:designnorm} to ensure they take values continuously in time.  See \S\ref{sec:RegCont}  below for more details on this procedure. 

It turns out that $\partial_t w^3/w^3 \approx \partial_t w/w$ (see Lemma \ref{dtw}) and so this will simplify 
the notation when defining the following high norm ``dissipation energies'': for $i \in \set{2,3}$,  
\begin{subequations} 
\begin{align} 
\mathcal{D}Q^i & = \nu \norm{\sqrt{-\Delta_{L}}A^{i} Q^i}_2^2 + CK_\lambda^i + CK_w^i + CK_{wL}^i \nonumber  \\ 
& = \nu \norm{\sqrt{-\Delta_{L}}A^{i} Q^i}_2^2  d\tau + \dot{\lambda}\norm{\abs{\grad}^{s/2}A^i Q^i}_2^2 + \norm{\sqrt{\frac{\partial_t w}{w}} \tilde{A}^i Q^i}_2^2 + \norm{\sqrt{\frac{\partial_t w_L}{w_L}}A^i Q^i}_2^2 \\ 
\mathcal{D}Q^1_{\neq} & = \nu \norm{\sqrt{-\Delta_{L}}A^{1} Q^1_{\neq} }_2^2 + CK_{\lambda;\neq}^{1} + CK_{w;\neq}^{1} + CK_{wL;\neq}^{1} \nonumber  \\ 
& = \nu \norm{\sqrt{-\Delta_{L}}A^{1} Q^1}_2^2 + \dot{\lambda}\norm{\abs{\grad}^{s/2}A^1 Q^1_{\neq}}_2^2 + \norm{\sqrt{\frac{\partial_t w}{w}} \tilde{A}^1 Q^1_{\neq}}_2^2 + \norm{\sqrt{\frac{\partial_tw_L}{w_L}}A^1 Q^1_{\neq}}_2^2 \\ 
\mathcal{D}g & = \nu \norm{\sqrt{-\Delta_{L}}A g }_2^2 + CK_L^g + CK_{\lambda}^{g} + CK_{w}^{g} \nonumber \\
& = \nu \norm{\sqrt{-\Delta_{L}}A g }_2^2 + \frac{2}{t}\norm{Ag}_2^2 + \dot{\lambda}\norm{\abs{\grad}^{s/2}A g}_2^2 + \norm{\sqrt{\frac{\partial_t w}{w}} \tilde{A} g}_2^2 \\ 
\mathcal{D}C^{i} & = \nu \norm{\sqrt{-\Delta_{L}}A C^i }_2^2 + CK_{\lambda}^{Ci} + CK_{w}^{Ci} \nonumber \\
& = \nu \norm{\sqrt{-\Delta_{L}}A C^i }_2^2 + \dot{\lambda}\norm{\abs{\grad}^{s/2}A C^i}_2^2 + \norm{\sqrt{\frac{\partial_t w}{w}} \tilde{A} C^i}_2^2 \\ 
CK_L^i & = \frac{1}{t}\norm{\mathbf{1}_{t \geq \jap{\grad_{Y,Z}}} A^i Q^i_{\neq}}_2^2 \\ 
\mathcal{D}Q^{\nu;i} & = \nu \norm{\sqrt{-\Delta_{L}}A^{\nu;i} Q^i}_2^2 + CK_\lambda^{\nu;i} + CK_{wL}^{\nu;i} \nonumber  \\ 
& := \nu \norm{\sqrt{-\Delta_{L}}A^{\nu;i} Q^i}_2^2 + \dot{\lambda}\norm{\abs{\grad}^{s/2}A^{\nu;i} Q^i}_2^2 + \norm{\sqrt{\frac{\partial_t w_L}{w_L}}A^{\nu;i} Q^i}_2^2 \\ 
\mathcal{D}Q^{\nu;1} & =  \nu \norm{\sqrt{-\Delta_{L}}A^{\nu;1} Q^{\nu;1} }_2^2 + CK_{\lambda}^{\nu;1} + CK_{wL}^{\nu;1} \nonumber  \\ 
& :=  \nu \norm{\sqrt{-\Delta_{L}}A^{\nu;1} Q^1_{\neq}}_2^2 + \dot{\lambda}\norm{\abs{\grad}^{s/2}A^{\nu;1} Q^1_{\neq}}_2^2 + \norm{\sqrt{\frac{\partial_t w_L}{w_L}}A^{\nu;1} Q^1_{\neq}}_2^2 \\ 
CK_L^{\nu;i} & := \frac{1}{t}\norm{\mathbf{1}_{t \geq \jap{\grad_{Y,Z}}} A^{\nu;i} Q^i}_2^2. 
\end{align} 
\end{subequations} 
Note the presence of $\tilde{A}^i$; this will mean that, unlike \cite{BGM15I}, the $CK_w$ terms only provide control in the range of frequencies $\abs{\partial_Y} \gtrsim \abs{\partial_Z}$. 
    
Using a bootstrap/continuity argument, we will propagate the following estimates. 
Fix constants $K_{Hi}, K_{H1\neq}, K_{HC1},K_{HC2}, K_{EDi}, K_{Li}, K_{ED2},K_{LC}$ for $i \in \set{1,3}$, sufficiently large determined by the proof, depending only on $\delta,\delta_1,s,\sigma,\gamma,\beta,\lambda^\prime,\lambda_0$ and $\alpha$.   
Further, fix $\sigma^\prime > 3$. 
Let $1 \leq T^\star < T^0$ be the largest time such that the following \emph{bootstrap hypotheses} hold (that $T^\star \geq 1$ is discussed below):
the high norm controls on $Q^i$,
\begin{subequations} \label{ineq:Boot_Hi} 
\begin{align} 
\norm{A^{1} Q^1_{0}(t)}_2^2  & \leq 4K_{H1} \epsilon^2 \label{ineq:Boot_Q1Hi1} \\
\norm{A^{1} Q^1_{\neq}(t)}_2^2 + \frac{1}{2}\int_1^t \mathcal{D}Q^1_{\neq}(\tau) d\tau  & \leq 4K_{H1\neq} \epsilon^2 \label{ineq:Boot_Q1Hi2} \\   
\norm{A^{2} Q^2}^2_2 + \int_1^t\frac{1}{2}\mathcal{D}Q^2(\tau) + CK_L^2(\tau) d\tau & \leq 4\epsilon^2 \label{ineq:Boot_Q2Hi} \\
\norm{A^{3} Q^3}^2_2 + \int_1^t\frac{1}{2} \mathcal{D}Q^3(\tau) d\tau & \leq 4K_{H3}\epsilon^2; \label{ineq:Boot_Q3Hi}
\end{align} 
\end{subequations}
the coordinate system controls, 
\begin{subequations} \label{ineq:Boot_CgHi}
\begin{align} 
\norm{A C^i}_2^2 + \frac{1}{2}\int_1^t \mathcal{D}C^i(\tau) d\tau  &\leq 4 K_{HC1} c^2_{0} \label{ineq:Boot_ACC}\\ 
\jap{t}^{-2}\norm{A C^i}_2^2 + \frac{1}{2} \int_1^t \jap{\tau}^{-2}\mathcal{D}C^i(\tau) d\tau &\leq 4K_{HC2}\epsilon^2 \log\jap{t} \label{ineq:Boot_ACC2}\\
\norm{Ag}_2^2 + \frac{1}{2}\int_1^t \mathcal{D}g  d\tau &\leq 4\epsilon^2 \label{ineq:Boot_Ag} \\ 
\norm{g}_{\G^{\lambda,\gamma}} & \leq  4 \frac{\epsilon}{\jap{t}^{2}} \label{ineq:Boot_gLow} \\
\norm{C}_{\G^{\lambda,\gamma}} & \leq  4K_{LC} \epsilon \jap{t} \label{ineq:Boot_LowC} 
\end{align} 
\end{subequations}
the enhanced dissipation estimates,
\begin{subequations} \label{ineq:Boot_ED}
\begin{align} 
\norm{A^{\nu;1} Q^1}_2^2 + \frac{1}{10} \int_1^t \mathcal{D}Q^{\nu;1}(\tau) d\tau & \leq 4K_{ED1}\epsilon^2 \label{ineq:Boot_ED1} \\ 
\norm{A^{\nu;2} Q^2}_2^2 + \int_1^t\frac{1}{10}\mathcal{D}Q^{\nu;2}(\tau) + CK_{L}^{\nu;2}(\tau) d\tau & \leq 4K_{ED2}\epsilon^2 \\ 
\norm{A^{\nu;3} Q^3}_2^2 + \frac{1}{10}\int_1^t \mathcal{D}Q^{\nu;3}(\tau) d\tau & \leq  4K_{ED3}\epsilon^2; \label{ineq:Boot_ED3}
\end{align} 
\end{subequations}
and the additional low frequency controls on the background streak  
\begin{subequations} \label{ineq:Boot_LowFreq} 
\begin{align} 
\norm{U_0^1}_{H^{\sigma^\prime}} & \leq 4K_{L1} \epsilon \jap{t} \label{ineq:Boot_U01_Low} \\ 
\norm{U_0^2}_{H^{\sigma^\prime}} & \leq 4 \epsilon  \label{ineq:Boot_U02_Low} \\ 
\norm{U_0^3}_{H^{\sigma^\prime}}  & \leq 4K_{L3}\epsilon. \label{ineq:Boot_U03_Low}
\end{align} 
\end{subequations}
For most steps of the proof we do not need to differentiate so precisely between different bootstrap constants so we define 
\begin{align} 
K_B = \max\left(K_{Hi}, K_{H1\neq}, K_{HC1},K_{HC2}, K_{EDi}, K_{Li}, K_{LC}\right). \label{def:KB}
\end{align}
 
By Lemma \ref{lem:Loc}, we have that $T^\star > t_\star > 0$ and it is a consequence 
of Lemma \ref{lem:Cont} that $T^\star < T^0$. 
It is relatively straightforward to prove that for $\epsilon$ sufficiently small, we have $1 \leq T^\star$; see \cite{BGM15I} for more discussion. 
Due to the real analyticity of the solution on $(0, T^0)$, it will follow from the ensuing proof that the quantities in the bootstrap hypotheses take values continuously in time for as long as the solution exists.  
Therefore, we may deduce $T^\star = T_F = c_0 \epsilon^{-1} < T^0$ via the following proposition, the proof of which is the main focus of the remainder of the paper. 

\begin{proposition}[Bootstrap] \label{prop:Boot}
Let $\epsilon < \nu^{2/3+\delta}$. 
For the constants appearing in the right-hand side of \eqref{def:KB} chosen sufficiently large and for $\nu$ and $c_0$ both chosen sufficiently small (depending only on $s,\lambda_0,\lambda^\prime,\alpha,\delta_1,\delta$ and arbitrary parameters such as  $\sigma,\beta, \ldots$), if $T^\star < T_F = c_0\epsilon^{-1}$ is such that the bootstrap hypotheses \eqref{ineq:Boot_Hi} \eqref{ineq:Boot_CgHi} \eqref{ineq:Boot_ED} \eqref{ineq:Boot_LowFreq} hold on $[1,T^\star]$, then on the same time interval all the inequalities in \eqref{ineq:Boot_Hi} \eqref{ineq:Boot_CgHi} \eqref{ineq:Boot_ED} \eqref{ineq:Boot_LowFreq} hold with constant `$2$' instead of `$4$'. 

\end{proposition} 

That Proposition \ref{prop:Boot} implies Theorem \ref{thm:SRS} is discussed briefly in \S\ref{sec:RegCont} below. 

\subsubsection{Bootstrap constants}
The relationship between the constants are similar to \cite{BGM15I} (although slightly simpler here since there are fewer). 
First, $K_{L1}$ and $K_{L3}$ are chosen sufficiently large relative to a universal constant depending only on $\sigma^\prime$. 
These in turn set $K_{H1},K_{H1\neq}$ and $K_{H3}$. These then imply $K_{HC1}$ which then implies $K_{HC2}$ and $K_{LC}$ followed finally by $K_{ED2}$ and then $K_{ED1}$ and $K_{ED3}$. 
Finally, $c_0$ and $\nu$ are chosen sufficiently small with respect to $K_B$, the max of all the bootstrap constants (as well as the parameters $s,\lambda_0,\lambda^\prime,\alpha,\delta_1$, and arbitrary parameters such as $\sigma,\beta$ etc).

\subsubsection{A priori estimates from the bootstrap hypotheses} \label{sec:AprioriBoot}
The motivation for the enhanced dissipation estimates \eqref{ineq:Boot_ED} is the following observation (which follows from \eqref{ineq:DLowB}): for any $f$,  
\begin{subequations} \label{ineq:AnuDecay}
\begin{align}
\norm{f_{\neq} (t)}_{\mathcal{G}^{\lambda(t),\beta}} & \lesssim_\alpha  \jap{t}^{2+\delta_1}\jap{\nu t^3}^{-\alpha} \norm{A^{\nu;1} f(t)}_2  \\ 
\norm{f_{\neq}(t)}_{\mathcal{G}^{\lambda(t),\beta}} & \lesssim_\alpha \jap{t} \jap{\nu t^3}^{-\alpha} \norm{A^{\nu;2} f(t)}_2 \\
\norm{f_{\neq}(t)}_{\mathcal{G}^{\lambda(t),\beta}} & \lesssim_\alpha \jap{t}^{2} \jap{\nu t^3}^{-\alpha} \norm{A^{\nu;3} f(t)}_2. 
\end{align} 
\end{subequations} 
Hence, \eqref{ineq:Boot_ED} expresses a rapid decay of $Q^i_{\neq}$ for $t \gtrsim \nu^{-1/3}$.  
Together with the ``lossy elliptic lemma'', Lemma \ref{lem:LossyElliptic}, we then get (under the bootstrap hypotheses), 
\begin{subequations}  \label{ineq:AprioriUneq}
\begin{align} 
\norm{U^1_{\neq} (t)}_{\mathcal{G}^{\lambda(t),\beta-2}} & \lesssim  \frac{\epsilon \jap{t}^{\delta_1}}{\jap{\nu t^3}^{\alpha}} \\ 
\norm{U^2_{\neq}(t)}_{\mathcal{G}^{\lambda(t),\beta-2}} & \lesssim \frac{\epsilon}{\jap{t}\jap{\nu t^3}^{\alpha}} \\ 
\norm{U^3_{\neq}(t)}_{\mathcal{G}^{\lambda(t),\beta-2}} & \lesssim \frac{\epsilon}{\jap{\nu t^3}^{\alpha}}.  
\end{align}
\end{subequations} 
 
For the zero frequencies of the velocity field we get from \eqref{ineq:Boot_Hi}, \eqref{ineq:Boot_LowFreq} and Lemma \ref{lem:PELbasicZero} (which allows to understand $\Delta_t^{-1}$ at zero $x$ frequencies) the matching a priori estimates
\begin{subequations} \label{ineq:AprioriU0}  
\begin{align} 
\norm{A U^1_0 (t)}_{2} & \lesssim \epsilon \jap{t} \\ 
\norm{A U^2_0 (t)}_{2} & \lesssim \epsilon \\ 
\norm{A^3 \jap{\grad}^2 U^3_0 (t)}_{2} & \lesssim \epsilon. 
\end{align} 
\end{subequations} 
Notice that no regularity loss is required to get the `correct' a priori estimates on the zero frequencies. 
However, unlike in our previous work \cite{BGM15I}, the natural regularity of the zero-frequency velocity fields are not all the same. 

\section{Regularization and continuation} \label{sec:RegCont} 
There are three preliminaries: (A) the instantaneous analytic regularization with initial data of the type \eqref{ineq:QuantGev2} (B) how to move estimates on these classical solutions between coordinate systems, and (C) the proof that Proposition \ref{prop:Boot} implies Theorem \ref{thm:SRS}.
The issues here are essentially the same as in \cite{BGM15I} so we will just give a brief summary. 

The proofs of Lemmas \ref{lem:Loc} and \ref{lem:Cont} are sketched in \cite{BGM15I}. 
Similarly, the following lemma is a variant of [Lemma 3.1 \cite{BGM15I}]. The proof is omitted for brevity as it follows via the same arguments.
\begin{lemma} \label{lem:BootStart}
We may take $2 \leq T^\star$ (defined in \S\ref{sec:energy} above) and for $t \leq 2$, the bootstrap estimates in \eqref{ineq:Boot_Hi}, \eqref{ineq:Boot_CgHi}, \eqref{ineq:Boot_ED}, and \eqref{ineq:Boot_LowFreq}, all hold with constant $5/4$ instead of $4$. 
\end{lemma}

In order to move estimates from $(X,Y,Z)$ to $(x,y,z)$ we may use the same methods described in \cite{BGM15I} (which are themselves essentially the same as those in \cite{BM13,BMV14}). 
we will first move to the coordinate system $(X,y,z)$. 
Writing $\bar{q}^i(t,X,y,z)  = Q^i(t,X,Y,Z) = q(t,x,y,z)$ and $\bar{u}^i(t,X,y,z) = U^i(t,X,Y,Z) = u^i(t,x,y,z)$ we derive the following, noting that $\bar{u}^i_0 = u^i_0$: 
\begin{align} 
\partial_t u_0^i + (u_0^2,u_0^3) \cdot \grad u_0^i & = (-u_0^2,0,0)^T - (0,\partial_y p_0^{NL0}, \partial_z p^{NL0}_0)^T + \nu \Delta u^i_0 + \mathcal{F}^i, \label{eq:u0i}
\end{align} 
where 
\begin{align*}
\Delta p_0^{NL0} = -\partial_i u_0^j \partial_j u^i_0
\end{align*}
and (using cancellations as in \eqref{eq:XavgCanc}), 
\begin{align} 
\mathcal{F}^i & = -\partial_y \left(\bar{u}^2_{\neq}\bar{u}_{\neq}^i\right)_0 - \partial_z\left(\bar{u}^3_{\neq} \bar{u}_{\neq}^i \right)_0. \label{eq:Fbaru}
\end{align} 

We then have the following lemma, analogous to [Lemma 3.2 \cite{BGM15I}], which holds here with an analogous proof. 
\begin{lemma} \label{lem:intermedSob}
For $\epsilon < \nu^{2/3 + \delta}$ and $c_0$ and $\nu$ sufficiently small (depending only on $s,\lambda_0,\lambda^\prime,\alpha$, $\delta_1$, and $\delta$), the bootstrap hypotheses imply the following for some $c \in (0,1)$ chosen such that $c\lambda(t) \in (\lambda^\prime,\lambda(t))$ for all $t$:   
\begin{subequations} \label{ineq:Xyzubds}
\begin{align} 
\norm{\bar{u}^1_{\neq}}_{\G^{c\lambda(t)}} & \lesssim \epsilon \jap{t}^{\delta_1}\jap{\nu t^3}^{-\alpha}  \\ 
\norm{\bar{u}_{\neq}^2}_{\G^{c\lambda(t)}} & \lesssim \epsilon \jap{t}^{-1} \jap{\nu t^3}^{-\alpha}  \\ 
\norm{\bar{u}_{\neq}^3}_{\G^{c\lambda(t)}} & \lesssim \epsilon  \jap{\nu t^3}^{-\alpha}, 
\end{align}  
\end{subequations}
and
\begin{subequations} \label{ineq:uzAPriori}
\begin{align} 
\norm{u^1_0(t)}_{\G^{c\lambda(t)}} & \lesssim \epsilon \jap{t} \label{ineq:uzApriori1} \\ 
\norm{u^2_0(t)}_{\G^{c\lambda(t)}} + \norm{u^3_0(t)}_{\G^{c\lambda(t)}} & \lesssim \epsilon.  \label{ineq:uzApriori23}
\end{align}  
\end{subequations} 
\end{lemma} 
 
Finally, the following lemma also follows analogously to the corresponding result in \cite{BGM15I}. 
Hence, the proof is omitted for the sake of brevity. 
 
\begin{lemma} \label{lem:PropBootThm}
For $\epsilon < \nu^{2/3+\delta}$ and $c_0$ and $\nu$ sufficiently small (depending only on $s,\lambda_0,\lambda^\prime,\alpha$, $\delta_1$, and $\delta$), Proposition \ref{prop:Boot} 
implies Theorem \ref{thm:SRS}.  
\end{lemma} 

\section{Multiplier and paraproduct tools} \label{sec:nrmuse}
In this section we outline some basic general inequalities regarding the multipliers which are used in the sequel. 
As in \cite{BGM15I}, the purpose is to set up a general framework that will make the large number of energy estimates later in the paper easier. 
Most of the estimates come in the general form $\int A^i Q^i A^i\left(f g\right) dV$. 
The goal of this section is to break the treatment of these terms into a four step procedure: 
\begin{enumerate} 
\item As in \cite{BGM15I}, the first step is to separate out zero/non-zero frequency interactions according to \S\ref{sec:NonlinHeuristics} and then expand with a paraproduct to divide the terms based on which of the nonlinear factors is dominant from the standpoint of frequency (paraproducts are explained in \S\ref{sec:paranote} below). 
\item Compare the norm for $Q^i$ with the norm of the dominant factor (also adding $\Delta_L^{-1} \Delta_L$ if the dominant factor is a velocity field) and commute it past the low frequency factor. Lemma \ref{lem:ABasic} below is the primary tool for this.  
\item Use Lemmas \ref{lem:MainFreqRat} and \ref{lem:MainFreqRat_RegImbalance} below to convert the ratio of the norms (together with possibly $\Delta_L^{-1}$) into multipliers that appear in the dissipation energies or integrate to $\lesssim \epsilon^2$ until $T_F = c_0\epsilon^{-1}$.  
\item Use Lemma \ref{gevreyparaproductlemma} or \ref{lem:ParaHighOrder} to re-combine the paraproduct decomposition into multiples of terms in the dissipation energy or other integrable errors. 
\end{enumerate}

\subsection{Basic inequalities regarding the multipliers} \label{sec:basicmult}
This section covers the key properties of the multipliers we are using and forms the core of the technical tools, however, it is very tedious and will likely appear unmotivated at first. 
A reader should consider skipping this section on the first reading and refer back to it whenever specific inequalities are needed. 
Note that this section is significantly more technical than the corresponding section in \cite{BGM15I}. 

In the lemmas which follow, one should imagine that frequencies $(k^\prime,\xi,l^\prime)$ and $(k-k^\prime,\eta-\xi,l-l^\prime)$ are interacting to force $(k,\eta,l)$, as will be occurring in the quadratic energy estimates.  

The first lemma gives us general estimates for how the $A$ and $A^i$ are related at different frequencies.
It is designed specifically for dealing with $f_{Hi}g_{Lo}$-type terms in the paraproducts (see \eqref{def:parapp}). 
 
\begin{lemma}[Frequency ratios for $A$ and $A^i$] \label{lem:ABasic}  
Let $\theta < 1/2$ and suppose 
\begin{align} 
\abs{k-k^\prime,\eta-\xi,l-l^\prime} \leq \theta\abs{k,\eta,l}. \label{ineq:AFreqLoc}
\end{align}
In what follows, define the frequency cut-offs (all functions of $(t,k,k^\prime,\eta,\xi,l,l^\prime)$),    
\begin{subequations} \label{def:freqcuts}
\begin{align} 
\chi^{R,NR} & = \mathbf{1}_{t \in \I_{k,\eta} \cap \I_{k,\xi}} \mathbf{1}_{k^\prime \neq k} \mathbf{1}_{\abs{l} < 5\abs{\eta}} \mathbf{1}_{\abs{l^\prime} < 5\abs{\xi}} \\ 
\chi^{NR,R} & = \mathbf{1}_{t \in \I_{k^\prime,\xi} \cap \I_{k^\prime,\eta}} \mathbf{1}_{k^\prime \neq k} \mathbf{1}_{\abs{l} < \frac{1}{5}\abs{\eta}}\mathbf{1}_{\abs{l^\prime} < \frac{1}{5}\abs{\xi}} \\
\chi^{r,NR} & = \mathbf{1}_{t \in \I_{r,\eta} \cap \I_{r,\xi}} \mathbf{1}_{k^\prime \neq r} \mathbf{1}_{\abs{l} < 5\abs{\eta}} \mathbf{1}_{\abs{l^\prime} < 5\abs{\xi}}  \\ 
\chi^{NR,r} & = \mathbf{1}_{t \in \I_{r,\eta} \cap \I_{r,\xi}} \mathbf{1}_{k \neq r} \mathbf{1}_{\abs{l} < \frac{1}{5}\abs{\eta}}\mathbf{1}_{\abs{l^\prime} < \frac{1}{5}\abs{\xi}}  \\ 
\chi^{\ast;33} & = 1 - \mathbf{1}_{t \in \I_{k,\eta} \cap \I_{k,\xi}} \mathbf{1}_{k \neq k^\prime}\mathbf{1}_{\abs{l} < \frac{1}{5}\abs{\eta}}\mathbf{1}_{\abs{l^\prime} < \frac{1}{5}\abs{\xi}}  - \chi^{NR,R} \\ 
\chi^{\ast;23} & = 1 - \sum_{r}\mathbf{1}_{t \in \I_{r,\eta} \cap \I_{r,\xi}} \mathbf{1}_{k^\prime \neq r}\mathbf{1}_{\abs{l} < \frac{1}{5}\abs{\eta}}\mathbf{1}_{\abs{l^\prime} < \frac{1}{5}\abs{\xi}} \\ 
\chi^{\ast;32} & = 1 - \sum_{r}\chi^{NR,r}, 
\end{align}
\end{subequations}
and for $i,j \in \set{1,2,3}$ and $a,b \in \set{0,\neq}$, the weight $\Gamma(i,j,a,b)$ given by, 
\begin{align*}
\Gamma(i,i,a,a) & = 1, & \Gamma(i,j,a,b)  & = \Gamma(j,i,b,a)^{-1}, \\
\Gamma(1,2,0,0) & = \jap{t}^{-1}, & \Gamma(1,2,\neq,\neq) & = \jap{t}^{-1} \jap{\frac{t}{\jap{\xi,l^\prime}}}^{-\delta_1}, \\ 
\Gamma(1,2,0,\neq) & = \jap{t}^{-1} \jap{\frac{t}{\jap{\xi,l^\prime}}}, & \Gamma(1,2,\neq,0) & = \jap{t}^{-1} \jap{\frac{t}{\jap{\xi,l^\prime}}}^{-1-\delta_1}, \\
\Gamma(1,3,0,0) & = \jap{t}^{-1},   & \Gamma(1,3,\neq,\neq) & = \jap{t}^{-1} \jap{\frac{t}{\jap{\xi,l^\prime}}}^{1-\delta_1}, \\ 
\Gamma(1,3,0,\neq) & = \jap{t}^{-1} \jap{\frac{t}{\jap{\xi,l^\prime}}}^2, & \Gamma(1,3,\neq,0) & = \jap{t}^{-1} \jap{\frac{t}{\jap{\xi,l^\prime}}}^{-1-\delta_1}, \\
\Gamma(2,3,\neq,\neq) & = \jap{\frac{t}{\jap{\xi,l^\prime}}}, & \Gamma(2,3,0,\neq) & = \jap{\frac{t}{\jap{\xi,l^\prime}}}^2, \\
\Gamma(2,3,\neq,0) & = \jap{\frac{t}{\jap{\xi,l^\prime}}}^{-1}, & \Gamma(2,3,0,0) & = 1, \\ 
\Gamma(1,1,0,\neq)&  = \jap{\frac{t}{\jap{\xi,l^\prime}}}^{1+\delta_1}, & \Gamma(2,2,0,\neq) & = \jap{\frac{t}{\jap{\xi,l^\prime}}}, \\ 
\Gamma(3,3,0,\neq) & = \jap{\frac{t}{\jap{\xi,l^\prime}}}^2. 
\end{align*}
Then there exists a $c = c(s) \in (0,1)$ such that for all $t$ we have the following for $i \in \set{1,2}$ and $a = \neq$ if $k \neq 0$ (otherwise $a=0$) and $b = \neq$ if $k^\prime \neq 0$ (otherwise $b = 0$), 
\begin{subequations} \label{ineq:ABasic}
\begin{align} 
A^{i}_k(t,\eta,l) & \lesssim \Gamma(i,j,a,b) A^{j}_{k^\prime}(t,\xi,l^\prime) e^{c\lambda\abs{k - k^\prime,\eta-\xi,l-l^\prime}^s} \label{ineq:ABasic12}\\ 
\left(A^3_k(t,\eta,l)\right)^2 & \lesssim \Gamma(3,3,a,b)\left(\tilde{A}^3_k(t,\eta,l) \tilde{A}^3_{k^\prime}(\xi,l^\prime)\chi^{R,NR}\frac{t}{\abs{k} + \abs{\eta-kt}} \right. \nonumber \\ & \left. \quad\quad + \tilde{A}^3_k(t,\eta,l)\tilde{A}^3_{k^\prime}(t,\xi,l^\prime) \chi^{NR,R}\frac{\abs{k^\prime} + \abs{\eta - k^\prime t}}{t} \right. \nonumber \\ & \quad\quad + \chi^{\ast;33} A^3_k(t,\eta,l) A^3_{k^\prime}(t,\xi,l^\prime)  \bigg) e^{c\lambda\abs{k - k^\prime,\eta-\xi,l-l^\prime}^s} \label{ineq:A3A3neqneq}\\  
\left(A^{i}_k(t,\eta,l)\right)^2 & \lesssim \Gamma(i,3,a,b)\left(\sum_{r} \tilde{A}^i_k(t,\eta,l)\tilde{A}^3_{k^\prime}(t,\xi,l^\prime)\chi^{r,NR}\frac{t}{\abs{r} + \abs{\eta- rt}} \right. \nonumber \\ &  \quad\quad + A^i_k(t,\eta,l) A^3_{k^\prime}(t,\xi,l^\prime) \chi^{\ast;23} \Bigg) e^{c\lambda\abs{k - k^\prime,\eta-\xi,l-l^\prime}^s} \label{ineq:A2A3neqneq} \\
\left(A^{3}_k(t,\eta,l)\right)^2 & \lesssim \Gamma(3,i,a,b)\left(\sum_{r} \tilde{A}^3_k(t,\eta,l) \tilde{A}^i_{k^\prime}(t,\xi,l^\prime)  \chi^{NR,r}\frac{\abs{r} + \abs{\eta-rt}}{t} \right. \nonumber \\ &  \quad\quad + \chi^{\ast;32} A^3_k(t,\eta,l) A^i_{k^\prime}(t,\xi,l^\prime) \Bigg)  e^{c\lambda\abs{k - k^\prime,\eta-\xi,l-l^\prime}^s}. \label{ineq:A3A2neqneq}
\end{align}
\end{subequations}
Analogous inequalities hold also with $A(t,\eta,l)$ using that $A(t,\eta,l) = \jap{\eta,l}^2 A_0^2(t,\eta,l)$. 
\end{lemma} 
\begin{remark} 
The terms involving $\chi^{R,NR}$, $\chi^{NR,R}$, $\chi^{r,NR}$, and $\chi^{NR,r}$ are arising from comparing ratios of $w^3_k$ and $w^3_{k^\prime}$ or $w^3$ and $w$; see e.g. \eqref{def:approxw} above.  
In particular, modulo details regarding the $Z$ frequencies, the three contributions to \eqref{ineq:A3A3neqneq} roughly correspond to the three possible regimes in Lemma \ref{lem:Jswap}: when a resonant frequency forces a non-resonant frequency, vice-versa, and neither.  
The inequalities \eqref{ineq:A2A3neqneq} and \eqref{ineq:A3A2neqneq} generally play a more crucial role in the proof of Theorem \ref{thm:SRS} and  correspond instead to what happens when one compares $w$ and $w^3$, rather than $w^3$ with itself (that is, in terms when $Q^3$ interacts with $Q^{1,2}$). 
We have chosen to write it in this manner as this is the form that is most natural for Lemma \ref{lem:MainFreqRat_RegImbalance} below. 
\end{remark}

\begin{remark} 
Note that a time/frequency combination is only considered truly ``resonant'' if $t \in \I_{k,\eta} \cap \I_{k,\xi}$. 
The reason for this is explained by Lemma \ref{lem:wellsep}: if $t \in \I_{k,\eta}$ but $t \not\in \I_{k,\xi}$, then either $\eta$ and $\xi$ are well-separated or the time/frequency combination is not really resonant, which results in $\jap{\eta-\xi}\jap{kt - \eta}\gtrsim t$.   
\end{remark} 

\begin{remark} 
Note that the definitions in \eqref{def:freqcuts} are not quite symmetric for minor technical reasons and that the decomposition defined by \eqref{def:freqcuts} is not quite a partition of unity, as there is an overlap region when $\abs{l} \approx \abs{\eta}$ or $\abs{l^\prime} \approx \abs{\xi}$.   
When losing due to the regularity imbalances, one must take the larger region $\abs{l} < 5\abs{\eta}$ and $\abs{l^\prime} < 5\abs{\xi}$ but when
gaining due to the regularity imbalances,  one must take the smaller region $\abs{l} < \frac{1}{5}\abs{\eta}$ and $\abs{l^\prime} < \frac{1}{5}\abs{\xi}$.
\end{remark}

\begin{remark} 
Note that some of the inequalities in Lemma \ref{lem:ABasic} are phrased on quadratic quantities (as opposed to \eqref{ineq:ABasic12} and the analogous lemma in \cite{BGM15I}).
This is to treat the overlapping regions $\abs{l} \approx \abs{\eta}$ and $\abs{l^\prime} \approx \abs{\xi}$ more carefully, in particular, it is to make sure that any losses or gains from the ratios of $w$ and $w^3$ come with $\tilde{A}^i$, even if it is a region of frequency where $A^i \not\approx \tilde{A}^i$ (see also Remark \ref{rmk:noFreqRestrict} below). 
This precision is only required in certain places, especially when we need to use the $CK_w^i$ terms, and in other cases less precise inequalities suffice.  
\end{remark}

\begin{proof} 
These inequalities are all more or less easy variants of each other so we will just consider one of the trickier inequalities and omit the rest for brevity.   
We will consider \eqref{ineq:A3A3neqneq}; further, we will consider just the case $a = b = \neq$ as the other cases are analogous.

The proof is divided into three regions (which do not exactly correspond to the three terms in \eqref{ineq:A3A3neqneq}). 
\\

\noindent
\textit{Case 1: $\abs{l} > 5\abs{\eta}$ or $\abs{l^\prime} > 5\abs{\xi}$ } \\ 
In this case, the $Z$ frequencies are dominant and hence one does not see the contributions from $w^3$ multipliers. 
Indeed, $\chi^{R,NR} = \chi^{NR,R} = 0$ and $\chi^{\ast;33} = 1$. 
If $\abs{l^\prime} > 3\abs{\eta}$ then by Lemma \ref{lem:totalGrowthw}, 
\begin{align*}
\frac{\left(\frac{e^{\mu \abs{\eta}^{1/2}}}{w^3_k(t,\eta)} + e^{\mu\abs{l}^{1/2}}\right)}{\left(\frac{e^{\mu \abs{\xi}^{1/2}}}{w^3_{k^\prime}(t,\xi)} + e^{\mu\abs{l^\prime}^{1/2}}\right)} & \lesssim \frac{1}{w_k^3(t,\eta)}e^{\mu\abs{\eta}^{1/2} - \mu\abs{l^\prime}^{1/2}} + e^{\mu\abs{l}^{1/2} - \mu \abs{l^\prime}^{1/2}} \\ 
& \lesssim e^{\frac{3\mu}{2}\abs{\eta}^{1/2} - \mu\abs{l^\prime}^{1/2}} + e^{\mu\abs{l-l^\prime}^{1/2}} \\ 
& \lesssim e^{\mu\abs{l-l^\prime}^{1/2}}.  
\end{align*}
Therefore, by \eqref{ineq:AFreqLoc} and \ref{lem:scon} (and that $w_L$ is $O(1)$ by \eqref{ineq:unifN} and \eqref{def:wL}), there is some $c^\prime = c^\prime(s) \in (0,1)$, 
\begin{align*}
A_k^3(t,\eta,l) & \lesssim e^{\mu\abs{l-l^\prime}^{1/2} + c^\prime\lambda\abs{k-k^\prime,\eta-\xi,l-l^\prime}^s}A_{k^\prime}(t,\xi,l^\prime). 
\end{align*}
Then in this case \eqref{ineq:A3A3neqneq} follows from \eqref{ineq:IncExp} for some $c^\prime <  c  < 1$.
If $\abs{l^\prime} \leq 3\abs{\eta}$ then it follows that either $\abs{l-l^\prime} \gtrsim \abs{\eta}$ or $\abs{\eta-\xi} \gtrsim \abs{l^\prime} \gtrsim \abs{\xi}$.
Therefore, Lemma \ref{lem:totalGrowthw}, for some $K$ there holds,   
\begin{align*}
\left(\frac{e^{\mu \abs{\eta}^{1/2}}}{w^3_k(t,\eta)} + e^{\mu\abs{l}^{1/2}}\right) & \lesssim e^{\frac{3}{2}\mu\abs{\eta}^{1/2}} + e^{\mu\abs{l}^{1/2}}  \lesssim e^{\mu\abs{l^\prime}^{1/2}} e^{K\mu\abs{\eta-\xi,l-l^\prime}^{1/2}}.
\end{align*} 
Therefore, by the frequency localizations, for some $c^\prime = c^\prime(s) \in (0,1)$, 
\begin{align*}
A_k^3(t,\eta,l) & \lesssim e^{K\mu\abs{\eta-\xi, l-l^\prime}^{1/2} + c^\prime\lambda\abs{k-k^\prime,\eta-\xi,l-l^\prime}^s}A^3_{k^\prime}(t,\xi,l^\prime), 
\end{align*}
from which again there follows \eqref{ineq:A3A3neqneq} from \eqref{ineq:IncExp} for some $c^\prime <  c  < 1$.
\\

\noindent
\textit{Case 2: ($\abs{l} < 5\abs{\eta}$ and $\abs{l^\prime} < 5\abs{\xi}$) and ($\abs{l} > \frac{1}{5}\abs{\eta}$ or $\abs{l^\prime} > \frac{1}{5}\abs{\xi}$)}\\ 
In this case, neither $l,l^\prime$ nor $\eta,\xi$ are necessarily dominant, and indeed $\abs{l} \approx \abs{\eta}$ or $\abs{l^\prime} \approx \abs{\xi}$.  
We have $\chi^{NR,R}=0$ but there are regions in frequency where $\chi^{R,NR} = \chi^{\ast;33} = 1$ and we have to consider contributions involving both $A^3$ and $\tilde{A^3}$ at the same time.  
By \eqref{ineq:AFreqLoc} and Lemma \ref{lem:scon} (and that $w_L$ is $O(1)$ by \eqref{ineq:unifN} and \eqref{def:wL}), there is some $c^\prime = c^\prime(s) \in (0,1)$, 
\begin{align*}
\left(A^3_k(t,\eta,l)\right)^2 & \lesssim \left(\frac{e^{2\mu \abs{\eta}^{1/2}}}{\left(w^3_k(t,\eta)\right)^2} + e^{2\mu\abs{l}^{1/2}}\right) \frac{1}{\left(w_L(k,\eta,l)\right)^2}\jap{k,\eta,l}^{2\sigma} e^{2\lambda\abs{k,\eta,l}^s} \\ 
& \lesssim \left(\frac{e^{\mu \abs{\eta}^{1/2} + \mu\abs{\xi}^{1/2} + \mu\abs{\eta-\xi}^{1/2}}}{\left(w^3_k(t,\eta)\right)^2} + e^{\mu\abs{l}^{1/2} + \mu\abs{l^\prime}^{1/2} + \mu\abs{l-l^\prime}^{1/2}}\right) \\ & \quad\quad \times \frac{1}{w_L(k,\eta,l)w_L(k^\prime,\xi,l^\prime)}\jap{k,\eta,l}^{\sigma}\jap{k^\prime,\xi,l^\prime}^{\sigma}  e^{\lambda\abs{k,\eta,l}^s + \lambda\abs{k^\prime,\xi,l^\prime}^s + c^\prime\lambda\abs{k-k^\prime,\eta-\xi,l-l^\prime}^s}
\end{align*}
Then, by \eqref{ineq:IncExp}, we have some $c^\prime < c  < 1$ such that 
\begin{align*}
\left(A^3_k(t,\eta,l)\right)^2 & \lesssim \left(\frac{w^3_{k^\prime}(\xi)}{ w^3_k(\eta)} \tilde{A}^3_k(t,\eta,l)\tilde{A}^3_{k^\prime}(t,\xi,l^\prime) + A^3_k(t,\eta,l) A^3_{k^\prime}(t,\xi,l^\prime)\right) e^{c\lambda\abs{k-k^\prime,\eta-\xi,l-l^\prime}^s}. 
\end{align*} 
Lemma \ref{lem:Jswap} implies for some $K > 0$ (in particular),  
\begin{align*}
\frac{w^3_{k^\prime}(\xi)}{w^3_k(\eta)} & \lesssim \left(1 + \frac{t}{\abs{k} + \abs{\eta-kt}}\mathbf{1}_{t \in \I_{k,\eta} \cap \I_{k,\xi}}\mathbf{1}_{k \neq k^\prime} \right)e^{K\mu \abs{\eta-\xi}^{1/2}}, 
\end{align*}
and so we may restrict the frequencies over which we have a loss involving the $\tilde{A}^3$ to $\chi^{R,NR}$ but there is an overlapping region where both $A^3$ and $\tilde{A}^3$ are necessary.  
This completes the proof of \eqref{ineq:A3A3neqneq} now in the range of frequencies $\abs{l} > \frac{1}{5}\abs{\eta}$ or $\abs{l^\prime} > \frac{1}{5}\abs{\xi}$. 
\\

\noindent
\textit{Case 3: $\abs{l} < \frac{1}{5}\abs{\eta}$ and $\abs{l^\prime} < \frac{1}{5}\abs{\xi}$} \\ 
In this case, we need to be able to gain from the regularity imbalance. 
Here we have $\chi^{\ast;33} = 0$ and the only contributions are those which involve $\tilde{A}^3$. 
We have here,  using $w_{k^\prime}(t,\xi) \leq 1$ by definition (see Appendix \ref{sec:Defw}), 
\begin{align*}
\frac{\left(\frac{e^{\mu \abs{\eta}^{1/2}}}{w^3_k(t,\eta)} + e^{\mu\abs{l}^{1/2}}\right)}{\left(\frac{e^{\mu \abs{\xi}^{1/2}}}{w^3_{k^\prime}(t,\xi)} + e^{\mu\abs{l^\prime}^{1/2}}\right)} & \lesssim \frac{w^3_{k^\prime}(t,\xi)}{w^3_k(t,\eta)}e^{\mu\abs{\eta-\xi}^{1/2}} + w^3_{k^\prime}(t,\xi)e^{\mu\abs{l}^{1/2} - \mu \abs{\xi}^{1/2}} \\ 
& \lesssim \frac{w^3_{k^\prime}(t,\xi)}{w^3_k(t,\eta)}e^{\mu\abs{\eta-\xi}^{1/2}} + e^{\mu\abs{l}^{1/2} - \mu \abs{\xi}^{1/2}} \\      
 & \lesssim \frac{w^3_{k^\prime}(t,\xi)}{w^3_k(t,\eta)}e^{2\mu\abs{\eta-\xi,l-l^\prime}^{1/2}}. 
\end{align*} 
Therefore, in this case we only have contributions from the ratio of $w^3$: as above, we have for some $c^\prime = c^\prime(s) \in (0,1)$: 
\begin{align*}
\left(A_k^3(t,\eta,l)\right)^2 & \lesssim \frac{w^3_{k^\prime}(t,\xi)}{w_k^3(t,\eta)}e^{2\mu\abs{\eta-\xi,l-l^\prime}^{1/2}} A_k^3(t,\eta,l) A_{k^\prime}^3(t,\xi,l^\prime) e^{c^\prime\lambda \abs{k-k^\prime,\eta-\xi,l-l^\prime}^s}. 
\end{align*}
then \eqref{ineq:A3A3neqneq} now follows from Lemma \ref{lem:Jswap} (followed by \eqref{ineq:IncExp}) and the fact that under these restrictions $A^3 \approx \tilde{A}^3$.
We then have that \eqref{ineq:A3A3neqneq} follows from Lemma \ref{lem:Jswap}. 
This completes the proof of \eqref{ineq:A3A3neqneq} over all possible frequencies, and as mentioned above, the other inequalities are similar or easier. 

\end{proof}

We also have the following for remainder terms in the paraproducts (see \eqref{def:parapp}); the proof is the same as the analogous [Lemma 4.2 \cite{BGM15I}], so we omit it here for brevity.  
\begin{lemma} \label{lem:Arem}
For all $K > 0$ there exists a $c = c(s,K) \in (0,1)$ such that if 
\begin{align*} 
\frac{1}{K}\abs{k^\prime,\xi,l^\prime} \leq \abs{k - k^\prime,\eta-\xi,l-l^\prime} \leq K\abs{k^\prime,\xi,l^\prime},
\end{align*}
then 
\begin{subequations} \label{ineq:ARemainderBasic}
\begin{align} 
A^1_k(t,\eta,l) & \lesssim  \jap{t}^{-2-\delta_1} e^{c\lambda \abs{k^\prime,\xi,l^\prime}^s}e^{c\abs{k-k^\prime,\eta-\xi,l-l^\prime}^s} \\ 
A^2_k(t,\eta,l) & \lesssim  \jap{t}^{-1} e^{c \lambda \abs{k^\prime,\xi,l^\prime}^s}e^{c\abs{k-k^\prime,\eta-\xi,l-l^\prime}^s} \\
A^3_k(t,\eta,l) & \lesssim  \jap{t}^{-2} e^{c \lambda \abs{k^\prime,\xi,l^\prime}^s}e^{c\abs{k-k^\prime,\eta-\xi,l-l^\prime}^s}, \label{ineq:A3remainderBasic} 
\end{align}  
\end{subequations}
and if $k = k^\prime = 0$ then
\begin{align} 
A(t,\eta,l) & \lesssim e^{c \lambda \abs{\xi,l^\prime}^s}e^{c \lambda \abs{\eta-\xi,l-l^\prime}^s}. \label{ineq:AARemainderBasic}
\end{align}
All implicit constants depend on $\kappa, \lambda, \sigma$ and $s$. 
\end{lemma}

The following is [Lemma 4.3 \cite{BGM15I}], see therein for a proof. 
\begin{lemma} [Frequency ratios for $\partial_t w$ and $\partial_t w_L$] \label{lem:CKwFreqRat} 
For all $t \geq 1$ we have
\begin{subequations} 
\begin{align} 
\left(\sqrt{\frac{\partial_t w(t,\eta)}{w(t,\eta)}} + \frac{\abs{k,\eta,l}^{s/2}}{\jap{t}^{s}}\right) & \lesssim \left(\sqrt{\frac{\partial_t w(t,\xi)}{w(t,\xi)}} + \frac{\jap{k^\prime,\xi,l^\prime}^{s/2}}{\jap{t}^{s}}\right) \jap{k-k^\prime,\eta-\xi,l-l^\prime}^2 \label{ineq:dtwBasicBrack} \\ 
\left(\sqrt{\frac{\partial_t w(t,\eta)}{w(t,\eta)}} + \frac{\abs{k,\eta,l}^{s/2}}{\jap{t}^{s}}\right) & \lesssim \left(\sqrt{\frac{\partial_t w(t,\xi)}{w(t,\xi)}} + \frac{\abs{k^\prime,\xi,l^\prime}^{s/2} + \abs{k-k^\prime,\eta-\xi,l-l^\prime}^{s/2}}{\jap{t}^{s}}\right) \nonumber \\ & \quad\quad \times \jap{k-k^\prime,\eta-\xi,l-l^\prime}^2 \label{ineq:dtwBasicBrack2} \\ 
\sqrt{\frac{\partial_t w_L(t,k,\eta,l)}{w_L(t,k,\eta,l)}} & \lesssim \sqrt{\frac{\partial_t w_L(t,k,\xi,l^\prime)}{w_L(t,k,\xi,l^\prime)}} \jap{\eta-\xi,l-l^\prime}^{3/2}. \label{ineq:dtNBasic}
\end{align}
\end{subequations}
Further, if $\abs{k^\prime,\xi,l^\prime} \gtrsim 1$ then \eqref{ineq:dtwBasicBrack} implies
\begin{align}
\left(\sqrt{\frac{\partial_t w(t,\eta)}{w(t,\eta)}} + \frac{\abs{k,\eta,l}^{s/2}}{\jap{t}^{s}}\right) & \lesssim \left(\sqrt{\frac{\partial_t w(t,\xi)}{w(t,\xi)}} + \frac{\abs{k^\prime,\xi,l^\prime}^{s/2}}{\jap{t}^{s}}\right) \jap{k-k^\prime,\eta-\xi,l-l^\prime}^2.  \label{ineq:dtwBasic}
\end{align}
Moreover, both \eqref{ineq:dtwBasicBrack} and \eqref{ineq:dtwBasic} hold if we replace $\abs{k,\eta,l}$ and $\abs{k,\xi,l^\prime}$ by $\abs{\eta}$ and $\abs{\xi}$ (respectively). 
\end{lemma} 

The next lemma is [Lemma 4.4, \cite{BGM15I}] and is immediate from the definition of $D$ \eqref{def:D}, but useful for separating the pre and post critical times in the enhanced dissipation estimates. 
\begin{lemma} 
For all $p \geq 0$ and $(k,\eta,l)$ there holds the following inequalities
\begin{subequations} 
\begin{align}
A^{\nu;i}_k(t,\eta,l) & \lesssim \jap{t}^{-p}\jap{k,\eta,l}^{\beta + 3\alpha + p}e^{\lambda\abs{k,\eta,l}^{s}} + A^{\nu;i}_k(t,\eta,l) \mathbf{1}_{t \geq 2\abs{\eta}} \label{ineq:AnuHiLowSep} \\ 
A^{\nu;i}_k(t,\eta,l) & \lesssim \jap{t}^{-p}\jap{k,\eta,l}^{\beta + 3\alpha + p}e^{\lambda\abs{k,\eta,l}^{s}} + \jap{t}^{-1}\left(\abs{k} + \abs{\eta-kt}\right)A^{\nu;i}_k(t,\eta,l) \mathbf{1}_{t \geq 2\abs{\eta}}.  \label{ineq:AnuHiLowSep2}
\end{align}
\end{subequations}
\end{lemma}

The next lemma tells us how to treat ratios involving $\Delta_L$. 
This lemma is a technical improvement of [Lemma 4.5, \cite{BGM15I}]. The adjustments are necessary as here we can only use the $CK_w$ terms in a certain sector of frequency due to the more non-trivial angular dependence of the norms we are employing. 

\begin{lemma}[Frequency ratios for $\Delta_L$] \label{lem:MainFreqRat}
If $t \gtrsim 1$ then for all $\eta,\xi,l,l^\prime, k^\prime$ and $k$ define the following 
\begin{align}
\chi_{NR;k} = 1- \mathbf{1}_{t \in \I_{k,\eta} \cap \I_{k,\xi}}\mathbf{1}_{\abs{l} < \frac{1}{5}\abs{\eta}}\mathbf{1}_{\abs{l^\prime} < \frac{1}{5}\abs{\xi}}. \label{def:chiNR}
\end{align}
Then, we have the following 
\begin{itemize}
\item Basic characterizations of non-resonance: for all $k \neq 0$,  
\begin{align}
\left(\frac{1}{\abs{k,\eta-kt,l}} + \frac{1}{\abs{k,\xi-kt,l^\prime}}\right)\chi_{NR;k} & \lesssim \frac{1}{\jap{k,t,l^\prime}} \jap{\eta-\xi,l-l^\prime}; \label{ineq:basicNR} 
\end{align}
\item Approximate integration by parts: for all $k \neq 0$,
\begin{align} 
\abs{\eta-kt} \lesssim \jap{\eta-\xi}\left(\abs{k} + \abs{\xi-kt}\right); \label{ineq:TriTriv}
\end{align}
\item For absorbing long-time losses: for all $k \neq 0$, 
\begin{align} 
\frac{1}{\abs{k,\eta-kt,l}} \jap{\frac{t}{\jap{\xi,l^\prime}}} & \lesssim \jap{\eta-\xi,l-l^\prime}; \label{ineq:ratlongtime}
\end{align}
\item For the linear stretching terms, for all $k \neq 0$, 
\begin{align}
\frac{\abs{k} \mathbf{1}_{t \leq 2\abs{\eta}}}{\abs{k} + \abs{l} + \abs{\eta-kt}} & \lesssim \kappa^{-1}\frac{\partial_t w(t,\eta)}{w(t,\eta)}\mathbf{1}_{\abs{l} \leq \frac{1}{5}\abs{\eta}} + \frac{\abs{l}^{1/2}}{t^{3/2}}; \label{ineq:CKwLS}
\end{align}
\item For nonlinear terms involving $\partial_X$ (for \textbf{(SI)} terms): if $p \in \mathbb{R}$ and $k \neq 0$,  
\begin{align} 
\frac{\abs{k,\eta-kt,l} \abs{k}}{k^2 + (l^\prime)^2 + \abs{\xi-kt}^2} \jap{\frac{t}{\jap{\xi,l^\prime}}}^p & \lesssim \left(\left(\sqrt{\frac{\partial_t w(t,\eta)}{w(t,\eta)}}\mathbf{1}_{\abs{l} \leq \frac{1}{5}\abs{\eta}}  + \frac{\abs{k,\eta}^{s/2}}{\jap{t}^s} \right)\left(\sqrt{\frac{\partial_t w(t,\xi)}{w(t,\xi)}}\mathbf{1}_{\abs{l^\prime} \leq \frac{1}{5}\abs{\xi}}  + \frac{\abs{k,\xi}^{s/2}}{\jap{t}^s} \right) \right. \nonumber \\ & \left. \quad +  \frac{\chi_{NR;k}}{\jap{t}}\min\left(1,\frac{\abs{k, \eta-kt,l}}{\jap{kt}}  \right) \jap{\frac{t}{\jap{\xi,l^\prime}}}^p \right) \jap{\eta-\xi,l-l^\prime}^{4}; \label{ineq:AiPartX}
\end{align}  
\item For terms with fewer derivatives (for \textbf{(3DE)} terms): if $a \in \set{1,2}$, $p \in \mathbb{R}$, and $k^\prime,k \neq 0$, then
\begin{align} 
\frac{1}{\abs{k^\prime,\xi-k^\prime t,l^\prime}^a} \jap{\frac{t}{\jap{\xi,l^\prime}}}^p & \lesssim \nonumber \\ & \hspace{-3cm} \left(\sqrt{\frac{\partial_t w(t,\eta)}{w(t,\eta)}} \mathbf{1}_{\abs{l} \leq \frac{1}{5}\abs{\eta}} + \frac{\abs{k,\eta}^{s/2}}{\jap{t}^s} \right)\left(\sqrt{\frac{\partial_t w(t,\xi)}{w(t,\xi)}} \mathbf{1}_{\abs{l^\prime} \leq \frac{1}{5}\abs{\xi}} + \frac{\abs{k^\prime, \xi}^{s/2}}{\jap{t}^s} \right) \jap{k-k^\prime, \eta-\xi,l-l^\prime}^3  \nonumber \\ & \hspace{-3cm} \quad\quad  + \frac{1}{\jap{t}^{a}}\jap{\frac{t}{\jap{\xi,l^\prime}}}^p \jap{k-k^\prime, \eta-\xi,l-l^\prime}^3  \label{ineq:AikDelLNoD} 
\end{align} 
\item For \textbf{(3DE)} terms in the nonlinear pressure and stretching: if $p \in \mathbb{R}$, $k k^\prime(k-k^\prime)  \neq 0$,  
\begin{subequations} \label{ineq:AikDelL2D}
\begin{align}
\frac{\abs{k,\eta-k t,l}\abs{k,\xi-k^\prime t,l^\prime}}{(k^\prime)^2 + (l^\prime)^2 + \abs{\xi-k^\prime t}^2} \jap{\frac{t}{\jap{\xi,l^\prime}}}^p & \lesssim \left(\jap{t} + \jap{\frac{t}{\jap{\xi,l^\prime}}}^p\right)\jap{k-k^\prime,\eta-\xi,l-l^\prime}^{2}  \\ 
\frac{\abs{k,\eta-k t,l}\abs{k^\prime,\xi-k^\prime t,l^\prime}}{(k^\prime)^2 + (l^\prime)^2 + \abs{\xi-k^\prime t}^2} \jap{\frac{t}{\jap{\xi,l^\prime}}}^p & \lesssim \left(\jap{t}\left(\sqrt{\frac{\partial_t w(t,\eta)}{w(t,\eta)}}\mathbf{1}_{\abs{l} \leq \frac{1}{5}\abs{\eta}} + \frac{\abs{k,\eta}^{s/2}}{\jap{t}^s} \right)\right. \nonumber \\ & \left.  \quad\quad \times \left(\sqrt{\frac{\partial_t w(t,\xi)}{w(t,\xi)}}\mathbf{1}_{\abs{l^\prime} \leq \frac{1}{5}\abs{\xi}} + \frac{\abs{k^\prime, \xi}^{s/2}}{\jap{t}^s} \right) \right. \nonumber  \\  & \left. \quad\quad + \min\left(1,\frac{\abs{k,\eta-kt,l}}{\jap{kt}} \right)\jap{\frac{t}{\jap{\xi,l^\prime}}}^p\right)\jap{k-k^\prime,\eta-\xi,l-l^\prime}^{2}. \label{ineq:AikDelL2D_CKw}  \\ 
\frac{\abs{l^\prime}\abs{k,\eta-kt,l}}{(k^\prime)^2 + (l^\prime)^2 + \abs{\xi-k^\prime t}^2} \jap{\frac{t}{\jap{\xi,l^\prime}}} & \lesssim \left(\jap{t}\left(\sqrt{\frac{\partial_t w(t,\eta)}{w(t,\eta)}}\mathbf{1}_{\abs{l} \leq \frac{1}{5}\abs{\eta}} + \frac{\abs{k,\eta}^{s/2}}{\jap{t}^s} \right) \right. \nonumber \\ & \left. \quad\quad \times \left(\sqrt{\frac{\partial_t w(t,\xi)}{w(t,\xi)}}\mathbf{1}_{\abs{l^\prime} \leq \frac{1}{5}\abs{\xi}} + \frac{\abs{k^\prime, \xi}^{s/2}}{\jap{t}^s} \right) \right. \nonumber  \\  &  \quad\quad + 1\bigg) \jap{k-k^\prime,\eta-\xi,l-l^\prime}^{2}. \label{ineq:AikDelL2D_CKw2}  
\end{align} 
\end{subequations}
\item For triple derivative terms (these arise in the treatment of \textbf{(F)} terms): if $p \in \mathbb{R}$ and $k \neq 0$,    
\begin{subequations} \label{ineq:AdelLij} 
\begin{align} 
\frac{\abs{l}^3}{(k)^2 + (l^\prime)^2 + \abs{\xi-k t}^2} \jap{\frac{t}{\jap{\xi,l^\prime}}}^p & \lesssim \abs{l}\left(\jap{l-l^\prime}^2  + \frac{\abs{l}^2}{\jap{l^\prime,t}^2} \jap{\frac{t}{\jap{\xi,l^\prime}}}^p\right)  \label{ineq:AdeZZZ} \\ 
\frac{\abs{\eta}\abs{l}^2 + \abs{\eta}^2\abs{l}}{(k)^2 + (l^\prime)^2 + \abs{\xi-k t}^2} \jap{\frac{t}{\jap{\xi,l^\prime}}}^p & \lesssim \left(\jap{t}^2\left(\sqrt{\frac{\partial_t w(t,\eta)}{w(t,\eta)}}\mathbf{1}_{\abs{l} \leq \frac{1}{5}\abs{\eta}} + \frac{\abs{\eta}^{s/2}}{\jap{t}^s}\right) \right. \nonumber \\ & \left. \quad\quad \times \left(\sqrt{\frac{\partial_t w(t,\xi)}{w(t,\xi)}}\mathbf{1}_{\abs{l^\prime} \leq \frac{1}{5}\abs{\xi}} + \frac{\abs{\xi}^{s/2}}{\jap{t}^s}\right) \right. \nonumber \\ & \left. \quad + \abs{l}\left(1  + \frac{\abs{\eta}\abs{l} + \abs{\eta}^2}{\jap{\xi,l^\prime,t}^2} \jap{\frac{t}{\jap{\xi,l^\prime}}}^p\right)\right) \jap{k,\eta-\xi,l-l^\prime}^3  \label{ineq:AdeYZZ}  \\ 
\frac{\abs{\eta}^3}{k^2 + (l^\prime)^2 + \abs{\xi-k t}^2} \jap{\frac{t}{\jap{\xi,l^\prime}}}^p & \lesssim \left(\jap{t}^3\left(\sqrt{\frac{\partial_t w(t,\eta)}{w(t,\eta)}}\mathbf{1}_{\abs{l} \leq \frac{1}{5}\abs{\eta}} + \frac{\abs{\eta}^{s/2}}{\jap{t}^s}\right) \right. \nonumber \\ & \left. \quad\quad \times \left(\sqrt{\frac{\partial_t w(t,\xi)}{w(t,\xi)}}\mathbf{1}_{\abs{l^\prime} \leq \frac{1}{5}\abs{\xi}} + \frac{\abs{\xi}^{s/2}}{\jap{t}^s}\right) \right. \nonumber \\ & \left. \quad  + \min(\abs{\eta},\jap{\xi-kt})\left(1  + \frac{\abs{\eta}^2}{\jap{\xi,l^\prime,t}^2} \jap{\frac{t}{\jap{\xi,l^\prime}}}^p\right) \right) \jap{k,\eta-\xi,l-l^\prime}^3. \label{ineq:AdeYYY}
\end{align} 
\end{subequations} 
\end{itemize} 
\end{lemma}
\begin{remark} 
As in \cite{BGM15I}, \eqref{ineq:AdelLij} implies 
\begin{align}
\frac{\abs{\eta,l}\jap{\eta,l}^2}{(k)^2 + (l^\prime)^2 + \abs{\xi-k t}^2} \jap{\frac{t}{\jap{\xi,l^\prime}}}^p & \lesssim \left(\jap{t}^3\left(\sqrt{\frac{\partial_t w(t,\eta)}{w(t,\eta)}}\mathbf{1}_{\abs{l} \leq \frac{1}{5}\abs{\eta}} + \frac{\abs{\eta}^{s/2}}{\jap{t}^s}\right) \right. \nonumber \\ & \left. \quad\quad \times \left(\sqrt{\frac{\partial_t w(t,\xi)}{w(t,\xi)}}\mathbf{1}_{\abs{l^\prime} \leq \frac{1}{5}\abs{\xi}} + \frac{\abs{\xi}^{s/2}}{\jap{t}^s}\right) \right. \nonumber \\ & \left. \quad  + \abs{\eta,l}\left(1  + \frac{\jap{\eta,l}^2}{\jap{\xi,l^\prime,t}^2} \jap{\frac{t}{\jap{\xi,l^\prime}}}^p\right)\right)\jap{k,\eta-\xi,l-l^\prime}^3 \label{ineq:AdeGen}
\end{align}
\end{remark}

\begin{proof}
First, note that for any fixed number $N \geq 1$, 
\begin{align}
\frac{\mathbf{1}_{\abs{l^\prime} \geq \frac{1}{N} \abs{\xi}} }{\abs{k^\prime,l^\prime,\xi-k^\prime t}} & \lesssim_{N} \frac{1}{\abs{l^\prime, k^\prime t}}, \label{ineq:Znonres}
\end{align}
and hence the sector in frequency where $l^\prime$ is dominant or comparable to $\xi$ is strongly non-resonant.  
Further, observe that for any $N \geq 1$, 
\begin{align*}
\abs{l}  \geq \frac{1}{N}\abs{\eta} \quad \textup{and} \quad\abs{l^\prime}  \leq \frac{1}{N+1}\abs{\xi}, 
\end{align*}
imply
\begin{align}
\abs{\xi,l^\prime} + \abs{\eta,l} \lesssim_N \abs{\eta-\xi,l-l^\prime}. \label{ineq:WellSepLEeta}
\end{align}
This ensures that if $(\eta,l)$ and $(\xi,l^\prime)$ are in separated sectors in frequency, then the entire multiplier can generally be absorbed by the $\jap{\eta-\xi,l-l^\prime}^m$ factors and one will not need $\partial_t w/w$.   
Furthermore, from \eqref{ineq:Znonres} and \eqref{ineq:WellSepLEeta}, we can derive \eqref{ineq:basicNR}. 
These observations allow us to refine the analogous lemma of \cite{BGM15I} to deduce Lemma \ref{lem:MainFreqRat}.
 
As a representative example, let us consider the proof of \eqref{ineq:AiPartX}. 
First consider the case $\abs{l^\prime} \leq \frac{1}{5}\abs{\xi}$ and $\abs{l} \leq \frac{1}{5}\abs{\eta}$.
Then, as in \cite{BGM15I} (see therein for a proof), we have
\begin{align*}
\frac{\abs{k,\eta-kt,l} \abs{k}}{k^2 + (l^\prime)^2 + \abs{\xi-kt}^2} \jap{\frac{t}{\jap{\xi,l^\prime}}}^p \mathbf{1}_{\abs{l^\prime} \leq \frac{1}{5}\abs{\xi}} \mathbf{1}_{\abs{l} \leq \frac{1}{5}\abs{\eta}} & \lesssim \\ &  \hspace{-6cm}  \mathbf{1}_{\abs{l^\prime} \leq \frac{1}{5}\abs{\xi}} \mathbf{1}_{\abs{l} \leq \frac{1}{5}\abs{\eta}}\left(\left(\sqrt{\frac{\partial_t w(t,\eta)}{w(t,\eta)}}\mathbf{1}_{\abs{l} \leq \frac{1}{5}\abs{\eta}}  + \frac{\abs{k,\eta}^{s/2}}{\jap{t}^s} \right)\left(\sqrt{\frac{\partial_t w(t,\xi)}{w(t,\xi)}}\mathbf{1}_{\abs{l^\prime} \leq \frac{1}{5}\abs{\xi}}  + \frac{\abs{k,\xi}^{s/2}}{\jap{t}^s} \right) \right.  \\ & \left. \hspace{-6cm} \quad +  \frac{\mathbf{1}_{t > 2\min(\abs{\eta},\abs{\xi})}}{\jap{t}}\min\left(1,\frac{\abs{k, \eta-kt,l}}{\jap{kt}}  \right) \jap{\frac{t}{\jap{\xi,l^\prime}}}^p \right) \jap{\eta-\xi,l-l^\prime}^{4},  
\end{align*}
which is consistent with \eqref{ineq:AiPartX}. 
 
Next, consider the case ($\abs{l^\prime} > \frac{1}{5}\abs{\xi}$ or $\abs{l} > \frac{1}{5}\abs{\eta}$).
If the former is true than we immediately have the following by \eqref{ineq:Znonres}:  
\begin{align}
\frac{\abs{k,\eta-kt,l} \abs{k}}{k^2 + (l^\prime)^2 + \abs{\xi-kt}^2} \jap{\frac{t}{\jap{\xi,l^\prime}}}^p & \lesssim \frac{\abs{k,\eta-kt,l} \abs{k}}{\jap{l^\prime,\xi,kt}^2} \jap{\frac{t}{\jap{\xi,l^\prime}}}^p\jap{\eta-\xi,l-l^\prime}^2,  \label{ineq:nonres1}  
\end{align}
which is consistent with \eqref{ineq:AiPartX}. 
Next, consider instead $\abs{l} > \frac{1}{5}\abs{\eta}$. If $\abs{l^\prime} > \frac{1}{6}\abs{\xi}$ then \eqref{ineq:nonres1} (and hence \eqref{ineq:AiPartX}) follows again by \eqref{ineq:Znonres}. 
However, if $\abs{l^\prime} < \frac{1}{6}\abs{\xi}$ then by \eqref{ineq:WellSepLEeta}, $\abs{\eta,l} + \abs{\xi,l^\prime} \lesssim \abs{\eta-\xi,l-l^\prime}$, and we again have \eqref{ineq:nonres1} by multiplying and dividing by $\jap{\xi,l^\prime}^2$. 

The other inequalities are dealt with in a similar fashion. 
\end{proof}

For the current work, we need an analogue of Lemma \ref{lem:MainFreqRat} which is more precise in order to handle (and take advantage of) the regularity imbalances in $A^3$. 

\begin{lemma}[Frequency ratios for $\Delta_L$ involving regularity imbalances] \label{lem:MainFreqRat_RegImbalance}
For $t \geq 1$ and $k,k^\prime,\eta,\xi,l,l^\prime$,

Then for $p \in \Real$, we have the following:
\begin{itemize} 
\item for \textbf{(SI)} (for $k^\prime = k \neq 0$; recall that definition \eqref{def:freqcuts} depends on both $k$ and $k^\prime$): 
\begin{align} 
\frac{\abs{k,\eta-kt,l} \abs{k}}{k^2 + (l^\prime)^2 + \abs{\xi-kt}^2}\left(\sum_{r}\chi^{r,NR}\frac{t}{\abs{r} + \abs{\eta-tr}}\right) & \nonumber \\ 
&  \hspace{-6cm} \lesssim \left(\sqrt{\frac{\partial_t w(t,\eta)}{w(t,\eta)}} + \frac{\abs{k,\eta}^{s/2}}{\jap{t}^s} \right)\left(\sqrt{\frac{\partial_t w(t,\xi)}{w(t,\xi)}} + \frac{\abs{k,\xi}^{s/2}}{\jap{t}^s} \right) \jap{\eta-\xi,l-l^\prime}^{4}; \label{ineq:AiPartXA23}
\end{align}
\item a simpler variant (for $k^\prime = k \neq 0$): 
\begin{align} 
 \sum_{r}\chi^{r,NR}\frac{t}{\abs{r} + \abs{\eta-tr}} \hspace{1cm}  & \nonumber \\ & \hspace{-4cm} \lesssim \jap{t}\left(\sqrt{\frac{\partial_t w(t,\eta)}{w(t,\eta)}} + \frac{\abs{\eta,l}^{s/2}}{\jap{t}^{s}}\right)\left(\sqrt{\frac{\partial_t w(t,\xi)}{w(t,\xi)}} + \frac{\abs{\xi,l^\prime}^{s/2}}{\jap{t}^{s}}\right)\jap{\eta-\xi,l-l^\prime}^4. \label{ineq:jNRBasic} 
\end{align}
\item if $k^\prime,k \neq 0$, $k \neq k^\prime$, and $a \in [1,2]$ (for \textbf{(3DE)} terms with few derivatives),  
\begin{subequations} \label{ineq:3DEoneRegBal} 
\begin{align}
\frac{1}{\abs{k^\prime}^2 + \abs{l^\prime}^2 + \abs{\xi - k^\prime t}^2} \left(\sum_{r} \chi^{NR,r}\frac{\abs{r} + \abs{\eta-tr}}{t}\right) \jap{\frac{t}{\jap{\xi,l^\prime}}}^p \hspace{4cm} & \nonumber \\ 
& \hspace{-10cm} \lesssim \frac{1}{\jap{t}}\left(\sqrt{\frac{\partial_t w(t,\eta)}{w(t,\eta)}} + \frac{\abs{\eta}^{s/2}}{\jap{t}^s} \right)\left(\sqrt{\frac{\partial_t w(t,\xi)}{w(t,\xi)}} + \frac{\abs{\xi}^{s/2}}{\jap{t}^s} \right) \jap{\eta-\xi,l-l^\prime}^3 \nonumber \\ & \hspace{-10cm} \quad + \frac{1}{\jap{t}^2}\jap{\frac{t}{\jap{\xi,l^\prime}}}^p \jap{\eta-\xi,l-l^\prime}^3 \label{ineq:A3ReacGain}  \\
\frac{1}{\abs{k^\prime}^2 + \abs{l^\prime}^2 + \abs{\xi - k^\prime t}^2}\left(\chi^{R,NR}\frac{t}{\abs{k} + \abs{\eta-kt}} + \chi^{NR,R}\frac{\abs{k^\prime}  + \abs{\eta- k^\prime t}}{t} + \chi^{\ast;33} \right) & \lesssim \frac{\jap{\eta-\xi}^2}{\jap{t}} \label{ineq:A33ReacGain} \\ 
\hspace{-2cm} \frac{1}{\abs{k^\prime, \xi - k^\prime t, l^\prime}^a}\left(\sum_{r}\chi^{r,NR}\frac{t}{\abs{r} + \abs{\eta-tr}} + \chi^{\ast;23} \right) \jap{\frac{t}{\jap{\xi,l^\prime}}}^a  \lesssim \jap{\eta-\xi}^2; \label{ineq:A1A3ReacGain}
\end{align}
\end{subequations}
\item if $k^\prime,k \neq 0$ and $k \neq k^\prime$ (for \textbf{(3DE)} terms with more derivatives),  
\begin{subequations} \label{ineq:3DEoneRegBalII} 
\begin{align} 
\jap{\frac{t}{\jap{\xi,l^\prime}}} \frac{\abs{k,\eta-kt,l} \abs{k^\prime,\xi-k^\prime t, l^\prime}}{\abs{k^\prime}^2 + \abs{l^\prime}^2 + \abs{\xi-k^\prime t}^2} \left(\sum_{r}\chi^{r,NR}\frac{t}{\abs{r} + \abs{\eta-tr}} + \chi^{\ast;23} \right) \hspace{1.5cm} & \nonumber \\ & \hspace{-4cm} \lesssim \abs{k,\eta-kt,l} \jap{\eta-\xi,l-l^\prime} \label{ineq:jNRPneqneq} \\
\frac{\abs{k^\prime,\xi-tk^\prime,l^\prime} \abs{k}}{\abs{k^\prime}^2 + \abs{l^\prime}^2 + \abs{\xi-tk^\prime}^2}\left(\chi^{R,NR}\frac{t}{\abs{k} + \abs{\eta-kt}} + \chi^{NR,R}\frac{\abs{k^\prime} + \abs{\eta - k^\prime t}}{t}\right) \hspace{1.5cm} & \nonumber \\ 
& \hspace{-11cm} \lesssim \left(\sqrt{\frac{\partial_t w(t,\eta)}{w(t,\eta)}} + \frac{\abs{\eta}^{s/2}}{\jap{t}^s} \right)\left(\sqrt{\frac{\partial_t w(t,\xi)}{w(t,\xi)}} + \frac{\abs{\xi}^{s/2}}{\jap{t}^s} \right) \jap{k-k^\prime, \eta-\xi,l-l^\prime}^4 \label{ineq:A33PartX} \\  
\frac{\abs{k^\prime,\xi-tk^\prime,l^\prime} \abs{l^\prime}}{\abs{k^\prime}^2 + \abs{l^\prime}^2 + \abs{\xi-tk^\prime}^2}\left(\chi^{R,NR}\frac{t}{\abs{k} + \abs{\eta-kt}} + \chi^{NR,R}\frac{\abs{k^\prime} + \abs{\eta - k^\prime t}}{t}\right) \hspace{1.5cm} & \nonumber \\ 
& \hspace{-11cm} \lesssim \jap{t}\left(\sqrt{\frac{\partial_t w(t,\eta)}{w(t,\eta)}} + \frac{\abs{\eta}^{s/2}}{\jap{t}^s} \right)\left(\sqrt{\frac{\partial_t w(t,\xi)}{w(t,\xi)}} + \frac{\abs{\xi}^{s/2}}{\jap{t}^s} \right) \jap{k-k^\prime, \eta-\xi,l-l^\prime}^4 \label{ineq:A33PartXZ1} \\  
\frac{\abs{l k}}{\abs{k^\prime}^2 + \abs{l^\prime}^2 + \abs{\xi-tk^\prime}^2} \left(\chi^{R,NR}\frac{t}{\abs{k} + \abs{\eta-kt}} + \chi^{NR,R}\frac{\abs{k^\prime} + \abs{\eta - k^\prime t}}{t} + \chi^{\ast;33} \right) \nonumber \\
& \hspace{-10cm} \lesssim \jap{k-k^\prime,\eta-\xi,l-l^\prime}^3. \label{ineq:A3neqA3neqZX}
\end{align}
\end{subequations} 
\item for terms of type \textbf{(F)}, (with $k = 0$ and $k^\prime \neq 0$),   
\begin{subequations} \label{ineq:AdelLijRegBal}  
\begin{align} %
\frac{\abs{l} \jap{\eta,l}^2 \jap{\frac{t}{\jap{\xi,l^\prime}}}^{2}}{(k^\prime)^2 + (l^\prime)^2 + \abs{\xi-k^\prime t}^2}\left(\sum_{r}\chi^{r,NR}\frac{t}{\abs{r} + \abs{\eta - tr}}\right) \hspace{3.5cm} & \nonumber  \\ 
&  \hspace{-10cm} \lesssim \jap{t}^2\left(\sqrt{\frac{\partial_t w(t,\eta)}{w(t,\eta)}} + \frac{\abs{\eta}^{s/2}}{\jap{t}^s} \right)\left(\sqrt{\frac{\partial_t w(t,\xi)}{w(t,\xi)}} + \frac{\abs{\xi}^{s/2}}{\jap{t}^s} \right) \jap{k^\prime,\eta-\xi,l-l^\prime}^3 \label{ineq:AA3Z} \\ 
\frac{\abs{\eta} \jap{\eta,l}^2}{(k^\prime)^2 + (l^\prime)^2 + \abs{\xi-k^\prime t}^2}\left(\sum_{r}\chi^{NR,r}\frac{\abs{r} + \abs{\eta - tr}}{t} \right) \jap{\frac{t}{\jap{\xi,l^\prime}}} \hspace{2cm} & \nonumber \\ 
& \hspace{-10cm} \lesssim \left(\jap{t}^2\left(\sqrt{\frac{\partial_t w(t,\eta)}{w(t,\eta)}} + \frac{\abs{\eta}^{s/2}}{\jap{t}^s} \right)\left(\sqrt{\frac{\partial_t w(t,\xi)}{w(t,\xi)}} + \frac{\abs{\xi}^{s/2}}{\jap{t}^s} \right)  + \abs{\eta} \right) \jap{k^\prime,\eta-\xi,l-l^\prime}^3; \label{ineq:A03A2YYY} \\
\frac{\abs{\eta} \jap{\eta,l}^2}{(k^\prime)^2 + (l^\prime)^2 + \abs{\xi-k^\prime t}^2} \chi^{\ast;32} \jap{\frac{t}{\jap{\xi,l^\prime}}} \hspace{6cm} & \nonumber \\ 
& \hspace{-12cm} \lesssim \left(\jap{t}^2\left(\sqrt{\frac{\partial_t w(t,\eta)}{w(t,\eta)}}\mathbf{1}_{\abs{l} \leq \frac{1}{5}\abs{\eta}} + \frac{\abs{\eta}^{s/2}}{\jap{t}^s} \right)\left(\sqrt{\frac{\partial_t w(t,\xi)}{w(t,\xi)}}\mathbf{1}_{\abs{l^\prime} \leq \frac{1}{5}\abs{\xi}} + \frac{\abs{\xi}^{s/2}}{\jap{t}^s} \right)  + \abs{\eta} \right) \jap{k^\prime,\eta-\xi,l-l^\prime}^3; \label{ineq:A03A2YYY2}
\end{align} 
\end{subequations} 
\end{itemize} 
\end{lemma}
\begin{remark} \label{rmk:noFreqRestrict}
Note the lack of frequency restrictions to $\abs{l} < \frac{1}{5}\abs{\eta}$ and $\abs{l^\prime} < \frac{1}{5}\abs{\xi}$. This is due to the fact that these inequalities need to sometimes be applied in the overlap regions where $\abs{l} \approx \abs{\eta}$ and $\abs{l^\prime} \approx \abs{\xi}$. 
\end{remark}
\begin{proof} 
The proofs are very similar to Lemma \ref{lem:MainFreqRat} with some minor changes. 
Consider \eqref{ineq:AiPartXA23} (the analogue of \eqref{ineq:AiPartX}). 
We have, by Lemma \ref{lem:dtw},
\begin{align}
\frac{\abs{k,\eta-kt,l} \abs{k}}{k^2 + (l^\prime)^2 + \abs{\xi-kt}^2}\left(\sum_{r}\chi^{r,NR} \frac{t}{\abs{r} + \abs{\eta-tr}}\right) 
 & \nonumber \\ & \hspace{-6cm} 
\lesssim \sum_{r} \chi^{r,NR} \frac{\abs{k}t}{t\abs{k-r}\abs{r}}\frac{\partial_t w(t,\eta)}{w(t,\eta)}\jap{\eta-\xi,l-l^\prime} 
\label{ineq:jNRBasicPfineq}
\end{align}
from which the result follows by Lemma \ref{lem:CKwFreqRat} (and that $\chi^{r,NR}$ form a partition of unity for a certain region of frequencies).
The proof of \eqref{ineq:jNRBasic} is essentially the same. 

Consider \eqref{ineq:A3ReacGain}; the other inequalities in \eqref{ineq:3DEoneRegBal} are easy  variants of this and the proofs of \eqref{ineq:jNRBasic} above.
First, in the case $t \not\in \I_{k^\prime,\eta} \cap \I_{k^\prime,\xi}$, we have $\jap{\eta-\xi}\jap{\xi - k^\prime t} \gtrsim t$ by Lemma \ref{lem:wellsep}, and so \eqref{ineq:A3ReacGain} follows.    
Next, consider the case that $t \in \I_{k^\prime,\eta} \cap \I_{k^\prime,\xi}$. 
Then, since $k \neq k^\prime$, $t \not\in \I_{k,\eta}$ and this contribution appears in the sum as $\chi^{NR,k^\prime}$ (recall the definition \eqref{def:freqcuts}). 
In this case \eqref{ineq:A3ReacGain} follows by Lemma \ref{lem:CKwFreqRat}. This now covers all cases.  

Let us comment briefly on the proof of \eqref{ineq:A03A2YYY}. The term such that $r = k^\prime$ follows due to the Lemma \ref{dtw} together with the 
frequency restrictions ensuring $\abs{\eta}\jap{\eta,l}^2 \lesssim \jap{kt}^3$. For the terms $r \neq k^\prime$, we have
\begin{align*}
\frac{\abs{\eta} \jap{\eta,l}^2}{(k^\prime)^2 + (l^\prime)^2 + \abs{\xi-k^\prime t}^2}\chi^{NR,r} \lesssim \frac{\jap{rt}^3}{t^2 \abs{k-r}^2} \lesssim \jap{t}^2 \frac{\abs{r}}{t} \abs{k}^2, 
\end{align*}
which is consistent with \eqref{ineq:A03A2YYY} by Lemma \ref{dtw} again. 

The remaining estimates follow by similar arguments combined with the arguments used in the proof of Lemma \ref{lem:MainFreqRat} (see also \cite{BGM15I}). 
Hence, these are omitted for the sake of brevity. 
\end{proof} 

\subsection{Paraproducts and related notations} \label{sec:paranote}
We briefly recall the short-hands introduced in \cite{BGM15I}. 
For paraproducts we use the homogeneous variant of the paraproduct and utilize the following short-hand to suppress the appearance of Littlewood-Paley projections:
\begin{align} 
fg & = f_{Hi} g_{Lo} + f_{Lo} g_{Hi} + (fg)_{\mathcal{R}} \nonumber \\ 
& = \sum_{M \in 2^\Integers} f_{M} g_{<M/8} + \sum_{M \in 2^{\Integers}} f_{<M/8} g_{M} + \sum_{M \in 2^{\Integers}} \sum_{M/8 \leq M^\prime \leq 8M} f_{M} g_{M^\prime}. \label{def:parapp}
\end{align} 
We recall the following lemma from \cite{BGM15I} for using the paraproducts in $L^2$ estimates. 
\begin{lemma}[Paraproducts for quadratic nonlinearities] \label{gevreyparaproductlemma}
Let $s\in[0,1)$, $\mu \geq 0$, $p \geq 0$.  Then, there exists a $c = c(s) \in (0,1)$ such that the following holds,   
\begin{subequations} \label{ineq:paraquad} 
\begin{align}
\norm{f_{Hi} g_{Lo}}_{\G^{\mu,p}} & \lesssim \norm{f}_{\G^{\mu,p}}\norm{g}_{\G^{c\mu,3/2+}}  \label{ineq:quadHL} \\ 
\norm{(fg)_{\mathcal{R}}}_{\G^{\mu,p}} & \lesssim \norm{f}_{\G^{c\mu,p}} \norm{g}_{\G^{c\mu,3/2+}} \label{ineq:quadR} \\ 
\int e^{\mu\abs{\grad}^s}\jap{\grad}^{p} h \, e^{\mu\abs{\grad}^s}\jap{\grad}^{p} \left(f_{Hi} g_{Lo}\right) dV & \lesssim \norm{h}_{\G^{\mu,p}}\norm{f}_{\G^{\mu,p}}\norm{g}_{\G^{c\mu,3/2+}}.   \label{ineq:triQuadHL}
\end{align}
\end{subequations}  
\end{lemma} 
\begin{remark} 
In most places in the proof, $\mu = 0$ as normally the multipliers $A^i$ or $A^{\nu;i}$ are playing the role of the norm.
\end{remark} 

Many of the nonlinear terms are higher order (up to quintic). 
For expanding cubic nonlinear terms, we use the short-hand from \cite{BGM15I}:  
\begin{align} 
fgh & = \sum_{N \in 2^{\Integers}} f_N g_{<N/8}h_{<N/8} + g_{N} f_{<N/8} h_{<N/8} + f_{< N/8} g_{<N/8} h_N + (fgh)_{\mathcal{R}}  \nonumber \\ 
& := f_{Hi}(gh)_{Lo} + g_{Hi}(fh)_{Lo}  + h_{Hi}(gf)_{Lo} + (fgh)_{\mathcal{R}}, \label{def:cubic}
\end{align} 
where the remainder term $(fgh)_{\mathcal{R}}$, includes all of the frequency contributions not included in the leading order terms. 
Note the short-hand $(gh)_{Lo} = g_{Lo} h_{Lo}$. 
By iterating this pattern, we obtain also decompositions for quartic and quintic terms. 
We also have the equivalents of \eqref{ineq:quadHL}, \eqref{ineq:quadR} and \eqref{ineq:triQuadHL}. 
\begin{lemma}[Paraproducts for higher order nonlinear terms] \label{lem:ParaHighOrder}
For all $\mu \geq 0$ and $p \geq 0$, there is some $c = c(s) \in (0,1)$ such that 
\begin{subequations}  \label{ineq:quin}
\begin{align} 
\norm{g_{Hi} (fhkj)_{Lo}}_{\G^{\mu,p}} & \lesssim_{p} \norm{g}_{\G^{\mu,p}}\norm{f}_{\G^{c\mu,3/2+}} \nonumber \\ & \quad\quad \times \norm{h}_{\G^{c\mu,3/2+}} \norm{k}_{\G^{c\mu,3/2+}} \norm{j}_{\G^{c\mu,3/2+}} \label{ineq:quinHL} \\ 
\norm{(fghkj)_{\mathcal{R}}}_{\G^{\mu,p}} & \lesssim_{p} \norm{g}_{\G^{c\mu,3/2+}}\norm{f}_{\G^{c\mu,3/2+}} \norm{h}_{\G^{c\mu,3/2+}} \nonumber \\ & \quad\quad \times \norm{k}_{\G^{c\mu,3/2+}} \norm{j}_{\G^{c\mu,3/2+}} \label{ineq:quinR} \\  
\int e^{\mu \abs{\grad}^s}\jap{\grad}^p q e^{\mu\abs{\grad}^s} \jap{\grad}^p (g_{Hi}(fhkj)_{Lo}) dV & \lesssim_{p} \norm{q}_{\G^{\mu,p}}\norm{g}_{\G^{\mu,p}} \norm{f}_{\G^{c\mu,3/2+}} \nonumber \\ & \quad\quad \times \norm{h}_{\G^{c\mu,3/2+}} \norm{k}_{\G^{c\mu,3/2+}} \norm{j}_{\G^{c\mu,3/2+}}. 
\end{align}
\end{subequations}
Analogous estimates hold also for the cubic and quartic decompositions. 
\end{lemma} 

One final short-hand we recall from \cite{BGM15I} involves the inner products that appear naturally in energy estimates. 
Consider, for example, a typical Gevrey energy estimate involving three quantities $f,g,h$, where generally $h$ will be a product of several low frequency terms: 
\begin{align*} 
\int e^{\lambda\abs{\grad}^s}f  e^{\lambda\abs{\grad}^s}\left(g_{Hi} h_{Lo}\right) dV & = \frac{1}{(2\pi)^{3/2}}\sum_{k,l,k^\prime,l^\prime} \int_{\eta,\xi} e^{\lambda\abs{k,\eta,l}^s}\overline{\hat{f}}_k(\eta,l)  e^{\lambda\abs{k,\eta,l}^s} \hat{g}_{k^\prime}(\xi,l^\prime)_{Hi} \hat{h}_{k-k^\prime}(\eta-\xi,l-l^\prime)_{Lo} d\eta d\xi. 
\end{align*} 
By the frequency localizations inherent in the shorthand and \eqref{lem:scon}, for some $c = c(s) \in (0,1)$ we have (by \eqref{ineq:triQuadHL}),  
\begin{align*}
\int e^{\lambda\abs{\grad}^s}f  e^{\lambda\abs{\grad}^s}\left(g_{Hi} h_{Lo}\right) dV & \lesssim  \sum_{k,l,k^\prime,l^\prime} \int_{\eta,\xi} e^{\lambda\abs{k,\eta,l}^s}\abs{\hat{f}_k(\eta,l)}  e^{\lambda\abs{k^\prime,\xi,l^\prime}^s} \abs{\hat{g}_{k^\prime}(\xi,l^\prime)_{Hi}} \\  & \quad\quad \times e^{c\lambda\abs{k-k^\prime,\eta-\xi,l-l^\prime}^s}\abs{\hat{h}_{k-k^\prime}(\eta-\xi,l-l^\prime)_{Lo}} d\eta d\xi \\ 
& \lesssim \norm{f}_{\G^{\lambda}}\norm{g}_{\G^{\lambda}} \norm{h}_{\G^{c\lambda,3/2+}}.
\end{align*}  
The low frequency factors will generally all be put in a norm $\G^{\lambda,3/2+}$ (once the estimates are over we do not need to worry about the $c$) and hence it makes sense 
to use a short-hand for the low-frequency factor as $\norm{h}_{\G^{\lambda,3/2+}}Low(k-k^\prime,\eta-\xi,l-l^\prime)$ where the function $Low$ is taken as an $O(1)$ function in $\G^{\lambda,3/2+}$ (and which can change line-to-line as implicit constants). 
For example,
\begin{align} 
\int e^{\lambda\abs{\grad}^s}f  e^{\lambda\abs{\grad}^s}\left(g_{Hi} h_{Lo}\right) dV & := \norm{h}_{\G^{\lambda,3/2+}}\sum_{k,l,k^\prime,l^\prime} \int_{\eta,\xi} e^{\lambda\abs{k,\eta,l}^s}\overline{\hat{f}}_k(\eta,l)  e^{\lambda\abs{k,\eta,l}^s} \hat{g}_{k^\prime}(\xi,l^\prime)_{Hi} \nonumber \\ & \quad\quad \times Low(k-k^\prime,\eta-\xi,l-l^\prime) d\eta d\xi \nonumber \\ 
 & \lesssim \norm{h}_{\G^{\lambda,3/2+}}\sum_{k,l,k^\prime,l^\prime} \int_{\eta,\xi} e^{\lambda\abs{k,\eta,l}^s}\abs{\hat{f}_k(\eta,l)}  e^{\lambda\abs{k^\prime,\xi,l^\prime}^s} \abs{\hat{g}_{k^\prime}(\xi,l^\prime)_{Hi}}\nonumber \\ & \quad\quad \times Low(k-k^\prime,\eta-\xi,l-l^\prime) d\eta d\xi \nonumber \\
& \lesssim \norm{f}_{\G^{\lambda}}\norm{g}_{\G^{\lambda}} \norm{h}_{\G^{c\lambda,3/2+}}. \label{def:Low}
\end{align} 
The utility of this short-hand will quickly become clear in the course of the proof. 

\subsection{Product lemmas and a few immediate consequences} 
First, note the following product lemma is an immediate consequence of Lemma \ref{gevreyparaproductlemma}. 
\begin{lemma}[Gevrey product lemma] \label{lem:GevProdAlg}
For all $s \in (0,1)$, $\mu \geq 0$, and $p \geq 0$, there exists $c = c(s) \in (0,1)$ such that the following holds for all $f,g \in \mathcal{G}^{\mu,p}$:
\begin{subequations}
\begin{align} 
\norm{fg}_{\G^{\mu,p}} & \lesssim_{p} \norm{f}_{\G^{c\mu,3/2+}} \norm{g}_{\G^{\mu,p}} + \norm{g}_{\G^{c\mu,3/2+}} \norm{f}_{\G^{\mu,p}}, \label{ineq:GProduct}
\end{align}
\end{subequations}
in particular, if $\mu > 0$, then $\mathcal{G}^{\mu,p}$ is an algebra for all $p \geq 0$ by \eqref{ineq:SobExp}:  
\begin{align} 
\norm{f g}_{\G^{\mu,\sigma}} & \lesssim_{p,\mu} \norm{f}_{\G^{\mu,p}} \norm{g}_{\G^{\mu,p}}. \label{ineq:GAlg} 
\end{align} 
\end{lemma}

Next we have the following, which is a simple variant of the analogous lemma from \cite{BGM15I}. 
\begin{lemma}[Product lemma for $A$ and $A^i$] \label{lem:AAiProd}
Let $p \geq 0$ and $r \geq -\sigma$. Then there exists a $c = c(s) \in (0,1)$ such that for $i \in \set{1,2}$, for all $f,g$, 
\begin{subequations} 
\begin{align} 
\norm{\abs{\grad}^p \jap{\grad}^{r}A^i(fg)}_2 & \lesssim \norm{f}_{\G^{c\lambda,3/2+}}\norm{\abs{\grad}^p \jap{\grad}^{r} A^ig}_2 \nonumber \\ & \quad + \norm{g}_{\G^{c\lambda,3/2+}}\norm{\abs{\grad}^p \jap{\grad}^{r} A^if}_2 \label{ineq:Aprodi} \\ 
\norm{\left(\sqrt{\frac{\partial_t w}{w}}\tilde{A}^i + \frac{\abs{\grad}^{s/2}}{\jap{t}^s}A^i \right)(fg)}_2 & \lesssim \norm{f}_{\G^{c\lambda,3/2+}}\norm{\left(\sqrt{\frac{\partial_t w}{w}}\tilde{A}^i + \frac{\abs{\grad}^{s/2}}{\jap{t}^s}A^i \right) g}_2 \nonumber \\ & \quad + \norm{g}_{\G^{c\lambda,3/2+}}\norm{\left(\sqrt{\frac{\partial_t w}{w}}\tilde{A}^i + \frac{\abs{\grad}^{s/2}}{\jap{t}^s}A^i \right)f}_2. 
\end{align}
\end{subequations}
If $f$ and $g$ are both independent of $X$, then the above holds also with $A^i$ replaced by either $A$ or $A^3$.  
\end{lemma}

\begin{remark} 
Notice the crucial detail that Lemma \ref{lem:AAiProd} does \emph{not} hold for $A^3$ if $f$ or $g$ depend on $X$ due to the regularity imbalances near the critical times.  
\end{remark} 

Together with \eqref{def:psi2sqrBrack}, Lemma \ref{lem:AAiProd} and Lemma \ref{lem:GevProdAlg} imply the following lemma (as long as $C^i$ remains sufficiently small). 
The proof is straightforward so we omit it for the sake of brevity. 

\begin{lemma}[Coefficient control] \label{lem:CoefCtrl} 
Let 
\begin{subequations} \label{def:G}
\begin{align} 
G_{yy} & = \left((1+\psi_y)^2 + \psi_z^2\right) - 1 \\
G_{yz} & =2\phi_y(1+\psi_y) + 2\psi_z(1 + \phi_z) \\ 
G_{zz} & = \left((1 + \phi_z)^2 + \phi_y^2 \right) - 1. 
\end{align}
\end{subequations} 
Under the bootstrap hypotheses, for $c_0$ sufficiently small, we have for any $G \in \set{\psi_y,\psi_z,\phi_y,\phi_z,G_{yy},G_{yz},G_{zz}}$, 
\begin{subequations} 
\begin{align} 
\norm{\jap{\grad}^{-1} A G}_2 & \lesssim \norm{A C}_2 \\ 
\norm{A G}_2 & \lesssim \norm{\grad A C}_2 \\ 
\norm{\jap{\grad}^{-1} \sqrt{\frac{\partial_t w}{w}} \tilde{A} G}_2  & \lesssim \norm{\left(\sqrt{\frac{\partial_t w}{w}}\tilde{A} + \frac{\abs{\grad}^{s/2}}{\jap{t}^s}A\right) C}_2 \\
\norm{\jap{\grad}^{-1} \abs{\grad}^{s/2} A G}_2 & \lesssim \norm{\abs{\grad}^{s/2}A C}_2. 
\end{align}
\end{subequations}
Further, 
\begin{subequations} \label{ineq:CCDeltatC}
\begin{align}
\norm{\jap{\grad}^{-2} A \Delta_t C^i}_2 & \lesssim \norm{A C}_2 \\ 
\norm{\jap{\grad}^{-1} A \Delta_t C^i}_2 & \lesssim \norm{\grad A C}_2 \\ 
\norm{\sqrt{\frac{\partial_t w}{w}} \jap{\grad}^{-2} \tilde{A} \Delta_t C^i}_2  & \lesssim \norm{\left(\sqrt{\frac{\partial_t w}{w}}\tilde{A} + \frac{\abs{\grad}^{s/2}}{\jap{t}^s}A\right) C}_2 \\
\norm{\abs{\grad}^{s/2} \jap{\grad}^{-2} A \Delta_t C^i}_2 & \lesssim \norm{\abs{\grad}^{s/2}A C}_2. 
\end{align}
\end{subequations}
Similarly, for any $\lambda(t) \geq \mu > 0$ and $\sigma \geq p \geq 0$ (the constant can be taken independent of $\mu$  for $p > 1$): 
\begin{align} 
\norm{G}_{\G^{\mu,p}} + \norm{\Delta_t C}_{\G^{\mu,p-1}} \lesssim \norm{\grad C}_{\G^{\mu,p}}. \label{ineq:psiCLow} 
\end{align}
\end{lemma}  

\begin{remark} \label{rmk:CoefCtrlLow} 
As discussed in \cite{BGM15I}, a consequence of \eqref{ineq:psiCLow} together with \eqref{ineq:Boot_LowC} implies that when coefficients appear in `low frequency' in a paraproduct they satisfy the a priori estimate $O(\epsilon \jap{t})$. 
Together with $\epsilon t \jap{\nu t^3}^{-1} \lesssim \jap{t}^{-1}$, this implies that when there is enhanced dissipation present, we generally need only treat the leading order terms that arise from the approximation $\partial_i^t \approx \partial_i^L$ or the terms that arise when the coefficients are in high frequency. 
\end{remark}

\begin{remark} \label{rmk:SIcoefneglect} 
Even when enhanced dissipation is not present, the coefficients do not depend on $X$ and hence the presence of the coefficients do not shift the frequencies in $X$. 
This will mean that even when there are no powers of $\jap{\nu t^3}^{-1}$, terms in which coefficients appear in low frequency are generally treatable with an easy variant of the treatment used on the leading order terms. 
There are a few exceptions, when the structure of the term is changed by the coefficients, and otherwise these terms are generally omitted. 
\end{remark}

We recall the following lemma from \cite{BGM15I}. 
 
\begin{lemma}[$A^\nu$ Product Lemma] \label{lem:AnuProd} 
The following holds for all $f^1$ and $f^2$ such that $f^2_{\neq} = f^2$, 
\begin{align} 
\norm{A^{\nu;i}(f^1 f^2)}_2 & \lesssim \norm{f^1}_{\G^{\lambda,\beta + 3\alpha+ 3/2+}}\norm{A^{\nu;i}f^2}_2. \label{ineq:AnuiDistri}  
\end{align} 
Moreover, if also $f^1_{\neq} = f^1$ then we have the product-type inequalities 
\begin{subequations} \label{ineq:AnuiDistriDecay}
\begin{align} 
\norm{A^{\nu;1}(f^1 f^2)}_2 \lesssim \frac{\jap{t}^{2+\delta_1}}{\jap{\nu t^3}^\alpha}\left(\norm{\jap{\grad}^{2-\beta} A^{\nu;1} f^1}_{2}\norm{A^{\nu;1} f^2}_2 + \norm{A^{\nu;1} f^1}_{2}\norm{\jap{\grad}^{2-\beta} A^{\nu;1} f^2}_2 \right) \label{ineq:AnuiDistriDecay1} \\ 
\norm{A^{\nu;2}(f^1 f^2)}_2 \lesssim \frac{\jap{t}}{\jap{\nu t^3}^\alpha}\left(\norm{\jap{\grad}^{2-\beta} A^{\nu;2} f^1}_{2}\norm{A^{\nu;2} f^2}_2 + \norm{A^{\nu;2} f^1}_{2}\norm{\jap{\grad}^{2-\beta} A^{\nu;2} f^2}_2 \right) \label{ineq:AnuiDistriDecay2} \\ 
\norm{A^{\nu;3}(f^1 f^2)}_2 \lesssim \frac{\jap{t}^{2}}{\jap{\nu t^3}^\alpha}\left(\norm{\jap{\grad}^{2-\beta} A^{\nu;3} f^1}_{2}\norm{A^{\nu;3} f^2}_2 + \norm{A^{\nu;3} f^1}_{2}\norm{\jap{\grad}^{2-\beta} A^{\nu;3} f^2}_2 \right). \label{ineq:AnuiDistriDecay3}
\end{align} 
\end{subequations}  
\end{lemma}

\section{High norm estimate on $Q^2$}
First compute the time evolution of $A^2 Q^2$ in $L^2$: 
\begin{align} 
\frac{1}{2}\frac{d}{dt}\norm{A^{2} Q^2}_2^2 & \leq \dot{\lambda}\norm{\abs{\grad}^{s/2}A^{2} Q^2}_2^2 - \norm{\sqrt{\frac{\partial_t w}{w}}\tilde{A}^{2} Q^2}_2^2  - \frac{1}{t}\norm{\mathbf{1}_{t > \jap{\grad_{Y,Z}}} A^{2}Q_{\neq}^2}_2^2 \nonumber \\
 & \quad -\norm{\sqrt{\frac{\partial_t w_L}{w_L}} A^{2} Q^2}_2^2  + \nu \int A^{2} Q^2 A^{2} \left(\tilde{\Delta_t} Q^2\right) dV -\int A^{2} Q^2 A^{2} \left( \tilde U \cdot \grad Q^2 \right) dV \nonumber \\
& \quad - \int A^{2} Q^2 A^{2}\left(Q^j \partial_j^t U^2  + 2\partial_i^t U^j \partial_{i}^t \partial_{j}^t U^i - \partial_Y^t\left(\partial_i^t U^j \partial_j^t U^i\right) \right) dV \nonumber \\ 
& = - \mathcal{D}Q^2 - CK^2_L + \mathcal{D}_E + \mathcal{T} + NLS1 + NLS2 + NLP, \label{eq:A2Q2Evo}
\end{align} 
where we used the definition
\begin{align}
\mathcal{D} = -\nu\norm{\sqrt{-\Delta_L}A^2 Q^2}_2^2 + \mathcal{D}_E. 
\end{align} 
Recall the following enumerations from \cite{BGM15I}. 
For $i,j\in \set{1,2,3}$ and $a,b \in \set{0,\neq}$: 
\begin{subequations} \label{def:Q2Enums}
\begin{align}
NLP(i,j,a,b) &= \int A^2 Q^2_{\neq} A^2\left( \partial_Y^t \left(\partial_j^t U^i_a \partial_i^t U^j_b \right) \right) dV \\ 
NLS1(j,a,b) & = -\int A^2 Q^2_{\neq} A^2\left(Q^j_a\partial_j^t U^2_{b}\right) dV \\
NLS2(i,j,a,b) & = -\int A^2 Q^2_{\neq} A^2\left(\partial_i^t U^j_a \partial_i^t\partial_j^t U^2_{b}\right) dV \\
NLP(i,j,0) & = \int A^2 Q^2_{0} A^2\left( \partial_Y^t \left(\partial_j^t U^i_0 \partial_i^t U^j_0 \right) \right) dV \\ 
NLS1(j,0) & = -\int A^2 Q^2_{0} A^2\left(Q^j_0\partial_j^t U^2_{0}\right) dV \\
NLS2(i,j,0) & = -\int A^2 Q^2_{0} A^2\left(\partial_i^t U^j_0 \partial_i^t\partial_j^t U^2_{0}\right) dV \\ 
\mathcal{F} & = -\int A^2 Q^2_{0} A^2\left(\partial_i^t \partial_i^t \partial_j^t \left(U^j_{\neq} U^2_{\neq}\right)_0 - \partial_{Y}^t \partial_j^t \partial_i^t \left(U^i_{\neq} U^j_{\neq}\right)_0\right)dV \\ 
\mathcal{T}_0 & = -\int A^{2} Q_0^2 A^{2} \left( g \partial_Y Q_{0}^2 \right) dV  \\ 
\mathcal{T}_{\neq} & = -\int A^{2} Q_{\neq}^2 A^{2} \left( \tilde U \cdot \grad Q^2 \right) dV 
\end{align}
\end{subequations} 

Note that we have split $\mathcal{T}$ into three contributions: $\mathcal{T}_0$ (the \textbf{(2.5NS)} interactions), $\mathcal{T}_{\neq}$ (the \textbf{(SI)} and \textbf{(3DE)} interactions), and a contribution that is grouped with $\mathcal{F}$ (the \textbf{(F)} interactions).  Similarly, we have split the $NLS$ and $NLP$ terms into several contributions: $NLS1(j,0)$, $NLS2(i,j,0)$, and $NLP(i,j,0)$ (the \textbf{(2.5NS)} interactions), the $NLS1(j,a,b)$, $NLS2(i,j,a,b)$, and $NLP(i,j,a,b)$ (the \textbf{(SI)} and \textbf{(3DE)} interactions), and a contribution that is grouped with $\mathcal{F}$ (the \textbf{(F)} interactions). 
This kind of subdivision will be used repeatedly in the sequel.  

\subsection{Zero frequencies} \label{sec:AQ2Zero}

\subsubsection{Transport nonlinearity} \label{sec:TransQ20}
Turn first to $\mathcal{T}_0$, the \textbf{(2.5NS)} contribution to the transport nonlinearity. 
From Lemma \ref{lem:AAiProd},  
\begin{align*}  
\mathcal{T}_{0} & \lesssim \norm{A^2 Q_0^2}_2 \left( \norm{Ag}_2 \norm{Q^2_0}_{\G^{\lambda,\gamma}} + \norm{g}_{\G^{\lambda,\gamma}} \norm{\grad A^2 Q^2_0}_2 \right) \\ 
& \lesssim \norm{A^2 Q_0^2}_2 \left(\epsilon \norm{Q^2_0}_{\G^{\lambda,\gamma}} + \frac{\epsilon}{\jap{t}^{2}} \norm{\grad A^2 Q^2_0}_2 \right) \\ 
& \lesssim \epsilon^{3/2}\norm{\grad A^2 Q^2}_2^2 + \left(\frac{\epsilon^{1/2}}{\jap{t}^{4}} + \epsilon\right)\norm{A^2 Q^2}_2^2, 
\end{align*}
which is consistent with Proposition \ref{prop:Boot} by absorbing first term with the dissipation and integrating in time, provided $c_0$, and $\epsilon$ (equivalently $\nu$) are chosen sufficiently small. 

\subsubsection{Nonlinear pressure and stretching} \label{sec:NLPSQ20}
These terms correspond to the nonlinear zero frequency interactions in the pressure and stretching terms, and so are of type \textbf{(2.5NS)}.
Unlike in \cite{BGM15I}, $A^2_0 \neq A^{3}_0$: near the critical times, we have less control over $Q^3_0$. 
Therefore, the most difficult contributions will come from terms which involve two derivatives of $Q^3$. 
Consider $NLP(3,3,0)$ as a representative example; the other contributions are all treated with a similar approach (or are easier) and hence are omitted for the sake of brevity. 
We expand with a paraproduct and group any terms where the coefficients appear in low frequency with the remainders: 
\begin{align*}
NLP(3,3,0) & = 2\int A^2 Q^2_0 A^2 \partial_Y\left( (\partial_Z U^3_0)_{Hi} (\partial_Z U^3_0)_{Lo} \right) dV \\ 
& \quad + \int A^2 Q^2_0 A^2\left(\left((\psi_y)_{Hi} \partial_Y + (\phi_y)_{Hi}\partial_Z \right) \left( (\partial_Z U^3_0)_{Lo} (\partial_Z U^3_0)_{Lo} \right) \right) dV \\ 
& \quad + \int A^2 Q^2_0 A^2 \partial_Y\left( \left((\phi_z)_{Hi}\partial_Z + (\psi_z)_{Hi} \partial_Y\right) (U^3_0)_{Lo} (\partial_Z U^3_0)_{Lo} \right) dV \\ 
& \quad + P_{\mathcal{R},C} \\ 
& = P_{HL} + P_{C1} + P_{C2} + P_{\mathcal{R},C}. 
\end{align*}
Turn to $P_{HL}$ first. 
By \eqref{ineq:AprioriU0} and \eqref{ineq:ABasic} we have
\begin{align*}
P_{HL} & \lesssim \epsilon\sum_{l,l^\prime} \int \abs{\widehat{Q_0^2}(\eta,l)} \left(\sum_{r} \tilde{A}^2_0(\eta,l) \tilde{A}_0^3(\xi,l^\prime) \chi^{r,NR}\frac{t}{\abs{r} + \abs{\eta-tr}} + \chi^{\ast;23}  A^2_0(\eta,l) A_0^3(\xi,l^\prime) \right) \\ & \quad\quad \times \abs{\Delta_L \widehat{U^3_0}(\xi,l^\prime)_{Hi}} Low(\eta-\xi,l-l^\prime) d\eta d\xi,  
\end{align*}
which by \eqref{ineq:jNRBasic}, \eqref{ineq:triQuadHL} gives (along with $\epsilon t \leq c_0$), 
\begin{align}
P_{HL} & \lesssim  \epsilon t \norm{\left(\sqrt{\frac{\partial_t w}{w}} \tilde{A}^2 + \frac{\abs{\grad}^{s/2}}{\jap{t}^{s}}A^2\right) Q^2_0}_2 \norm{\left(\sqrt{\frac{\partial_t w}{w}}\tilde{A}^3 + \frac{\abs{\grad}^{s/2}}{\jap{t}^{s}} A^3\right) \Delta_L U^3_0}_2 \nonumber \\ 
& \quad + \epsilon \norm{A^2 Q^2_0}_2 \norm{\Delta_L A^3 U_0^3}_2 \nonumber \\  
& \lesssim c_0 \norm{\left(\sqrt{\frac{\partial_t w}{w}}\tilde{A}^2 + \frac{\abs{\grad}^{s/2}}{\jap{t}^{s}}A^2\right)  Q^2_0}_2^2 +  c_0\norm{\left(\sqrt{\frac{\partial_t w}{w}}\tilde{A}^3 + \frac{\abs{\grad}^{s/2}}{\jap{t}^{s}}A^3\right) \Delta_L U^3_0}_2^2 \nonumber \\ 
& \quad + \epsilon \norm{A^2 Q^2_0}_2^2 + \epsilon \norm{\Delta_L A^3 U_0^3}_2^2. \label{ineq:PHLQ2} 
\end{align}
By Lemmas \ref{lem:PELbasicZero} and \ref{lem:PELCKZero}, this is consistent with Proposition \ref{prop:Boot}  for $c_0$ sufficiently small and $t \leq c_0 \epsilon^{-1}$ by absorbing the leading terms with the dissipation energies and integrating in time. 

Of the coefficient error terms, $P_{C2}$ is the most difficult; we treat only this case and omit the others. 
By \eqref{ineq:ABasic}, \eqref{ineq:AprioriU0}, and \eqref{ineq:triQuadHL}, followed by Lemma \ref{lem:CoefCtrl},
\begin{align*}
P_{C2} & \lesssim \epsilon^2 \sum_{l,l^\prime} \int \abs{A^2_0 \widehat{Q_0^2}(\eta,l) \frac{A(\xi,l^\prime) \abs{\eta}}{\jap{\xi,l^\prime}^2} \left( \abs{\widehat{\psi_y}(\xi,l^\prime)_{Hi}} + \abs{\widehat{\phi_y}(\xi,l^\prime)_{Hi}} \right)}Low(\eta-\xi,l-l^\prime) d\eta d\xi \\   
& \lesssim \epsilon^2 \norm{A^2 Q^2_0}_2 \left( \norm{\jap{\grad}^{-1} A \psi_y}_2 +  \norm{\jap{\grad}^{-1} A \phi_y}_2\right) \\
& \lesssim \epsilon^2 \norm{A^2 Q^2_0}_2 \norm{AC}_2 \\ 
& \lesssim \epsilon \norm{A^2 Q^2_0}_2^2 + \epsilon^3 \norm{AC}_2^2, 
\end{align*}
which is consistent with Proposition \ref{prop:Boot} for $c_0$ sufficiently small after integrating in time.  

The remainder terms are similar, or easier than, the terms treated above and hence these are omitted for brevity. This completes $NLP(3,3,0)$; the other $NLP$ terms are similar or easier and are hence omitted as well.  

\subsubsection{Forcing from non-zero frequencies} \label{sec:NzeroForcing}
Turn next to nonlinear interactions of type \textbf{(F)}: the interaction of two $X$ frequencies $k$ and $-k$ and sub-divide via 
\begin{align*} 
\mathcal{F}  & = -\int A^2 Q^2_0 A^2\left(\partial_Z^t \partial_Z^t \partial_j^t \left(U^j_{\neq} U^2_{\neq}\right)_0 - \partial_Y^t \partial_Y^t \partial_Z^t \left(U^3_{\neq} U^2_{\neq}\right)_0  - \partial_Y^t \partial_Z^t \partial_Z^t \left(U^3_{\neq} U^3_{\neq}\right)_0\right) dV  \\ 
& = F^1 + F^2 + F^3.     
\end{align*}
As in \cite{BGM15I}, all three are treated via variants of the same basic approach which will ultimately come down to applying the appropriate multiplier estimate in \eqref{ineq:AdelLij} or \eqref{ineq:AdelLijRegBal} depending on the combination of derivatives present. However, the situation here is more complicated than in \cite{BGM15I} due to the additional regularity loss in non-resonant modes of $Q^3$ near the critical times. 
 
We expand $F^3$ with a paraproduct and group terms where the coefficients appear in low frequency with the remainder: 
\begin{align*}  
F^3 & = 2\sum_{k\neq 0} \int A^{2}Q^2_0 A_0^{2} \partial_Y\partial_Z\partial_Z\left( \left(U^3_{-k}\right)_{Hi} \left( U^3_k\right)_{Lo}\right)dV \\
& \quad +\sum_{k\neq 0}  \int A^{2}Q^2_0 A_0^{2} \left( \left( (\psi_y)_{Hi} \partial_Y + (\phi_y)_{Hi}\partial_Z \right) \partial_Y\partial_Z\left( \left(U^3_{-k}\right)_{Lo} \left( U^3_k\right)_{Lo}\right)\right) dV \\ 
& \quad + \sum_{k\neq 0} \int A^{2}Q^2_0 A_0^{2} \left( \partial_Y \left( (\psi_z)_{Hi}\partial_Y + (\phi_z)_{Hi}\partial_Z \right) \partial_Z\left( \left(U^3_{-k}\right)_{Lo} \left( U^3_k\right)_{Lo}\right) \right) dV \\
& \quad + \sum_{k\neq 0} \int A^{2}Q^2_0 A_0^{2} \left( \partial_Y \partial_Z \left((\psi_z)_{Hi} \partial_Y + (\phi_z)_{Hi}\partial_Z \right) \left( \left(U^3_{-k}\right)_{Lo} \left( U^3_k\right)_{Lo}\right) \right) dV \\ 
& \quad + F^2_{\mathcal{R},C} \\ 
& = F^3_{HL} + F^3_{C1} + F^3_{C2} + F^3_{C3} + F^3_{\mathcal{R},C},    
\end{align*} 
where here $F^3_{\mathcal{R},C}$ includes all of the remainders from the quintic paraproduct as well as the higher order terms involving coefficients as low frequency factors. 

Turn first to $F^3_{HL}$ (recall \eqref{ineq:AprioriUneq} and the shorthand discussed in \eqref{def:Low} above) which by \eqref{ineq:ABasic} is given by 
\begin{align*} 
F_{HL}^3 & \lesssim \frac{\epsilon}{\jap{\nu t^3}^{\alpha}} \sum_{k\neq 0} \sum_{l,l^\prime} \int \abs{A^{2} \widehat{Q^2_0}(\eta,l) A_0^2(\eta,l) \frac{\abs{l}^2\abs{\eta} }{k^2 + (l^\prime)^2 + \abs{\xi-kt}^2} \Delta_L \widehat{U^3_{k}}(\xi,l^\prime)_{Hi}} Low(-k,\eta-\xi,l-l^\prime) d\eta d\xi \\ 
& \lesssim \frac{\epsilon}{\jap{\nu t^3}^{\alpha}} \sum_{k\neq 0} \sum_{l,l^\prime} \int \abs{\widehat{Q^2_0}(\eta,l)} \frac{\abs{l}^2\abs{\eta} }{k^2 + (l^\prime)^2 + \abs{\xi-kt}^2} \jap{\frac{t}{\jap{\xi,l^\prime}}}^{2} \\ & \quad\quad \times \left(\sum_{r} \chi^{r,NR}\frac{t}{\abs{r} + \abs{\eta-tr}}\tilde{A}_0^2(\eta,l)\tilde{A}^3_k(\xi,l^\prime) + \chi^{\ast;23} A_0^2(\eta,l) A^3_k(\xi,l^\prime) \right) \\ & \quad\quad \times \abs{ \Delta_L \widehat{U^3_{k}}(\xi,l^\prime)_{Hi}} Low(-k,\eta-\xi,l-l^\prime) d\eta d\xi.
\end{align*} 
By \eqref{ineq:AA3Z} and \eqref{ineq:AdeYZZ} (for the $\chi^{\ast;23}$ contribution), followed by \eqref{ineq:triQuadHL}, there holds 
\begin{align*}
F_{HL}^{3} & \lesssim \frac{\epsilon \jap{t}^2}{\jap{\nu t^3}^{\alpha}} \norm{\left(\sqrt{\frac{\partial_t w}{w}}\tilde{A}^2 + \frac{\abs{\grad}^{s/2}}{\jap{t}^{s}}A^2\right)
Q_0^2}_2 \norm{\left(\sqrt{\frac{\partial_t w}{w}}\tilde{A}^3 + \frac{\abs{\grad}^{s/2}}{\jap{t}^{s}}A^3\right) \Delta_L U^3_{\neq}}_2 \\ & \quad +  \frac{\epsilon }{\jap{\nu t^3}^{\alpha}} \norm{\sqrt{-\Delta_L}A^2 Q_0^2}_2\norm{A^3 \Delta_L U^3_{\neq}}_2 \\ 
& \lesssim \frac{\epsilon \jap{t}^2}{\jap{\nu t^3}^{\alpha}} \norm{\left(\sqrt{\frac{\partial_t w}{w}}\tilde{A}^2 + \frac{\abs{\grad}^{s/2}}{\jap{t}^{s}}A^2\right) Q_0^2}_2^2 +  \frac{\epsilon \jap{t}^2}{\jap{\nu t^3}^{\alpha}} \norm{\left(\sqrt{\frac{\partial_t w}{w}}\tilde{A}^3 + \frac{\abs{\grad}^{s/2}}{\jap{t}^{s}}A^3\right) \Delta_L U^3_{\neq}}^2_2 \\ & \quad +\frac{\epsilon^{3/2}}{\jap{\nu t^3}^{\alpha}} \norm{\sqrt{-\Delta_L}A^2 Q_0^2}^2_2 + \frac{\epsilon^{1/2}}{\jap{\nu t^3}^{\alpha}}\norm{A^3 \Delta_L U^3_{\neq}}_2^2, 
\end{align*}
which after Lemmas \ref{lem:SimplePEL} and \ref{lem:PEL_NLP120neq} is consistent with Proposition \ref{prop:Boot} for $c_0$ and $\epsilon$ sufficiently small. 

Turn next to the coefficient error terms. Due to the high number of derivatives, the most difficult one is $F_{C3}^3$, hence, we focus only on this one and omit the others for brevity. 
We have by \eqref{ineq:AprioriUneq}, Lemma \ref{lem:ABasic}, and \eqref{ineq:triQuadHL}, 
\begin{align*}
F_{C3}^3 & \lesssim \frac{\epsilon^2}{\jap{\nu t^3}^{2\alpha}} \sum_{k\neq 0} \sum_{l,l^\prime} \int \abs{A^{2} \widehat{Q^2_0}(\eta,l) \frac{\abs{\eta,l}^2}{\jap{\xi,l^\prime}^2} A\left( \abs{\widehat{\psi_z}(\xi,l^\prime)_{Hi}} + \abs{\widehat{\phi_z}(\xi,l^\prime)_{Hi}} \right) } Low(-k,\eta-\xi,l-l^\prime) d\eta d\xi \\
& \lesssim \frac{\epsilon^2}{\jap{\nu t^3}^{2\alpha}}\norm{\sqrt{-\Delta_L} A^2 Q^2_0}_2 \left(\norm{\jap{\grad}^{-1}A \phi_z}_2 + \norm{\jap{\grad}^{-1}A\psi_z}_2\right) \\ 
& \lesssim \epsilon^{3/2}\norm{\sqrt{-\Delta_L} A^2 Q^2_0}_2^2 +  \frac{\epsilon^{5/2}}{\jap{\nu t^3}^{4\alpha}}\norm{AC}_2^2, 
\end{align*}
which is consistent with Proposition \ref{prop:Boot} for $\epsilon$ sufficiently small. 
The remaining coefficient error terms are similar or easier and are hence omitted. 
The remainder terms are easy variants of the above treatments. The one which may require comment is the error term of the form 
\begin{align*}
2\sum_{k\neq 0} \int A^{2}Q^2_0 A_0^{2}\left((\phi_y)_{Lo}\partial_Z\partial_Z\partial_Z\left( \left(U^3_{-k}\right)_{Hi} \left( U^3_k\right)_{Lo}\right)\right)dV, 
\end{align*}
as the structure of the nonlinearity has changed and it is less clear how to absorb the losses due to the unbalance of regularities. However, since $\norm{C}_{\G^{\lambda,\gamma}} \lesssim \epsilon t$, the presence of the coefficients gains a power of $t$ and absorbs the loss via $\epsilon t \jap{\nu t^3}^{-1}  \lesssim t^{-1}$. From there the proof applies \eqref{ineq:AdeZZZ}; for more details, see the treatment of $F^{1;3}$ below where a similar argument is carried out. 
This completes the treatment of $F^3$. 

Consider next the contribution from $F^{1}$ and $j = 3$ (denoted $F^{1;3}$) which requires further explanation. As above, we expand with a paraproduct, 
\begin{align*}
F^{1;3} & = -\sum_{k\neq 0} \int A^{2}Q^2_0 A_0^{2} \partial_Z\partial_Z\partial_Z\left( \left(U^3_{-k}\right)_{Hi} \left( U^2_k\right)_{Lo}\right)dV \\
& \quad - \sum_{k\neq 0} \int A^{2}Q^2_0 A_0^{2} \partial_Z\partial_Z\partial_Z\left( \left(U^3_{-k}\right)_{Lo} \left( U^2_k\right)_{Hi}\right) dV \\ 
& \quad -\sum_{k\neq 0}  \int A^{2}Q^2_0 A_0^{2} \left( \left( (\psi_z)_{Hi} \partial_Y + (\phi_z)_{Hi}\partial_Z \right) \partial_Z\partial_Z\left( \left(U^3_{-k}\right)_{Lo} \left( U^2_k\right)_{Lo}\right)\right) dV \\ 
& \quad - \sum_{k\neq 0} \int A^{2}Q^2_0 A_0^{2} \left( \partial_Z \left( (\psi_z)_{Hi}\partial_Y + (\phi_z)_{Hi}\partial_Z \right) \partial_Z\left( \left(U^3_{-k}\right)_{Lo} \left( U^2_k\right)_{Lo}\right) \right) dV \\
& \quad - \sum_{k\neq 0} \int A^{2}Q^2_0 A_0^{2} \left( \partial_Z \partial_Z \left((\psi_z)_{Hi} \partial_Y + (\phi_z)_{Hi}\partial_Z \right) \left( \left(U^3_{-k}\right)_{Lo} \left( U^2_k\right)_{Lo}\right) \right) dV \\ 
& \quad - F^{1;3}_{\mathcal{R},C} \\ 
& = F^{1;3}_{HL} + F^{1;3}_{LH} + F^{1;3}_{C1} + F^{1;3}_{C2} + F^{1;3}_{C3} + F^{1;3}_{\mathcal{R},C}, 
\end{align*}
The coefficient error terms $F^{1;3}_{Ci}$ and remainder terms $F^{1;3}_{\mathcal{R},C}$ are all easier than the $F^3$ case treated above and are hence omitted for brevity. 
Of the two leading order terms, $F_{LH}^{1;3}$ is easier as there is no additional regularity loss near critical times (despite the larger low frequency factor); indeed it is treated by a straightforward variant of the treatment of $F_{HL}^{1;3}$. Hence, turn to the latter, which by \eqref{ineq:AprioriUneq} and Lemma \ref{lem:ABasic} is given by 
\begin{align*}
F_{HL}^{1;3} & \lesssim \frac{\epsilon}{\jap{t}\jap{\nu t^3}^{\alpha}} \sum_{k \neq 0} \sum_{l,l^\prime} \int \abs{A^{2} \widehat{Q^2_0}(\eta,l) A_0^2(\eta,l) \frac{\abs{l}^3}{k^2 + (l^\prime)^2 + \abs{\xi-kt}^2} \Delta_L \widehat{U^3_{k}}(\xi,l^\prime)_{Hi}} Low(-k,\eta-\xi,l-l^\prime) d\eta d\xi \\ 
& \lesssim \frac{\epsilon}{\jap{t} \jap{\nu t^3}^{\alpha}} \sum_{k \neq 0} \sum_{l,l^\prime} \int \abs{\widehat{Q^2_0}(\eta,l)}\frac{\abs{l}^3}{k^2 + (l^\prime)^2 + \abs{\xi-kt}^2} \jap{\frac{t}{\jap{\xi,l^\prime}}}^{2} \\ & \quad\quad \times \left(\sum_{r} \chi^{r,NR}\frac{t}{\abs{r} + \abs{\eta-tr}}\tilde{A}^2_0(\eta,l) \tilde{A}^3_k(\xi,l^\prime) + \chi^{\ast;23} A^2_0(\eta,l) A^3_k(\xi,l^\prime) \right)  \\ & \quad\quad \times \abs{\Delta_L \widehat{U^3_{k}}(\xi,l^\prime)_{Hi}} Low(-k,\eta-\xi,l-l^\prime) d\eta d\xi \\ 
& \lesssim \frac{\epsilon}{\jap{\nu t^3}^{\alpha}} \sum_{k \neq 0} \sum_{l,l^\prime} \int \abs{A^{2} \widehat{Q^2_0}(\eta,l)}\frac{\abs{l}^3}{k^2 + (l^\prime)^2 + \abs{\xi-kt}^2} \jap{\frac{t}{\jap{\xi,l^\prime}}}^{2} \\ & \quad\quad \times \abs{A^3_{k} \Delta_L \widehat{U^3_{k}}(\xi,l^\prime)_{Hi}} Low(-k,\eta-\xi,l-l^\prime) d\eta d\xi. 
\end{align*}
By \eqref{ineq:AdeZZZ} (with $p = 2$) and \eqref{ineq:triQuadHL} we have,  
\begin{align*}
F_{HL}^{1;3} & \lesssim \frac{\epsilon}{\jap{\nu t^3}^{\alpha}} \norm{\sqrt{-\Delta_L} A^2 Q_0^2}_2 \norm{A^3 \Delta_L U^3_{\neq}}_2 \lesssim \epsilon^{3/2}\norm{\sqrt{-\Delta_L}A^2 Q_0^2}_2^2 + \frac{\epsilon^{1/2}}{\jap{\nu t^3}^{2\alpha}}\norm{A^3 \Delta_L U^3_{\neq}}^2_2, 
\end{align*}
which by Lemma \ref{lem:SimplePEL} is consistent with Proposition \ref{prop:Boot} for $\epsilon$ sufficiently small. 
This completes $F^{1;3}$. The remaining forcing terms are relatively easy variants of those already treated and are hence omitted for brevity. 

\subsubsection{Dissipation error terms} \label{sec:DEQ02}
Recalling the definitions of the dissipation error terms and the short-hand \eqref{def:G}, we have 
\begin{align} 
\mathcal{D}_E & =  \nu\int A_0^2 Q^2_0 A_0^2\left(G_{yy} \partial_{YY}Q^2_0 + G_{zz}\partial_{ZZ}Q^2_0 + G_{yz}\partial_{YZ}Q^2_0 \right) dV. \label{def:DEQ2}
\end{align} 
All three error terms are essentially the same and are treated in the same manner as the analogous terms in \cite{BGM15I}. Hence,  we omit the treatments and simply state the results
\begin{align} 
\mathcal{D}_{E} & \lesssim c_0^{-1} \nu \epsilon^2 \norm{\grad A C}_2^2 + c_0 \nu \norm{\sqrt{-\Delta_L }A^2 Q^2}_{2}^2. \label{ineq:DEbd}
\end{align} 
Note that as in \cite{BGM15I}, by \eqref{ineq:Boot_ACC}
\begin{align*} 
\int_{1}^{T^\star} c_0^{-1} \nu \epsilon^2 \norm{\grad A C(t)}_2^2 dt & \lesssim c_{0} \epsilon^2 K_B.  
\end{align*} 
Hence, for $c_{0}$ sufficiently small, \eqref{ineq:DEbd} is consistent with Proposition \ref{prop:Boot}. 

\subsection{Non-zero frequencies}
Next we consider the contributions to \eqref{eq:A2Q2Evo} which come from the evolution of non-zero $X$ frequencies.   

\subsubsection{Nonlinear pressure $NLP$} \label{sec:NLPQ2} 
\paragraph{Treatment of $NLP(1,j,0,\neq)$} \label{sec:NLP213}
Here $j \in \set{2,3}$ due to the structure of the nonlinearity.
The case $j = 3$ was singled out in \cite{BGM15I} as one of the leading order nonlinear interactions of type $\textbf{(SI)}$ (see also \S\ref{sec:Toy}). 
We will concentrate on this case and omit the treatment of $j=2$, which is treated with the same method and moreover is simpler due to the lack of a regularity imbalance in $A^2$ near the critical times.

This term is quartic (in the sense that the nonlinearity is order $4$) and we will use the paraproduct decomposition described in \S\ref{sec:paranote}. We will group terms where the coefficients appear in `low  frequency' with the remainder (see Remarks \ref{rmk:CoefCtrlLow} and \ref{rmk:SIcoefneglect}). 
Therefore, the expansion is
\begin{align*} 
NLP(1,3,0,\neq) & =  \sum_{k \neq 0} \int A^2 Q^2_{k} A^2\left( (\partial_Y - t\partial_X) \left( \left(\partial_Z U^1_0\right)_{Lo} (\partial_XU^3_{k})_{Hi}  \right) \right) dV \\ 
& \quad  + \sum_{k \neq 0} \int A^2 Q^2_{k} A^2\left( (\partial_Y - t\partial_X) \left( (\partial_Z U^1_0)_{Hi} (\partial_XU^3_{k})_{Lo} \right) \right) dV \\ 
& \quad + \sum_{k \neq 0} \int A^2 Q^2_{k} A^2\left( \left((\psi_y)_{Hi}(\partial_Y - t\partial_X) + (\phi_y)_{Hi}\partial_Z\right) \left(\partial_XU^3_{k} \partial_Z U^1_0\right)_{Lo} \right) dV \\ 
 & \quad  + \sum_{k \neq 0} \int A^2 Q^2_{k} A^2 \left( (\partial_Y-t\partial_X) \left( (\partial_XU^3_{k})_{Lo} \left( \left((\psi_z)_{Hi}\partial_Y + (\phi_z)_{Hi}\partial_Z \right) (U^1_0)_{Lo}\right) \right) \right) dV \\
& \quad + P_{\mathcal{R},C} \\  
& = P_{LH} + P_{HL} + P_{C1} + P_{C2} + P_{\mathcal{R},C},  
\end{align*} 
where $P_{\mathcal{R},C}$ includes all of the remainders from the quartic paraproduct as well as the higher order terms involving coefficients as low frequency factors.  

Turn first to $P_{LH}$, which by \eqref{ineq:AprioriU0} and \eqref{ineq:A2A3neqneq} is bounded by (recall the shorthand \eqref{def:Low}), 
\begin{align*} 
P_{LH} & \lesssim  \epsilon t \sum_{k \neq 0} \int \abs{A^2 \widehat{Q^2_{k}}(\eta,l) A^2_k(\eta,l) \frac{(\eta - tk)k}{k^2 + \abs{l^\prime}^2 + \abs{\xi-tk}^2} \Delta_L\widehat{U^3_{k}}(\xi,l^\prime)_{Hi}} Low(\eta-\xi,l-l^\prime) d\eta d\xi \\ 
& \lesssim  \epsilon t \sum_{k \neq 0} \int \abs{\widehat{Q^2_{k}}(\eta,l)} \frac{(\eta - tk)k}{k^2 + \abs{l^\prime}^2 + \abs{\xi-tk}^2} \jap{\frac{t}{\jap{\xi,l^\prime}}} \\ 
& \quad\quad \times \left(\sum_r \chi^{r,NR} \frac{t}{\abs{r} + \abs{\eta-tr}} \tilde{A}_k^2(\eta,l)\tilde{A}_k^3(\xi,l^\prime) + \chi^{\ast;23} A_k^2(\eta,l) A_k^3(\xi,l^\prime) \right) \\ & \quad\quad \times \abs{\Delta_L\widehat{U^3_{k}}(\xi,l^\prime)_{Hi}} Low(\eta-\xi,l-l^\prime) d\eta d\xi. 
\end{align*}
Note that by \eqref{ineq:basicNR}, the following holds on the support of the integrand: 
\begin{align}
\chi_{NR;k} & \lesssim \mathbf{1}_{t \leq \epsilon^{-1/2+\delta/100}} + \mathbf{1}_{t \geq \epsilon^{-1/2+\delta/100}}\epsilon^{1/2} \jap{\xi-kt,l^\prime} \jap{\eta-\xi,l-l^\prime}.  \label{ineq:teps12trick}
\end{align}
Hence, by \eqref{ineq:AiPartXA23} (with $p = 1$), followed by \eqref{ineq:triQuadHL}, 
\begin{align} 
P_{LH} & \lesssim \epsilon t\norm{\left(\sqrt{\frac{\partial_t w}{w}}\tilde{A}^2 + \frac{\abs{\grad}^{s/2}}{\jap{t}^s}A^2\right) Q^2}_2\norm{\left(\sqrt{\frac{\partial_t w}{w}}\tilde{A}^3 + \frac{\abs{\grad}^{s/2}}{\jap{t}^s}A^3\right)\Delta_L U^3_{\neq}}_2 \nonumber  \\ 
& \quad + \epsilon^{3/2} \norm{\sqrt{-\Delta_L} A^2 Q^2}_2^2 + \mathbf{1}_{t \leq \epsilon^{-1/2+\delta/100}}\epsilon^{1/2}\norm{A^3 \Delta_L U^3_{\neq}}_2^2 + \epsilon^{3/2-\delta/50}\norm{\sqrt{-\Delta_L }A^3 \Delta_L U^3_{\neq}}_2^2,  \label{ineq:PLHNLP13}
\end{align}
which is consistent with Proposition \ref{prop:Boot} by Lemmas \ref{lem:PEL_NLP120neq}, \ref{lem:SimplePEL}, and \ref{lem:PELED} for $\epsilon$ and $\epsilon t \leq c_0$ sufficiently small. 
 
Turn next to the contribution of $P_{HL}$, which can be treated in the same manner as in \cite{BGM15I}. Indeed, by \eqref{ineq:AprioriUneq} followed by Lemma \ref{lem:ABasic} and \eqref{ineq:triQuadHL}, we have, 
\begin{align*} 
P_{HL} & \lesssim \frac{\epsilon}{\jap{\nu t^3}^{\alpha}}\sum_{k \neq 0} \sum_{l, l^\prime} \int \abs{A^2 \widehat{Q^2_{k}}(\eta,l) A^2_k(\eta,l) \abs{\eta - kt} \abs{l^\prime} \widehat{U^1_0}(\xi,l^\prime)_{Hi} Low(k,\eta-\xi,l-l^\prime)} d\eta d\xi \\ 
 & \lesssim \frac{\epsilon}{\jap{\nu t^3}^{\alpha}} \sum_{k \neq 0} \sum_{l, l^\prime} \int \abs{ A^2 \widehat{Q^2_{k}} \frac{\abs{\eta - kt}}{\jap{\xi,l^\prime}^2 \jap{\frac{t}{\jap{\xi,l^\prime}}}} \abs{l^\prime} A\widehat{U^1_0}(\xi,l^\prime)_{Hi} Low(k,\eta-\xi,l-l^\prime)} d\eta d\xi \\ 
& \lesssim \frac{\epsilon}{\jap{t} \jap{\nu t^3}^{\alpha}} \norm{\sqrt{-\Delta_L} A^2 Q^2}_2 \norm{AU^1_0}_2 \\ 
& \lesssim \epsilon^{3/2} \norm{\sqrt{-\Delta_L} A^2 Q^2}_2^2 + \frac{\epsilon^{1/2}}{\jap{\nu t^3}^{2\alpha}} \left(\frac{1}{\jap{t}^2}\norm{AU^1_0}_2^2\right). 
\end{align*} 
This is consistent with Proposition \ref{prop:Boot} after applying Lemma \ref{lem:PELbasicZero}. 

Turn first to $P_{C1}$, which is also treated in the same manner as in \cite{BGM15I}. By \eqref{ineq:AprioriUneq}, \eqref{ineq:AprioriU0}, and Lemma \ref{lem:ABasic} we have
\begin{align*} 
P_{C1} & \lesssim \frac{\epsilon^2 t^2}{\jap{\nu t^3}^{\alpha}} \sum_{k \neq 0} \sum_{l, l^\prime} \int \abs{A^2 \widehat{Q^2_{k}}(\eta,l) A^2_k(\eta,l) \left(\abs{\widehat{\psi_y}(\xi,l^\prime)_{Hi}} + \abs{\widehat{\phi_y}(\xi,l^\prime)_{Hi}} \right)} Low(k,\eta-\xi,l-l^\prime) d\eta d\xi \\ 
& \lesssim \frac{\epsilon^2 t^2}{\jap{\nu t^3}^{\alpha}} \sum_{k \neq 0} \sum_{l, l^\prime} \int \abs{A^2 \widehat{Q^2_{k}}(\eta,l) \frac{1}{\jap{\xi,l^\prime}^2 \jap{\frac{t}{\jap{\xi,l^\prime}}}} A\left( \abs{\widehat{\psi_y}(\xi,l^\prime)_{Hi}} + \abs{\widehat{\phi_y}(\xi,l^\prime)_{Hi}}\right)} Low(k,\eta-\xi,l-l^\prime) d\eta d\xi\\
 & \lesssim \frac{\epsilon^2 t}{\jap{\nu t^3}^{\alpha}} \norm{A^2 Q^2_{\neq}}_2 \left(\norm{\jap{\grad}^{-1 }A \psi_y}_2  + \norm{\jap{\grad}^{-1}A \phi_y}_2\right) \\ 
& \lesssim \frac{\epsilon}{\jap{\nu t^3}^{\alpha}} \norm{ A^2 Q^2}_2^2 + \frac{\epsilon^{3} t^4}{\jap{\nu t^3}^{\alpha}} \left(\frac{1}{\jap{t}^2} \norm{A C}^2_2\right), 
\end{align*} 
which is consistent with Proposition \ref{prop:Boot} for $\epsilon$ sufficiently small. 
This completes the treatment of $P_{C1}$.  
The second coefficient term, $P_{C2}$, is very similar: there is one extra derivative landing on the coefficient but there is one less power of time from the low frequency factor. By Lemma \ref{lem:ABasic}, we will be able to balance the loss by the gain and apply essentially the same treatment as we did for $P_{C1}$. Hence, this is omitted for the sake of brevity.  

Similarly, the remainder and coefficient terms $P_{\mathcal{R},C}$ are omitted as they are easier or very similar. 
This completes the treatment of $NLP(1,3,0,\neq)$. 

\paragraph{Treatment of $NLP(i,j,0,\neq)$ with $i \in \set{2,3}$} \label{sec:NLPQ2_0neq_notX}
We will demonstrate how to deal with these terms by the example of $NLP(2,3,0,\neq)$ (recall \eqref{def:Q2Enums}), which is one of the leading order terms.
Expanding with a quintic paraproduct and grouping the low frequency coefficient terms with the remainder: 
\begin{align*} 
NLP(2,3,0,\neq) & = \sum_{k \neq 0} \int A^2 Q^2_{k} A^2\left((\partial_Y-t\partial_X)((\partial_Y - t\partial_X)(U^3_{k})_{Hi} (\partial_Z^t U^2_0)_{Lo}) \right) dV \\ 
& \quad +\sum_{k \neq 0} \int A^2 Q^2_{k} A^2\left((\partial_Y-t\partial_X)((\partial_Y-t\partial_X)(U^3_{k})_{Lo} (\partial_Z^t U^2_0)_{Hi}) \right) dV \\
& \quad +\sum_{k \neq 0} \int A^2 Q^2_{k} A^2\left( \left((\psi_y)_{Hi}(\partial_Y - t\partial_X) + (\phi_y)_{Hi}\partial_Z \right)(\partial_Y^t (U^3_{k})_{Lo} (\partial_Z^t U^2_0)_{Lo}) \right) dV \\
& \quad +\sum_{k \neq 0} \int A^2 Q^2_{k} A^2\left( (\partial_Y - t\partial_X)( \left((\psi_y)_{Hi}(\partial_Y-t\partial_X) + (\phi_y)_{Hi}\partial_Z \right) (U^3_{k})_{Lo} (\partial_Z^t U^2_0)_{Lo}) \right) dV \\
& \quad +\sum_{k \neq 0} \int A^2 Q^2_{k} A^2\left( (\partial_Y - t\partial_X)( (\partial_Y^t U^3_{k})_{Lo} ( \left((\phi_z)_{Hi}\partial_Z + (\psi_z)_{Hi}\partial_Y \right) U^2_0)_{Lo}) \right) dV \\
& \quad + P_{\mathcal{R},C} \\  
& = P_{HL} + P_{LH} + P_{C1} + P_{C2} + P_{\mathcal{R},C}, 
\end{align*}
where the term $P_{\mathcal{R},C}$ contains the remainders from the quintic paraproducts and the higher order terms where the coefficients are in low frequency. 

Consider first $P_{HL}$, which by \eqref{ineq:AprioriU0}, \eqref{ineq:TriTriv}, and \eqref{ineq:A2A3neqneq}, 
\begin{align*} 
P_{HL} & \lesssim \epsilon\sum_{k \neq 0} \int \abs{A^2 \widehat{Q^2_{k}}(\eta,l) A^2_k(\eta,l) \abs{\eta - tk} \abs{\xi-tk} \widehat{U^3_{k}}(\xi,l^\prime)_{Hi} } Low(\eta-\xi,l-l^\prime) d\eta d\xi \\ 
& \lesssim \epsilon\sum_{k \neq 0} \int \abs{\widehat{Q^2_{k}}(\eta,l)} \\ & \quad\quad \times \left(\sum_{r}\chi^{r,NR}\frac{t}{\abs{r} + \abs{\eta - tr}}\tilde{A}^2_k(\eta,l)\tilde{A}^3_k(\xi,l^\prime) + \chi^{\ast;23} A^2_k(\eta,l) A^3_k(\xi,l^\prime) \right) \jap{\frac{t}{\jap{\xi,l^\prime}}} \\ & \quad\quad \times \jap{\xi-tk}^2 \abs{\widehat{U^3_{k}}(\xi,l^\prime)_{Hi}} Low(\eta-\xi,l-l^\prime) d\eta d\xi. 
\end{align*}   
By \eqref{ineq:jNRBasic}, \eqref{ineq:teps12trick}, and \eqref{ineq:quadHL} we have,  
\begin{align*} 
P_{HL} & \lesssim \epsilon \jap{t} \norm{\left(\sqrt{\frac{\partial_t w}{w}}\tilde{A}^2 + \frac{\abs{\grad}^{s/2}}{\jap{t}^{s}}A^2\right)  Q^2_{\neq}}_2 \norm{\left(\sqrt{\frac{\partial_t w}{w}}\tilde{A}^3 + \frac{\abs{\grad}^{s/2}}{\jap{t}^{s}}A^3 \right) \Delta_L U^3_{\neq}}_2 \\ 
& \quad + \epsilon^{3/2}\norm{\sqrt{-\Delta_L}A^2 Q^2}_2^2 + \mathbf{1}_{t \leq \epsilon^{-1/2+\delta/100}}\epsilon^{1/2}\norm{\Delta_L A^3 U^3_{\neq}}^2_2 + \epsilon^{3/2-\delta/50}\norm{\sqrt{-\Delta_L} \Delta_L A^3 U^3_{\neq}}^2_2,  
\end{align*}
 which is consistent with Proposition \ref{prop:Boot} by Lemmas \ref{lem:PEL_NLP120neq}, \ref{lem:SimplePEL}, and \ref{lem:PELED}. 

Turn next to $P_{LH}$. As in \cite{BGM15I}, this term is treated as in the analogous term in $NLP(1,3,0,\neq)$, using that extra loss of time from the second $\partial_Y^t$ derivative replaces the gain in $t$ from the presence of $U_0^2$ as opposed to $U_0^1$. 
We omit the analogous details and simply conclude that 
\begin{align*} 
P_{LH} & \lesssim \epsilon^{3/2} \norm{\sqrt{-\Delta_L} A^2 Q^2}_2^2 + \frac{\epsilon^{1/2}}{\jap{\nu t^3}^{2\alpha}}\norm{AU_0^2}_2^2, 
\end{align*} 
which after Lemma \ref{lem:PELbasicZero} is consistent with Proposition \ref{prop:Boot} for $\epsilon$ sufficiently small.  

The coefficient error terms, $P_{Ci}$, are also similar to \cite{BGM15I} and the corresponding terms in the treatment of $NLP(1,3,0,\neq)$ above in \S\ref{sec:NLP213}.
We omit the details for brevity. 
Similarly, the remainder terms and low frequency coefficient terms are relatively easy to deal with or are easy variants of the above treatments and are hence omitted. This completes the treatment of $NLP(2,3,0,\neq)$, which is the leading order term in $NLP(i,j,0,\neq)$ with $i \in \set{2,3}$. 

\paragraph{Treatment of $NLP(i,j,\neq,\neq)$ terms} \label{sec:NLPQ2_neqneq}
These are pressure interactions of type \textbf{(3DE)}.  
All of these terms can be treated in a similar fashion, however the terms involving $U^3$ are slightly harder due to the regularity imbalances. 
We will focus on the case $i=1$ and $j=3$ and omit the others, which follow analogously.  
As usual, this term is quartic, but when we expand with the paraproduct we will keep the coefficients only when they appear in high frequency and group the other terms with the remainder.
Hence, 
\begin{align*} 
NLP(1,3,\neq,\neq) & = \int A^2 Q^2_{\neq} A^2\left( (\partial_Y-t\partial_X)( (\partial_Z U^1_{\neq})_{Lo} (\partial_X U^3_{\neq})_{Hi} )\right) dV \\ 
& \quad + \int A^2 Q^2_{\neq} A^2\left( (\partial_Y-t\partial_X)( (\partial_Z U^1_{\neq})_{Hi} (\partial_X U^3_{\neq})_{Lo} )\right) dV \\ 
& \quad + \int A^2 Q^2_{\neq} A^2\left( \left((\psi_y)_{Hi}(\partial_Y-t\partial_X) + (\phi_y)_{Hi}\partial_Z\right)( (\partial_Z U^1_{\neq})_{Lo} (\partial_X U^3_{\neq})_{Lo} )\right) dV \\ 
& \quad +\sum_{k} \int  A^2 Q^2_{\neq} A^2\left( (\partial_Y-t\partial_X)( \left((\psi_z)_{Hi}(\partial_Y-t\partial_X) + (\phi_z)_{Hi}\partial_Z \right)(U^1_{\neq})_{Lo} (\partial_X U^3_{\neq})_{Lo} )\right) dV \\ 
& = P_{LH} + P_{HL} + P_{C1} + P_{C2} + P_{\mathcal{R},C},
\end{align*}
where $P_{\mathcal{R},C}$ contains the paraproduct remainders and the terms where coefficients appear in low frequency. 
By \eqref{ineq:AprioriUneq}, \eqref{ineq:A2A3neqneq}, and \eqref{ineq:jNRPneqneq},  
\begin{align*} 
P_{LH} & \lesssim \frac{\epsilon \jap{t}^{\delta_1}}{\jap{\nu t^3}^\alpha} \sum_{k}\int \abs{A^2 \widehat{Q^2_k}(\eta,l)A^2_k(\eta,l) (\eta-k t)k^\prime \widehat{U^3_{k^\prime}}(\xi,l^\prime)} Low(k-k^\prime,\eta-\xi,l-l^\prime) d\eta d\xi \\ 
& \lesssim \frac{\epsilon \jap{t}^{\delta_1}}{\jap{\nu t^3}^\alpha} \sum_{k}\int \abs{A^2 \widehat{Q^2_k}(\eta,l)} \frac{(\eta-k t)k^\prime}{\abs{k^\prime}^2 + \abs{l^\prime}^2 + \abs{\xi-k^\prime t}^2} \jap{\frac{t}{\jap{\xi,l^\prime}}} \\ & \quad\quad \times \left(\sum_{r}\chi^{r,NR}\frac{t}{\abs{r} + \abs{\eta-tr}} + \chi^{\ast;23} \right) \abs{A^3\Delta_L\widehat{U^3_{k^\prime}}(\xi,l^\prime)} Low(k-k^\prime,\eta-\xi,l-l^\prime) d\eta d\xi \\
& \lesssim \frac{\epsilon \jap{t}^{\delta_1}}{\jap{\nu t^3}^\alpha}\norm{\sqrt{-\Delta_L}A^2 Q^2_{\neq}}_2 \norm{\Delta_L A^3 U^3_{\neq}}_2 \\ 
& \lesssim \epsilon^{3/2} \norm{\sqrt{-\Delta_L}A^2 Q^2_{\neq}}_2^2 +  \frac{\epsilon^{1/2} \jap{t}^{2\delta_1}}{\jap{\nu t^3}^{2\alpha}}\norm{\Delta_L A^3 U^3_{\neq}}_2^2, 
\end{align*} 
which is consistent with Proposition \ref{prop:Boot} by Lemma \ref{lem:SimplePEL} for $\epsilon$ sufficiently small. 

By \eqref{ineq:AprioriUneq}, \eqref{ineq:ABasic}, and \eqref{ineq:AikDelL2D_CKw}, followed by \eqref{ineq:triQuadHL}, we have
\begin{align*} 
P_{HL} & \lesssim \frac{\epsilon}{\jap{\nu t^3}^\alpha} \sum_{k}\int \abs{A^2 \widehat{Q^2_k}(\eta,l)A^2_k(\eta,l) (\eta-k t)l^\prime \widehat{U^1_{k^\prime}}(\xi,l^\prime)} Low(k-k^\prime,\eta-\xi,l-l^\prime) d\eta d\xi \\ 
& \lesssim \frac{\epsilon}{\jap{\nu t^3}^\alpha} \sum_{k}\int \abs{A^2 \widehat{Q^2_k}(\eta,l) \jap{\frac{t}{\jap{\xi,l^\prime}}}^{\delta_1}\frac{\jap{t} (\eta-k t) l^\prime}{\abs{k^\prime}^2 + \abs{l^\prime}^2 + \abs{\xi-k^\prime t}^2} \Delta_L A^1\widehat{U^1_{k^\prime}}(\xi,l^\prime)} \\ & \quad\quad \times  Low(k-k^\prime,\eta-\xi,l-l^\prime) d\eta d\xi \\
& \lesssim \frac{\epsilon t^2}{\jap{\nu t^3}^{\alpha}}\norm{\left(\sqrt{\frac{\partial_t w}{w}}\tilde{A}^2 + \frac{\abs{\grad}^{s/2}}{\jap{t}^{s}}A^2\right) Q^2}_2 \norm{\left(\sqrt{\frac{\partial_t w}{w}}\tilde{A}^1 + \frac{\abs{\grad}^{s/2}}{\jap{t}^{s}}A^1 \right) \Delta_L U^1_{\neq}}_2 \\ & \quad + \frac{\epsilon \jap{t}^{\delta_1}}{\jap{\nu t^3}^\alpha}\norm{\sqrt{-\Delta_L} A^2 Q^2_{\neq}}_2 \norm{\Delta_L A^1 U^1_{\neq}}_2, 
\end{align*} 
which after Lemmas \ref{lem:SimplePEL} and \ref{lem:PEL_NLP120neq}, is consistent with Proposition \ref{prop:Boot}. 

As in \cite{BGM15I}, the coefficient error terms are straightforward here and are hence omitted for the sake of brevity. 
As discussed above, the remainder terms $P_{\mathcal{R}.C}$ are much easier than the leading order terms, and these are hence omitted. 
This completes the treatment of $NLP(1,3,\neq,\neq)$. 
Other $i,j$ combinations can be treated via a simple variant of this (one will also use \eqref{ineq:jNRPneqneq} for this). 
    
\subsubsection{Nonlinear stretching $NLS$}\label{sec:NLSQ2}

\paragraph{Treatment of $NLS1(j,0,\neq)$ and $NLS1(j,\neq,0)$} \label{sec:NLS1Q20neq}
Recall the definition of $NLS1(j,0,\neq)$ from \eqref{def:Q2Enums}. 
These terms can essentially be treated in the same manner as the $NLP(j,2,0,\neq)$ nonlinear pressure terms in \S\ref{sec:NLP213} and \S\ref{sec:NLPQ2_0neq_notX} and hence we omit them for brevity.
 
Consider the $NLS1(j,\neq,0)$ terms. 
Notice that the $j = 1$ term disappears due to the usual null structure. The $j=3$ term is then the most dangerous remaining term as we must contend with the loss of regularity near critical times as well as a large low-frequency growth.  
Expanding this term with a paraproduct and focusing on the highest order terms gives: 
\begin{align*} 
NLS1(3,\neq,0) & = -\int A^2 Q^2 A^2\left( (Q^3_{\neq})_{Hi} (\partial_Z U^2_{0})_{Lo}\right) dV - \int A^2 Q^2 A^2\left( (Q^3_{\neq})_{Lo} (\partial_Z U^2_{0})_{Hi}\right) dV \\ 
& \quad - \int A^2 Q^2 A^2\left( (Q^3_{\neq})_{Lo} \left((\psi_z)_{Hi} \partial_Y + (\phi_z)_{Hi}\partial_Z\right) (U^2_{0})_{Lo}\right) dV + S_{\mathcal{R},C} \\ 
& = S_{HL} + S_{LH} + S_{C} + S_{\mathcal{R},C}, 
\end{align*}
where $S_{\mathcal{R},C}$ contains the paraproduct remainders and the terms where the coefficients appear in low frequency. 
By \eqref{ineq:AprioriU0}, \eqref{ineq:A2A3neqneq}, \eqref{ineq:jNRBasic}, \eqref{ineq:teps12trick}, and \eqref{ineq:triQuadHL} we have
\begin{align*} 
S_{HL} & \lesssim \epsilon \sum_{k}\int \abs{\widehat{Q^2_k}(\eta,l)} \jap{\frac{t}{\jap{\xi,l^\prime}}}
\\ & \quad\quad \times \left(\sum_{r}\chi^{r,NR}\frac{t}{\abs{r} + \abs{\eta-tr}}\tilde{A}^2_k(\eta,l) \tilde{A}^3_k(\xi,l^\prime) 
 + \chi^{\ast;23} A^2_k(\eta,l) A^3_k(\xi,l^\prime) \right)\\
& \quad\quad \times \abs{\widehat{Q^3_{k}}(\xi,l^\prime)} Low(\eta-\xi,l-l^\prime) d\eta d\xi \\ 
& \lesssim \epsilon t \norm{\left(\sqrt{\frac{\partial_t w}{w}}\tilde{A}^2 + \frac{\abs{\grad}^{s/2}}{\jap{t}^{s}}A^2\right) Q^2}_2\norm{\left(\sqrt{\frac{\partial_t w}{w}}\tilde{A}^3 + \frac{\abs{\grad}^{s/2}}{\jap{t}^{s}}A^3\right)  Q^3_{\neq}}_2 \\ 
& \quad + \epsilon^{3/2}\norm{\sqrt{-\Delta_L}A^2 Q^2}_2^2 + \mathbf{1}_{t \leq \epsilon^{-1/2+\delta/100}}\epsilon^{1/2}\norm{A^3 Q^3_{\neq}}_2^2 + \epsilon^{3/2-\delta/50}\norm{\sqrt{-\Delta_L} A^3 Q^3_{\neq}}_2^2, 
\end{align*}
which is consistent with Proposition \ref{prop:Boot} for $\epsilon$ and $c_0$ sufficiently small. 

The treatment of $S_{LH}$ is the same as \cite{BGM15I}: by \eqref{ineq:Boot_ED}, Lemma \ref{lem:ABasic}, and \eqref{ineq:triQuadHL},
\begin{align*} 
S_{LH} & \lesssim   \frac{\epsilon \jap{t}^2}{\jap{\nu t^3}^\alpha} \sum_{k}\int \abs{A^2 \widehat{Q^2_k}(\eta,l) A^2_k(\eta,l) l^\prime \widehat{U^2_{0}}(\xi,l^\prime)_{Hi}} Low(k,\eta-\xi,l-l^\prime) d\eta d\xi \\ 
& \lesssim \frac{\epsilon \jap{t}}{\jap{\nu t^3}^\alpha} \norm{A^2 Q^2}_2\norm{AU_0^2}_2, 
\end{align*}
which is consistent with Proposition \ref{prop:Boot} for $\epsilon$ sufficiently small. 
The coefficient error term, $S_C$, is treated as in \cite{BGM15I}: by \eqref{ineq:AprioriU0}, \eqref{ineq:AprioriUneq}, and Lemma \ref{lem:ABasic}, and Lemma \ref{lem:CoefCtrl}, 
\begin{align*} 
S_C & \lesssim \frac{\epsilon^2 \jap{t}^2}{\jap{\nu t^3}^\alpha} \sum_{k}\int \abs{A^2 \widehat{Q^2_k}(\eta,l) \frac{1}{\jap{t} \jap{\xi,l^\prime} } A\widehat{\psi_y}(\xi,l^\prime)_{Hi}} Low(k,\eta-\xi,l-l^\prime) d\eta d\xi \\ 
& \lesssim \frac{\epsilon \jap{t}}{\jap{\nu t^3}^\alpha} \norm{A^2 Q^2}_2^2 + \frac{\epsilon^3 \jap{t}}{\jap{\nu t^3}^{\alpha}}\norm{AC}_2^2, 
\end{align*}
which is consistent with Proposition \ref{prop:Boot} for $\epsilon$ sufficiently small. 
    
As usual, the remainders and coefficient error terms in $S_{\mathcal{R},C}$ are significantly easier to treat and hence 
are omitted for brevity. 
This completes the treatment of $NLS1(3,\neq,0)$; the other term, $NLS1(2,\neq,0)$ is easier and is treated the same way, hence we omit this for brevity.

\paragraph{Treatment of $NLS1(j,\neq,\neq)$} 
The most problematic terms are $j = 3$ and $j = 1$. 
The other terms will be treated in a similar fashion, so we focus on the $j = 3$ for brevity.
We expand the term with a paraproduct and only keep the coefficients to leading order when they appear in high frequency:
\begin{align*}
NLS1(3,\neq,\neq) & = -\int A^2 Q^2 A^2\left( (Q^3_{\neq})_{Hi} (\partial_Z U^2_{\neq})_{Lo}\right) dV -\int A^2 Q^2 A^2\left( (Q^3_{\neq})_{Lo} (\partial_Z U^2_{\neq})_{Hi}\right) dV \\ 
& \quad - \int A^2 Q^2 A^2\left( (Q^3_{\neq})_{Lo} \left((\psi_z)_{Hi} (\partial_Y - t\partial_X) + (\phi_z)_{Hi}\partial_Z \right)U^2_{\neq})_{Lo}\right) dV  + S_{\mathcal{R},C} \\ 
& = S_{HL} + S_{LH} + S_{C} + S_{\mathcal{R},C}, 
\end{align*}
where $S_{\mathcal{R},C}$ contains the paraproduct remainders and the terms where the coefficients appear in low frequency. 
By \eqref{ineq:AprioriUneq}, \eqref{ineq:A2A3neqneq}, and \eqref{ineq:triQuadHL} we have
\begin{align*} 
S_{HL} & \lesssim \frac{\epsilon}{\jap{t}\jap{\nu t^3}^{\alpha}} \sum_{k}\int \abs{A^2 \widehat{Q^2_k}(\eta,l) \jap{\frac{t}{\jap{\xi,l^\prime}}}\left(\sum_{r}\chi^{r,NR}\frac{t}{\abs{r} + \abs{\eta-tr}} + \chi^{\ast;23} \right)  A^3 \widehat{Q^3_{k^\prime}}(\xi,l^\prime)}  \\ & \quad\quad \times Low(k-k^\prime,\eta-\xi,l-l^\prime) d\eta d\xi \\ 
& \lesssim \frac{\epsilon}{\jap{\nu t^3}^{\alpha}}\norm{A^2 Q^2}_2\norm{A^3 Q^3}_2
\end{align*}
which is consistent with Proposition \ref{prop:Boot} for $\epsilon$ sufficiently small. 

Turn next to the $S_{LH}$ term. By \eqref{ineq:Boot_Hi}, \eqref{ineq:ABasic}, \eqref{ineq:AiPartX}, and \eqref{ineq:triQuadHL} we have
\begin{align*} 
S_{LH} & \lesssim \frac{\epsilon \jap{t}^2}{\jap{\nu t^3}^{\alpha}} \sum_{k}\int \abs{A^2 \widehat{Q^2_k}(\eta,l) \frac{\abs{l^\prime}}{\abs{k^\prime}^2 + \abs{l^\prime}^2 + \abs{\xi - k^\prime t}^2} \Delta_LA^2 \widehat{U^2_{k^\prime}}(\xi,l^\prime)_{Hi}} \\ & \quad\quad \times  Low(k-k^\prime,\eta-\xi,l-l^\prime) d\eta d\xi \\ 
& \lesssim \frac{\epsilon \jap{t}^2}{\jap{\nu t^3}^{\alpha}}\norm{\left(\sqrt{\frac{\partial_t w}{w}}\tilde{A}^2 + \frac{\abs{\grad}^{s/2}}{\jap{t}^{s}}A^2 \right) Q^2}_2 \norm{\left(\sqrt{\frac{\partial_t w}{w}}\tilde{A}^2 + \frac{\abs{\grad}^{s/2}}{\jap{t}^{s}}A^2\right) \Delta_L U^2_{\neq}}_2 \\ 
& \quad + \frac{\epsilon \jap{t}}{\jap{\nu t^3}^{\alpha}}\norm{A^2 Q^2}_2 \norm{A^2 \Delta_L U^2_{\neq}}_2, 
\end{align*}
which is consistent with Proposition \ref{prop:Boot} for $\epsilon$ small by Lemmas \ref{lem:SimplePEL} and \ref{lem:PEL_NLP120neq}.
For the coefficient error term is treated in the same fashion as the corresponding error term associated with $NLS1(3,\neq,0)$ in \S\ref{sec:NLS1Q20neq} above. 
Hence, the treatment is omitted. Similarly, the remainder and coefficient low frequency terms in $S_{\mathcal{R},C}$ are also omitted. 
This completes the treatment of the $NLS1(3,\neq,\neq)$ term; the other $NLS1(j,\neq,\neq)$ terms are treated similarly. 

\paragraph{Treatment of $NLS2(i,1,0,\neq)$} \label{sec:NLS2i1neq0}
Recall the definition of these terms from \eqref{def:Q2Enums}. 
The non-zero contributions come from $i  = 2$ and $i = 3$ and these can be treated as in \cite{BGM15I} (note $U^3$ does not appear in either).  
We hence omit the treatment for the sake of brevity (it roughly parallels $NLP(1,2,0,\neq)$ in \S\ref{sec:NLP213}, which was omitted since this was slightly easier than the leading order $NLP(1,3,0,\neq)$). 

\paragraph{Treatment of $NLS2(i,j,0,\neq)$ with $j \neq 1$} \label{sec:NLS2ijneq0}
Recall \eqref{def:Q2Enums} and note that $i \neq 1$. Unlike in \cite{BGM15I}, not all the cases are quite the same.
However, the losses due to the regularity imbalances in $Q^3$ can be easily absorbed by the low frequency growth of $Q^2$. 
Otherwise, the treatment is similar to that used in \cite{BGM15I}. Hence the details are omitted for brevity. 

\paragraph{Treatment of $NLS2(i,j,\neq,0)$}
Recall \eqref{def:Q2Enums} and note that $j \neq 1$.
These terms can all be treated in a manner similar to the treatment of $NLS2(i,j,0,\neq)$ above and are hence omitted for the sake of brevity. 

\paragraph{Treatment of $NLS2(i,j,\neq,\neq)$}
First note that the contribution $i = j = 2$ cancels with the $NLP$ terms. 
These terms are treated similar to $NLP(i,j,\neq,\neq)$, however they are generally easier as the regularity imbalances in $Q^3$ and the large growth in $Q^1$ arises 
on the factor with fewer derivatives. 
Moreover, if $U^1$ or $U^3$ are in high frequency, than the decay of the low frequency factor $U^2$ is better by a $t^{-1}$.
Hence, it is straightforward to show that for all choices of $i$ and $j$, 
\begin{align*}
NLS2(i,j,\neq,\neq) & \lesssim \frac{\epsilon t}{\jap{\nu t^3}^{\alpha}}\norm{A^2 Q^2_{\neq}}_2\left(\norm{A^j \Delta_L U^j_{\neq}}_2 +  \norm{A^2 \Delta_L U^2_{\neq}}_2\right), 
\end{align*}  
which is consistent with Proposition \ref{prop:Boot} by Lemma \ref{lem:SimplePEL} for $\epsilon$ sufficiently small.

\subsubsection{Transport nonlinearity $\mathcal{T}$} \label{sec:Q2_TransNon}
Next, we treat $\mathcal{T}_{\neq}$ (recall \eqref{def:Q2Enums}). 
Begin with a paraproduct decomposition: 
\begin{align*} 
\mathcal{T}_{\neq} & = -\int A^{2} Q^2_{\neq} A^{2} \left( \tilde U_{Lo} \cdot \grad Q^2_{Hi} \right) dV  -\int A^{2} Q^2_{\neq} A^{2} \left( \tilde U_{Hi} \cdot \grad Q^2_{Lo} \right) dV - \int A^{2} Q^2_{\neq} A^{2} \left( \tilde U \cdot \grad Q^2 \right)_{\mathcal{R}} dV \\  
& = \mathcal{T}_T + \mathcal{T}_{R} + \mathcal{T}_{\mathcal{R}},
\end{align*}
where, as in \cite{BGM15I}, `T' and `R' stand for \emph{transport} and \emph{reaction} respectively.
Decompose the transport and reaction terms into subcomponents depending on the $X$ frequencies: 
\begin{align*} 
\mathcal{T}_{T} 
 & = -\int A^{2} Q^2_{\neq} A^{2} \left( (\tilde U_{\neq})_{Lo}  \cdot (\grad Q^2_0)_{Hi} \right) dV - \int A^{2}_{\neq} Q^2 A^{2} \left( g_{Lo} \partial_Y (Q^2_{\neq})_{Hi} \right) dV \\ & \quad - \int A^{2} Q^2_{\neq} A^{2} \left( (\tilde U_{\neq})_{Lo} \cdot \grad (Q^2_{\neq})_{Hi} \right) dV \\ 
& = \mathcal{T}_{T;\neq 0}+ \mathcal{T}_{T;0 \neq}+ \mathcal{T}_{T;\neq \neq}, 
\end{align*} 
and,
\begin{align*} 
\mathcal{T}_{R} & =  -\int A^{2} Q^2_{\neq} A^{2} \left( (\tilde U_{\neq})_{Hi}  \cdot (\grad Q^2_0)_{Lo} \right) dV - \int A^{2} Q^2_{\neq} A^{2} \left( g_{Hi} \partial_Y (Q^2_{\neq})_{Lo} \right) dV \\ & \quad - \int A^{2} Q^2_{\neq} A^{2} \left( (\tilde U_{\neq})_{Hi} \cdot \grad (Q^2_{\neq})_{Lo} \right) dV \\ 
& = \mathcal{T}_{R;\neq 0} + \mathcal{T}_{R;0 \neq}+ \mathcal{T}_{R;\neq \neq}. 
\end{align*} 

\paragraph{Transport by zero frequencies: $\mathcal{T}_{T;0 \neq}$}
Turn first to $\mathcal{T}_{T;0 \neq}$, which is the transport by $g$. 
On the Fourier side, 
\begin{align*} 
\mathcal{T}_{T;0 \neq} & \lesssim \sum_k \sum_{l,l^\prime} \int \abs{A^2 \widehat{Q^2_k}(\eta,l) A_k^2(\eta,l)\hat{g}(\eta-\xi,l-l^\prime)_{Lo} \xi \widehat{Q^2_k}(\xi,l^\prime)_{Hi}} d\eta d\xi. 
\end{align*} 
Hence, by \eqref{ineq:ABasic}, $\abs{\xi} \leq \abs{\xi - kt} + \abs{kt}$, and \eqref{ineq:triQuadHL}, 
\begin{align*} 
\mathcal{T}_{T;0 \neq} & \lesssim \norm{g}_{\G^{\lambda}}\norm{A^2 Q^2_{\neq}}_2 \left(\norm{(\partial_Y - t\partial_X) A^2 Q^2}_2 + t\norm{\partial_X A^2 Q^2}_2\right) \\
& \lesssim \jap{t}\norm{g}_{\G^{\lambda}}\norm{A^2 Q^2_{\neq}}_2\norm{\sqrt{-\Delta_L} A^2 Q^2_{\neq}}_2 \\ 
& \lesssim \epsilon^{3/2}\norm{\sqrt{-\Delta_L} A^2 Q^2_{\neq}}_2^2 + \frac{\epsilon^{1/2}}{\jap{t}^{2}}\norm{A^2 Q^2_{\neq}}_2^2, 
\end{align*} 
where the last line followed from the low norm control on $g$, \eqref{ineq:Boot_gLow}. 
This contribution is hence consistent with Proposition \ref{prop:Boot} for $\epsilon$ sufficiently small. 

\paragraph{Transport by non-zero frequencies, $\mathcal{T}_{T;\neq \neq}$ and $\mathcal{T}_{T;\neq 0}$}
Turn next to $\mathcal{T}_{T;\neq \neq}$. 
Indeed, going back to \eqref{def:tildeU2}, 
\begin{align*} 
\mathcal{T}_{T;\neq \neq} & =  \int A^2 Q^2_{\neq} A^2 \left( \begin{pmatrix} (U_{\neq}^1)_{Lo}  \\ \left((1+\psi_y)U^2_{\neq}\right)_{Lo} + \left(\psi_zU^3_{\neq}\right)_{Lo} \\ \left((1 + \phi_z)U^3_{\neq} \right)_{Lo} + \left(\phi_y U^2_{\neq}\right)_{Lo} \end{pmatrix} \cdot \begin{pmatrix} \partial_X  \\ \partial_Y - t\partial_X \\ \partial_Z \end{pmatrix} (Q_{\neq}^2)_{Hi} \right)  dV. 
\end{align*} 
The presence of the coefficients is irrelevant by Lemma \ref{lem:GevProdAlg} and Lemma \ref{lem:CoefCtrl} so let us ignore them. 
By \eqref{ineq:ABasic}, \eqref{ineq:triQuadHL}. and \eqref{ineq:AprioriUneq} we have
\begin{align*} 
\mathcal{T}_{T;\neq \neq} & \lesssim \left(\norm{U^1_{\neq}}_{\G^{\lambda}} +\norm{U^2_{\neq}}_{\G^{\lambda}} + \norm{U^3_{\neq}}_{\G^{\lambda}}\right)\norm{A^{2} Q^2}_2 \norm{\sqrt{-\Delta_L}A^{2} Q^2}_2 \\ 
& \lesssim \frac{\epsilon^{1/2} \jap{t}^{2\delta_{1}}}{\jap{\nu t^3}^{2\alpha}} \norm{A^{2} Q^2}^2_2 +  \epsilon^{3/2}\norm{\sqrt{-\Delta_L}A^{2} Q^2}^2_2, 
\end{align*} 
which is consistent with Proposition \ref{prop:Boot} for $\delta_1$ and $\epsilon$ sufficiently small. 
The contribution from $\mathcal{T}_{T;\neq,0}$ is treated similarly and yields 
\begin{align*} 
\mathcal{T}_{T;\neq 0} & \lesssim \frac{\epsilon }{\jap{\nu t^3}^{\alpha}}\norm{A^2 Q^2}_2\norm{\grad A^2 Q^2_0}_2  \lesssim \frac{\epsilon^{1/2}}{\jap{\nu t^3}^{2\alpha}}\norm{A^2 Q^2}_2^2 + \epsilon^{3/2}\norm{\grad A^2 Q^2_0}^2_2, 
\end{align*}
which is consistent with Proposition \ref{prop:Boot} for $\epsilon$ sufficiently small. 
This completes the treatment of the `transport' contribution to the transport nonlinearity. 
 
\paragraph{Reaction term $\mathcal{T}_{R;0 \neq}$}
It is in the reaction terms where things get more interesting. 
We begin with the trivial one, $\mathcal{T}_{R;0 \neq}$. 
By Lemma \ref{lem:ABasic}, \eqref{ineq:AprioriUneq}, and \eqref{ineq:triQuadHL}, 
\begin{align*} 
\mathcal{T}_{R;0 \neq} & \lesssim \frac{\epsilon t}{\jap{\nu t^3}^{\alpha}} \sum \int \abs{A^2 \widehat{Q^2_k}(\eta,l) \frac{1}{\jap{\frac{t}{\jap{\xi,l^\prime}}} \jap{\xi,l^\prime}^2} A\hat{g}(\xi,l^\prime)_{Hi}} Low(k,\eta-\xi,l-l^\prime) d\xi d\eta \\ 
& \lesssim \frac{\epsilon}{\jap{\nu t^3}^{\alpha}} \norm{Ag}_2 \norm{A^2 Q^2_{\neq}}_2, 
\end{align*} 
which is consistent with Proposition \ref{prop:Boot} for $\epsilon$ sufficiently small. 

\paragraph{Reaction term $\mathcal{T}_{R;\neq 0}$} \label{sec:Q2TRneq0}
First consider $\mathcal{T}_{R;\neq 0}$, which is further divided via  (recall this shorthand notation from \S\ref{sec:paranote} and the a priori estimates \eqref{ineq:AprioriUneq}, \eqref{ineq:AprioriU0})
\begin{align*} 
\mathcal{T}_{R;\neq 0} & \lesssim  \epsilon \sum_{k \neq 0}\int \abs{A^{2} \hat{Q}^2_k(\eta,l) A^{2}_{k}(\eta,l) \hat{U}_k^2 (\xi,l^\prime)_{Hi}} Low(\eta-\xi,l-l^\prime) d\eta d\xi \\ 
& \quad + \epsilon \sum_{k \neq 0}\int \abs{A^{2} \hat{Q}^2_k(\eta,l) A^{2}_{k}(\eta,l) \hat{U}_k^3 (\xi,l^\prime)_{Hi}} Low(\eta-\xi,l-l^\prime) d\eta d\xi \\  
& \quad + \frac{\epsilon^2}{\jap{t}\jap{\nu t^3}^{\alpha}} \sum_{k \neq 0}\int \abs{A^{2} \hat{Q}^2_k(\eta,l) A^{2}_{k}(\eta,l) \left( \abs{\widehat{\psi_y}(\xi,l^\prime)_{Hi}} + \abs{\widehat{\phi_y}(\xi,l^\prime)_{Hi}}\right)} Low(\eta-\xi,l-l^\prime) d\eta d\xi \\  
& \quad + \frac{\epsilon^2}{\jap{\nu t^3}^{\alpha}} \sum_{k \neq 0}\int \abs{A^{2} \hat{Q}^2_k(\eta,l) A^{2}_{k}(\eta,l) \left( \abs{\widehat{\psi_z}(\xi,l^\prime)_{Hi}} + \abs{\widehat{\phi_z}(\xi,l^\prime)_{Hi}} \right) } Low(\eta-\xi,l-l^\prime) d\eta d\xi \\  
& \quad + \mathcal{T}_{R;\neq 0;\mathcal{R}} \\ 
& = \mathcal{T}_{R;\neq 0;2} + \mathcal{T}_{R;\neq 0;3} + \mathcal{T}_{R;\neq 0;C1} + \mathcal{T}_{R;\neq 0;C2}+ \mathcal{T}_{R;\neq 0;\mathcal{R}}. 
\end{align*} 
Turn first to $\mathcal{T}_{R;\neq 0;2}$. 
By \eqref{ineq:ABasic}, \eqref{ineq:triQuadHL}, and the projection to non-zero frequencies,  
\begin{align*} 
\mathcal{T}_{R;\neq 0;2} & \lesssim \epsilon\norm{A^2 Q^2_{\neq}}_2 \norm{A^2 U^2_{\neq}}_2 \lesssim \epsilon\norm{A^2 Q^2_{\neq}}_2 \norm{A^2 \Delta_L U^2_{\neq}}_2, 
\end{align*}  
which by Lemma \ref{lem:SimplePEL} is consistent with Proposition \ref{prop:Boot} for $c_0$ sufficiently small. 

Turn next to $\mathcal{T}_{R;\neq 0;3}$. By \eqref{ineq:A2A3neqneq}, \eqref{ineq:A1A3ReacGain}, and \eqref{ineq:triQuadHL}, 
\begin{align*} 
\mathcal{T}_{R;\neq 0;3} & \lesssim \epsilon \sum_{k \neq 0}\int \abs{A^{2} \hat{Q}^2_k(\eta,l)}  \frac{\jap{\frac{t}{\jap{\xi,l^\prime}}}}{k^2 + (l^\prime)^2 + \abs{\xi-kt}^2} \\ 
& \quad\quad \times \left(\sum_{r}\chi^{r,NR}\frac{t}{\abs{r} + \abs{\eta -tr} } + \chi^{\ast;23} \right)\abs{ \Delta_LA^3 \widehat{U^3_k }(\xi,l^\prime)_{Hi}}  Low(\eta-\xi,l-l^\prime) d\xi d\eta \\ 
& \lesssim \epsilon \norm{A^2 Q^2}_2 \norm{\Delta_L A^3 U^3_{\neq}}_2, 
\end{align*} 
which by Lemma \ref{lem:SimplePEL} is consistent Proposition \ref{prop:Boot} for $c_0$ sufficiently small. 

The two coefficients are straightforward and are hence omitted for the sake of brevity. 
The remainder terms are even simpler and are hence omitted. This completes the treatment of the reaction term
$\mathcal{T}_{R;\neq 0}$. 

\paragraph{Reaction term $\mathcal{T}_{R; \neq \neq}$} \label{sec:Q2TRneqneq}
Turn finally to $\mathcal{T}_{R;\neq \neq}$, which is more problematic here than in \cite{BGM15I} due to the low frequency growth of $Q^2$ and the lower regularity of $Q^3$. 
As in the treatment of  $\mathcal{T}_{R;\neq 0}$ above in \S\ref{sec:Q2TRneq0}, we sub-divide in frequency more carefully,  
\begin{align*} 
\mathcal{T}_{R;\neq \neq} & \lesssim \frac{\epsilon \jap{t}}{\jap{\nu t^3}^\alpha} \sum\int \mathbf{1}_{k k^\prime(k-k^\prime) \neq 0} \abs{A^{2} \hat{Q}^2_k(\eta,l) A^{2}_{k}(\eta,l) \hat{U}_{k^\prime}^1 (\xi,l^\prime)_{Hi}} Low(k-k^\prime, \eta-\xi, l - l^\prime) d\eta d\xi \\   
& \quad +  \frac{\epsilon \jap{t}^{2}}{\jap{\nu t^3}^\alpha} \sum \int \mathbf{1}_{k k^\prime(k-k^\prime) \neq 0} \abs{A^{2} \hat{Q}^2_k(\eta,l) A^{2}_{k}(\eta,l) \hat{U}_{k^\prime}^2 (\xi,l^\prime)_{Hi}} Low(k-k^\prime, \eta-\xi, l - l^\prime) d\eta d\xi \\  
& \quad +  \frac{\epsilon \jap{t}}{\jap{\nu t^3}^{\alpha-1}} \sum \int \mathbf{1}_{k k^\prime(k-k^\prime) \neq 0} \abs{A^{2} \hat{Q}^2_k(\eta,l) A^{2}_{k}(\eta,l) \hat{U}_{k^\prime}^3 (\xi,l^\prime)_{Hi}} Low(k-k^\prime, \eta-\xi, l - l^\prime) d\eta d\xi \\   
& \quad + \frac{\epsilon^2 \jap{t}}{\jap{\nu t^3}^{\alpha}} \sum \int \mathbf{1}_{k k^\prime(k-k^\prime) \neq 0} \abs{A^{2} \hat{Q}^2_k(\eta,l)} A^{2}_{k}(\eta,l) \left(\abs{\hat{\psi_y}(\xi,l^\prime)_{Hi}} + \abs{\hat{\phi_z}(\xi,l^\prime)_{Hi}} + \abs{\hat{\phi_y}(\xi,l^\prime)_{Hi}}\right) \\ & \quad\quad\quad \times Low(k-k^\prime, \eta-\xi, l - l^\prime) d\eta d\xi \\   
& \quad + \frac{\epsilon^2 \jap{t}^{2}}{\jap{\nu t^3}^{\alpha}} \sum \int \mathbf{1}_{k k^\prime(k-k^\prime) \neq 0} \abs{A^{2} \hat{Q}^2_k(\eta,l) A^{2}_{k}(\eta,l) \left(\hat{\psi_z}(\xi,l^\prime)_{Hi} \right)} Low(k-k^\prime, \eta-\xi, l - l^\prime) d\eta d\xi \\   
& \quad + \mathcal{T}_{R;\neq \neq;\mathcal{R}} \\ 
& = \mathcal{T}_{R;\neq\neq}^{1} + \mathcal{T}_{R;\neq\neq}^2 + \mathcal{T}_{R;\neq\neq}^3  + \mathcal{T}_{R;\neq\neq}^{C1} + \mathcal{T}_{R;\neq\neq}^{C2} + \mathcal{T}_{R;\neq\neq;\mathcal{R}}, 
\end{align*} 
where we used $\epsilon t \jap{\nu t^3}^{-1} \lesssim t^{-1}$ in $\mathcal{T}_{R;\neq\neq}^3$ to reduce the power of time of the $(U^3)_{Hi} \left(\psi_z(\partial_Y-t\partial_X)Q^2\right)_{Lo}$ term. 

Turn first to $\mathcal{T}_{R;\neq\neq}^{1}$, which by \eqref{ineq:ABasic}, \eqref{ineq:AikDelLNoD} and \eqref{ineq:triQuadHL} is given by 
\begin{align*} 
\mathcal{T}_{R;\neq \neq}^{1} & \lesssim \frac{\epsilon \jap{t}^2}{\jap{\nu t^3}^\alpha} \sum \int \abs{A^{2} \widehat{Q^2_k}(\eta,l) \frac{1}{(k^\prime)^2 + (l^\prime)^2 + \abs{\xi - k^\prime t}^2} \jap{\frac{t}{\jap{\xi,l^\prime}}}^{1+\delta_1}} \\ & \quad\quad \times \abs{ A^{1} \Delta_L \widehat{U^1_{k^\prime}}(\xi,l^\prime)_{Hi}} Low(k-k^\prime,\eta-\xi,l-l^\prime) d\eta d\xi \\ 
& \lesssim \frac{\epsilon \jap{t}^2}{\jap{\nu t^3}^{\alpha}}\norm{\left(\sqrt{\frac{\partial_t w}{w}}\tilde{A}^2 + \frac{\abs{\grad}^{s/2}}{\jap{t}^{s}} A^2 \right)  Q^2}_2\norm{\left(\sqrt{\frac{\partial_t w}{w}}\tilde{A}^1 + \frac{\abs{\grad}^{s/2}}{\jap{t}^{s}}A^1\right) \Delta_L U^1_{\neq}}_2 \\ 
& \quad + \frac{\epsilon \jap{t}^{1+\delta_1}}{\jap{\nu t^3}^\alpha} \norm{A^2 Q^2_{\neq}}_2 \norm{A^1 \Delta_L U^1_{\neq}}_2, 
\end{align*}  
which by Lemmas \ref{lem:SimplePEL} and \ref{lem:PEL_NLP120neq}, is consistent with Proposition \ref{prop:Boot} by the bootstrap hypotheses for $\epsilon$ and $\delta_1$ sufficiently small. 
The treatment of $\mathcal{T}_{R;\neq\neq}^{2}$ is essentially the same as $\mathcal{T}^1_{R;\neq\neq}$ and yields 
\begin{align*} 
\mathcal{T}_{R;\neq \neq}^{2} & \lesssim \frac{\epsilon \jap{t}^2}{\jap{\nu t^3}^{\alpha}}\norm{\left(\sqrt{\frac{\partial_t w}{w}}\tilde{A}^2 + \frac{\abs{\grad}^{s/2}}{\jap{t}^{s}}A^2\right) Q^2}_2\norm{\left(\sqrt{\frac{\partial_t w}{w}}\tilde{A}^2 + \frac{\abs{\grad}^{s/2}}{\jap{t}^{s}}A^2\right) \Delta_L U^2_{\neq}}_2 \\ 
& \quad + \frac{\epsilon}{\jap{\nu t^3}^\alpha} \norm{A^2 Q^2_{\neq}}_2 \norm{A^2 \Delta_L U^2_{\neq}}_2,  
\end{align*} 
which again by Lemmas \ref{lem:SimplePEL} and \ref{lem:PEL_NLP120neq}, is  consistent with Proposition \ref{prop:Boot} by the bootstrap hypotheses for $\epsilon$ sufficiently small.
 
Turn next to $\mathcal{T}_{R;\neq\neq}^{3}$.
By \eqref{ineq:A2A3neqneq}, \eqref{ineq:A1A3ReacGain}, and \eqref{ineq:triQuadHL}, we have 
\begin{align*}
\mathcal{T}_{R;\neq\neq}^{3} & \lesssim \frac{\epsilon \jap{t}}{\jap{\nu t^3}^\alpha} \int \abs{A^{2} \widehat{Q^2_k}(\eta,l)} \frac{1}{(k^\prime)^2 + (l^\prime)^2 + \abs{\xi - k^\prime t}^2} \jap{\frac{t}{\jap{\xi,l^\prime}}} \\ 
& \quad\quad \times \left(\sum_{r}\chi^{r,NR}\frac{t}{\abs{r} + \abs{\eta-tr}} + \chi^{\ast;23} \right) \abs{A^{3} \Delta_L \widehat{U^3_{k^\prime}}(\xi,l^\prime)_{Hi}} Low(k-k^\prime,\eta-\xi,l-l^\prime) d\eta d\xi \\ 
& \lesssim \frac{\epsilon \jap{t}}{\jap{\nu t^3}^{\alpha}}\norm{A^2 Q^2}_2 \norm{\Delta_L A^3 U^3_{\neq}}_2, 
\end{align*} 
which after Lemma \ref{lem:SimplePEL}, is consistent with Proposition \ref{prop:Boot}. 

The coefficient error terms are treated the same as in \S\ref{sec:Q2TRneq0}; hence we omit the treatments for brevity and simply conclude 
\begin{align*}
\mathcal{T}_{R;\neq\neq}^{C1} + \mathcal{T}_{R;\neq\neq}^{C2} & \lesssim \frac{\epsilon^2 t}{\jap{\nu t^3}^{\alpha-1}}\norm{A^2 Q^2}_2\norm{AC}_2. 
\end{align*}
The remainder terms $\mathcal{T}_{R;\neq\neq}$ are similarly straightforward and are omitted for brevity as well. 
This completes the treatment of the transport nonlinearity for $Q^2$.

\subsubsection{Dissipation error terms $\mathcal{D}$} \label{sec:DEneqQ2}
Recalling the dissipation error terms and the short-hand \eqref{def:G}, we have 
\begin{align*} 
\mathcal{D}_E & = \nu\sum_{k \neq 0}\int A^2 Q^2_k A^2_k\left( G_{yy}(\partial_{Y} - t\partial_X)^2 Q^2_k + G_{yz}(\partial_Y - t \partial_X)\partial_{Z}Q^2_k + G_{zz}\partial_{ZZ} Q^2_k \right) dV. 
\end{align*} 
These terms can be treated in the same manner as the analogous terms in \cite{BGM15I}; therefore, we omit the treatment for brevity 
and simply conclude the final result: 
\begin{align*}
\mathcal{D}_E & \lesssim c_0 \nu \norm{\sqrt{-\Delta_L} A^2 Q^2}_2^2 + \frac{\epsilon^{1/2}}{\jap{\nu t^3}^{\alpha}} \norm{A^2 Q^2_{\neq}}_2^2 + \frac{\nu^2 \epsilon^{3/2} t^4 }{\jap{\nu t^3}^{\alpha}}\norm{AC}_2^2, 
\end{align*}
which is consistent  with Proposition \ref{prop:Boot} for $\epsilon$ sufficiently small. 

\section{High norm estimate on $Q^3$}
Computing the evolution of $A^{3}Q^3$:  
\begin{align} 
\frac{1}{2}\frac{d}{dt}\norm{A^{3} Q^3}_2^2 & \leq \dot{\lambda}\norm{\abs{\grad}^{s/2}A^{3} Q^3}_2^2 - \norm{\sqrt{\frac{\partial_t w}{w}} \tilde{A}^{3} Q^3}_2^2 -\norm{\sqrt{\frac{\partial_t w_L}{w_L}} A^{3} Q^3}_2^2 
-\frac{2}{t}\norm{\mathbf{1}_{t > \jap{\grad_{Y,Z}}} A^{3} Q^3}_2^2\nonumber \\
& \quad -2 \int A^{3} Q^3 A^{3} \partial_{YX}^t U^3 dV + 2 \int A^{3} Q^3 A^{3} \partial_{ZX}^t U^2 dV \nonumber\\  & \quad
 + \nu \int A^{3} Q^{3} A^{3} \left(\Delta_t Q^3\right) dV
-\int A^{3} Q^3 A^{3}\left( \tilde U \cdot \grad Q^3 \right) dv \nonumber \\ 
& \quad - \int A^{3} Q^3 A^{3} \left[ Q^j \partial_j^t U^3 + 2\partial_i^t U^j \partial_{ij}^t U^3  - \partial_Z^t\left(\partial_i^t U^j \partial_j^t U^i\right) \right] dV \nonumber \\ 
& = \mathcal{D}Q^3 - CK_{L}^3 + LS3 + L3 + \mathcal{D}_E + \mathcal{T} + NLS1 + NLS2 + NLP,  \label{ineq:AQ3_Evo} 
\end{align} 
where we are again using 
\begin{align*}
\mathcal{D}_E = \nu \int A^{3} Q^{3} A^{3} \left((\tilde{\Delta_t} - \Delta_L) Q^3\right) dV. 
\end{align*}
As in \eqref{def:Q2Enums}, let us here recall the following enumerations from \cite{BGM15I}: for $i,j\in \set{1,2,3}$ and $a,b \in \set{0,\neq}$: 
\begin{subequations} \label{def:Q3Enums}
\begin{align}
NLP(i,j,a,b) &= \int A^3 Q^3_{\neq} A^3\left( \partial_Z^t \left(\partial_j^t U^i_a \partial_i^t U^j_b \right) \right) dV \\ 
NLS1(j,a,b) & = -\int A^3 Q^3_{\neq} A^3\left(Q^j_a\partial_j^t U^3_{b}\right) dV \\
NLS2(i,j,a,b) & = -\int A^3 Q^3_{\neq} A^2\left(\partial_i^t U^j_a \partial_i^t\partial_j^t U^3_{b}\right) dV \\
NLP(i,j,0) & = \int A^3 Q^3_{0} A^3\left( \partial_Z^t \left(\partial_j^t U^i_0 \partial_i^t U^j_0 \right) \right) dV \\ 
NLS1(j,0) & = -\int A^3 Q^3_{0} A^3\left(Q^j_0\partial_j^t U^3_{0}\right) dV \\
NLS2(i,j,0) & = -\int A^3 Q^3_{0} A^3\left(\partial_i^t U^j_0 \partial_i^t\partial_j^t U^3_{0}\right) dV \\ 
\mathcal{F} & = -\int A^3 Q^3_{0} A^3\left(\partial_i^t \partial_i^t \partial_j^t \left(U^j_{\neq} U^3_{\neq}\right)_0 - \partial_{Z}^t \partial_j^t \partial_i^t \left(U^i_{\neq} U^j_{\neq}\right)_0 \right)dV \\ 
\mathcal{T}_0 & = -\int A^{3} Q_0^3 A^{3} \left( \tilde U_{0} \cdot \grad Q_{0}^3 \right) dV \\ 
\mathcal{T}_{\neq} & = -\int A^{2} Q_{\neq}^3 A^{3} \left( \tilde U \cdot \grad Q^3 \right) dV. 
\end{align}
\end{subequations} 

Note we have split the nonlinearity up analogously to what is done in \eqref{def:Q2Enums} above. 
 
\subsection{Zero frequencies}
As in the treatment of $A^2Q^2$ in \S\ref{sec:AQ2Zero}, the estimate on $Q_0^3$ is very different than the estimate on $Q^3_{\neq}$ and are hence naturally separated. 

\subsubsection{Transport nonlinearity} 
The treatment of $\mathcal{T}_0$, the \textbf{(2.5NS)} contribution to the transport nonlinearity, goes through exactly the same as the corresponding treatment for $Q_0^2$ in \S\ref{sec:TransQ20} (as the main problems in $A^3$ will only arise when changing the $X$ frequencies) and hence, for the sake of brevity this term is omitted. 

\subsubsection{Nonlinear pressure and stretching}
The treatment of zero frequency pressure and stretching contributions in \eqref{ineq:AQ3_Evo} is very similar to the treatment
used for $Q^2_0$ in \S\ref{sec:NLPSQ20} except that since we are estimating with $A^3$, there is no loss on factors involving $U^3$ as there is in \S\ref{sec:NLPSQ20}. 
As the treatment here is analogous (except easier), we omit these terms for brevity. 

\subsubsection{Forcing from non-zero frequencies} \label{sec:NzeroForcingQ3}
Turn next to the treatment of $\mathcal{F}$ (defined above in \eqref{def:Q3Enums}), for nonlinear interactions of type \textbf{(F)}.
In accordance with the toy model in \S\ref{sec:Toy}, we will find that the forcing from non-zero frequencies on $Q^3_0$ is more extreme than those on $Q_0^2$.  In particular, unlike in \S\ref{sec:NzeroForcing} above, in order to treat the case $\nu \ll \epsilon$ we will need the regularity imbalances. 
Write
\begin{align*} 
\mathcal{F} & = -\int A^3 Q^3_0 A^3 \left(\partial_Y^t \partial_Y^t\partial_j^t \left(U^j_{\neq} U^3_{\neq}\right)_0 - \partial_Z^t \partial_Z^t \partial_Y^t \left(U^2_{\neq} U^3_{\neq}\right)_0 - \partial_Z^t \partial_Y^t \partial_Y^t \left(U^2_{\neq} U^2_{\neq}\right)_0 \right) dV \\ 
& = F^1 + F^2 + F^3.
\end{align*}
The most dangerous term is $F^1$; we omit the other two for brevity as they are easy variants of $F^1$ 
and the treatments in \S\ref{sec:NzeroForcing}. 
 Write 
\begin{align*}
F^1 = -\int A^3 Q^3 A^3\left( \partial_Y^t \partial_Y^t\partial_Y^t \left(U^2_{\neq} U^3_{\neq}\right)_0 + \partial_Y^t \partial_Y^t\partial_Z^t \left( U^3_{\neq} U^3_{\neq}\right)_0 \right) dV = F^{1;2} + F^{1;3}.
\end{align*}
The first term, $F^{1;2}$, is the leading order contribution (at least when $U^2$ is in high frequency) due to the $t^3$ that will be present near the critical times due to the $(\partial_Y)^3$ (near the critical times $\partial_Y \sim t\partial_X$), and hence let us focus on this and omit $F^{1;3}$ for brevity.    
Expand $F^{1;2}$ with a quintic paraproduct and group all of the terms where the coefficients appear in low frequency with the remainder:
\begin{align*}  
F^{1;2} & = -\sum_{k\neq 0} \int A^{3}Q^3_0 A_0^{3} \partial_Y\partial_Y\partial_Y\left( \left(U^3_{-k}\right)_{Hi} \left( U^2_k\right)_{Lo}\right)dV \\
& \quad - \sum_{k\neq 0} \int A^{3}Q^3_0 A_0^{3} \partial_Y \partial_Y\partial_Y\left( \left(U^3_{-k}\right)_{Lo} \left( U^2_k\right)_{Hi}\right) dV \\ 
& \quad - \sum_{k\neq 0}  \int A^{3}Q^3_0 A_0^{3} \left((\psi_y)_{Hi}\partial_Y + (\phi_y)_{Hi}\partial_Z\right)\partial_Y\partial_Y\left( \left(U^3_{-k}\right)_{Lo} \left( U^2_k\right)_{Lo}\right) dV \\ 
& \quad - \sum_{k\neq 0} \int A^{3}Q^3_0 A_0^{3} \partial_Y \left( \left((\psi_y)_{Hi}\partial_Y + (\phi_y)_{Hi}\partial_Z\right)\partial_Y\left( \left(U^3_{-k}\right)_{Lo} \left( U^2_k\right)_{Lo}\right)\right) dV \\
& \quad - \sum_{k\neq 0} \int A^{3}Q^3_0 A_0^{3} \partial_Y \partial_Y \left( \left((\psi_y)_{Hi}\partial_Y + (\phi_y)_{Hi}\partial_Z\right) \left( \left(U^3_{-k}\right)_{Lo} \left( U^2_k\right)_{Lo}\right)\right) dV \\ 
& \quad + F^1_{\mathcal{R},C} \\ 
& = F_{HL} + F_{LH} + F_{C1} + F_{C2} + F_{C3} + F_{\mathcal{R},C},    
\end{align*}
where here $F_{\mathcal{R}}$ includes the remainders from the paraproduct and terms where coefficients appear in low frequency. 

Turn first to the easier $F_{HL}$. 
From \eqref{ineq:AprioriUneq}, \eqref{ineq:ABasic}, \eqref{ineq:AdelLij}, and \eqref{ineq:triQuadHL} we have, 
\begin{align*} 
F_{HL} & \lesssim \frac{\epsilon}{\jap{t} \jap{\nu t^3}^{\alpha}} \sum_{k\neq 0} \sum_{l,l^\prime} \int \abs{A^{3} \widehat{Q^3_0}(\eta,l)} \frac{\abs{\eta}^3 \jap{\frac{t}{\jap{\xi,l^\prime}}}^2}{k^2 + (l^\prime)^2 + \abs{\xi-kt}^2} \abs{ \Delta_L A^{3}\widehat{U^3_{k}}(\xi,l^\prime)_{Hi}} Low\left(-k,\eta-\xi,l-l^\prime \right) d\eta d\xi \\ 
& \lesssim \frac{\epsilon \jap{t}^{2}}{\jap{\nu t^3}^{\alpha}}\norm{\left(\sqrt{\frac{\partial_t w}{w}}\tilde{A}^{3} + \frac{\abs{\grad}^{s/2}}{\jap{t}^{s}}A^{3} \right) Q^3}_2 \norm{\left(\sqrt{\frac{\partial_t w}{w}}\tilde{A}^{3} + \frac{\abs{\grad}^{s/2}}{\jap{t}^s}A^{3}\right) \Delta_L U_{\neq}^3}_2
\\ & \quad + \frac{\epsilon}{\jap{t}\jap{\nu t^3}^\alpha}\norm{\sqrt{-\Delta_L} A^{3}Q^3}_2 \norm{A^{3}\Delta_L U^3_{\neq}}_2,    
\end{align*} 
which, after the application of Lemmas \ref{lem:PEL_NLP120neq} and \ref{lem:SimplePEL}, is consistent with Proposition \ref{prop:Boot}. 
Notice the importance of the inviscid damping to reduce the power of $t$. 

Turn next to $F_{LH}$, which is the term appearing in the toy model in \S\ref{sec:Toy} as one of the leading order contributions to the nonlinear interaction \textbf{(F)}. 
Here, it is the regularity imbalance between $Q^2_{\neq}$ and $Q^3_0$ which will reduce the power of $t$. 
By \eqref{ineq:AprioriUneq} and Lemma \ref{lem:ABasic} we have 
\begin{align*}
F_{LH} & \lesssim  \frac{\epsilon}{\jap{\nu t^3}^{\alpha}}\sum_{k\neq 0} \sum_{l,l^\prime} \int \abs{A^{3} \widehat{Q^3_0}(\eta,l) A^3_0(\eta,l) \frac{\abs{\eta}^3}{k^2 + (l^\prime)^2 + \abs{\xi-kt}^2} \Delta_L \widehat{U^2_{k}}(\xi,l^\prime)_{Hi}} Low(-k,\eta-\xi,l-l^\prime) d\eta d\xi \\ 
& \lesssim \frac{\epsilon}{\jap{\nu t^3}^{\alpha}}\sum_{k\neq 0} \sum_{l,l^\prime} \int \abs{ \widehat{Q^3_0}(\eta,l)} \frac{\abs{\eta}^3}{k^2 + (l^\prime)^2 + \abs{\xi-kt}^2} \\ & \quad\quad \times \left(\sum_{r}\chi^{NR,r}\frac{\abs{r} + \abs{\eta-tr}}{t} \tilde{A}^{3}_0(\eta,l) \tilde{A}_k^2(\xi,l^\prime) + \chi^{\ast;32} A^{3}_0(\eta,l) A_k^2(\xi,l^\prime) \right) \\ & \quad\quad \times \jap{\frac{t}{\jap{\xi,l^\prime}}}  \abs{A_k^2 \Delta_L \widehat{U^2_{k}}(\xi,l^\prime)_{Hi}} Low(-k,\eta-\xi,l-l^\prime) d\eta d\xi. 
\end{align*}
Therefore, by \eqref{ineq:A03A2YYY} and \eqref{ineq:A03A2YYY2}, followed by \eqref{ineq:triQuadHL}, we have 
\begin{align*}
F_{LH}^1 & \lesssim \frac{\epsilon \jap{t}^2}{\jap{\nu t^3}^{\alpha}} \norm{\left(\sqrt{\frac{\partial_t w}{w}}\tilde{A}^3 + \frac{\abs{\grad}^{s/2}}{\jap{t}^{s}}A^3\right)  Q^3_0}_2\norm{\left(\sqrt{\frac{\partial_t w}{w}}\tilde{A}^2 + \frac{\abs{\grad}^{s/2}}{\jap{t}^{s}}A^2\right) \Delta_L U^2_{\neq}}_2 \\ 
& \quad + \frac{\epsilon}{\jap{\nu t^3}^{\alpha}}\norm{\sqrt{-\Delta_L}A^3 Q^3_0}_2 \norm{A^2 \Delta_L U_{\neq}^2}_{2}, 
\end{align*}
which by Lemmas \ref{lem:PEL_NLP120neq} and \ref{lem:SimplePEL} is consistent with Proposition \ref{prop:Boot} for $\epsilon$ sufficiently small. 
 
The terms associated with the coefficient terms are treated the same as the corresponding terms in \S\ref{sec:NzeroForcing} and are hence omitted for brevity and we simply conclude the results
\begin{align*} 
F^1_{C1} + F^1_{C2} + F^1_{C3} & \lesssim \epsilon^{3/2}\norm{\sqrt{-\Delta_L} A^{3}Q^3}_2^2 + \frac{\epsilon^{5/2}}{\jap{\nu t^3}^{2\alpha}}\norm{AC}^2_2.   
\end{align*} 
The remainder terms are similarly straightforward or easy variants of the other treatments and are hence omitted as well.  
This completes the treatment of $F^1$. 
As mentioned above, the treatments of $F^2$ and $F^3$ are similar (but easier) and hence also omitted. 

\subsubsection{Zero frequency dissipation error terms}
The treatment of the dissipation error terms for $Q^3_0$ is the same as $Q_0^2$ as outlined in \S\ref{sec:DEQ02}, and therefore is omitted for the sake of brevity.

\subsection{Non-zero frequencies}
\subsubsection{Nonlinear pressure $NLP$} \label{sec:NLP3}

\paragraph{Treatment of $NLP(1,j,0,\neq)$} \label{sec:NLPQ3ij_0neq}
This term is the analogue of the nonlinear terms treated in \S\ref{sec:NLP213}. 
Note that $j \neq 1$ by the zero frequency assumption. 
We can essentially use the same treatment, although here it is easier since $Y$ derivatives are slightly harder than $Z$ derivatives 
and because we are imposing one less power of time control on $Q^3_{\neq}$ than on $Q^2_{\neq}$. 
For this reason, we omit the treatment for brevity and simply conclude the result: 
\begin{align*} 
NLP(1,j,0,\neq) & \lesssim c_{0}\norm{\left(\sqrt{\frac{\partial_t w}{w}}\tilde{A}^3 + \frac{\abs{\grad}^{s/2}}{\jap{t}^{s}}A^3 \right) Q^3}_2\norm{\left(\sqrt{\frac{\partial_t w}{w}}\tilde{A}^j + \frac{\abs{\grad}^{s/2}}{\jap{t}^s}A^j \right) \Delta_L U^j_{\neq}}_2 \\
& \quad + \epsilon \norm{\sqrt{-\Delta_L} A^3 Q^3}_2^2 + \mathbf{1}_{t \ll \epsilon^{-1/2}}\epsilon^{1/2}\norm{A^j \Delta_L U^j_{\neq}}_2^2 + \epsilon^{3/2}\norm{\sqrt{-\Delta_L} A^j \Delta_L U^j_{\neq}}_2^2 \\ 
& \quad + \frac{\epsilon \jap{t}}{\jap{\nu t^3}^{\alpha}}\norm{A^3 Q^3_{\neq}}_2 \norm{A^1 \jap{\grad}^2 U_0^1}_2 + \frac{\epsilon^2}{\jap{\nu t^3}^{\alpha-1}}\norm{A^3 Q^3_{\neq}}_2 \norm{AC}_2,  
\end{align*}   
 which, after Lemmas \ref{lem:PELbasicZero}, \ref{lem:PEL_NLP120neq}, \ref{lem:SimplePEL}, and \ref{lem:PELED}, is consistent Proposition \ref{prop:Boot} for $\epsilon$ sufficiently small. 

\paragraph{Treatment of $NLP(i,j,0,\neq)$ with $i \in \set{2,3}$} \label{sec:NLPQ3_0neq_notX} 
This is the analogue of the nonlinear terms treated in \S\ref{sec:NLPQ2_0neq_notX} above.  
These can treated analogously to the treatment in \S\ref{sec:NLPQ2_0neq_notX}, but in fact it is much easier here due to the fact that $Q^3$ is growing quadratically at `low' frequencies.
In particular, we can deduce (using also $j \neq 1$), 
\begin{align*}
NLP(i,j,0,\neq) & \lesssim \epsilon \norm{A^3Q^3_{\neq}}_2 \norm{\Delta_L A^j U^j_{\neq}}_2 + \frac{\epsilon \jap{t}}{\jap{\nu t^3}^{\alpha-1}} \norm{A^3Q^3_{\neq}}_2 \norm{A U^i_{0}}_2 \\ 
& \quad + \frac{\epsilon^2\jap{t}}{\jap{\nu t^3}^{\alpha-1}}\norm{A^3 Q^3}_2 \norm{AC}_2,  
\end{align*}   
which after Lemmas \ref{lem:PELbasicZero}, \ref{lem:PEL_NLP120neq}, and \ref{lem:SimplePEL}, is consistent with Proposition \ref{prop:Boot} for $\epsilon$ sufficiently small.  

\paragraph{Treatment of $NLP(i,j,\neq,\neq)$} \label{sec:NLPQ3_neqneq}
These terms are fairly straightforward. 
The term with $i = j = 3$ cancels with the $NLS$ terms. 
Let us just treat $NLP(1,3,\neq,\neq)$ and omit the others for brevity, which follow by similar arguments.
Expand with a paraproduct, as usual grouping higher order terms involving the coefficients in low frequency with the remainder 
\begin{align*}
NLP(1,3,\neq,\neq) & =  \int  A^3 Q^3_{\neq} A^3 \partial_Z \left( \left(\partial_Z U^1_{\neq}\right)_{Lo} (\partial_XU^3_{\neq})_{Hi}  \right) dV \\ 
& \quad + \int A^3 Q^3_{\neq} A^3 \partial_Z \left( \left(\partial_Z U^1_{\neq}\right)_{Hi} (\partial_XU^3_{\neq})_{Lo}  \right) dV \\ 
& \quad + \int A^3 Q^3_{\neq} A^3 \left( \left( (\phi_z)_{Hi}\partial_Z + (\psi_z)_{Hi}(\partial_Y - t\partial_X)  \right)  \left( \left(\partial_Z U^1_{\neq}\right)_{Lo} (\partial_XU^3_{\neq})_{Lo}  \right) \right) dV \\ 
& \quad + \int A^3 Q^3_{\neq} A^3 \partial_Z\left(\left( \left( (\phi_z)_{Hi}\partial_Z + (\psi_z)_{Hi}(\partial_Y - t\partial_X)  \right) U^1_{\neq}\right)_{Lo} (\partial_XU^3_{\neq})_{Lo}  \right) dV \\ 
& \quad + P_{\mathcal{R},C} \\  
& = P_{LH} + P_{HL} + P_{C1} + P_{C2} + P_{\mathcal{R},C}, 
\end{align*}
where $P_{\mathcal{R},C}$ includes all of the remainders from the quartic paraproduct as well as the higher order terms involving coefficients as low frequency factors.  

Consider $P_{LH}$ first.
By \eqref{ineq:AprioriUneq} followed by \eqref{ineq:A3A3neqneq}, by \eqref{ineq:A3neqA3neqZX} and \eqref{ineq:triQuadHL} it follows that  
\begin{align*}
P_{LH} & \lesssim \frac{\epsilon \jap{t}^{\delta_1}}{\jap{\nu t^3}^{\alpha}} \sum_{k \neq 0} \int \abs{A^3 \widehat{Q^3_{k}}(\eta,l)} \frac{\abs{l k^\prime}}{\abs{k^\prime}^2 + \abs{l^\prime}^2 + \abs{\xi-tk^\prime}^2} \\ & \quad\quad \times
\left(\chi^{R,NR}\frac{t}{\abs{k} + \abs{\eta-kt}} + \chi^{NR,R}\frac{\abs{k^\prime} + \abs{\eta - k^\prime t}}{t} + \chi^{\ast;33} \right) \abs{A^3\Delta_L\widehat{U^3_{k^\prime}}(\xi,l^\prime)_{Hi}} \\ & \quad\quad \times Low(k - k^\prime,\eta-\xi,l-l^\prime) d\eta d\xi \\ 
& \lesssim \frac{\epsilon \jap{t}^{1+\delta_1}}{\jap{\nu t^3}^{\alpha}}\norm{A^3Q^3}_2\norm{A^3 \Delta_L U_{\neq}^3}_2, 
\end{align*}
which after Lemma \ref{lem:SimplePEL}, is consistent with Proposition \ref{prop:Boot} for $\delta_1$ and $\epsilon$ sufficiently small. 

Consider next $P_{HL}$. 
By \eqref{ineq:AprioriUneq} followed by \eqref{ineq:ABasic}, we have
\begin{align*}
P_{HL} & \lesssim \frac{\epsilon}{\jap{\nu t^3}^{\alpha}} \sum_{k \neq 0} \int \abs{A^3 \widehat{Q^3_{k}}(\eta,l)} \frac{\abs{l l^\prime} \jap{t}}{\abs{k^\prime}^2 + \abs{l^\prime}^2 + \abs{\xi-tk^\prime}^2} \jap{\frac{t}{\jap{\xi,l^\prime}}}^{\delta_1-1} \\ & \quad\quad \times \abs{A^1\Delta_L\widehat{U^1_{k^\prime}}(\xi,l^\prime)_{Hi}} Low(k - k^\prime,\eta-\xi,l-l^\prime) d\eta d\xi \\ 
& \lesssim \frac{\epsilon \jap{t}}{\jap{\nu t^3}^{\alpha}}\norm{A^3 Q^3}_2 \norm{A^1 \Delta_L U^1_{\neq}}_2, 
\end{align*}
which is consistent with Proposition \ref{prop:Boot} for $\epsilon$ sufficiently small after applying Lemma \ref{lem:SimplePEL}. 

The coefficient error terms and the remainder terms are straightforward (easier) variants of the treatment in \S\ref{sec:NLPQ2_neqneq} or of the above treatments of $P_{HL}$ and $P_{LH}$, and hence are omitted for brevity.   
The other nonlinear pressure terms are similar to, or easier than, the above, and are hence omitted for brevity. 

\subsubsection{Nonlinear stretching $NLS$} \label{sec:NLSQ3}
Controlling the $NLS$ terms in the evolution of $Q^3$ is in general 
slightly harder than for $Q^2$ (treated above in \S\ref{sec:NLSQ2}), due to the fact that $U^3$ is larger than $U^2$. 
Moreover, we occasionally have to deal with the imbalance in the regularities inherent to $A^3$. 

\paragraph{Treatment of $NLS1(j,\neq,0)$ and  $NLS1(j,0,\neq)$} \label{sec:NLS1Q3_neq0} 
Consider first the  $NLS1(j,0,\neq)$ terms. 
Due to the large size of $Q_0^1$, it turns out $j = 1$ is the hardest case, and hence we only treat this case (the case $j = 3$ is complicated by the regularity imbalance of $A^3_k$ compared to $A^3_0$ (see Lemma \ref{lem:ABasic}), however, even at the critical time, the loss is at most $\jap{t}$, which is still not more than what is lost when comparing $A^3_k$ to $A^1_k$).  
Expanding with a paraproduct
\begin{align*}
NLS1(1,0,\neq) & = -\sum_{k \neq 0}\int A^3 Q^3_k A^3_k \left( (Q^1_0)_{Hi} (\partial_X U^3_{k})_{Lo}    \right) dV - \sum_{k \neq 0}\int A^3 Q^3_k A^3_k \left( (Q^1_0)_{Lo} (\partial_X U^3_{k})_{Hi}    \right) dV + S_{\mathcal{R}}  \\ 
& = S_{HL} + S_{LH} + S_{\mathcal{R}}.
\end{align*}
For the $S_{HL}$ term, it follows from \eqref{ineq:AprioriUneq}, Lemma \ref{lem:ABasic}, and \eqref{ineq:triQuadHL},  
\begin{align*}
S_{HL} & \lesssim \frac{\epsilon t}{\jap{\nu t^3}^{\alpha}}\norm{A^3 Q^3}_2 \norm{A^1 Q^1_0}_2.  
\end{align*}
For the $S_{LH}$ term, by \eqref{ineq:AprioriU0} and \eqref{ineq:ABasic}, followed by \eqref{ineq:AiPartX} and \eqref{ineq:triQuadHL},   
\begin{align*}
S_{LH} & \lesssim \epsilon t \sum \int \abs{A^3 \widehat{Q^3_k}(\eta,l)} \frac{\abs{k}}{k^2 + \abs{l^\prime}^2 + \abs{\xi-kt}^2} \abs{A^3 \Delta_L \widehat{U^3_k}(\xi,l^\prime)_{Hi}} Low(\eta-\xi,l-l^\prime) d\eta d\xi \\ 
& \lesssim \epsilon t \norm{\left(\sqrt{\frac{\partial_t w}{w}}\tilde{A}^3 + \frac{\abs{\grad}^{s/2}}{\jap{t}^{s}}A^3\right) Q^3}_2 \norm{\left(\sqrt{\frac{\partial_t w}{w}}\tilde{A}^3 + \frac{\abs{\grad}^{s/2}}{\jap{t}^{s}}A^3\right) \Delta_L U^3_{\neq}}_2 + \epsilon\norm{A^3 Q^3}_2\norm{A^3 \Delta_L U^3_{\neq}}_2,
\end{align*}
which is consistent with Proposition \ref{prop:Boot} for $\epsilon$ sufficiently small by Lemmas \ref{lem:PEL_NLP120neq} and \ref{lem:SimplePEL}.  
The  remainder term is straightforward and is hence omitted. 
As mentioned above, the remaining $NLS1(j,0,\neq)$ terms are omitted as well as they are similar. 

Consider next the  $NLS1(j,\neq,0)$ terms. 
Notice that $j \neq 1$ by the nonlinear structure. 
The remaining contributions are not quite the same: due to the regularity imbalances in $A^3$, the case $j = 3$ is slightly harder (note this does not cancel with the other pressure/stretching terms). 
Hence, we treat this term and omit the $j = 2$ contribution. 
As usual, begin with a paraproduct and group the terms where the coefficients appear in low frequency with the remainder: 
\begin{align*}
NLS1(3,\neq,0) & = -\sum_{k \neq 0} \int A^3 Q^3_{k} A^3\left( (Q^3_{k})_{Hi} (\partial_Z U_0^3)_{Lo} \right) dV - \sum\int A^3 Q^3_{k} A^3\left( (Q^3_{k})_{Lo} (\partial_Z U_0^3)_{Hi} \right) dV \\  
& \quad - \sum_{k \neq 0}\int A^3 Q^3_{k} A^3\left( (Q^3_{k})_{Lo} ((\phi_z)_{Hi}\partial_Z + (\psi_z)_{Hi} \partial_Y)(U_0^3)_{Lo} \right) dV + S_{\mathcal{R},C} \\ 
& = S_{HL}  + S_{LH} + S_{C} + S_{\mathcal{R},C}. 
\end{align*}
For the first term, $S_{HL}$, from \eqref{ineq:AprioriU0} and \eqref{ineq:ABasic} we have
\begin{align*}
S_{HL} & \lesssim \epsilon\norm{A^3 Q^3}_2^2, 
\end{align*}
which is consistent with Proposition \ref{prop:Boot} for $c_0$ sufficiently small. 
For the second term, $S_{LH}$, we have by Lemma \ref{lem:ABasic}, Lemma \ref{dtw}, and \eqref{ineq:triQuadHL} (note that the zero frequency is never resonant and hence the $\chi^{NR,R}$ term disappears),  
\begin{align*}
S_{LH} & \lesssim \frac{\epsilon t^2}{\jap{\nu t^3}^{\alpha}} \sum_{k \neq 0} \int \abs{\widehat{Q^3_k}(\eta,l)} \jap{\frac{t}{\jap{\xi,l^\prime}}}^{-2}\\ 
& \quad\quad \times \left( \chi^{R,NR}\frac{t}{\abs{k} + \abs{\eta-kt}}\tilde{A}_k^3(\eta,l)\tilde{A}_0^3(\xi,l^\prime) + A_k^3(\eta,l) A_0^3(\xi,l^\prime) \right) \\ 
& \quad\quad \times \frac{\abs{l^\prime}}{\jap{\xi,l^\prime}^2} \abs{A^3 \jap{\grad}^2 \widehat{U^3_0}(\xi,l^\prime)} d\xi d\eta \\ 
& \lesssim \frac{\epsilon t^2}{\jap{\nu t^3}^{\alpha}} \norm{\left(\sqrt{\frac{\partial_t w}{w}}\tilde{A}^3 + \frac{\abs{\grad}^{s/2}}{\jap{t}^{s}}A^3\right) Q^3}_2 \norm{\left(\sqrt{\frac{\partial_t w}{w}}\tilde{A}^3 + \frac{\abs{\grad}^{s/2}}{\jap{t}^{s}}A^3\right) \jap{\grad}^2 U^3_0}_2 \\ 
& \quad + \frac{\epsilon t}{\jap{\nu t^3}^{\alpha}} \norm{A^3 Q^3}_2 \norm{A^3 \jap{\grad}^2 U^3_0}_2, 
\end{align*}
which, by Lemmas \ref{lem:PELCKZero} and \ref{lem:PELbasicZero}, is consistent with Proposition \ref{prop:Boot} for $\epsilon$ and $c_0$ sufficiently small.

\paragraph{Treatment of $NLS1(j,\neq,\neq)$} \label{sec:NLS1Q3_neqneq} 
All of these terms can be treated in a similar fashion, in fact, $j=3$ is the hardest due to the regularity losses together with a $\partial_Z$ (as opposed to $\partial_X$ as in $j=1$). 
Hence, let us just consider the case $j=3$ and omit the others for brevity. 
Expand the term with a paraproduct, as usual leaving the terms with coefficients in low frequency with the remainder,
\begin{align*}
NLS1(1,\neq,\neq) & = -\int A^3 Q^3_{\neq} A^3\left( (Q^3_{\neq})_{Hi} (\partial_Z U_{\neq}^3)_{Lo} \right) dV  - \int A^3 Q^3_{\neq} A^3\left( (Q^3_{\neq})_{Lo} (\partial_Z U_{\neq}^3)_{Hi} \right) dV \\
& \quad - \int A^3 Q^3_{\neq} A^3\left( (Q^3_{\neq})_{Lo} \left( \left( (\phi_z)_{Hi}\partial_Z + (\psi_z)_{Hi}\partial_Y \right) (U_{\neq}^3)_{Lo}\right) \right) dV + S_{\mathcal{R}} \\ 
& = S_{HL} + S_{LH} +  S_{\mathcal{R}}.
\end{align*}
By \eqref{ineq:AprioriUneq}, Lemma \ref{lem:ABasic}, and \eqref{ineq:triQuadHL} (the loss of $t$ is due to the regularity imbalances), 
\begin{align*}
S_{HL} & \lesssim \frac{\epsilon \jap{t}}{\jap{\nu t^3}^{\alpha}}\norm{A^3Q^3}_2^2, 
\end{align*}
which is consistent with Proposition \ref{prop:Boot} for $\epsilon$ sufficiently small by Lemma \ref{lem:SimplePEL}. 

For $S_{LH}$, we have to be a little more careful.
By \eqref{ineq:A3A3neqneq}, \eqref{ineq:Boot_ED} 
\begin{align*}
S_{LH} & \lesssim \frac{\epsilon \jap{t}^{2}}{\jap{\nu t^3}^{\alpha}} \sum_{k \neq 0} \int \abs{\widehat{Q^3_{k}}(\eta,l)} \frac{\abs{l^\prime}}{\abs{k^\prime}^2 + \abs{l^\prime}^2 + \abs{\xi-tk^\prime}^2}\\ 
& \quad\quad \times \left(\chi^{R,NR}\frac{t}{\abs{k} + \abs{\eta-kt}}\tilde{A}^3_k(\eta,l) \tilde{A}^3_{k^\prime}(\xi,l^\prime)  + \chi^{NR,R}\frac{\abs{k^\prime}  + \abs{\eta - k^\prime t}}{t} \tilde{A}^3_k(\eta,l) \tilde{A}^3_{k^\prime}(\xi,l^\prime) \right. \\ & \quad\quad\quad + \chi^{\ast;33} A^3_k(\eta,l) A^3_{k^\prime}(\xi,l^\prime)  \bigg) \abs{A^3\Delta_L\widehat{U^3_{k^\prime}}(\xi,l^\prime)_{Hi}} Low(k - k^\prime,\eta-\xi,l-l^\prime) d\eta d\xi. 
\end{align*}
Therefore by \eqref{ineq:A33PartX}, \eqref{ineq:AikDelLNoD}, and \eqref{ineq:triQuadHL}, there holds 
\begin{align*}
S_{LH} & \lesssim \frac{\epsilon \jap{t}^2}{\jap{\nu t^3}^{\alpha}}\norm{\left(\sqrt{\frac{\partial_t w}{w}}\tilde{A}^3 + \frac{\abs{\grad}^{s/2}}{\jap{t}^{s}}A^3\right) Q^3}_2\norm{\left(\sqrt{\frac{\partial_t w}{w}}\tilde{A}^3 + \frac{\abs{\grad}^{s/2}}{\jap{t}^{s}}A^3\right) \Delta_L U^3_{\neq}}_2 \\ 
& \quad + \frac{\epsilon \jap{t}}{\jap{\nu t^3}^{\alpha}}\norm{A^3 Q^3}_2 \norm{A^3 \Delta_L U^3_{\neq}}_2, 
\end{align*}
which is consistent with Proposition \ref{prop:Boot} after Lemmas \ref{lem:PEL_NLP120neq} and Lemma \ref{lem:SimplePEL}. 

The coefficient error term $S_C$ and remainder term $S_{\mathcal{R}}$ are both straightforward or easy variants of estimates already performed and hence are omitted for brevity. 

\paragraph{Treatment of $NLS2(i,j,0,\neq)$} 
Recall \eqref{def:Q3Enums} and notice that $i \neq 1$. 
These terms are treated in essentially the same way as  $NLS1(3,\neq,0)$ (or $NLS1(2,\neq,0)$) and hence we omit the treatment for brevity. 

\paragraph{Treatment of $NLS2(i,j,\neq,0)$} 
Recall \eqref{def:Q3Enums} and notice that neither $i$ nor $j$ can be $1$ in this case. 
These terms are very similar to $NLS1(2,0,\neq)$ and are hence omitted for brevity. 

\paragraph{Treatment of $NLS2(i,j,\neq,\neq)$}
First, notice that $i = j = 3$ cancels with the $NLS$ terms. 
The most difficult term is $i = 2$ and $j = 3$; let us briefly comment on this term and omit the others for brevity. 
Expanding with a paraproduct
\begin{align*}
NLS2(2,3,\neq,\neq) & =  -\int  A^3 Q^3_{\neq} A^3 \left( \left( (\partial_Y-t\partial_X) U^3_{\neq}\right)_{Lo} ( (\partial_Y-t\partial_X) \partial_Z U^3_{\neq})_{Hi}  \right) dV \\ 
& \quad - \int A^3 Q^3_{\neq} A^3 \partial_Z \left( \left( (\partial_Y - t\partial_X) U^3_{\neq}\right)_{Hi} (\partial_Z(\partial_Y - t\partial_X)U^3_{\neq})_{Lo}  \right) dV \\ 
& \quad + S_{C1} + S_{C2} + S_{C3} + S_{\mathcal{R},C} \\ 
& = S_{LH} + S_{HL} + S_{C1} + S_{C2} + S_{C3} + S_{\mathcal{R},C}, 
\end{align*}
where $S_{\mathcal{R},C}$ denotes the remainders and $S_{Ci}$ denote terms in which the coefficients appear in high frequency; these are very similar to many terms we have already treated and are hence omitted.
The leading order terms are treated in essentially the same manner; the $S_{LH}$ term is clearly the harder one, so let us just show the treatment of this one. 
For $LH$ term we have, by \eqref{ineq:3DEoneRegBalII},
\begin{align*}
S_{LH} & \lesssim \frac{\epsilon \jap{t}}{\jap{\nu t^3}^{\alpha}} \sum_{k \neq 0} \int \abs{\widehat{Q^3_{k}}(\eta,l)} \frac{\abs{\xi - tk^\prime} \abs{l^\prime}}{\abs{k^\prime}^2 + \abs{l^\prime}^2 + \abs{\xi-tk^\prime}^2}\\ 
& \quad\quad \times \left(\chi^{R,NR}\frac{t}{\abs{k} + \abs{\eta-kt}}\tilde{A}^3_k(\eta,l) \tilde{A}^3_{k^\prime}(\xi,l^\prime)  + \chi^{NR,R}\frac{\abs{k^\prime}  + \abs{\eta - k^\prime t}}{t} \tilde{A}^3_k(\eta,l) \tilde{A}^3_{k^\prime}(\xi,l^\prime) \right. \\ & \quad\quad\quad + \chi^{\ast;33} A^3_k(\eta,l) A^3_{k^\prime}(\xi,l^\prime)  \bigg) \abs{A^3\Delta_L\widehat{U^3_{k^\prime}}(\xi,l^\prime)_{Hi}} Low(k - k^\prime,\eta-\xi,l-l^\prime) d\eta d\xi \\ 
& \lesssim \frac{\epsilon t^2}{\jap{\nu t^3}^\alpha}\norm{\left(\sqrt{\frac{\partial_t w}{w}}\tilde{A}^3 + \frac{\abs{\grad}^{s/2}}{\jap{t}^{s}}A^3\right)Q^3}_2 \norm{\left(\sqrt{\frac{\partial_t w}{w}}\tilde{A}^3 + \frac{\abs{\grad}^{s/2}}{\jap{t}^{s}}A^3\right)\Delta_L U^3_{\neq}}_2 \\ 
& \quad + \frac{\epsilon t}{\jap{\nu t^3}^{\alpha}}\norm{A^3 Q^3}_2\norm{A^3 \Delta_L U^3_{\neq}}_2, 
\end{align*}
which is consistent with Proposition \ref{prop:Boot} by Lemmas \ref{lem:PEL_NLP120neq} and \ref{lem:SimplePEL}. 

\subsubsection{Transport nonlinearity $\mathcal{T}$} \label{sec:Q3_TransNon}
Begin with a paraproduct decomposition: 
 \begin{align*} 
 \mathcal{T}_{\neq}  & = -\int A^{3} Q^3_{\neq} A^{3} \left( \tilde U_{Lo} \cdot \grad Q^3_{Hi} \right) dV -\int A^{3} Q^3_{\neq} A^{3} \left( \tilde U_{Hi} \cdot \grad Q^3_{Lo} \right) dV + \mathcal{T}_{\mathcal{R}} \\  
& = \mathcal{T}_T + \mathcal{T}_{R} + \mathcal{T}_{\mathcal{R}}, 
\end{align*}
where $\mathcal{T}_{\mathcal{R}}$ includes the remainder (as above in \S\ref{sec:Q2_TransNon}, we use the terminology `transport' and `reaction' for the first two terms respectively).
There are two interesting challenges here. 
First, the additional $ + e^{\mu \abs{l}^{1/2}}$ was added in \eqref{def:A} because large regularity imbalances caused by $w^3$ would have been problematic at high $Z$ frequencies in the in the `transport' contribution.  
Second, we will see that the `reaction' contribution is significantly more difficult and, as predicted in \S\ref{sec:Toy}, we will need to take advantage of the regularity imbalances to close an estimate. 

Decompose the reaction terms based on the $X$ dependence of each factor:  
\begin{align*} 
\mathcal{T}_{R} & =  -\int A^{3} Q^3 A^{3} \left( (\tilde U_{\neq})_{Hi}  \cdot (\grad Q^3_0)_{Lo} \right) dV - \int A^{3} Q^3 A^{3} \left( g_{Hi} \cdot \partial_Y (Q^3_{\neq})_{Lo} \right) dV \\ & \quad - \int A^{3} Q^3 A^{3} \left( (\tilde U_{\neq})_{Hi} \cdot \grad (Q^3_{\neq})_{Lo} \right) dV \\ 
& = \mathcal{T}_{R;\neq 0} + \mathcal{T}_{R;0 \neq}+ \mathcal{T}_{R;\neq \neq}, 
\end{align*} 
and also the transport terms:
\begin{align*} 
\mathcal{T}_{T} & =  -\int A^{3} Q^3 A^{3} \left( (\tilde U_{\neq})_{Lo}  \cdot (\grad Q^3_0)_{Hi} \right) dV - \int A^{3} Q^3 A^{3} \left( g_{Lo} \partial_Y (Q^3_{\neq})_{Hi} \right) dV \\ & \quad - \int A^{3} Q^3 A^{3} \left( (\tilde U_{\neq})_{Lo} \cdot \grad (Q^3_{\neq})_{Hi} \right) dV \\ 
& = \mathcal{T}_{T;\neq 0} + \mathcal{T}_{T;0 \neq}+ \mathcal{T}_{T;\neq \neq}. 
\end{align*} 

\paragraph{Transport term $\mathcal{T}_{T;0 \neq}$} 
This term can be treated the same as the corresponding term in \S\ref{sec:Q2_TransNon}: because the velocity field is independent of $X$, there are no regularity losses associated with the regularity imbalances in the norm $A^3$ -- these only occur if one changes the $X$ frequency, as $\chi^{R,NR} = \chi^{NR,R} = 0$ if $k =k^\prime$ in Lemma \ref{lem:ABasic}. Hence, as above, 
\begin{align*}
\mathcal{T}_{T;0\neq} & \lesssim \epsilon^{3/2}\norm{\sqrt{-\Delta_L}A^3 Q^3}_2^2 + \frac{\epsilon^{1/2}}{t^{2}}\norm{A^3 Q^3}_2^2.
\end{align*}

\paragraph{Transport term $\mathcal{T}_{T;\neq 0}$}  
This is one of the terms where it is crucial that we include the $ + e^{\mu\abs{l}^{1/2}}$ correction to the norm. 
By Lemma \ref{lem:ABasic} and \eqref{ineq:jNRBasic}, we have by $\abs{\xi,l^\prime} \chi^{R,NR} \lesssim \abs{kt}\chi^{R,NR}$ (it is here we are using that regularity imbalances only occur for $\abs{\partial_Z} \lesssim \abs{\partial_Y}$),  
\begin{align*}
\mathcal{T}_{T;\neq 0} & \lesssim \frac{\epsilon}{\jap{\nu t^3}^\alpha}\sum \int \abs{A^3 \widehat{Q^3_{k}}(\eta,l) A^3_k(\eta,l) \abs{\xi,l^\prime} \widehat{Q^3_0}(\xi,l^\prime)_{Hi} Low(k,\eta-\xi,l-l^\prime)} d\eta d\xi \\ 
& \lesssim \frac{\epsilon}{\jap{\nu t^3}^\alpha}\sum \int \abs{\widehat{Q^3_{k}}(\eta,l)}\left(\frac{t \chi^{R,NR}}{\abs{k} + \abs{\eta-kt}} \tilde{A}^3_k(\eta,l)\tilde{A}^3_0(\xi,l^\prime)  + A^3_k(\eta,l) A^3_0(\xi,l^\prime) \right) \\ & \quad\quad \times \abs{\xi,l^\prime} \abs{\widehat{Q^3_0}(\xi,l^\prime)_{Hi}} Low(k,\eta-\xi,l-l^\prime) d\eta d\xi \\
& \lesssim \frac{\epsilon t^2}{\jap{\nu t^3}^{\alpha}} \norm{\left(\sqrt{\frac{\partial_t w}{w}}\tilde{A}^3 + \frac{\abs{\grad}^{s/2}}{\jap{t}^{s}}A^3 \right) Q^3}_2^2 + \frac{\epsilon^{1/2}}{\jap{\nu t^3}^{2\alpha}} \norm{A^3 Q^3_{\neq}}_2^2 +  \epsilon^{3/2}\norm{\grad A^3 Q_0^3}_2, 
\end{align*}
which is consistent with Proposition \ref{prop:Boot}. 

\paragraph{Transport term $\mathcal{T}_{T;\neq \neq}$}
We will again use crucially that we have the $ + e^{\mu\abs{l}^{1/2}}$ correction to the norm. 
By \eqref{ineq:ABasic} we have 
\begin{align*}
\mathcal{T}_{T;\neq \neq} & \lesssim \frac{\epsilon}{\jap{\nu t^3}^{\alpha-1}}\sum \int \abs{A^3 \widehat{Q^3_{k}}(\eta,l) A^3_k(\eta,l) \abs{k t^{\delta_1}, t^{-1}(\xi-k^\prime t), l^\prime} \widehat{Q^3_{k^\prime}}(\xi,l^\prime)_{Hi}} \\  &\quad\quad \times Low(k-k^\prime,\eta-\xi,l-l^\prime) d\eta d\xi \\ 
& \lesssim \frac{\epsilon}{\jap{\nu t^3}^{\alpha-1}}\sum \int \abs{\widehat{Q^3_{k}}(\eta,l)} \left(\frac{t \chi^{R,NR}}{\abs{k} + \abs{\eta-kt}} \tilde{A}^3_k(\eta,l)\tilde{A}^3_{k^\prime}(\xi,l^\prime)  + A^3_k(\eta,l) A^3_{k^\prime}(\xi,l^\prime) \right) \\ 
& \quad\quad \times \left(\abs{k^\prime t^{\delta_1}} + t^{-1}\abs{\xi-k^\prime t} +  \abs{l^\prime}\right)  \abs{\widehat{Q^3_{k^\prime}}(\xi,l^\prime)_{Hi} Low(k-k^\prime,\eta-\xi,l-l^\prime)} d\eta d\xi \\ 
& = \mathcal{T}_{T;\neq \neq}^{X} + \mathcal{T}_{T;\neq \neq}^{Y}  + \mathcal{T}_{T;\neq \neq}^{Z}. 
\end{align*}
Note that we have used the inviscid damping on $U^2$ and the inequality $\norm{C}_{\G^{\lambda,\gamma}}\norm{U^3_{\neq}}_{\G^{\lambda,\beta-2}} \lesssim \epsilon^2 t \jap{\nu t^3}^\alpha \lesssim \epsilon t^{-1} \jap{\nu t^3}^{\alpha - 1}$ (see \S\ref{sec:AprioriBoot}) to reduce the power in front of the $\partial_Y - t\partial_X$ derivative.
Due to the gain in $\abs{k}$ at the critical times from $\chi^{R,NR}\abs{k}^{-1}$, we have
\begin{align*}
\mathcal{T}_{T;\neq \neq}^{X} & \lesssim \frac{\epsilon t^{1+\delta_1}}{\jap{\nu t^3}^{\alpha-1}}\norm{A^3 Q^3}_2^2 + \frac{\epsilon t^{\delta_1}}{\jap{\nu t^3}^{\alpha-1}}\norm{A^3 Q^3}_2\norm{\sqrt{-\Delta_L}A^3 Q^3}_2,
\end{align*}
which is consistent with Proposition \ref{prop:Boot} for $\epsilon$ and $\delta_1$ sufficiently small. 
Due to the extra $t^{-1}$, there are no losses in the $Y$ term and hence we have 
\begin{align*}
\mathcal{T}_{T;\neq \neq}^{Y} & \lesssim \frac{\epsilon}{\jap{\nu t^3}^{\alpha}}\norm{A^3 Q^3}_2\norm{\sqrt{-\Delta_L}A^3 Q^3}_2 \lesssim \frac{\epsilon^{1/2}}{\jap{\nu t^3}^{2\alpha}}\norm{A^3 Q^3}_2^2  + \epsilon^{3/2}\norm{\sqrt{-\Delta_L}A^3 Q^3}_2^2, 
\end{align*}
which is also consistent with Proposition \ref{prop:Boot}. 
For the $Z$ term we use $\chi^{R,NR} \abs{l^\prime} \lesssim \abs{kt} \chi^{R,NR}$ (it is here we are using that the losses only occur for $\abs{\partial_Z} \lesssim \abs{\partial_Y}$ due to the $+e^{\mu\abs{l}^{1/2}}$ correction) and \eqref{ineq:jNRBasic} to deduce 
\begin{align*}
\mathcal{T}_{T;\neq \neq}^{Z} & \lesssim \frac{\epsilon t^2}{\jap{\nu t^3}^{\alpha}}\norm{\left(\sqrt{\frac{\partial_t w}{w}}\tilde{A}^3 + \frac{\abs{\grad}^{s/2}}{\jap{t}^{s}}A^3\right) Q^3}_2^2 + \frac{\epsilon}{\jap{\nu t^3}^{\alpha}}\norm{A^3 Q^3}_2\norm{\sqrt{-\Delta_L}A^3 Q^3}_2, 
\end{align*}
which is consistent with Proposition \ref{prop:Boot}. 

\paragraph{Reaction term $\mathcal{T}_{R;0 \neq}$}
Turn first to the easiest, $\mathcal{T}_{R;0 \neq}$. 
By\eqref{ineq:Boot_ED3} and Lemma \ref{lem:ABasic}, we get (also noting \eqref{def:tildeU2}):  
\begin{align*} 
\mathcal{T}_{R;0 \neq} & \lesssim \frac{\epsilon \jap{t}^2}{\jap{\nu t^3}^{\alpha}}\sum_{k \neq 0}\int \abs{A^{3} \hat{Q_k^3} (\eta,l) A^{3}_{k}(\eta,l) \widehat{g}(\xi,l^\prime)_{Hi}} Low(k,\eta-\xi,l-l^\prime) d\xi d\eta \\
& \lesssim \frac{\epsilon}{\jap{\nu t^3}^{\alpha}}\norm{A^3 Q^3_{\neq}}_2 \norm{Ag}_2, 
\end{align*} 
which is consistent with Proposition \ref{prop:Boot}. 

\paragraph{Reaction terms $\mathcal{T}_{R;\neq 0}$} \label{sec:Q3TRneq0}
Next consider $\mathcal{T}_{R;\neq 0}$.
In fact, since $Q^3_0$ is the same order of magnitude as $Q^2_0$, and $A^3 \lesssim A^2$, this term can be treated in the same fashion as was done in \S\ref{sec:Q2TRneq0}. Hence, we omit the details for brevity.

\paragraph{Reaction term $\mathcal{T}_{R;\neq \neq}$} \label{sec:Q3TRneqneq}
Turn next to $\mathcal{T}_{R;\neq \neq}$. This includes terms isolated in \S\ref{sec:Toy} as leading order contributions to the \textbf{(3DE)} nonlinear interactions (see \cite{BGM15I} and \S\ref{sec:NonlinHeuristics}) and these terms are one of the places where we will need the regularity imbalances in $A^3$. 
As in \S\ref{sec:Q2TRneqneq} above,  we further decompose in terms of frequency: 
\begin{align*} 
\mathcal{T}_{R;\neq \neq} & \lesssim \frac{\epsilon \jap{t}^2}{\jap{\nu t^3}^\alpha} \sum_{k,k^\prime}\int \mathbf{1}_{k k^\prime(k-k^\prime) \neq 0} \abs{A^{3} \hat{Q}^3_k(\eta,l) A^{3}_{k}(\eta,l) \hat{U}_{k^\prime}^1 (\xi,l^\prime)_{Hi}} Low(k-k^\prime, \eta-\xi, l - l^\prime) d\eta d\xi \\   
& \quad +  \frac{\epsilon \jap{t}^{3}}{\jap{\nu t^3}^\alpha} \sum_{k,k^\prime}\int \mathbf{1}_{k k^\prime(k-k^\prime) \neq 0} \abs{A^{3} \hat{Q}^3_k(\eta,l) A^{3}_{k}(\eta,l) \hat{U}_{k^\prime}^2 (\xi,l^\prime)_{Hi}} Low(k-k^\prime, \eta-\xi, l - l^\prime) d\eta d\xi \\  
& \quad +  \frac{\epsilon \jap{t}^2}{\jap{\nu t^3}^{\alpha-1}} \sum_{k,k^\prime}\int \mathbf{1}_{k k^\prime(k-k^\prime) \neq 0} \abs{A^{3} \hat{Q}^3_k(\eta,l) A^{3}_{k}(\eta,l) \hat{U}_{k^\prime}^3 (\xi,l^\prime)_{Hi}} Low(k-k^\prime, \eta-\xi, l - l^\prime) d\eta d\xi \\   
& \quad + \frac{\epsilon^2 \jap{t}^2}{\jap{\nu t^3}^{\alpha}} \sum_{k,k^\prime}\int \mathbf{1}_{k k^\prime(k-k^\prime) \neq 0} \abs{A^{3} \hat{Q}^3_k(\eta,l)} A^{3}_{k}(\eta,l) \left(\abs{\hat{\phi_z}(\xi,l^\prime)_{Hi}} + \abs{\hat{\psi_z}(\xi,l^\prime)_{Hi}} + \abs{\hat{\phi_y}(\xi,l^\prime)_{Hi}}\right) \\ & \quad\quad\quad \times Low(k-k^\prime, \eta-\xi, l - l^\prime) d\eta d\xi \\   
& \quad + \frac{\epsilon^2 \jap{t}^{3}}{\jap{\nu t^3}^{\alpha}} \sum_{k,k^\prime}\int \mathbf{1}_{k k^\prime(k-k^\prime) \neq 0} \abs{A^{3} \hat{Q}^3_k(\eta,l) A^{3}_{k}(\eta,l)\left(\hat{\psi_y}(\xi,l^\prime)_{Hi}\right)} Low(k-k^\prime, \eta-\xi, l - l^\prime) d\eta d\xi \\   
& \quad + \mathcal{T}_{R;\neq \neq;\mathcal{R}} \\ 
& = \mathcal{T}_{R;\neq\neq}^{1} + \mathcal{T}_{R;\neq\neq}^2 + \mathcal{T}_{R;\neq\neq}^3  + \mathcal{T}_{R;\neq\neq}^{C1} + \mathcal{T}_{R;\neq\neq}^{C2} + \mathcal{T}_{R;\neq\neq;\mathcal{R}}. 
\end{align*} 
Consider $\mathcal{T}^2_{R;\neq \neq}$, which is one of the terms in the toy model. 
In particular, we will use the regularity imbalance between $Q^2$ and $Q^3$ to reduce the power of $t$. 
By \eqref{ineq:A3A2neqneq}, 
\begin{align*}
\mathcal{T}_{R;\neq\neq}^{2} & \lesssim \frac{\epsilon \jap{t}^3}{\jap{\nu t^3}^{\alpha}} \sum_{k,k^\prime}\int \mathbf{1}_{k,k^\prime,k-k^\prime \neq 0} \abs{\hat{Q}^3_k(\eta,l)} \frac{1}{\abs{k^\prime}^2 + \abs{l^\prime}^2 + \abs{\xi - k^\prime t}^2} \\   
& \quad\quad \times \left(\sum_{r} \chi^{NR,r}\frac{\abs{r} + \abs{\eta-tr}}{t} \tilde{A}^3_k(\eta,l) \tilde{A}^2_{k^\prime}(\xi,l^\prime) + \chi^{\ast;32} A^3_k(\eta,l) A^2_{k^\prime}(\xi,l^\prime) \right) \\  
& \quad\quad \times \jap{\frac{t}{\jap{\xi,l^\prime}}}^{-1} \abs{ \Delta_L \hat{U}_{k^\prime}^2 (\xi,l^\prime)_{Hi}} Low(k-k^\prime, \eta-\xi, l - l^\prime) d\eta d\xi.   
\end{align*} 
Therefore, by \eqref{ineq:A3ReacGain} followed by \eqref{ineq:triQuadHL}, 
\begin{align*}
\mathcal{T}_{R;\neq\neq}^{2} & \lesssim \frac{\epsilon \jap{t}^2}{\jap{\nu t^3}^{\alpha}}\norm{\left(\sqrt{\frac{\partial_t w}{w}}\tilde{A}^3 + \frac{\abs{\grad}^{s/2}}{\jap{t}^{s}}A^3\right) Q^3}_2 \norm{\left(\sqrt{\frac{\partial_t w}{w}}\tilde{A}^2 + \frac{\abs{\grad}^{s/2}}{\jap{t}^{s}}A^2\right) \Delta_L U^2_{\neq}}_2 \\ 
& \quad + \frac{\epsilon \jap{t}}{\jap{\nu t^3}^{\alpha}}\norm{A^3 Q^3}_2 \norm{A^2 \Delta_L U^2_{\neq}}_2,  
\end{align*}
which, by Lemmas \ref{lem:PEL_NLP120neq} and \ref{lem:SimplePEL}, is consistent with Proposition \ref{prop:Boot}. 
The term $\mathcal{T}^1_{R;\neq \neq}$ is treated in essentially the same way (matching the intuition that $Q^1 \sim t Q^2$ near the critical times) and is hence omitted. 

Next, turn to the treatment of $\mathcal{T}^3_{R;\neq \neq}$. 
By \eqref{ineq:A3A3neqneq} we have
\begin{align*}
\mathcal{T}_{R;\neq\neq}^{3} & \lesssim \frac{\epsilon \jap{t}^2}{\jap{\nu t^3}^{\alpha-1}} \sum_{k,k^\prime}\int \mathbf{1}_{k k^\prime(k-k^\prime) \neq 0} \abs{\hat{Q}^3_k(\eta,l)} \frac{1}{\abs{k^\prime}^2 + \abs{l^\prime}^2 + \abs{\xi - k^\prime t}^2} \\   
& \quad\quad \times \left(\chi^{R,NR}\frac{t}{\abs{k} + \abs{\eta-kt}}\tilde{A}^3_k(\eta,l)\tilde{A}^3_{k^\prime}(\xi,l^\prime) + \chi^{NR,R}\frac{\abs{k^\prime} + \abs{\eta - k^\prime t}}{t} \tilde{A}^3_k(\eta,l)\tilde{A}^3_{k^\prime}(\xi,l^\prime) \right. \\ &  \quad\quad  + \chi^{\ast;33} A^3_k(\eta,l) A^3_{k^\prime}(\xi,l^\prime) \bigg) \abs{A^3 \Delta_L \hat{U}_{k^\prime}^3 (\xi,l^\prime)_{Hi}} Low(k-k^\prime, \eta-\xi, l - l^\prime) d\eta d\xi, 
\end{align*}
which by \eqref{ineq:A33ReacGain} and \eqref{ineq:triQuadHL} is 
\begin{align*}
\mathcal{T}_{R;\neq\neq}^{3} & \lesssim \frac{\epsilon \jap{t}}{\jap{\nu t^3}^{\alpha-1}}\norm{A^3 Q^3}_2 \norm{A^3 \Delta_L U^3_{\neq}}_2, 
\end{align*}
which is consistent with Proposition \ref{prop:Boot} by Lemma \ref{lem:SimplePEL}. 

Finally, turn to $\mathcal{T}_{R;\neq\neq}^{C1}$ and $\mathcal{T}_{R;\neq\neq}^{C2}$. By Lemma \ref{lem:ABasic} and \eqref{ineq:quadHL} (and Lemma \ref{lem:CoefCtrl}), we have
\begin{align*} 
\mathcal{T}_{R;\neq\neq}^{C1}+ \mathcal{T}_{R;\neq\neq}^{C2}  & \lesssim \frac{\epsilon^2 \jap{t}^2}{\jap{\nu t^3}^{\alpha-1}}\norm{A^3 Q^3}_2 \norm{AC}_2  \lesssim \frac{\epsilon \jap{t}}{\jap{\nu t^3}^{\alpha-1}}\norm{A^3 Q^3}^2_2 + \frac{\epsilon^3 \jap{t}^3}{\jap{\nu t^3}^{\alpha-1}} \norm{AC}_2^2,  
\end{align*}
which is consistent with Proposition \ref{prop:Boot} for $\alpha$ sufficiently large, $\epsilon$ sufficiently small, and $\delta > 0$. 
 This completes the treatment of $\mathcal{T}_{R;\neq\neq}$ and hence all of $\mathcal{T}$.

\subsubsection{Dissipation error terms  $\mathcal{D}$} \label{sec:DEneqQ3}
Due to the quadratic growth at low frequencies of $Q^3$ and the much larger size of $\epsilon$, 
these terms cannot be treated as they were in \cite{BGM15I}. 
However, we will adapt a treatment from \cite{BMV14} which treats the critical times with increased precision. 
Recalling the dissipation error terms and the short-hand \eqref{def:G}, we have 
\begin{align*} 
\mathcal{D}_E & = \nu\sum_{k \neq 0}\int A^3 Q^3_k A^3_k\left(G_{yy}(\partial_{Y} - t\partial_X)^2 Q^3_k + G_{yz}(\partial_Y - t \partial_X)\partial_{Z}Q^3_k + G_{zz}\partial_{ZZ}Q^3_k \right) dV \\ 
& = \mathcal{D}_E^1 + \mathcal{D}_E^2  + \mathcal{D}_E^3.  
\end{align*} 
We will only treat $\mathcal{D}_E^1$; $\mathcal{D}_E^2$ and $\mathcal{D}_E^3$ are slightly easier and are hence omitted. 
As usual, we expand with a paraproduct: 
\begin{align*} 
\mathcal{D}_E^1 & = \nu\sum_{k \neq 0}\int A^3 Q^3_k A^3_k \left( (G_{yy})_{Hi} (\partial_{Y} - t\partial_X)^2 (Q^3_k)_{Lo} \right) dV + \nu\sum_{k \neq 0}\int A^3 Q^3_k A^3_k \left((G_{yy})_{Lo} (\partial_{Y} - t\partial_X)^2 (Q^3_k)_{Hi} \right) dV \\ 
& \quad + \nu\sum_{k \neq 0}\int A^3 Q^3_k A^3_k \left( \left(G_{yy}(\partial_{Y} - t\partial_X)^2 Q^3_k \right)_{\mathcal{R}} \right) dV \\ 
& = \mathcal{D}_{E;HL}^1 + \mathcal{D}_{E;LH}^1 + \mathcal{D}_{E;\mathcal{R}}^1. 
\end{align*} 
As in \S\ref{sec:DEneqQ2} and \cite{BGM15I}, we can control the latter two terms by the dissipation; we omit the details for brevity. 
Next, turn to the treatment of $\mathcal{D}_{E;HL}^1$. 
By Lemma \ref{lem:ABasic}, there is some $c = c(s) \in (0,1)$ such that 
\begin{align*} 
\mathcal{D}_{E;HL}^{1} & \lesssim \nu \sum_{k \neq 0} \int \abs{A^3 \widehat{Q^3_k}(\eta,l) A^3_k(\eta,l) \widehat{G_{yy}}(\xi,l^\prime)_{Hi} (\eta-\xi - tk)^2 \widehat{Q^3_k}(\eta-\xi,l-l^\prime)_{Lo}} d\eta d\xi \\ 
& \lesssim \nu \sum_{k \neq 0} \int\left[\chi_R + \chi_{NR;k} \right]\abs{A^3 \widehat{Q^3_k}(\eta,l)  \frac{1}{\jap{\xi,l^\prime} \jap{t}} A \widehat{G_{yy}}(\xi,l^\prime)_{Hi} } \\ 
& \quad\quad \times \abs{(\eta - \xi - tk)^2 e^{c\lambda\abs{k,\eta-\xi,l-l^\prime}^s} \widehat{Q^3_k} (\eta-\xi,l-l^\prime)_{Lo} } d\eta d\xi, \\ 
& = \mathcal{D}_{E;HL}^{1;R} + \mathcal{D}_{E;HL}^{1;NR}, 
\end{align*}
where $\chi_{R;k} = \mathbf{1}_{t \in \I_{k,\eta}\cap \I_{k,\xi}} \mathbf{1}_{\abs{l} \leq \frac{1}{5}\abs{\eta}} \mathbf{1}_{\abs{l^\prime} \leq \frac{1}{5}\abs{\xi}}$ and $\chi_{NR;k} = 1-\chi_{R;k}$ is defined in \eqref{def:chiNR}.  
For the non-resonant term $\mathcal{D}_{E;HL}^{1;NR}$, since $\jap{t} \lesssim (\abs{k} + \abs{l} + \abs{\eta-kt})\jap{\eta-\xi,l-l^\prime}$ 
and $\abs{\eta-\xi-kt} \lesssim \jap{t}\jap{k,\eta-\xi}$ on the support of the integrand by \eqref{ineq:basicNR}, 
\begin{align*} 
\mathcal{D}_{E;HL}^{1;NR} & \lesssim \nu \sum_{k \neq 0} \int \chi_{NR;k} \abs{\sqrt{-\Delta_L} A^3 \widehat{Q^3_k}(\eta,l)  \frac{1}{\jap{\xi,l^\prime} \jap{t}^2} A \widehat{G_{yy}}(\xi,l^\prime)_{Hi}} \\ & \quad\quad \times \abs{(\eta - \xi - tk)^2 e^{c\lambda\abs{k,\eta-\xi,l-l^\prime}^s} \widehat{Q^3_k} (\eta-\xi,l-l^\prime)_{Lo} } d\eta d\xi, \\ 
& \lesssim \nu t  \norm{\jap{\grad}^{-1}AG}_{2}\norm{\sqrt{-\Delta_L} A^3 Q^3}_2 \jap{t}^{-2}\norm{\sqrt{-\Delta_L} Q^3_{\neq}}_{\G^{\lambda}}. 
\end{align*} 
It follows by \eqref{ineq:AnuDecay}, Lemma \ref{lem:CoefCtrl} and \eqref{ineq:Boot_ACC2}, we have
\begin{align*} 
\mathcal{D}_{E;HL}^{1;NR} & \lesssim \frac{\nu t}{\jap{\nu t^3}^{\alpha}} \norm{AC}_{2}\norm{\sqrt{-\Delta_L} A^3 Q^3}_2\norm{\sqrt{-\Delta_L} A^{\nu;3} Q^3}_2 \\ 
& \lesssim  \frac{\nu \epsilon t^2 \abs{\log c_0\epsilon^{-1}} }{\jap{\nu t^3}^{\alpha}} \norm{\sqrt{-\Delta_L} A^3 Q^3}_2\norm{\sqrt{-\Delta_L} A^{\nu;3} Q^3}_2 \\ 
& \lesssim  \epsilon^{\delta/4} \nu \norm{\sqrt{-\Delta_L} A^3 Q^3}_2^2 + \epsilon^{\delta/4} \nu \norm{\sqrt{-\Delta_L} A^{\nu;3} Q^3}_2^2,
\end{align*}
which is consistent with Proposition \ref{prop:Boot} by the bootstrap hypotheses for $\epsilon$ sufficiently small and $\delta > 0$. 
For the resonant term $\mathcal{D}_{E;HL}^{1;R}$ we have by Lemma \ref{lem:dtw} and  \eqref{ineq:quadHL} (also using that $A(\xi,l^\prime) \approx \tilde{A}(\xi,l^\prime)$ on the support of the integrand due to the definition of $\chi^{R}$ and $\abs{\eta-\xi-kt} \lesssim \jap{t} \jap{k,\eta-\xi}$),   
\begin{align*} 
\mathcal{D}_{E;HL}^{1;R} & \lesssim \nu \sum_{k \neq 0} \int \chi_{R;k} \left(\abs{k} + \abs{\eta-kt}\right)^{1/2}\abs{A^3 \widehat{Q^3_k}(\eta,l)  \frac{1}{\jap{\xi,l^\prime} \jap{t}} \sqrt{\frac{\partial_t w}{w}}\tilde{A}\widehat{G_{yy}}(\xi,l^\prime)_{Hi}} \\ & \quad\quad \times \jap{t}^{3/2}\jap{k,\eta-\xi}^{5/2}\abs{(\eta -\xi - tk)^{1/2} e^{c\lambda\abs{k,\eta-\xi,l-l^\prime}^s} \widehat{Q^3_k}(\eta-\xi,l-l^\prime)_{Lo}}  d\eta d\xi \\ 
& \lesssim \nu t^{1/2} \norm{\sqrt{-\Delta_L} A^3 Q^3}_2^{1/2} \norm{A^3 Q^3_{\neq}}_2^{1/2} \norm{Q^3_{\neq}}_{\G^\lambda}^{1/2}\norm{\sqrt{-\Delta_L} Q^3_{\neq}}_{\G^\lambda}^{1/2} \norm{\sqrt{\frac{\partial_t w}{w}}\jap{\grad}^{-1}\tilde{A}G_{yy}}_2.
\end{align*} 
Then, by \eqref{ineq:AnuDecay}, \eqref{def:A}, and \eqref{def:Anu}, followed by Lemma \ref{lem:CoefCtrl}, for some small $\delta^\prime > 0$ 
\begin{align*}  
& \lesssim \frac{\nu t^{5/2}}{\jap{\nu t^3}^{\alpha/2}} \norm{\sqrt{-\Delta_L} A^3 Q^3}_2 \norm{A^3 Q^3_{\neq}}_2^{1/2}\norm{A^{\nu;3} Q^3}_{2}^{1/2} \norm{\sqrt{\frac{\partial_t w}{w}}\jap{\grad}^{-1}\tilde{A}G_{yy}}_2 \\ 
& \lesssim \frac{\nu^{1/2} t \epsilon}{\jap{\nu t^3}^{\alpha/2-1}} \norm{\sqrt{-\Delta_L} A^3 Q^3}_2 \norm{\sqrt{\frac{\partial_t w}{w}}\jap{\grad}^{-1}\tilde{A}G_{yy}}_2 \\ 
& \lesssim \epsilon^{\delta^\prime}\nu\norm{\sqrt{-\Delta_L} A^3 Q^3}_2^2 +  \frac{\epsilon^{2-\delta^\prime} \jap{t}^4}{\jap{\nu t^3}^{\alpha-2}} \left(\frac{1}{\jap{t}^2}\norm{\left(\sqrt{\frac{\partial_t w}{w}}\tilde{A} + \frac{\abs{\grad}^{s/2}}{\jap{t}^{s}}A\right)C}_2^2 \right),  
\end{align*} 
which is now consistent with Proposition \ref{prop:Boot} for $\delta^\prime$ and $\epsilon$ small. 
Note that the hypothesis $\epsilon \lesssim \nu^{2/3+\delta}$ with $\delta > 0$ is essentially sharp for controlling this term.    
This completes the treatment of $\mathcal{D}_{E}^{1}$ and hence of the dissipation error terms.

\subsubsection{Linear stretching term $LS3$} \label{sec:LS30_Hi}
First separate into two parts (to be sub-divided further below), 
\begin{align*} 
LS3 & = -2\int A^{3} Q^3 A^{3}\partial_X(\partial_Y - t\partial_X) U^3  dV - 2\int A^{3} Q^3 A^{3} \partial_X \left(\psi_y(\partial_Y - t\partial_X) + \phi_y\partial_Z \right) U^3 dV \\ 
& = LS3^0 + LS3^{C}. 
\end{align*} 

\paragraph{Treatment of $LS3^C$} \label{sec:LS3C}
Expand with a paraproduct, 
\begin{align*} 
LS3^{C} & = -2\int A^{3} Q^3 A^{3} \partial_X \left((\psi_y)_{Hi}(\partial_Y - t\partial_X) + (\phi_y)_{Hi}\partial_Z \right)\left(U^3\right)_{Lo} dV \\ 
& \quad - 2\int A^{3} Q^3 A^{3} \partial_X \left( (\psi_y)_{Lo}(\partial_Y - t\partial_X) + (\phi_y)_{Lo}\partial_Z \right) \left(U^3\right)_{Hi} dV \\ 
& \quad - 2\int A^{3} Q^3 A^{3} \partial_X \left( \left(\psi_y(\partial_Y - t\partial_X)  + \phi_y\partial_Z \right) U^3\right)_{\mathcal{R}} dV \\ 
& = LS3^{C}_{HL} + LS3^{C}_{LH} + LS3^{C}_{\mathcal{R}}.  
\end{align*}
The main issue is $LS3^{C}_{HL}$, where the coefficients appear in `high frequency', so turn to this term first.    
By Lemma \ref{lem:ABasic}, \eqref{ineq:triQuadHL}, and Lemma \ref{lem:CoefCtrl},  
\begin{align*} 
LS3^C_{HL} & \lesssim \frac{\epsilon\jap{t}}{\jap{\nu t^3}^{\alpha}}\sum_{k \neq 0}\int \abs{A^{3} \widehat{Q^3_k}(\eta,l)}\frac{1}{\jap{\xi,l^\prime}\jap{t}}A\left( \abs{\widehat{\psi_y}(\xi,l^\prime)} +  \abs{\widehat{\phi_y}(\xi,l^\prime)}\right) Low(k,\eta-\xi,l-l^\prime) d\xi \\ 
& \lesssim \frac{\epsilon^{1/2}}{\jap{\nu t^3}^{\alpha}}\norm{A^3 Q^3}^2 +  \frac{\epsilon^{3/2}}{\jap{\nu t^3}^{\alpha}}\norm{AC}^2_2, 
\end{align*}
which is consistent with Proposition \ref{prop:Boot} by \eqref{ineq:Boot_ACC2} for $\epsilon$ sufficiently small and $\delta > 0$ (hence $\epsilon \lesssim \nu^{2/3+\delta}$ is essentially sharp here). 

Turn next to the $LS3^C_{LH}$, which is reminiscent of $NLP(1,3,0,\neq)$ in \S\ref{sec:NLP213}. 
Indeed, by Lemma \ref{lem:CoefCtrl}, \eqref{ineq:AiPartX}, and \eqref{ineq:triQuadHL} we have,   
\begin{align} 
LS3^C_{LH} & \lesssim \epsilon\jap{t} \sum \int \abs{ A^3 \widehat{Q^3_k}(\eta,l) A^3_k(\eta,l) \frac{k\abs{\xi - kt,l^\prime}}{k^2 + (l^\prime)^2 + \abs{\xi-kt}^2} \Delta_L \widehat{U^3_k}(\xi,l^\prime)} Low(\eta-\xi,l-l^\prime) d\eta dx \nonumber \\ 
& \lesssim c_{0}\norm{\left(\sqrt{\frac{\partial_t w}{w}}\tilde{A}^3 + \frac{\abs{\grad}^{s/2}}{\jap{t}^{s}}A^3 \right) Q^3}_2 \norm{\left(\sqrt{\frac{\partial_t w}{w}}\tilde{A}^3 + \frac{\abs{\grad}^{s/2}}{\jap{t}^{s}}A^3 \right) \Delta_L U^3_{\neq}}_2 \nonumber  \\ & \quad + \frac{\epsilon}{\jap{t}}\norm{\sqrt{-\Delta_L}A^3 Q^3}_2 \norm{A^3 \Delta_L U^3_{\neq}}_2, \label{ineq:LS3CLH} 
\end{align} 
which, after the application of Lemmas \ref{lem:PEL_NLP120neq} and \ref{lem:SimplePEL}, is consistent with Proposition \ref{prop:Boot} for $\epsilon$ and $c_0$ sufficiently small. 

The remainder $LS3^C_{\mathcal{R}}$ follows easily and is hence omitted. 

\paragraph{Leading order term, $LS3^0$} \label{sec:LS30}
As in \cite{BGM15I}, the $2$ in the leading order term is crucially important and cannot be altered; it is the origin of the quadratic growth of $Q^3$ at low (relative to time) frequencies and any alteration would cause faster growth and a collapse of the bootstrap. 
For this reason we have to treat this term more precisely. 
Begin by isolating the leading order contribution: by the definition of $\Delta_t$ (see \eqref{def:Deltat} and the shorthand \eqref{def:G}), 
\begin{align} 
LS3^0 & = -2\int A^{3} Q^3 A^{3}\partial_X(\partial_Y - t\partial_X) \Delta_{L}^{-1} \Delta_L \Delta_t^{-1}Q^3  dV \nonumber \\
& =  -2\int A^{3} Q^3 A^{3}\partial_X(\partial_Y - t\partial_X) \Delta_{L}^{-1} \left( Q^3 - G_{yy}(\partial_Y - t\partial_X)^2\Delta_t^{-1}Q^3  \right. \nonumber \\ 
& \quad \left. \quad - G_{yz}\partial_Z (\partial_Y - t\partial_X) U^3 - G_{zz}\partial_{ZZ}U^3 - \Delta_t C^1 (\partial_Y - t\partial_X) U^3 - \Delta_t C^2 \partial_Z U^3 \right)  dV \nonumber \\  
& = LS3^{0;0} +  \sum_{j = i}^5 LS3^{0;Ci}.      \label{eq:LS30} 
\end{align} 

The treatment of $LS3^{0;0}$ is essentially the same as in \cite{BGM15I}. 
The only minor difference is that one must separate high frequencies in $Z$ from high frequencies of $Y$ when using $CK_w^3$. 
Due to the uniform ellipticity in $Z$, this does not make a major difference and this contribution can be absorbed by the existing terms. 
Divide into long-time and short-time regimes 
\begin{align*}
LS3^{0;0} & = -2\int \left[\mathbf{1}_{t \leq 2\abs{\eta}}  + \mathbf{1}_{t > 2\abs{\eta}}  \right] \abs{A^{3} \widehat{Q^3_k}(\eta,l)}^2 \frac{k(\eta-kt)}{k^2 + l^2 + \abs{\eta-kt}^2} d\eta \\ 
& = LS3^{0;0,ST} + LS3^{0;0, LT}. 
\end{align*}
The long-time regime is treated the same as in \cite{BGM15I} (see therein for a proof), and hence for some universal $K > 0$:  
\begin{align*}
LS3^{0;0,LT} & \leq CK^3_{L} + \frac{\delta_\lambda}{10\jap{t}^{3/2}}\norm{\abs{\grad}^{s/2} A^{3}Q^3}_2^2 +  \frac{K}{\delta_\lambda^{\frac{1}{2s-1}}\jap{t}^{3/2}}\norm{ A^{3}Q^3}_2^2, 
\end{align*}
which, for $\delta_\lambda$ sufficiently small and $K_{H3}$ sufficiently large, is consistent with Proposition \ref{prop:Boot}. 
For the short-time regime we apply \eqref{ineq:CKwLS} to deduce for some $K > 0$,  
\begin{align*}
LS3^{0;0,ST} & \lesssim \kappa^{-1}\norm{\sqrt{\frac{\partial_t w}{w}} \tilde{A}^{3} Q^3}_2^2 + \frac{1}{\jap{t}^{3/2}}\norm{\abs{\grad}^{1/4} A^3 Q^3}_2^2 \\ 
& \leq K\kappa^{-1}\norm{\sqrt{\frac{\partial_t w}{w}} \tilde{A}^{3} Q^3}_2^2 + \frac{\delta_\lambda}{10\jap{t}^{3/2}} \norm{\abs{\grad}^{s/2} A^3 Q^3}_2^2 + \frac{K}{\delta_\lambda^{\frac{1}{2s-1}}\jap{t}^{3/2}} \norm{A^3 Q^3}_2^2,  
\end{align*}
which is consistent with Proposition \ref{prop:Boot} for $\kappa$ sufficiently large, $\delta_\lambda$ sufficiently small (so that the first term is absorbed by $CK_\lambda^3$) and $K_{H3}$ is sufficiently large. 

Consider the first error term in \eqref{eq:LS30}, $LS3^{0;C1}$; here we will need a more refined treatment than in \cite{BGM15I}.  
Expanding $LS3^{0;C1}$ gives 
\begin{align*} 
LS3^{0;C1} & = -2\int A^{3} Q^3 A^{3}\partial_X(\partial_Y - t\partial_X) \Delta_{L}^{-1} \left( (G_{yy})_{Hi} (\partial_Y - t\partial_X)^2 U^3_{Lo}\right) dV \\ 
& \quad -2\int A^{3} Q^3 A^{3}\partial_X(\partial_Y - t\partial_X) \Delta_{L}^{-1} \left( (G_{yy})_{Lo} (\partial_Y - t\partial_X)^2 U^3_{Hi}\right) dV \\ 
& \quad -2\int A^{3} Q^3 A^{3}\partial_X(\partial_Y - t\partial_X) \Delta_{L}^{-1} \left( (G_{yy}) (\partial_Y - t\partial_X)^2 U^3 \right)_{\mathcal{R}} dV \\ 
& = LS3^{0;C1}_{HL} + LS3^{0;C1}_{LH} + LS3^{0;C1}_{\mathcal{R}}.  
\end{align*} 
The most interesting contribution is the $HL$ term. 
By \eqref{ineq:AprioriUneq} and Lemma \ref{lem:ABasic}, we have 
\begin{align*}
LS3^{0;C1}_{HL} & \lesssim \frac{\epsilon\jap{t}^2}{\jap{\nu t^3}^{\alpha}}\sum_{l,l^\prime,k\neq 0}\int \abs{A^{3} \widehat{Q^3_k}(\eta,l)  \frac{\abs{\eta-kt}}{\left(k^2 + l^2 + \abs{\eta-kt}^2\right) \jap{t} \jap{\xi,l^\prime} }A\widehat{G_{yy}}(\xi,l^\prime)_{Hi}} \\ & \quad\quad \times Low(k,\eta-\xi,l-l^\prime) d\xi d\eta. 
\end{align*}
Therefore, by \eqref{ineq:AikDelLNoD} followed by \eqref{ineq:triQuadHL}, and Lemma \ref{lem:CoefCtrl}, we have 
\begin{align*}
LS3^{0;C1}_{HL} & \lesssim  \frac{\epsilon t}{\jap{\nu t^3}^\alpha } \norm{\left(\sqrt{\frac{\partial_t w}{w}}\tilde{A}^{3} + \frac{\abs{\grad}^{s/2}}{\jap{t}^{s}}A^{3}\right) Q^3}_2 \norm{\left(\sqrt{\frac{\partial_t w}{w}}\tilde{A} + \frac{\abs{\grad}^{s/2}}{\jap{t}^{s}}A\right) \jap{\grad}^{-1} G_{yy}}_2 \\ 
& \quad + \frac{\epsilon}{\jap{\nu t^3}} \norm{A^{3}Q^3}_2 \norm{\jap{\grad}^{-1}AG_{yy}}_2 \\ 
& \lesssim \frac{\epsilon t^2}{\jap{\nu t^3}^{\alpha}} \norm{\left(\sqrt{\frac{\partial_t w}{w}}\tilde{A}^3 + \frac{\abs{\grad}^{s/2}}{\jap{t}^{s}}A^3\right) Q^3}_2^2 + \frac{\epsilon}{\jap{\nu t^3}^{\alpha}} \norm{\left(\sqrt{\frac{\partial_t w}{w}}\tilde{A} + \frac{\abs{\grad}^{s/2}}{\jap{t}^{s}}A\right) C}_2^2 \\ 
& \quad + \frac{\epsilon \jap{t}}{\jap{\nu t^3}}\left(\norm{A^{3}Q^3}^2_2 +  \left(\frac{1}{\jap{t}}\norm{AC}_2\right)^2 \right). 
\end{align*} 
This is consistent with Proposition \ref{prop:Boot} by \eqref{ineq:Boot_ACC2} for $\delta > 0$ and $\epsilon$ sufficiently small. 

Turn to $LS3^{0;C1}_{LH}$, which by Lemma \ref{lem:CoefCtrl}, 
\begin{align*} 
LS3^{0;C1}_{LH} \lesssim \epsilon \jap{t}  \sum_{k,l}\int \abs{A^{3} \widehat{Q^3_k}(\eta,l) A^{3}_k(\eta,l) \frac{\abs{k}\abs{\eta-kt}}{k^2 + l^2 + \abs{\eta-kt}^2} \left(\Delta_L U^3_k\right)_{Hi}(\xi,l^\prime)} Low(\eta-\xi,l-l^\prime) d \eta d\xi.   
\end{align*}  
We can treat this term roughly like $NLP(1,3,0,\neq)$ on $Q^2$ in \S\ref{sec:NLP213}: by \eqref{ineq:AiPartX} and \eqref{ineq:triQuadHL},  
\begin{align*} 
LS3^{0;C1}_{LH} & \lesssim c_{0}\norm{\left(\sqrt{\frac{\partial_t w}{w}} \tilde{A}^{3} + \frac{\abs{\grad}^{s/2}}{\jap{t}^s}A^{3}\right) Q^3}_2^2 + c_{0}\norm{\left(\sqrt{\frac{\partial_t w}{w}} \tilde{A}^{3} + \frac{\abs{\grad}^{s/2}}{\jap{t}^s}A^{3}\right) \Delta_L U^3_{\neq}}_2^2 \\ 
& \quad + \epsilon\norm{A^3 Q^3_{\neq}}_2\norm{\Delta_L A^3 U^3_{\neq}}_2.
\end{align*}
By Lemmas \ref{lem:PEL_NLP120neq} and  \ref{lem:SimplePEL}, this is consistent with Proposition \ref{prop:Boot} by the bootstrap hypotheses. 
 
The remainder $LS3^{0;C1}_{\mathcal{R}}$ is straightforward and is omitted for the sake of brevity. 
This completes the first error term in \eqref{eq:LS30}, $LS3^{0;C1}$. 

The second and third error terms, $LS3^{0;C2}$ and $LS3^{0;C3}$, are similar to $LS3^{0;C1}$ but slightly easier, and yield similar contributions. 
Hence, we omit the treatment for brevity.   
 
The last two coefficient errors, $LS3^{0;C4}$ and $LS3^{0;C5}$, are also similar but require a slight adjustment. 
In particular, due to the two derivatives on the coefficients, we cannot gain any powers of time from $A^3$ as in the treatment of 
$LS3^{0;C1}$ above. However, this is balanced by the fact that there is one less power of $\partial_Y - t\partial_X$. Hence, the above treatment adapts in a straightforward manner and so we omit the details for brevity. 
 
This concludes the treatment of the linear stretching term $LS3$.

\subsubsection{Linear pressure term $LP3$} \label{ineq:LP3_Hi} 
As in $LS3$, we first separate the coefficient corrections and expand with a paraproduct:
\begin{align*} 
LP3 & = 2 \int A^{3} Q^3 A^{3} \partial_Z \partial_X U^2 dV + 2 \int A^{3} Q^3 A^{3} \left((\psi_{z})_{Lo}(\partial_Y - t\partial_X) + (\phi_z)_{Lo}\partial_Z \right)  \left(\partial_X U^2\right)_{Hi} dV \\
& \quad + 2\int A^{3} Q^3 A^{3} \left((\psi_{z})_{Hi}(\partial_Y - t\partial_X) + (\phi_z)_{Hi}\partial_Z \right)\left( \partial_X U^2\right)_{Lo} dV \\ 
& \quad + 2\int A^{3} Q^3 A^{3} \left( \left(\psi_{z}(\partial_Y - t\partial_X) + \phi_z\partial_Z \right) \partial_X  U^2\right)_{\mathcal{R}} dV \\ 
& = LP3^{0} + LP3^{C}_{LH} + LP3^{C}_{HL} + LP3^{C}_{\mathcal{R}}. 
\end{align*} 

\paragraph{Treatment of $LP3^{0}$}
As in \cite{BGM15I}, from \eqref{def:wL}, 
\begin{align*} 
LP3^{0} \leq \frac{1}{2\kappa}\norm{\sqrt{\frac{\partial_t w_L}{w_L}} A^{3} Q^3}_2^2 + \frac{1}{2\kappa}\norm{\sqrt{\frac{\partial_t w_L}{w_L}}A^{3} \Delta_L U^2_{\neq}}^2_2.    
\end{align*}
The first term is absorbed by the $CK_{wL}^3$ term in \eqref{ineq:AQ3_Evo}. For the latter term we apply Lemma \ref{lem:QPELpressureI}, which yields contributions which are integrable or are absorbed by the $CK$ terms. 

\paragraph{Treatment of $LP3^C$}
Turn first to $LP3^{C}_{HL}$, in which the coefficient is in `high frequency'. 
By \eqref{ineq:AprioriUneq}, Lemma \ref{lem:ABasic}, \eqref{ineq:triQuadHL}, and Lemma \ref{lem:CoefCtrl}, we have 
\begin{align*} 
LP3^{C}_{HL} &\lesssim \frac{\epsilon}{\jap{\nu t^3}^{\alpha}}\sum_{k,l}\int \abs{A^{3} \widehat{Q^3_k}(\eta,l)} \frac{1}{\jap{\xi,l^\prime}\jap{t}} A\left(\abs{\widehat{\psi_z}(\xi,l)_{Hi}} + \jap{t}^{-1}\abs{\widehat{\phi_z}(\xi,l)_{Hi}} \right) Low(k,\eta-\xi,l-l^\prime) d\eta d\xi \\
& \lesssim  \frac{\epsilon}{\jap{t} \jap{\nu t^3}^{\alpha}}\norm{A^{3}Q^3}_2 \norm{A C}_2,
\end{align*}  
which is consistent with Proposition \ref{prop:Boot} for $\epsilon$ sufficiently small.

Next turn to $LP3^{C}_{LH}$, which by Lemma \ref{lem:CoefCtrl}  and \eqref{ineq:ABasic}, is controlled via
\begin{align*} 
LP3^{C}_{LH}  & \lesssim \epsilon \jap{t} \sum_{k \neq 0,l}\int_\eta  \abs{A^{3} \widehat{Q^3_k}(\eta,l)} \frac{\abs{k}\abs{\xi-kt,l^\prime}}{\left(k^2 + (l^\prime)^2 + \abs{\xi-kt}^2\right) \jap{\frac{t}{\jap{\xi,l^\prime}}}} \abs{A^{2}\widehat{\left(\Delta_{L} U^2 \right)}_k(\xi,l^\prime)} Low(\eta-\xi,l-l^\prime) d\eta. 
\end{align*} 
We may treat this in a manner similar to the canonical $NLP(1,3,0,\neq)$ on $Q^2$ in \S\ref{sec:NLP213}.
Indeed, by \eqref{ineq:AiPartX} and \eqref{ineq:triQuadHL} we have,  
\begin{align*} 
LP3^{C}_{HL} & \lesssim c_{0}\norm{\left(\sqrt{\frac{\partial_t w}{w}} \tilde{A}^{3} + \frac{\abs{\grad}^{s/2}}{\jap{t}^s}A^{3}\right) Q^3}_2^2 + c_{0}\norm{\left(\sqrt{\frac{\partial_t w}{w}} \tilde{A}^{2}+ \frac{\abs{\grad}^{s/2}}{\jap{t}^s}A^{2}\right) \Delta_L U^2_{\neq}}_2^2 \\ 
& \quad + \epsilon\norm{A^3 Q^3_{\neq}}_2\norm{\Delta_L A^2 U^2_{\neq}}_2, 
\end{align*}  
which by Lemmas \ref{lem:PEL_NLP120neq} and  \ref{lem:SimplePEL}, is consistent with Proposition \ref{prop:Boot} by the bootstrap hypotheses. 
The remainder term $LP3^{\mathcal{R}}$ is straightforward and is omitted for the sake of brevity; this completes the treatment of $LP3$. 

\section{High norm estimate on $Q^1_0$} \label{sec:Q1Hi1}
As in \cite{BGM15I}, the improvement of \eqref{ineq:Boot_Q1Hi1} proceeds slightly differently than most other estimates we are making. 
The goal is to obtain exactly $O(\epsilon \jap{t})$ growth, rather than any logarithmic losses in $t$ or $\epsilon$. 
We will deduce an estimate like
\begin{align} 
\frac{1}{2}\frac{d}{dt}\norm{A^{1}Q^1_0}^2_2  & \leq -\frac{t}{\jap{t}^2}\norm{A^1 Q^1_0}_2^2 + \frac{1}{\jap{t}}\norm{A^{1}Q^1_0}_2\norm{A^2 Q^2_0}_2 + c_{0}\epsilon^2\mathcal{I}(t) \nonumber \\  
& \leq -\frac{t}{\jap{t}^2}\norm{A^1Q^1_0}_2^2 +  \frac{4\epsilon}{\jap{t}}\norm{A^{1}Q^1_0}_2 + c_{0}\epsilon^2\mathcal{I}(t),  \label{ineq:basic_Q1}
\end{align} 
where $\int_1^{c_0 \epsilon^{-1}} \mathcal{I}(t) dt = O(K_B)$ uniformly in $\epsilon$. 
This yields the desired bound by comparing $X(t) = \norm{A^1 Q^1_0(t)}_2^2$ to the super-solution of the inequality given by $Y(t) = \max(\frac{3}{2}K_{H10}, 6\sqrt{2})\epsilon + c_{0}\epsilon^2 \int_{1}^t\mathcal{I}(\tau) d\tau$ and choosing $c_0$ sufficiently small. 
Indeed (for $K_{H10}$ sufficiently large), 
\begin{align*} 
\partial_t Y(t) = c_{0}\epsilon^2\mathcal{I}(t) \geq \left(-\frac{t}{\jap{t}^2}Y(t) +  \frac{4\epsilon}{\jap{t}}\right) Y(t) + c_{0}\epsilon^2\mathcal{I}(t), 
\end{align*}
as the additional two terms on the RHS sum to something negative by the choice of $Y(t)$ (recall $t \geq 1$).
By Lemma \ref{lem:BootStart}, $X(1) < Y(1)$, and therefore by comparison and \eqref{ineq:basic_Q1},  $X(t) \leq Y(t)$ for all $t \in [1,T_\star)$. 

Therefore, improving \eqref{ineq:Boot_Q1Hi1} reduces to proving an estimate like \eqref{ineq:basic_Q1}.
From the evolution equation for $Q^1_0$, using enumerations analogous to \eqref{def:Q2Enums} and \eqref{def:Q3Enums} above,  
\begin{align} 
\frac{1}{2}\frac{d}{dt}\norm{A^{1} Q^1_0}_2^2 & \leq \dot{\lambda}\norm{\abs{\grad}^{s/2}A^{1} Q_0^1}_2^2 - \norm{\sqrt{\frac{\partial_t w}{w}} \tilde{A}^{1} Q_0^1}_2^2 - \frac{t}{\jap{t}^2}\norm{A^1 Q_0^1}_2^2 \nonumber \\ 
& \quad - \int A^{1}Q^1_0 A^1 Q^2_0 dV + \nu  \int A^{1} Q^{1}_0 A^{1} \left(\tilde{\Delta_t} Q^1_0\right) dV 
-  \int A^{1} Q^1_0 A^{1}\left(\tilde U_0 \cdot \grad Q^1_0\right) dV \nonumber \\  
& \quad -  \int A^{1} Q^1_0 A^{1} \left(Q^j_0 \partial_j^t U^1_0 + 2\partial_i^t U^j_0 \partial_{ij}^t U^1_0\right) dV \nonumber \\ 
& \quad -  \int A^{1} Q^1_0 A^{1} \left(Q^j_{\neq} \partial_j^t U^1_{\neq} + 2\partial_i^t U^j_{\neq} \partial_{ij}^t U^1_{\neq}\right)_0 dV \nonumber \\ 
& = -\mathcal{D}Q_0^1 + CK_L^1  + LU + \mathcal{D}_E + \mathcal{T}_0 + NLS1(j,0) + NLS2(i,j,0) + \mathcal{F}, \label{eq:AevoQ10}
\end{align} 
where we are denoting 
\begin{align*}
\mathcal{D}_E = \nu  \int A^{1} Q^{1}_0 A^{1} \left((\tilde{\Delta_t} - \Delta_L) Q^1_0\right) dV. 
\end{align*}
As above in \eqref{def:Q2Enums} and \eqref{def:Q3Enums}, we have decomposed the nonlinear terms based on the heuristics in \S\ref{sec:NonlinHeuristics}.   

Notice that, due to the $X$ average, the linear pressure and stretching terms both disappear along with the nonlinear pressure. 
Hence the main growth of $Q^1_0$ is caused by the lift-up effect term, $LU$.
This term is treated by Cauchy-Schwarz:
\begin{align*}
LU \leq \jap{t}^{-1} \norm{A^{1} Q^1_0}_2 \norm{A^{2} Q^2_0}_2,   
\end{align*}  
which, together with \eqref{ineq:Boot_Q2Hi} is responsible for the leading order linear term in \eqref{ineq:basic_Q1}. 
It remains to see how to control the nonlinear terms.

\subsection{Transport nonlinearity} 
By Lemma \ref{lem:AAiProd} (with \eqref{ineq:Boot_Ag} and \eqref{ineq:Boot_gLow}), 
\begin{align*}
\mathcal{T}_{0} & \lesssim \norm{g}_{\G^{\lambda,\gamma}}\norm{A^1 Q_0^1}_2\norm{\grad A^1 Q_0^1}_2 + \norm{Ag}_2 \norm{A^1 Q_0^1}_2^2 \lesssim \epsilon^{3/2}\norm{\grad A^1 Q_0^1}^2_2 +  \left(\frac{\epsilon^{1/2}}{\jap{t}^{4}} + \epsilon\right)\norm{A^1 Q_0^1}^2_2, 
\end{align*}
which is consistent with Proposition \ref{prop:Boot} for $c_0$ and $\epsilon$ sufficiently small.

\subsection{Nonlinear stretching}
This term is the analogue of those treated in \S\ref{sec:NLPSQ20} and corresponds to the nonlinear stretching effects on $Q_0^1$ involving only zero frequencies (the pressure disappears due to the $X$ average). 
The treatment of this term can be made in the same way as the corresponding treatment for $Q^2$ in \S\ref{sec:TransQ20} and \S\ref{sec:NLPSQ20}, although it is slightly easier here as we are permitting growth on $Q_0^1$, unlike $Q_0^2$ (in particular $A_0^1 \approx \jap{t}^{-1}A^2_0$). 
Hence, these contributions are omitted for brevity. 

\subsection{Forcing from non-zero frequencies}
In this section we consider interactions of type \textbf{(F)} (see \S\ref{sec:NonlinHeuristics}): the forcing of non-zero frequencies directly back onto $Q_0^1$. 
Recall from \eqref{eq:XavgCanc},
\begin{align*} 
\mathcal{F} & = -\int A^1 Q^1 A^1 \left(\partial_Y^t \partial_Y^t \partial_Y^t \left(U^2_{\neq} U^1_{\neq}\right)_{0} + \partial_Y^t \partial_Y^t \partial_Z^t \left(U^3_{\neq} U^1_{\neq}\right)_{0} \right) dV \\
& \quad - \int A^1 Q^1 A^1\left(\partial_Z^t \partial_Z^t \partial_Z^t \left(U^3_{\neq} U^1_{\neq}\right)_{0} + \partial_Z^t \partial_Z^t \partial_Y^t \left(U^2_{\neq} U^1_{\neq}\right)_{0}\right) dV \\
& = F^1 + F^2 + F^3 + F^4. 
\end{align*}
Let us begin with $F^2$ (corresponding to $i = 2$ and $j = 3$); the treatment is also essentially the same as $F^3$. Note the terms involving $U^3$ are expected to be the worst due to the regularity imbalances. 
Decompose the $F^2$ with a paraproduct; as usual we group contributions where the coefficients appear in low frequency with the remainder:
\begin{align}  
F^2 & = -\sum_{k\neq 0} \int A^{1}Q^1_0 A_0^{1} \partial_Y\partial_Y\partial_Z\left( \left(U^3_{-k}\right)_{Hi} \left( U^1_k\right)_{Lo}\right)dV \nonumber \\
& \quad - \sum_{k\neq 0} \int A^{1}Q^1_0 A_0^{1} \partial_Y \partial_Y\partial_Z\left( \left(U^3_{-k}\right)_{Lo} \left( U^1_k\right)_{Hi}\right) dV \nonumber \\
& \quad -\sum_{k\neq 0}  \int A^{1}Q^1_0 A_0^{1} \left(\left( (\psi_y)_{Hi} \partial_Y + (\phi_{y})_{Hi} \partial_Z \right)\partial_Y\partial_Z\left( \left(U^3_{-k}\right)_{Lo} \left( U^1_k\right)_{Lo}\right)\right) dV \nonumber \\
& \quad - \sum_{k\neq 0} \int A^{1}Q^1_0 A_0^{1} \partial_Y\left( \left( (\psi_y)_{Hi} \partial_Y + (\phi_{y})_{Hi} \partial_Z \right) \partial_Z\left( \left(U^3_{-k}\right)_{Lo} \left( U^1_k\right)_{Lo}\right)\right) dV \nonumber \\
& \quad - \sum_{k\neq 0} \int A^{1}Q^1_0 A_0^{1} \partial_Y \partial_Y \left( \left( (\psi_z)_{Hi} \partial_Y + (\phi_{z})_{Hi} \partial_Z \right)\left( \left(U^3_{-k}\right)_{Lo} \left( U^1_k\right)_{Lo}\right)\right) dV \nonumber \\
& \quad + F^2_{\mathcal{R},C} \nonumber \\
& = F^2_{HL} + F^2_{LH} + F^2_{C1} + F^2_{C2} + F^2_{C3} + F^2_{\mathcal{R},C}. \label{def:F2ppQ10}
\end{align} 
Turn first to $F^{2}_{HL}$. 
By \eqref{ineq:AprioriUneq}, Lemma \ref{lem:ABasic}, \eqref{ineq:AA3Z}, \eqref{ineq:AdelLij},  and \eqref{ineq:triQuadHL}, 
\begin{align*} 
F_{HL}^{2} & \lesssim \frac{\epsilon}{\jap{\nu t^3}^{\alpha}}   \jap{t}^{\delta_1}\sum_{l,l^\prime,k \neq 0} \int \abs{A^{1} \widehat{Q^1_0}(\eta,l) A^1_{0}(\eta,l) \frac{\abs{\eta}^2\abs{l}}{k^2 + (l^\prime)^2 + \abs{\xi-kt}^2} \Delta_L \widehat{U^3_{k}}(\xi,l^\prime)_{Hi}} \\ & \hspace{5cm} \times Low(-k,\eta-\xi,l-l^\prime) d\eta d\xi \\ 
& \lesssim \frac{\epsilon}{\jap{\nu t^3}^{\alpha}} \jap{t}^{\delta_1-1}\sum_{l,l^\prime,k \neq 0} \int \abs{\widehat{Q^1_0}(\eta,l)} \frac{\abs{\eta}^2\abs{l} \jap{\frac{t}{\jap{\xi,l^\prime}}}^2}{k^2 + (l^\prime)^2 + \abs{\xi-kt}^2}  \\ 
& \quad\quad \times \left(\sum_{r} \chi^{r,NR}\frac{t}{\abs{r} + \abs{\eta-tr}}\tilde{A}^1_0(\eta,l)\tilde{A}^3_k(\xi,l^\prime) + \chi^{\ast;23} A^1_0(\eta,l) A^3_k(\xi,l^\prime) \right) \\ 
& \quad\quad \times \abs{ \Delta_L A^{3}\widehat{U^3_{k}}(\xi,l^\prime)_{Hi}}Low(-k,\eta-\xi,l-l^\prime) d\eta d\xi \\ 
& \lesssim \frac{\epsilon \jap{t}^{1+\delta_1}}{\jap{\nu t^3}^{\alpha}}\norm{\left(\sqrt{\frac{\partial_t w}{w}}\tilde{A}^1 + \frac{\abs{\grad}^{s/2}}{\jap{t}^{s}}A^1\right) Q^1_0}_2 \norm{\left(\sqrt{\frac{\partial_t w}{w}}\tilde{A}^3 + \frac{\abs{\grad}^{s/2}}{\jap{t}^s}A^3 \right) \Delta_L U^3_{\neq}}_2
\\ & \quad + \frac{\epsilon}{t^{1-\delta_1}\jap{\nu t^3}^{\alpha}}\norm{\sqrt{-\Delta_L} A^{1}Q^1}_2 \norm{A^{3}\Delta_L U^3_{\neq}}_2,
\end{align*} 
which, after the application of Lemmas \ref{lem:PEL_NLP120neq} and \ref{lem:SimplePEL}, is consistent with \eqref{ineq:basic_Q1} for $\epsilon$ sufficiently small.  
 
Turn next to $F_{LH}^{2}$, which also by \eqref{ineq:ABasic}, \eqref{ineq:AdelLij} and \eqref{ineq:quadHL} we have 
\begin{align*} 
F_{LH}^{2} & = \frac{\epsilon}{\jap{\nu t^3}^{\alpha}} \sum_{l,l^\prime,k\neq 0} \int \abs{A^{1} \widehat{Q^1_0}(\eta,l) \frac{\abs{\eta}^2\abs{l} \jap{\frac{t}{\jap{\xi,l^\prime}}}^{1+\delta_1}}{k^2 + (l^\prime)^2 + \abs{\xi-kt}^2} \Delta_L A^{1}\widehat{U^1_{k}}(\xi,l^\prime)_{Hi}} \\ & \hspace{5cm} \times Low(-k,\eta-\xi,l-l^\prime) d\eta d\xi \\ s
& \lesssim \frac{\epsilon \jap{t}^2}{\jap{\nu t^3}^{\alpha}}\norm{\left(\sqrt{\frac{\partial_t w}{w}} + \frac{\abs{\grad}^{s/2}}{\jap{t}^{s}}\right) A^{1}Q^1_0}_2 \norm{\left(\sqrt{\frac{\partial_t w}{w}}\tilde{A}^1 + \frac{\abs{\grad}^{s/2}}{\jap{t}^s}A^1 \right) \Delta_L U^1_{\neq}}_2
\\ & \quad + \frac{\epsilon}{\jap{\nu t^3}^\alpha}\norm{\sqrt{-\Delta_L} A^{1}Q^1}_2 \norm{A^{1}\Delta_L U^1_{\neq}}_2,   
\end{align*} 
which, after the application of the Lemmas \ref{lem:PEL_NLP120neq} and \ref{lem:SimplePEL}, is consistent with \eqref{ineq:basic_Q1}. 

The most difficult coefficient error term in \eqref{def:F2ppQ10} is $F^2_{C3}$. 
By Lemma \ref{lem:ABasic}, \eqref{ineq:triQuadHL}, and Lemma \ref{lem:CoefCtrl},
\begin{align*} 
F^2_{C3} & \lesssim  \epsilon^2 \jap{t}^{\delta_1} \jap{\nu t^3}^{-2\alpha} \sum_{l,l^\prime,k\neq 0} \int \abs{A^{1} \widehat{Q^1_0}(\eta,l)} A_0^{1}(\eta,l) \abs{\eta}^2 \\ & \quad\quad \times\left( \abs{\widehat{\psi_z}(\xi,l^\prime)_{Hi}} + \abs{\widehat{\phi_z}(\xi,l^\prime)_{Hi}}\right) Low(k,\eta-\xi,l-l^\prime) d\eta d\xi \\ 
& \lesssim \epsilon^{3/2} \norm{\grad A^{1}Q^1_0}_2^2 + \frac{\epsilon^{5/2} \jap{t}^{2\delta_1}}{\jap{\nu t^3}^{4\alpha}}\norm{AC}^2_2,
\end{align*} 
which is consistent with \eqref{ineq:basic_Q1} for $c_{0}$ and $\epsilon$ sufficiently small by \eqref{ineq:Boot_ACC2}. 
The other coefficient terms in \eqref{def:F2ppQ10}, $F_{C1}^2$ and $F_{C2}^2$ are easier and give similar contributions. Hence, these are omitted for the sake of brevity.
The remainder term in \eqref{def:F2ppQ10}, $F_{\mathcal{R}}^2$, is similarly straightforward and is omitted as well. 
This completes the treatment of $F^2$. 
Despite appearing rather different, in fact the treatment of $F^3$ is essentially the same. 
Indeed, the regularity imbalances are restricted to where $\abs{\partial_Z} \lesssim \abs{\partial_Y}$ and hence, for frequencies where the regularity imbalances are occurring, $F^3$ looks roughly like $F^2$ and the same treatment applies. Outside of the regularity imbalances, one simply uses that $\Delta_L$ is uniformly elliptic in $Z$ in the same way non-resonance is used above in the treatment of $F^2$ (see \S\ref{sec:basicmult} for more details). As the details are exactly the same as above, we omit them for brevity.  

The other $\mathcal{F}$ terms, $F^1$ and $F^4$, are treated as in \cite{BGM15I}; $F^1$ is slightly harder. 
The main idea is similar to the treatment of $F^2$ above, however one instead 
uses \eqref{ineq:AdeYYY} for $F^1$ (and \eqref{ineq:AdeYZZ} for $F^4$)
and hence deduce 
\begin{align*}
F^1 + F^4 & \lesssim \frac{\epsilon \jap{t}^{2+\delta_1}}{\jap{\nu t^3}^{\alpha}}\norm{\left(\sqrt{\frac{\partial_t w}{w}}\tilde{A}^{1} + \frac{\abs{\grad}^{s/2}}{\jap{t}^{s}} A^{1} \right)Q^1_0}_2 \norm{\left(\sqrt{\frac{\partial_t w}{w}}\tilde{A}^2 + \frac{\abs{\grad}^{s/2}}{\jap{t}^s}A^2 \right) \Delta_L U^2_{\neq}}_2
\\ & \quad + \frac{\epsilon \jap{t}^{\delta_1}}{\jap{\nu t^3}^\alpha}\norm{\sqrt{-\Delta_L} A^{1}Q^1_0}_2 \norm{A^{2}\Delta_L U^2_{\neq}}_2 \\ 
& \quad + \frac{\epsilon \jap{t}^{2}}{\jap{\nu t^3}^{\alpha}}\norm{\left(\sqrt{\frac{\partial_t w}{w}}\tilde{A}^1 + \frac{\abs{\grad}^{s/2}}{\jap{t}^{s}}A^1 \right)Q^1_0}_2 \norm{\left(\sqrt{\frac{\partial_t w}{w}}\tilde{A}^1 + \frac{\abs{\grad}^{s/2}}{\jap{t}^s}A^1 \right) \Delta_L U^1_{\neq}}_2
\\ & \quad + \frac{\epsilon \jap{t}^{\delta_1}}{\jap{\nu t^3}^\alpha}\norm{\sqrt{-\Delta_L} A^{1}Q^1}_2 \norm{A^{1}\Delta_L U^1_{\neq}}_2,   
\end{align*}
which, after applying Lemmas \ref{lem:PEL_NLP120neq} and \ref{lem:SimplePEL},  is consistent with \ref{ineq:basic_Q1} under the bootstrap hypotheses for $c_{0}$ and $\epsilon$ chosen sufficiently small. 
This completes all of the forcing terms. 
  
\subsection{Dissipation error terms}
As in \cite{BGM15I}, these can be treated in the same manner as the dissipation error terms on $Q^2_0$ were treated in \S\ref{sec:DEQ02}. 
We omit the details for brevity: 
\begin{align} 
\mathcal{D}_E & \lesssim c_0\nu \norm{\sqrt{-\Delta_L}A^1 Q^1_0}_2^2 + \nu\epsilon^2 c_0^{-1} \norm{\grad A C}_{2}^2,  
\end{align}  
which for $c_{0}$ sufficiently small, is consistent with Proposition \ref{prop:Boot}. This completes the high norm estimate on $Q_0^1$. 

\section{High norm estimate on $Q^1_{\neq}$}
Consider from the evolution equation for $Q^1_{\neq}$: 
\begin{align} 
\frac{1}{2}\frac{d}{dt} \norm{A^{1} Q^1_{\neq}}_2^2  & \leq \dot{\lambda}\norm{\abs{\grad}^{s/2}A^{1} Q^1_{\neq}}_2^2 - \norm{\sqrt{\frac{\partial_t w}{w}} \tilde{A}^{1} Q^1_{\neq}}_2^2 \nonumber \\ 
& \quad - \norm{\sqrt{\frac{\partial_t w_L}{w_L}} A^{1} Q^1_{\neq}}_2^2
- \frac{t}{\jap{t}^{2}}\norm{A^{1}Q^1_{\neq}}_2^2 -\frac{(1+\delta_1)}{t} \norm{\mathbf{1}_{t > \jap{\grad_{Y,Z}}} A^{1} Q^1_{\neq}}_2^2 \nonumber \\
& \quad - \int A^{1}Q^1_{\neq} A^1 Q^2_{\neq} dV -2 \int A^{1} Q^1 A^{1} \partial_{YX}^t U^1_{\neq} dV + 2 \int A^{1} Q^1_{\neq} A^{1} \partial_{XX} U^2_{\neq} dV \nonumber \\  
 & \quad + \nu \int A^{1} Q^{1}_{\neq} A^{1} \left(\tilde{\Delta_t} Q^1_{\neq}\right) dV  - \int A^{1} Q^1_{\neq} A^{1}\left( \tilde U \cdot \grad Q^1  \right) dV \nonumber \\ 
& \quad -  \int A^{1} Q^1_{\neq} A^{1} \left[ Q^j \partial_j^t U^1 + 2\partial_i^t U^j \partial_{ij}^t U^1  - \partial_X\left(\partial_i^t U^j \partial_j^t U^i\right) \right] dV \nonumber \\ 
& = -\mathcal{D}Q^1_{\neq} - CK_{L1}^1 - (1+\delta_1)CK_{L2}^1 + LU + LS1 + LP1 \nonumber \\ & \quad + \mathcal{D}_E + \mathcal{T} + NLS1 + NLS2 + NLP,  \label{ineq:Q1HneqEvo}
\end{align}
where as usual 
\begin{align*}
\mathcal{D}_E = \int A^1 Q^1_{\neq} A^1 \left((\tilde{\Delta_t} - \Delta_L)Q^1_{\neq}\right) dV. 
\end{align*}
We define enumerations of the nonlinear terms analogous to those in \eqref{def:Q2Enums} and \eqref{def:Q3Enums}.   

\subsection{Linear stretching term $LS1$} \label{sec:LS1_Hi}
As discussed in \cite{BGM15I}, one of the difficulties in deducing the high norm estimate on $Q^1_{\neq}$ is the linear stretching term $LS1$. 
First separate into two parts (to be sub-divided further), 
\begin{align*} 
LS1 & = -2\int A^{1} Q^1 A^{1}\partial_X(\partial_Y - t\partial_X) U^1  dV - 2\int A^{1} Q^1 A^{1} \partial_X\left( (\psi_y)(\partial_Y - t\partial_X) + (\phi_{y})\partial_Z\right) U^1 dV \\ 
& = LS1^0 + LS1^{C}.   
\end{align*} 

\subsubsection{Treatment of $LS1^C$}
The $LS1^C$ term can be treated in essentially the same manner as the corresponding $LS3^C$ in \S\ref{sec:LS3C}. 
Hence, we omit the details for brevity.  

\subsubsection{Leading order term $LS1^0$}
As in \eqref{eq:LS30} of \S\ref{sec:LS30}, we first expand by writing out $\Delta_t^{-1}$ in terms of $\Delta_L$: 
\begin{align*} 
LS1^0 & =  -2 \int A^{1} Q^1 A^{1}\partial_X(\partial_Y - t\partial_X) \Delta_{L}^{-1} \left[Q^1 - G_{yy} (\partial_Y - t\partial_X)^2 U^1 - G_{yz}\partial_Z (\partial_Y - t\partial_X) U^1 \right. \\ & \quad\quad\quad \left. - G_{zz}\partial_{ZZ} U^1  -\Delta_t C^1 (\partial_Y - t\partial_X) U^1 -\Delta_t C^2 \partial_Z U^1 \right]  dV \\ 
& = LS1^{0;0}  + \sum_{i = 1}^5 LS1^{0;Ci}. 
\end{align*} 
The leading order term is treated as in \cite{BGM15I} (with the slight variation for large $Z$ frequencies as used in \S\ref{sec:LS30} above), and hence we omit the treatment and  conclude the following for some $K > 0$,  
\begin{align*}
LS1^{0,0} & \leq (1-\delta_1)CK_{L1}^1 + (1+\delta_1)CK^1_{L2} + \frac{\delta_\lambda}{10\jap{t}^{3/2}} \norm{\abs{\grad}^{s/2}A^{1}Q^1_{\neq}}_2^2 + \frac{K}{\kappa}\norm{\sqrt{\frac{\partial_t w}{w}}  \tilde{A}^{1} Q^1_{\neq}}_2^2 \\ & \quad + \frac{K}{\delta_\lambda^{\frac{1}{2s-1}} \jap{t}^{3/2}} \norm{A^{1}Q^1_{\neq}}_2^2 +  K\frac{1+\delta_1}{\jap{t}^2 t} \norm{A^1 Q^1_{\neq}}_2^2, 
\end{align*}
which is consistent with Proposition \ref{prop:Boot} under the bootstrap hypotheses for $K_{H1\neq}$ sufficiently large relative to $\exp(K\delta_\lambda^{-\frac{1}{2s-1}})$ (also, $\kappa$ must be chosen sufficiently large, but relative only to a universal constant).

The error terms $LS1^{0;Ci}$ are treated in a manner similar to the analogous terms in $LS3$ in \S\ref{sec:LS30} and hence the details are omitted for brevity (indeed $A^1_k$ is a weaker norm than $A^3_k$ due to the extra $\jap{t}^{-1}$ decay). This completes the treatment of the $LS1$ term. 

\subsection{Lift-up effect term $LU$} \label{sec:LUQhi2}
This follows as in \cite{BGM15I}, and hence we omit the details:  
\begin{align*} 
LU  & \leq \delta_1t\jap{t}^{-2}\norm{A^{1} Q^1_{\neq}}^2_2 + \frac{\delta_\lambda}{4\delta_1 t^{3/2}}\norm{\abs{\grad}^{s/2}A^2 Q^2_{\neq}}_2^2 + \frac{1}{4\delta_\lambda^{\frac{1}{2s-1}} t^{3/2}}\norm{A^2 Q^2_{\neq}}_2^2 +  \frac{1}{4\delta_1 t} \norm{\mathbf{1}_{t > \jap{\grad_{Y,Z}}}A^{2}Q^2_{\neq}}_2^2. 
\end{align*}  
The first term is absorbed by the remaining piece of $CK_{L1}^1$ left over in \eqref{ineq:Q1HneqEvo} from the treatment of $LS1$.
The others are consistent with Proposition \ref{prop:Boot} via \eqref{ineq:Boot_Q2Hi} for $K_{H1\neq}$ large relative to $\delta_1^{-1}$ and $\delta_\lambda^{-1}$. 
Hence, this suffices to treat $LU$.  

\subsection{Linear pressure term $LP1$}
The linear pressure term $LP3$ treated in \S\ref{ineq:LP3_Hi} is significantly harder than $LP1$ here, as here only $X$ derivatives are involved. 
Therefore, from Lemma \ref{dtw}, we get (the implicit constant is independent of $\kappa$), 
\begin{align*}
LP1 & \leq 2 \sum\int \abs{A^1\widehat{Q^1_k}(\eta,l) \frac{\abs{k}^2}{k^2 + l^2 + \abs{\eta-kt}^2}  A^1\Delta_L \widehat{U^2_k}(\eta,l)} d\eta \\ 
& \lesssim \kappa^{-1} \jap{t}^{-1} \norm{\left(\sqrt{\frac{\partial_t w}{w}}\tilde{A}^1 + \frac{\abs{\grad}^{s/2}}{\jap{t}^s}A^1\right)  Q^1_{\neq}}_2\norm{\left(\sqrt{\frac{\partial_t w}{w}}\tilde{A}^2 + \frac{\abs{\grad}^{s/2}}{\jap{t}^s}A^2\right)\Delta_L U_{\neq}^2}_2 \\ & \quad + \jap{t}^{-3}\norm{A^1Q^1}_2\norm{A^2\Delta_L U_{\neq}^2}_2.  
\end{align*}
Therefore for $\kappa$ and $K_{H1\neq}$ sufficiently large and $c_{0}$ sufficiently small, this is consistent with Proposition \ref{prop:Boot} by the bootstrap hypotheses after applying Lemmas \ref{lem:PEL_NLP120neq} and \ref{lem:SimplePEL}.
 
\subsection{Nonlinear pressure $NLP$}

After cancellations, none of the existing terms here are worse than those appearing in $Q^2$ in \S\ref{sec:NLPQ2} or $Q^3$ in \S\ref{sec:NLP3}. Moreover, on $Q^1$ we are imposing less control (since $A^1$ is weaker than $A^{2,3}$ at high frequencies due to $\jap{t}^{-1}$) and the leading derivative is an $X$ derivative, which is generally less dangerous than those associated with $Y$ and $Z$. 
Therefore, the treatment of the $NLP$ contributions here are an easy variant of the treatments in \S\ref{sec:NLPQ2} and \S\ref{sec:NLP3}. Accordingly, the details are omitted for the sake of brevity. 

\subsection{Nonlinear stretching $NLS$}
These terms can be slightly more dangerous than the corresponding $NLS$ terms in $Q^{2,3}$ due to the persistent presence of $U^1$, however, this will be naturally balanced by the allowed linear growth of $Q^1$ at high frequencies.
 
\subsubsection{Treatment of $NLS1$}
Consider first the $NLS1(j,\neq,0)$ terms. Note $j \neq 1$ due to the zero frequencies (a crucial nonlinear structure). The case $j = 3$ is worse than $j=2$ due to the large growth and regularity imbalances in $Q^3$. 
Hence, let us just focus on the case $j = 3$. 
As usual,  with a paraproduct and group any terms with coefficients in low frequencies in with the remainder: 
\begin{align*} 
NLS1(3,\neq,0) & = -\sum \int A^1 Q^1_k A^1\left( (Q^3_{k})_{Hi} (\partial_Z U_0^1)_{Lo} \right) dV - \sum \int A^1 Q^1_k A^1\left( (Q^3_{k})_{Lo} (\partial_Z U_0^1)_{Hi} \right) dV \\ 
& \quad - \sum \int A^1 Q^1_k A^1\left( (Q^3_{k})_{Lo} \left((\psi_z)_{Hi} \partial_Y + (\phi_z)_{Hi} \partial_Z \right) (U_0^1)_{Lo} \right) dV + S_{\mathcal{R},C} \\ 
& = S_{HL} + S_{LH} + S_{C} + S_{\mathcal{R},C}, 
\end{align*}
where $S_{\mathcal{R},C}$ includes the remainders from the paraproduct and the low frequency coefficient terms. 
By \eqref{ineq:AprioriU0}, Lemma \ref{lem:ABasic},
followed by \eqref{ineq:jNRBasic}, \eqref{ineq:teps12trick}, and \eqref{ineq:triQuadHL},
\begin{align*}
S_{HL} & \lesssim \epsilon \jap{t} \sum_{k} \int \abs{\widehat{Q^1_k}(\eta,l) \jap{t}^{-1} \jap{\frac{t}{\jap{\xi,l^\prime}}}^{1-\delta_1}} \\ 
& \quad\quad \times \left(\sum_{r}\chi^{r,NR}\frac{t}{\abs{r} + \abs{\eta-tr}}\tilde{A}^1_k(\eta,l) \tilde{A}^3_k(\xi,l^\prime) + \chi^{\ast;23} A^1_k(\eta,l) A^3_k(\xi,l^\prime) \right) \\ & \quad\quad \times \abs{\widehat{Q^3_{k}}(\xi,l^\prime)_{Hi}} Low(\eta-\xi,l-l^\prime) d\eta d\xi \\ 
& \lesssim \epsilon \jap{t}\norm{\left(\sqrt{\frac{\partial_t w}{w}}\tilde{A}^1 + \frac{\abs{\grad}^{s/2}}{\jap{t}^{s}} A^1 \right) Q^1}_2 \norm{\left(\sqrt{\frac{\partial_t w}{w}}\tilde{A}^3 + \frac{\abs{\grad}^{s/2}}{\jap{t}^{s}}A^3 \right) Q^3}_2 \\ 
& \quad + \epsilon^{3/2} \norm{\sqrt{-\Delta_L} A^1 Q^1_{\neq}}_2 + \mathbf{1}_{t \leq \epsilon^{-1/2+\delta/100}}\epsilon^{1/2} \norm{A^3 Q^3}_2^2 + \epsilon^{3/2-\delta/50}\norm{\sqrt{-\Delta_L} A^3 Q^3}_2^2, 
\end{align*}
which is consistent with Proposition \ref{prop:Boot} for $\epsilon$ and $\epsilon t \leq c_0$ sufficiently small. 

For $S_{LH}$ we use Lemma \ref{lem:ABasic} and \eqref{ineq:AprioriUneq} followed by \eqref{ineq:triQuadHL}, 
\begin{align*}
S_{LH} & \lesssim \frac{\epsilon \jap{t}^2 }{\jap{\nu t^3}^{\alpha}} \sum_{k} \int \abs{ A^1 \widehat{Q^1_k}(\eta,l) \frac{1}{\jap{\xi,l^\prime}} \jap{\frac{t}{\jap{\xi,l^\prime}}}^{-1-\delta_1} A^1 \jap{\grad}^{2} \widehat{U^1_{0}}(\xi,l^\prime)_{Hi} Low(k,\eta-\xi,l-l^\prime) } d\eta d\xi \\ 
& \lesssim \frac{\epsilon \jap{t}}{\jap{\nu t^3}^{\alpha}}\norm{A^1 Q^1}_2 \norm{A^1 \jap{\grad}^2 U_0^1}_2,
\end{align*}
which is consistent with Proposition \ref{prop:Boot} for $\epsilon$ and $c_0$ sufficiently small. 
Similarly, 
\begin{align*}
S_{C} & \lesssim \frac{\epsilon^2 \jap{t}}{\jap{\nu t^3}^{\alpha}}\norm{A^1 Q^1}_2 \norm{A C}_2 \lesssim \frac{\epsilon^{1/2}}{\jap{\nu t^3}^{\alpha}}\norm{A^1 Q^1}_2^2 +  \frac{\epsilon^{5/2} t^2}{\jap{\nu t^3}^{\alpha}} \norm{A C}^2_2,  
\end{align*}
which is consistent with Proposition \ref{prop:Boot} for $\delta > 0$ and $\epsilon$ sufficiently small.  
The remainder terms are similar to the above and are hence omitted for brevity. This completes the treatment of the $NLS1(3,\neq,0)$ terms; the other $j$ are simpler. 

Next consider the $NLS1(j,0,\neq)$ terms. 
The most difficult is naturally the case $j = 1$ (which does not cancel); the others are simpler and are hence omitted for brevity. 
Expand with a paraproduct, 
\begin{align*}
NLS1(1,0,\neq) & = -\sum \int A^1 Q^1_k A^1\left( (Q^1_{0})_{Hi} ( \partial_X U_k^1)_{Lo} \right) dV \\ 
& \quad - \sum \int A^1 Q^1_k A^1\left( (Q^2_{0})_{Lo} ( \partial_XU_k^1)_{Hi} \right) dV + S_{\mathcal{R},C} \\ 
& = S_{HL} + S_{LH} + S_{\mathcal{R},C}. 
\end{align*}
From Lemma \ref{lem:ABasic} and \eqref{ineq:AprioriUneq}, 
\begin{align*}
S_{HL} & \lesssim \frac{\epsilon \jap{t}^{\delta_1}}{\jap{\nu t^3}^{\alpha}}\sum_{k} \int \abs{ A^1 \widehat{Q^1_k}(\eta,l) \jap{\frac{t}{\jap{\xi,l^\prime}}}^{-1-\delta_1} A^1 \widehat{Q^1_{0}}(\xi,l^\prime)_{Hi} Low(\eta-\xi,l-l^\prime) } d\eta d\xi \\ 
& \lesssim \frac{\epsilon \jap{t}^{\delta_1}}{\jap{\nu t^3}^{\alpha}} \norm{A^1 Q^1_{\neq}}_2 \norm{A^1 Q^1_0}_2. 
\end{align*}
From \eqref{ineq:ABasic}, \eqref{ineq:AprioriU0}, and \eqref{ineq:AiPartX},  
\begin{align*}
S_{LH} & \lesssim \epsilon t \sum_{k} \int \abs{ A^1 \widehat{Q^1_k}(\eta,l) \frac{\abs{k}}{k^2 + (l^\prime)^2 + \abs{\xi - kt}^2} \Delta_L A^1 \widehat{U^1_{k}}(\xi,l^\prime)_{Hi} Low(\eta-\xi,l-l^\prime) } d\eta d\xi \\ 
& \lesssim \epsilon t \norm{\left(\sqrt{\frac{\partial_t w}{w}}\tilde{A}^1 + \frac{\abs{\grad}^{s/2}}{\jap{t}^{s}}A^1\right)Q^1}_2\norm{\left(\sqrt{\frac{\partial_t w}{w}}\tilde{A}^1 + \frac{\abs{\grad}^{s/2}}{\jap{t}^{s}}A^1\right)\Delta_L U^1_{\neq}}_2 \\ 
& \quad + \epsilon \norm{A^1 Q^1}_2\norm{A^1 \Delta_L U^1_{\neq}}_2,  
\end{align*}
which is consistent with Proposition \ref{prop:Boot} by Lemmas \ref{lem:PEL_NLP120neq} and \ref{lem:SimplePEL}.  
This completes the $NLS1(j,0,\neq)$ terms. 

Finally consider the $NLS1(j,\neq,\neq)$ terms. 
All these terms are treated similarly, hence, consider just $j = 3$. 
Expand as above 
\begin{align*}
NLS1(3,\neq,\neq) & = -\int A^1 Q^1_{\neq} A^1\left( (Q^3_{\neq})_{Hi} (\partial_Z U_{\neq}^1)_{Lo} \right) dV - \int A^1 Q^1_{\neq} A^1\left( (Q^3_{\neq})_{Lo} (\partial_Z U_{\neq}^1)_{Hi} \right) dV \\ 
& \quad - \sum \int A^1 Q^1_{\neq} A^1\left( (Q^3_{\neq})_{Lo} \left((\psi_z)_{Hi} (\partial_Y - t\partial_X) + (\phi_z)_{Hi}\partial_Z\right)(U_{\neq}^1)_{Lo} \right) dV \\ 
& \quad + S_{\mathcal{R},C} \\ 
& = S_{HL} + S_{LH} + S_{C} + S_{\mathcal{R},C}. 
\end{align*}
For $S_{HL}$, we have by \eqref{ineq:AprioriUneq}, Lemma \ref{lem:ABasic}, and \eqref{ineq:triQuadHL} (a power of $t$ is lost due to the regularity imbalances), 
\begin{align*}
S_{HL} & \lesssim \frac{\epsilon \jap{t}^{\delta_1}}{\jap{\nu t^3}^{\alpha}} \sum \int \mathbf{1}_{k k^\prime(k-k^\prime) \neq 0} \abs{A^1 \widehat{Q^1_k}(\eta,l)}
  \\ & \quad\quad \times \jap{\frac{t}{\jap{\xi,l^\prime}}}^{1-\delta_1} \abs{A^3 \widehat{Q^3_{k^\prime}}(\xi^\prime,l^\prime)_{Hi}} Low(k-k^\prime,\eta-\xi,l-l^\prime) d\eta d\xi \\ 
& \lesssim \frac{\epsilon t}{\jap{\nu t^3}^{\alpha}}\norm{A^1 Q^1_{\neq}}_2 \norm{A^3 Q^3}_2,
\end{align*} 
which is consistent with Proposition \ref{prop:Boot}. 
For $S_{LH}$ we have by \eqref{ineq:Boot_Hi},  \eqref{ineq:ABasic}, \eqref{ineq:AikDelLNoD}, and \eqref{ineq:triQuadHL}, 
\begin{align*} 
S_{LH} & \lesssim \frac{\epsilon \jap{t}^2}{\jap{\nu t^3}^{\alpha}} \sum_{k} \int \mathbf{1}_{k k^\prime(k-k^\prime) \neq 0}  \abs{ A^1 \widehat{Q^1_k}(\eta,l) \frac{1}{\abs{k^\prime} + \abs{l^\prime} + \abs{\eta-k^\prime t}} \Delta_L A^1 \widehat{U^1_{k^\prime}}(\xi,l^\prime)_{Hi}} \\ & \quad \quad \times Low(k-k^\prime, \eta-\xi,l-l^\prime) d\eta d\xi \\ 
& \lesssim \frac{\epsilon \jap{t}^2}{\jap{\nu t^3}^{\alpha}}\norm{\left(\sqrt{\frac{\partial_t w}{w}}\tilde{A}^1 + \frac{\abs{\grad}^{s/2}}{\jap{t}^{s}}A^1 \right) Q^1}_2\norm{\left(\sqrt{\frac{\partial_t w}{w}}\tilde{A}^1 + \frac{\abs{\grad}^{s/2}}{\jap{t}^{s}}A^1 \right) \Delta_L U_{\neq}^1}_2 \\  
& \quad + \frac{\epsilon \jap{t}}{\jap{\nu t^3}^{\alpha}}\norm{A^1 Q^1}_2 \norm{A^1 \Delta_L U_{\neq}^1}_2, 
\end{align*}
which by Lemmas \ref{lem:SimplePEL} and \ref{lem:PEL_NLP120neq}, is consistent with Proposition \ref{prop:Boot} for $\epsilon$ sufficiently small. 
The coefficient error terms are similar to those that arise in e.g. $NLP(i,j,\neq,\neq)$ and are hence omitted for brevity (although they require the hypothesis $\epsilon \lesssim \nu^{2/3+\delta}$ for $\delta > 0$).   
The remainder terms are either easier or similar to the above treatments and hence can also be omitted. 

As discussed above, the remaining $NLS1$ terms are similar or easier and hence are safely omitted. This completes the $NLS1$ terms. 

\subsubsection{Treatment of the $NLS2$ terms} 

Turn to the $NLS2$ terms. These terms are all treated via easy variants of the treatments of the $NLS1$ and $NLP$ terms. They are hence omitted for the sake of brevity. 

\subsection{Transport nonlinearity $\mathcal{T}$} 
In this section, we treat the \textbf{(SI)} and \textbf{(3DE)} contributions to the transport nonlinearity, given by $\mathcal{T}_{\neq}$. 
Begin with a paraproduct decomposition: 
 \begin{align*} 
 \mathcal{T}_{\neq} & = -\int A^{1} Q^1_{\neq} A^{1} \left( \tilde U_{Lo} \cdot \grad Q^3_{Hi} \right) dV -\int A^{3} Q^3 A^{3} \left( \tilde U_{Hi} \cdot \grad Q^3_{Lo} \right) dV + \mathcal{T}_{\mathcal{R}} \\ 
& = \mathcal{T}_T + \mathcal{T}_{R} + \mathcal{T}_{\mathcal{R}}, 
\end{align*}
where $\mathcal{T}_{\mathcal{R}}$ includes the remainder.  
Due to the lack of regularity imbalances in $A^1$, the transport and remainder contributions, $\mathcal{T}_T$ and $\mathcal{T}_{\mathcal{R}}$ respectively, 
are treated as in \S\ref{sec:Q2_TransNon}. Hence, we omit the treatments and conclude
\begin{align} 
\mathcal{T}_T + \mathcal{T}_{\mathcal{R}} & \lesssim \epsilon^{3/2}\norm{\sqrt{-\Delta_L}A^1 Q^1}_2^2 + \left(\frac{\epsilon^{1/2}}{\jap{t}^{2}} + \frac{\epsilon^{1/2} \jap{t}^{2\delta_1}}{\jap{\nu t^3}^{2\alpha}} \right)\norm{A^1 Q^1}_2^2.  
\end{align}
Turn to the reaction contribution. 
As in \S\ref{sec:Q2_TransNon} and \S\ref{sec:Q3_TransNon}, decompose the reaction term based on the $X$ dependence of each factor:  
\begin{align*} 
\mathcal{T}_{R} & =  -\int A^{1} Q^1 A^{1} \left( (\tilde U_{\neq})_{Hi}  \cdot (\grad Q^1_0)_{Lo} \right) dV - \int A^{1} Q^1 A^{1} \left( g_{Hi} \partial_Y (Q^1_{\neq})_{Lo} \right) dV \\ & \quad - \int A^{1} Q^1 A^{1} \left( (\tilde U_{\neq})_{Hi} \cdot (\grad Q^1_{\neq})_{Lo} \right) dV \\ 
& = \mathcal{T}_{R;\neq 0} + \mathcal{T}_{R;0 \neq}+ \mathcal{T}_{R;\neq \neq}. 
\end{align*} 

\subsubsection{Reaction term $\mathcal{T}_{R;0 \neq}$}
By\eqref{ineq:Boot_ED1} and Lemma \ref{lem:ABasic}, we get (also noting \eqref{def:tildeU2}):  
\begin{align*} 
\mathcal{T}_{R;0 \neq} & \lesssim \frac{\epsilon \jap{t}^{2+\delta_1}}{\jap{\nu t^3}^{\alpha}}\sum_{k \neq 0}\int \abs{A^{1} \hat{Q}^1_k (\eta,l) A^{1}_{k}(\eta,l) \widehat{g}(\xi,l^\prime)_{Hi}} Low(k,\eta-\xi,l-l^\prime) d\xi d\eta \\
& \lesssim \frac{\epsilon}{\jap{\nu t^3}^{\alpha}}\norm{A^1 Q^1_{\neq}}_2 \norm{Ag}_2, 
\end{align*} 
which is consistent with Proposition \ref{prop:Boot}. 

\subsubsection{Reaction term $\mathcal{T}_{R;\neq 0}$}
For this term we use a slight variant of the treatment found in \S\ref{sec:Q2TRneq0}.
Note that $Q^1_0$ is $O(t)$ larger than $Q^2_0$ but $A^1 \lesssim \jap{t}^{-1} A^2$, and hence the allowed growth in $A^1$ will balance the extra growth in these terms. 
Therefore, these can be treated in the same fashion as was done in \S\ref{sec:Q2TRneq0}. Hence, we omit the details for brevity.  
 
\subsubsection{Reaction term $\mathcal{T}_{R;\neq \neq}$} \label{sec:Q1TRneqneq}
This reaction term is slightly different than the analogous terms in \S\ref{sec:Q2TRneqneq} and \S\ref{sec:Q3TRneqneq}, as we are not allowing a norm imbalance like in \S\ref{sec:Q3TRneqneq} but instead are allowing a steady linear growth.
As in \S\ref{sec:Q2TRneqneq} above,  we further decompose in terms of frequency: 
\begin{align*} 
\mathcal{T}_{R;\neq \neq} & \lesssim \frac{\epsilon \jap{t}^{2+\delta_1}}{\jap{\nu t^3}^\alpha} \sum_{k,k^\prime}\int \mathbf{1}_{k k^\prime(k-k^\prime) \neq 0} \abs{A^{1} \hat{Q}^1_k(\eta,l) A^{1}_{k}(\eta,l) \hat{U}_{k^\prime}^1 (\xi,l^\prime)_{Hi}} Low(k-k^\prime, \eta-\xi, l - l^\prime) d\eta d\xi \\   
& \quad +  \frac{\epsilon \jap{t}^{3+\delta_1}}{\jap{\nu t^3}^\alpha} \sum_{k,k^\prime}\int \mathbf{1}_{k k^\prime(k-k^\prime) \neq 0} \abs{A^{1} \hat{Q}^1_k(\eta,l) A^{1}_{k}(\eta,l) \hat{U}_{k^\prime}^2 (\xi,l^\prime)_{Hi}} Low(k-k^\prime, \eta-\xi, l - l^\prime) d\eta d\xi \\  
& \quad +  \frac{\epsilon \jap{t}^{2+\delta_1}}{\jap{\nu t^3}^{\alpha-1}} \sum_{k,k^\prime}\int \mathbf{1}_{k k^\prime(k-k^\prime) \neq 0} \abs{A^{1} \hat{Q}^1_k(\eta,l) A^{1}_{k}(\eta,l) \hat{U}_{k^\prime}^3 (\xi,l^\prime)_{Hi}} Low(k-k^\prime, \eta-\xi, l - l^\prime) d\eta d\xi \\   
& \quad + \frac{\epsilon^2 \jap{t}^{2+\delta_1}}{\jap{\nu t^3}^{\alpha}} \sum_{k,k^\prime}\int \mathbf{1}_{k k^\prime(k-k^\prime) \neq 0} \abs{A^{1} \hat{Q}^1_k(\eta,l)} A^{1}_{k}(\eta,l) \left(\abs{\hat{\psi_y}(\xi,l^\prime)_{Hi}} + \abs{\hat{\phi_y}(\xi,l^\prime)_{Hi}} + \abs{\hat{\phi_z}(\xi,l^\prime)_{Hi}}\right) \\ & \quad\quad\quad \times Low(k-k^\prime, \eta-\xi, l - l^\prime) d\eta d\xi \\   
& \quad + \frac{\epsilon^2 \jap{t}^{3+\delta_1}}{\jap{\nu t^3}^{\alpha}} \sum_{k,k^\prime}\int \mathbf{1}_{k k^\prime(k-k^\prime) \neq 0} \abs{A^{1} \hat{Q}^1_k(\eta,l) A^{1}_{k}(\eta,l) \hat{\psi_z}(\xi,l^\prime)_{Hi} } Low(k-k^\prime, \eta-\xi, l - l^\prime) d\eta d\xi \\   
& \quad + \mathcal{T}_{R;\neq \neq;\mathcal{R}} \\ 
& = \mathcal{T}_{R;\neq\neq}^{1} + \mathcal{T}_{R;\neq\neq}^2 + \mathcal{T}_{R;\neq\neq}^3  + \mathcal{T}_{R;\neq\neq}^{C1} + \mathcal{T}_{R;\neq\neq}^{C2} + \mathcal{T}_{R;\neq\neq;\mathcal{R}}. 
\end{align*} 
Consider $\mathcal{T}^1_{R;\neq \neq}$. 
By \eqref{ineq:ABasic} followed by \eqref{ineq:AikDelLNoD} and \eqref{ineq:triQuadHL}, 
\begin{align*}
\mathcal{T}_{R;\neq\neq}^{1} & \lesssim \frac{\epsilon \jap{t}^{2+\delta_1}}{\jap{\nu t^3}^{\alpha}} \sum_{k,k^\prime}\int \mathbf{1}_{k k^\prime(k-k^\prime) \neq 0} \abs{A^{1} \hat{Q}^1_k(\eta,l)} \frac{1}{\abs{k^\prime}^2 + \abs{l^\prime}^2 + \abs{\xi - k^\prime t}^2}\abs{A^1 \Delta_L \hat{U}_{k^\prime}^1 (\xi,l^\prime)_{Hi}} \\ & \quad\quad \times Low(k-k^\prime, \eta-\xi, l - l^\prime) d\eta d\xi \\ 
& \lesssim \frac{\epsilon \jap{t}^{2+\delta_1}}{\jap{\nu t^3}^{\alpha}}\norm{\left(\sqrt{\frac{\partial_t w}{w}}\tilde{A}^1 + \frac{\abs{\grad}^{s/2}}{\jap{t}^{s}}A^1\right) Q^1}_2 \norm{\left(\sqrt{\frac{\partial_t w}{w}}\tilde{A}^1 + \frac{\abs{\grad}^{s/2}}{\jap{t}^{s}}A^1 \right) \Delta_L U^1_{\neq}}_2 \\ 
& \quad + \frac{\epsilon \jap{t}^{\delta_1}}{\jap{\nu t^3}^{\alpha}}\norm{A^1 Q^1}_2 \norm{A^1 \Delta_L U^1_{\neq}}_2, 
\end{align*} 
which by Lemmas \ref{lem:PEL_NLP120neq} and \ref{lem:SimplePEL}, is consistent with Proposition \ref{prop:Boot}. 
The term $\mathcal{T}^2_{R;\neq \neq}$ is treated in essentially the same way and is hence omitted (that $A^1 \lesssim \jap{t}^{-1}A^2$ recovers the additional power of $t$ from the low frequency factor in $\mathcal{T}^2_{R;\neq \neq}$). 

Next, turn to the treatment of $\mathcal{T}^3_{R;\neq \neq}$. 
By Lemma \ref{lem:ABasic} followed by \eqref{ineq:A1A3ReacGain} and \eqref{ineq:triQuadHL}, 
\begin{align*}
\mathcal{T}_{R;\neq\neq}^{3} & \lesssim \frac{\epsilon \jap{t}^{2+\delta_1}}{\jap{\nu t^3}^{\alpha}} \sum_{k,k^\prime}\int \mathbf{1}_{k k^\prime(k-k^\prime) \neq 0} \abs{A^{1} \hat{Q}^1_k(\eta,l)} \frac{1}{\abs{k^\prime}^2 + \abs{l^\prime}^2 + \abs{\xi - k^\prime t}^2} \\   
& \quad\quad \times  \left(\sum_{r}\chi^{r,NR}\frac{t}{\abs{r} + \abs{\eta-tr}} + \chi^{\ast;23} \right) \frac{1}{\jap{t}} \jap{\frac{t}{\jap{\xi,l^\prime}}}^{1-\delta_1} \\ & \quad\quad \times \abs{A^3 \Delta_L \hat{U}_{k^\prime}^3 (\xi,l^\prime)_{Hi}} Low(k-k^\prime, \eta-\xi, l - l^\prime) d\eta d\xi \\ 
& \lesssim \frac{\epsilon \jap{t}^{1 + \delta_1}}{\jap{\nu t^3}^{\alpha}} \norm{A^1 Q^1}_2 \norm{A^3 \Delta_L U^3_{\neq}}_2, 
\end{align*}
which is consistent by Lemma \ref{lem:SimplePEL}. 

The  coefficient and remainder terms can be treated in exactly the same manner as in \S\ref{sec:Q3TRneqneq} 
and are therefore omitted for the sake of brevity. 
This completes the treatment of $\mathcal{T}_{R;\neq\neq}$ and hence all of $\mathcal{T}$.

\subsection{Dissipation error terms $\mathcal{D}$}
These terms are treated in the same manner as the corresponding terms in $Q^3$, found in \S\ref{sec:DEneqQ3}.
The results are analogous to those found therein and are hence here omitted for brevity. 

This completes the high norm estimate on $Q^1_{\neq}$. 

\section{Coordinate system controls} 
In this section we prove the necessary controls on $C^i$ and the auxiliary unknown $g$.
  
\subsection{High norm estimate on $g$}
We will begin by improving \eqref{ineq:Boot_Ag}, which roughly measures the time-oscillations between $U_0^1$ and $C^1$, and hence measures the time-oscillations of the $y$ component of the shear. 
From \eqref{def:gPDE2} we have
\begin{align*} 
\frac{1}{2}\frac{d}{dt}\norm{Ag}_2^2 & = -CK^g_\lambda - CK^g_w - \frac{2}{t}\norm{Ag}_2^2 - \int Ag A\left(g \partial_Y g\right) dV \\ 
& \quad + \int Ag A\left(\tilde{\Delta_t}g\right) dV - \frac{1}{t}\sum_{k \neq 0}\int Ag A \left( U_{-k} \cdot \grad^t U^1_{k} \right) dV \\ 
& = -\mathcal{D}g + \mathcal{T} + \mathcal{D}_E + \mathcal{F}. 
\end{align*} 

\subsubsection{Transport nonlinearity} \label{sec:TransNong}
By Lemma \ref{lem:AAiProd} and \eqref{ineq:Boot_gLow}, we have 
\begin{align*}
\mathcal{T} & \lesssim \norm{Ag}_2^2 \norm{g}_{\G^{\lambda}} + \norm{Ag}_2\norm{g}_{\G^{\lambda}}\norm{\grad A g}_2 \lesssim \frac{\epsilon^{1/2}}{\jap{t}^{2}}\norm{Ag}_2^2 + \epsilon^{3/2}\norm{\grad A g}_2^2, 
\end{align*}
which is consistent with Proposition \ref{prop:Boot} for $\epsilon$ sufficiently small, 

\subsubsection{Dissipation error terms} \label{sec:DissErrg}
Recall that the dissipation error terms are of the form
\begin{align*} 
\mathcal{D}_E = \nu \int Ag A \left( G_{yy} \partial_Y^2 g + G_{yz}\partial_{YZ} g + G_{zz} \partial_{ZZ}g \right) dV. 
\end{align*}
We may treat these as in \cite{BGM15I} (for which we use essentially the same treatment as in \S\ref{sec:DEQ02}, despite the higher regularity of $A$). 
Using that approach we have,
\begin{align*}
\mathcal{D}_E \lesssim c_0\nu\norm{\grad A g}_2^2 + \nu\norm{Ag}_2 \norm{\grad AC}_2\norm{\grad Ag}_2 \lesssim c_0\nu\norm{\grad Ag}_2^2 + c_0^{-1} \epsilon^2 \nu \norm{\grad AC}_2^2,  
\end{align*}
which is consistent with Proposition \ref{prop:Boot} for $c_0$ sufficiently small. 

\subsubsection{Forcing from non-zero frequencies} \label{sec:g_Forcing}
Analogous to \eqref{eq:XavgCanc}, by the divergence-free condition we have,  
\begin{align*} 
\mathcal{F} & = -\sum_{k \neq 0}\frac{1}{t}\int A g A \left( \partial_Y^t \left( U^2_{-k} U^1_{k}\right) + \partial_Z^t \left( U^3_{-k} U^1_{k}\right) \right) dV  = F_Y + F_Z. 
\end{align*}  
Consider $F_Y$ first. 
Expand with a paraproduct and group terms where the coefficients appear in low frequency with the remainder:  
\begin{align*} 
F_Y & = - \sum_{k \neq 0}\frac{1}{t}\int A g  A \partial_Y \left( (U^2_{-k})_{Lo} (U^1_{k})_{Hi} \right) dV 
- \sum_{k \neq 0}\frac{1}{t}\int A g  A \partial_Y \left( (U^2_{-k})_{Hi} (U^1_{k})_{Lo} \right) dV  \\ 
& \quad - \sum_{k \neq 0}\frac{1}{t}\int A g  A \left( \left((\psi_y)_{Hi}\partial_Y + (\phi_y)_{Hi}\partial_Z \right) \left((U^2_{-k})_{Lo} (U^1_{k})_{Lo}\right) \right) dV  + F_{Y;\mathcal{R},C} \\ 
& = F_{Y;LH} + F_{Y;HL} + F_{Y;C} + F_{Y;\mathcal{R},C}. 
\end{align*}
By \eqref{ineq:AprioriUneq} and \eqref{ineq:ABasic} 
\begin{align*}
F_{Y;LH} & \lesssim \frac{\epsilon}{t\jap{t} \jap{\nu t^3}^{\alpha}} \sum_{k \neq 0} \sum_{l,l^\prime} \int \abs{A \hat{g}(\eta,l)} \frac{\jap{t} \abs{\eta} \jap{\eta,l}^2 \jap{\frac{t}{\jap{\eta,l}}}^{1+\delta_1}}{k^2 + (l^\prime)^2 + \abs{\xi-kt}^2} \\ & \quad\quad \times \abs{A^{1} \Delta_L \widehat{U^1_{k}}(\xi,l^\prime)_{Hi}} Low(-k,\eta-\xi,l-l^\prime) d\eta d\xi.
\end{align*}
By \eqref{ineq:AdelLij} and \eqref{ineq:triQuadHL}, we therefore have 
\begin{align*}
F_{Y;LH} & \lesssim \frac{\epsilon \jap{t}^{2}}{\jap{\nu t^3}^{\alpha}}
 \norm{\left(\sqrt{\frac{\partial_t w}{w}}\tilde{A} + \frac{\abs{\grad}^{s/2}}{\jap{t}^s}A \right) g}_2^2 +  \frac{\epsilon \jap{t}^{2}}{\jap{\nu t^3}^{\alpha}}\norm{\left(\sqrt{\frac{\partial_t w}{w}}\tilde{A}^1 + \frac{\abs{\grad}^{s/2}}{\jap{t}^s}A^1 \right) \Delta_L U^1_{\neq}}_2 \\ 
& \quad + \epsilon^{3/2}\norm{\sqrt{-\Delta_L}Ag}_2^2 + \frac{\epsilon^{1/2}}{\jap{t}^{2} \jap{\nu t^3}^{2\alpha}} \norm{A^{1} \Delta_L U^1_{\neq}}^2_2,
\end{align*}
which, by Lemmas \ref{lem:PEL_NLP120neq} and \ref{lem:SimplePEL}, is consistent with Proposition \ref{prop:Boot} for $\delta_1$  and $\epsilon$ sufficiently small.  

Turn next to $F_{Y;HL}$. Similar to $F_{Y;LH}$, we get from \eqref{ineq:AprioriUneq}, \eqref{ineq:ABasic}, \eqref{ineq:AdelLij}, and \eqref{ineq:triQuadHL} we get
\begin{align*} 
F_{Y;HL} & \lesssim \frac{\epsilon \jap{t}^{\delta_1}}{t\jap{\nu t^3}^{\alpha}} \sum_{k \neq 0}\sum_{l,l^\prime} \int \abs{ A \hat{g} \frac{\abs{\eta} \jap{\eta,l}^2 \jap{\frac{t}{\jap{\xi,l^\prime}}}}{k^2 + (l^\prime)^2 + \abs{\xi-kt}^2} A^{2} \Delta_L \widehat{U^2_{k}}(\xi,l^\prime)_{Hi}}  Low(-k,\eta-\xi,l-l^\prime) dV \\ 
& \lesssim \frac{\epsilon t^{2+\delta_1}}{\jap{\nu t^3}^\alpha}\norm{\left(\sqrt{\frac{\partial_t w}{w}}\tilde{A} + \frac{\abs{\grad}^{s/2}}{\jap{t}^s}A\right)  g}_2^2 + \frac{\epsilon t^{2+\delta_1}}{\jap{\nu t^3}^\alpha} \norm{\left(\sqrt{\frac{\partial_t w}{w}}\tilde{A}^2 + \frac{\abs{\grad}^{s/2}}{\jap{t}^s}A^2 \right)  \Delta_L U^2_{\neq}}_2^2 \\ & \quad + \frac{\epsilon }{t^{1-\delta_1}\jap{\nu t^3}^\alpha} \norm{\sqrt{-\Delta_L}Ag}_2 \norm{A^{2} \Delta_L  U^2_{\neq}}_2,   
\end{align*} 
which is consistent with Proposition \ref{prop:Boot} by Lemmas \ref{lem:SimplePEL} and \ref{lem:PEL_NLP120neq}. The remainder terms $F_{Y;\mathcal{R},C}$ are similar, but simpler, and and are hence omitted for brevity. 

Consider finally $F_{Y;C}$. 
By \eqref{ineq:AprioriUneq}, Lemma \ref{lem:ABasic}, and \eqref{ineq:triQuadHL} (and Lemma \ref{lem:CoefCtrl}),  
\begin{align*} 
F_{Y;C} & \lesssim \frac{\epsilon^2 }{t\jap{t}^{1-\delta_1}\jap{\nu t^3}^{2\alpha}}\sum_{l,l^\prime} \int\abs{A \hat{g}(\eta,l)}  A \left( \abs{\widehat{\psi_y}} + \abs{\widehat{\phi_z}}\right)(\xi,l^\prime)_{Hi} Low(-k,\eta-\xi,l-l^\prime) d\eta d\xi \\ 
& \lesssim  \frac{\epsilon^{1/2}}{t^{4-2\delta_1} \jap{\nu t^3}^{2\alpha}} \norm{Ag}_2^2 + \epsilon^{7/2}\norm{\grad A C}^2_2.
\end{align*}
which is consistent with Proposition \ref{prop:Boot} for $\epsilon$ sufficiently small. 
This completes the treatment of $F_Y$. 

Next turn to $F_Z$, which has additional complications due to the regularity imbalances implying $U^3$ has worse regularity than $U^2$ near the critical times. 
Expand with a paraproduct and as usual with terms in which the coefficients appear in low frequency included in the remainder:
\begin{align*} 
F_Z & = - \sum_{k \neq 0}\frac{1}{t}\int A g  A \partial_Z \left( (U^3_{-k})_{Lo} (U^1_{k})_{Hi} \right) dV - \sum_{k \neq 0}\frac{1}{t}\int A g  A \partial_Z \left( (U^3_{-k})_{Hi} (U^1_{k})_{Lo} \right) dV \\ 
& \quad - \sum_{k \neq 0}\frac{1}{t}\int A g  A\left( \left((\psi_z)_{Hi} \partial_Y + (\phi_z)_{Hi}\partial_Z\right) (U^3_{-k} U^1_{k})_{Lo} \right) dV  +  F_{Z;\mathcal{R},C} \\  
& = F_{Z;LH} + F_{Z;HL} + F_{Z;C} + F_{Z;\mathcal{R}}. 
\end{align*}   
Consider first $F_{Z;LH}$, which is similar to the analogous term above in $F_{Y}$. 
Indeed, by \eqref{ineq:AprioriUneq}, \eqref{ineq:ABasic}, \eqref{ineq:AdelLij}, and \eqref{ineq:triQuadHL}, 
\begin{align*} 
F_{Z;LH} & \lesssim \frac{\epsilon}{t\jap{\nu t^3}^{\alpha}}\sum_{k \neq 0} \sum_{l,l^\prime} \int \abs{A\hat{g}(\eta,l) \frac{\abs{l} \jap{\eta,l}^2 \jap{t} \jap{\frac{t}{\jap{\xi,l^\prime}}}^{1+\delta_1}}{k^2 + (l^\prime)^2 + \abs{\xi-kt}^2} A^{1} \Delta_L \widehat{U^1_{k}}(\xi,l^\prime)_{Hi}} Low(-k,\eta-\xi,l-l^\prime) d\eta d\xi \\ 
& \lesssim \frac{\epsilon t^2}{\jap{\nu t^3}^\alpha} \norm{\left(\sqrt{\frac{\partial_t w}{w}}\tilde{A} + \frac{\abs{\grad}^{s/2}}{\jap{t}^s}A\right) g}_2^2 +  \frac{\epsilon t^2}{\jap{\nu t^3}^\alpha} \norm{\left(\sqrt{\frac{\partial_t w}{w}}\tilde{A}^1 + \frac{\abs{\grad}^{s/2}}{\jap{t}^s} A^1 \right) \Delta_L U^1_{\neq}}_2^2  \\ 
& \quad + \frac{\epsilon }{\jap{\nu t^3}^\alpha} \norm{\sqrt{-\Delta_L}Ag}_2 \norm{A^{1} \Delta_L U^1_{\neq}}_2,  
\end{align*}
which is consistent with Proposition \ref{prop:Boot} for $\epsilon$ sufficiently small by Lemmas \ref{lem:PEL_NLP120neq} and \ref{lem:SimplePEL}. 
Turn next to $F_{Z;HL}$, which is complicated by the regularity imbalance in $A^3$. 
Indeed, by \eqref{ineq:AprioriUneq}, Lemma \ref{lem:ABasic}, followed by \eqref{ineq:AA3Z}, \eqref{ineq:AdelLij}, and \eqref{ineq:triQuadHL}, we have
\begin{align*} 
F_{Z;HL} & \lesssim \frac{\epsilon}{t^{1-\delta_1}\jap{\nu t^3}^{\alpha}}\sum_{k \neq 0} \sum_{l,l^\prime} \int \abs{\hat{g}(\eta,l)} \frac{\abs{l} \jap{\eta,l}^2 \jap{\frac{t}{\jap{\xi,l^\prime}}}^{2}}{k^2 + (l^\prime)^2 + \abs{\xi-kt}^2} \\ 
& \quad \quad \times \left(\sum_{r}\chi^{r,NR}\frac{t}{\abs{r} + \abs{\eta-tr}}\tilde{A}(\eta,l) \tilde{A}^3_k(\xi,l^\prime)  + \chi^{\ast;23} A(\eta,l) A^3_k(\xi,l^\prime) \right) \\ & \quad\quad \times \abs{\Delta_L \widehat{U^3_{k}}(\xi,l^\prime)_{Hi}} Low(-k,\eta-\xi,l-l^\prime) d\eta d\xi \\ 
& \lesssim \frac{\epsilon \jap{t}^{1+\delta_1}}{\jap{\nu t^3}^\alpha}\norm{\left(\sqrt{\frac{\partial_t w}{w}}\tilde{A} + \frac{\abs{\grad}^{s/2}}{\jap{t}^{s}}A\right) g}_2 \norm{\left(\sqrt{\frac{\partial_t w}{w}}\tilde{A}^3 + \frac{\abs{\grad}^{s/2}}{\jap{t}^{s}}A^3 \right) \Delta_L U^3_{\neq}}_2 \\ 
& \quad + \frac{\epsilon}{t^{1-\delta_1}\jap{\nu t^3}^{\alpha}} \norm{\sqrt{-\Delta_L}Ag}_2 \norm{A^3 \Delta_L U^3_{\neq}}_2, 
\end{align*}
 which by Lemmas \ref{lem:PEL_NLP120neq} and \ref{lem:SimplePEL} is consistent with Proposition \ref{prop:Boot} for $\epsilon$ sufficiently small. 
The coefficient and remainder terms can be treated as in $F_Y$; hence these are omitted for brevity. 
This completes the treatment of the forcing terms and hence of the entire high norm estimate on $g$.  

\subsection{Low norm estimate on $g$}
Computing the evolution of $\norm{g}_{\G^{\lambda,\gamma}}$ (denoting $A^S = e^{\lambda(t)\abs{\grad}^s} \jap{\grad}^{\gamma}$) from \eqref{def:gPDE2},
\begin{align} 
\frac{1}{2}\frac{d}{dt}\left( t^{4} \norm{g}_{\G^{\lambda,\gamma}}^2\right) & \leq \dot{\lambda}t^4 \norm{\abs{\grad}^{s/2}g}_{\G^{\lambda,\gamma}}^2 - t^{4}\int A^S g A^S\left(g \partial_Y g\right) dV \nonumber \\ 
& \quad + t^{4} \int A^S g A^S \left(\tilde{\Delta_t}g\right) dV - t^3\int A^S g A^S \left( U_{\neq} \cdot \grad^t U^1_{\neq} \right)_0 dV \nonumber \\ 
& = -CK_\lambda^{g,L} + \mathcal{T} + \mathcal{D} + \mathcal{F}. \label{eq:gLowEvo} 
\end{align} 
The treatment of the transport nonlinearity $\mathcal{T}$ and the dissipation error terms in $\mathcal{D}$ are essentially same as in the previous section (in fact easier), so are hence omitted. 
It remains to see why the forcing $\mathcal{F}$ can treated better at lower regularity. 
Following the treatments in the previous section and \S\ref{sec:NzeroForcing}, we can use the divergence free condition to write
\begin{align*} 
\mathcal{F} & = -t^{3} \int A^S g A^S \left( \partial_Y^t \left( U^2_{\neq} U^1_{\neq}\right)_0 + \partial_Z^t \left( U^3_{\neq} U^1_{\neq}\right)_0 \right) dV. 
\end{align*} 
The two terms can be treated  together. 
Indeed, by Lemmas \ref{lem:GevProdAlg}, Lemma \ref{lem:CoefCtrl}, the bootstrap hypotheses, as well as Lemma \ref{lem:LossyElliptic} and \eqref{ineq:AprioriUneq}, 
\begin{align*} 
\mathcal{F} & \lesssim t^{3} \norm{g}_{\G^{\lambda,\gamma}}(1 + \norm{C}_{\G^{\lambda,\gamma+1}})\left(\norm{U^2_{\neq}}_{\G^{\lambda,\gamma+1}}\norm{U^1_{\neq}}_{\G^{\lambda,3/2+}} + \norm{U^3_{\neq}}_{\G^{\lambda,3/2+}}\norm{U^1_{\neq}}_{\G^{\lambda,\gamma+1}} \right. \\ & \left. \quad\quad + \norm{U^2_{\neq}}_{\G^{\lambda,3/2+}}\norm{U^1_{\neq}}_{\G^{\lambda,\gamma+1}} + \norm{U^3_{\neq}}_{\G^{\lambda,\gamma+1}}\norm{U^1_{\neq}}_{\G^{\lambda,3/2+}} \right) \\ 
&  \lesssim t^{3} \norm{g}_{\G^{\lambda,\gamma}} \left(\frac{\epsilon^2 \jap{t}^{\delta_1}}{\jap{\nu t^3}^\alpha} \right) \\ 
& \lesssim \frac{\epsilon^{1/2} t^{4}}{\jap{\nu t^3}^{\alpha}} \norm{g}^2_{\G^{\lambda,\gamma}} + \frac{\epsilon^{7/2} t^{2+2\delta_1}}{\jap{\nu t^3}^{\alpha}}, 
\end{align*}
which, for $\delta_1$ and $\epsilon$ sufficiently small (and $\delta > 0$), is consistent  with Proposition \ref{prop:Boot}  

\subsection{Long time, high norm estimate on $C^i$} \label{sec:ACC}
Next, we improve \eqref{ineq:Boot_ACC}.  
Computing the evolution equation on $C^i$, \eqref{def:CReal}, we get 
\begin{align} 
\frac{1}{2}\frac{d}{dt}\norm{A C^i}_2^2 & = \dot{\lambda}\norm{\abs{\grad}^{s/2}A C^i}_2^2 - \norm{\sqrt{\frac{\partial_t w}{w}} \tilde{A} C^i}_2^2 + \nu \int A C^i  A\left(\tilde{\Delta_t} C^i \right) dV \nonumber  \\ & \quad -\int A C^i A\left( g\partial_Y C^i \right) dV + \mathcal{L}^i \nonumber \\ 
& = -\mathcal{D}C^i + \mathcal{D}_E  + \mathcal{T} + \mathcal{L}^i, \label{ineq:C1Evo} 
\end{align}
where 
\begin{align*}
\mathcal{D}_E & = \nu \int AC^i A\left( (\tilde{\Delta_t} - \Delta) C^i \right) dV.  
\end{align*}
and
\begin{subequations} 
\begin{align}
\mathcal{L}^1 & = \int A C^1  A g  dV - \int A C^1 A U_0^2  dV \label{def:L1LT} \\ 
\mathcal{L}^2 & = -\int A C^2 A U_0^3  dV. \label{def:L2LT} 
\end{align}
\end{subequations}

\subsubsection{Linear driving terms}
\paragraph{Treatment of $\mathcal{L}^1$} \label{sec:LD1_Longtime}
Consider the first term in \eqref{def:L1LT}.
For this it suffices to use 
\begin{align*}
\int A C^1  A g  dV & \leq \frac{\epsilon}{2 c_0} \norm{AC^1}_2^2 + \frac{c_0}{2\epsilon} \norm{Ag}_2^2, 
\end{align*}
which, for $K_{HC1} \gg 1$, is consistent with Proposition \ref{prop:Boot} (via integrating factors). 

Turn to the second term in \eqref{def:L1LT}. 
From Lemma \ref{lem:PELbasicZero} (for some $K$ depending on $s,\sigma$ and $\lambda$),  
 \begin{align*} 
-\int A C^1  A U_0^2  dV & \leq \frac{\epsilon}{2 c_0}\norm{A C^1}_2^2 + \frac{c_0}{2\epsilon}\norm{A U^2_0}_2^2 \\
& \leq \frac{\epsilon}{2 c_0}\norm{AC^1}_2^2 + K \epsilon \norm{A C}_2^2 +  \frac{Kc_0}{\epsilon}\norm{A^2 Q^2_0}_2^2 + \frac{Kc_0}{\epsilon}\norm{U_0^2}_2^2, 
\end{align*} 
which for $\epsilon$ and $c_0$ sufficiently small and $K_{HC1}$ sufficiently large, is consistent with Proposition \ref{prop:Boot} (again, via integrating factors). 

\paragraph{Treatment of $\mathcal{L}^2$} \label{sec:L2_Longer}
Now consider the case $i = 2$. 
The issue here is that we want to propagate higher regularity on $C^2$ than we have on $U_0^3$ due to the regularity imbalance in $A^3$. 
First we have the following, independently of $\kappa$ (see \eqref{def:wNR3}), 
\begin{align*}
\mathcal{L}^2 & \lesssim \sum \int \abs{\widehat{C^2}(\eta,l)}\left(\sum_{r}\mathbf{1}_{t \in \I_{r,\eta}} \frac{t}{\abs{r} + \abs{\eta-tr}}\tilde{A}(\eta,l)\tilde{A}^3_0(\eta,l) + \chi^{\ast} A(\eta,l) A^3(\eta,l) \right) \jap{\eta,l}^2 \abs{ \widehat{U^3_0}(\eta,l)} d\eta,  
\end{align*} 
where $\chi^\ast = 1 - \sum_{r \neq 0} \mathbf{1}_{t \in \I_{r,\eta}}$. 
Therefore, by Lemma \ref{dtw} and orthogonality, 
\begin{align*} 
\mathcal{L}^2 & \lesssim \frac{\jap{t}}{\kappa} \sum_{r \neq 0} \norm{\sqrt{\frac{\partial_t w}{w}} \mathbf{1}_{t \in \I_{r,\partial_Y}} \tilde{A} C^2}_2 \norm{\sqrt{\frac{\partial_t w}{w}} \mathbf{1}_{t \in \I_{r,\partial_Y}} \jap{\grad}^2 \tilde{A}^3 U_0^3}_2 + \norm{AC^2}_2 \norm{\jap{\grad}^2 A^3 U_0^3}_2 \\
& \lesssim \frac{1}{\kappa} \norm{\sqrt{\frac{\partial_t w}{w}} \tilde{A}C^2}_2^2 + \frac{\jap{t}^2}{\kappa} \norm{\sqrt{\frac{\partial_t w}{w}} \jap{\grad}^2 \tilde{A}^3 U_0^3}_2^2 + c_0^{-1} \epsilon\norm{AC^2}_2^2 + \frac{c_0}{\epsilon}  \norm{\jap{\grad}^2 A^3 U_0^3}_2^2.  
\end{align*}  
This is consistent with Proposition \ref{prop:Boot} for $K_{HC1} \gg K_{H3}$ (using $t \leq T_F < c_0\epsilon^{-1}$), $c_0$ and $\epsilon$ sufficiently small and $\kappa$ sufficiently large (the latter relative only to a universal constant independent of all other parameters). 

\subsubsection{Transport nonlinearity}\label{sec:TransNon_ACC}
By Lemma \ref{lem:AAiProd}, \eqref{ineq:Boot_gLow}, and \eqref{ineq:Boot_LowC}, 
\begin{align*}
\mathcal{T} & \lesssim \norm{AC^i}_{2}\left(\norm{AC^i}_{\G^{\lambda,\gamma}} \norm{Ag}_{2} + \norm{g}_{\G^{\lambda,\gamma}}\norm{\grad A C^i}_2\right) \\ 
& \lesssim  \left(\epsilon + \frac{\epsilon^{1/2}}{\jap{t}^{4}}\right) \norm{AC^i}_2^2 + \epsilon^{3/2} \norm{\grad A C^i}_2^2,    
\end{align*}
which is consistent with Proposition \ref{prop:Boot} for $\epsilon$ and $c_0$ sufficiently small.  

\subsubsection{Dissipation error terms} \label{sec:DissC}
For these terms, as in \cite{BGM15I}, we may use an easy variant of the treatment in \S\ref{sec:DissErrg}. 
We omit the details for brevity: 
\begin{align*} 
\mathcal{D}_E & \lesssim \nu \norm{AC}_{2} \norm{\grad A C^i}_2^2 + \nu \norm{A C^i}_{2} \norm{\grad A C}_{2} \norm{\grad C^i}_{\G^{\lambda,\gamma-1}} \lesssim c_{0} \nu \norm{\grad A C}_2^2, 
\end{align*} 
which is then absorbed by the dissipation by choosing $c_{0}$ sufficiently small.

\subsection{Shorter time, high norm estimate on $C^i$}
The improvement of \eqref{ineq:Boot_ACC2} is essentially the same as that of \eqref{ineq:Boot_ACC} with a few slight changes. 
From \eqref{def:CReal}, 
\begin{align} 
\frac{1}{2}\frac{d}{dt}\left(\jap{t}^{-2}\norm{A C^i}_2^2\right) & = -\frac{t}{\jap{t}^4}\norm{A C^i}_2^2 + \jap{t}^{-2}\dot{\lambda}\norm{\abs{\grad}^{s/2}A C^i}_2^2 - \jap{t}^{-2}\norm{\sqrt{\frac{\partial_t w}{w}}\tilde{A} C^i}_2^2 \nonumber \\ & \quad + \jap{t}^{-2}\nu \int A C^i  A \left(\tilde{\Delta_t} C^i\right) dv \nonumber 
-\jap{t}^{-2}\int A C^1 A\left( g \partial_Y C^i \right) dV + \mathcal{L}^i \nonumber \\
& = -CK_{L}^{C} - \jap{t}^{-2}\mathcal{D}C^i + \mathcal{D}_E + \mathcal{T} + \mathcal{L}^i, \label{ineq:CCEvo} 
\end{align}
where 
\begin{align*}
\mathcal{D}_E = \jap{t}^{-2} \int AC^i A\left((\tilde{\Delta_t} - \Delta)C^i\right) dV 
\end{align*}
and 
\begin{subequations} 
\begin{align}
\mathcal{L}^1 & = \jap{t}^{-2}\int A C^1  A g  dV - \jap{t}^{-2}\int A C^1 A U_0^2  dV \label{def:L1} \\ 
\mathcal{L}^2 & = - \jap{t}^{-2}\int A C^2 A U_0^3  dV. \label{def:L2} 
\end{align}
\end{subequations} 
The only real difference between the estimates \eqref{ineq:Boot_ACC2} versus \eqref{ineq:Boot_ACC} is in the linear driving terms $\mathcal{L}^i$. Hence, we omit the treatment of $\mathcal{T}$ and $\mathcal{D}_E$, as these can be treated in essentially the same manner as in the improvement of \eqref{ineq:Boot_ACC}.

\subsubsection{Linear driving terms}

\paragraph{Treatment of $\mathcal{L}^1$} \label{sec:LD1_shorttime}
Consider first the case $i =1$. 
By Cauchy-Schwarz, 
\begin{align*}
\jap{t}^{-2}\int A C^1  A g  dV & \leq \frac{t}{2\jap{t}^4}\norm{AC^1}_2^2 + \frac{1}{2 t}\norm{Ag}_2^2 \leq \frac{1}{2}CK_{L}^{C,1} + \frac{1}{4}CK_L^g. 
\end{align*}
Hence the first term is absorbed by the $CK_L^{C,1}$ term in \eqref{ineq:CCEvo} whereas the second term is controlled by \eqref{ineq:Boot_Ag} and hence this is consistent with Proposition \ref{prop:Boot} provided $K_{HC2}$ is sufficiently large.

Consider the second term in \eqref{def:L1}.
By a similar argument but now applying Lemma \ref{lem:PELbasicZero}, we have for some $K > 0$,  
\begin{align*}
-\jap{t}^{-2}\int A C^1  A U_0^2   dV & \leq  \frac{t}{10\jap{t}^4}\norm{A C^1 }_2^2 + \frac{5}{t}\norm{A U_0^2}_2^2 \\ 
& \leq  \frac{t}{10\jap{t}^4}\norm{A C^1}_2^2 + \frac{K}{\jap{t}}\norm{A^2 Q_0^2}_2^2 + \frac{K}{\jap{t}}\norm{U_0^2}_2^2 + \frac{K\epsilon^2}{\jap{t}}\norm{AC}_2^2. 
\end{align*}
Hence for $K_{HC2}$ sufficiently large relative to $K_{HC1}$, this is consistent with Proposition \ref{prop:Boot} for $c_0$ and $\epsilon$ sufficiently small.

\paragraph{Treatment of $\mathcal{L}^2$} 
As in \S\ref{sec:L2_Longer}, we have (again defining $\chi^\ast = 1 - \sum_{r \neq 0}\mathbf{1}_{t \in \I_{r,\eta}}$), 
\begin{align*}
-\jap{t}^{-2} \int A C^2  A U_0^3  dV & \lesssim  \jap{t}^{-2} \sum \int \abs{\widehat{C^2}(\eta,l)}\left(\sum_{r}\mathbf{1}_{t \in \I_{r,\eta}} \frac{t}{\abs{r} + \abs{\eta-tr}}\tilde{A}(\eta,l)\tilde{A}^3_0(\eta,l) + \chi^{\ast} A(\eta,l)A^3_0(\eta,l) \right) \\ & \quad\quad \times \jap{\eta,l}^2 \abs{ \widehat{U^3_0}(\eta,l)} d\eta \\  
& \lesssim \kappa^{-1}\jap{t}^{-1} \sum_{r \neq 0} \norm{\sqrt{\frac{\partial_t w}{w}} \mathbf{1}_{t \in \I_{r,\partial_Y}} \tilde{A} C^2}_2 \norm{\sqrt{\frac{\partial_t w}{w}} \mathbf{1}_{t \in \I_{r,\partial_Y}} \jap{\grad}^2 \tilde{A}^3 U_0^3}_2 \\ & \quad + \jap{t}^{-2}\norm{AC^2}_2 \norm{\jap{\grad}^2 A^3 U_0^3}_2 \\ 
& = T_1 + T_2. 
\end{align*}
To treat the first term we use orthogonality and Lemma \ref{lem:PELCKZero} to deduce the following (where $K$ is a universal constant depending only on $\lambda$ and $s$ and differs from line to line),  
\begin{align*}
T_1 & \leq \frac{1}{10}\jap{t}^{-2}\norm{\sqrt{\frac{\partial_t w}{w}} \tilde{A}C^2}^2_2 + K\norm{\sqrt{\frac{\partial_t w}{w}} \jap{\grad}^2 \tilde{A}^3 U_0^3}_2^2 \\ 
& \leq \frac{1}{10}\jap{t}^{-2}\norm{\sqrt{\frac{\partial_t w}{w}} \tilde{A}C^2}^2_2 + K\norm{\left(\sqrt{\frac{\partial_t w}{w}}\tilde{A}^3 + \frac{\abs{\grad}^{s/2}}{\jap{t}^{s}} A^3 \right)  Q_0^3}_2^2 \\ 
& \quad + \frac{K}{\jap{t}^{2s}}\norm{U_0^3}_2^2 + K\epsilon^2\norm{\left(\sqrt{\frac{\partial_t w}{w}}\tilde{A} + \frac{\abs{\grad}^{s/2}}{\jap{t}^{s}}A\right) C}_2^2,  
\end{align*}
which is consistent with Proposition \ref{prop:Boot} for $c_0$ and $\epsilon$ sufficiently small together with $K_{HC2} \gg K_{H3}$. 
Turn next to $T_2$, which is treated in the same manner as the second term in \eqref{def:L1} (where $K$ is a universal constant depending only on $\lambda$ and $s$ and differs from line to line), 
\begin{align*}
T_2  & \leq \frac{t}{10\jap{t}^4}\norm{AC^2}_2^2 + \frac{5}{2t}\norm{A^3 \jap{\grad}^2 U_0^3}_2^2 \\
& \leq \frac{t}{10\jap{t}^4}\norm{AC^2}_2^2 + \frac{K}{\jap{t}}\norm{A^3 Q_0^3}_2^2 + \frac{K}{\jap{t}}\norm{U_0^3}_2^2 + \frac{K\epsilon^2}{\jap{t}}\norm{AC}_2^2 \\ 
& \leq \frac{t}{10\jap{t}^4}\norm{AC}_2^2 + \frac{4K K_{H3}}{\jap{t}}\epsilon^2 + \frac{4 K_{HC1} K\epsilon^2 c_0^2}{\jap{t}},
\end{align*}
which is sufficient provided $c_0$ and $\epsilon$ are chosen small and $K_{HC1} \gg K_{H3}$.

\subsection{Low norm estimate on $C$} 
The improvement of \eqref{ineq:Boot_LowC} estimate is an easy variation of that applied to improve \eqref{ineq:Boot_ACC} and \eqref{ineq:Boot_ACC2} except 
one uses the super-solution method discussed in \S\ref{sec:Q1Hi1} used to improve \eqref{ineq:Boot_Q1Hi1}. 

\section{Enhanced dissipation estimates} \label{sec:ED}
In this section we improve the enhanced dissipation estimates \eqref{ineq:Boot_ED}. 
A recurring theme here will be the gain in $t$ from Lemma \ref{lem:AnuLossy} when $\partial_X$ derivatives are present, a kind of ``null'' structure. 

\subsection{Enhanced dissipation of $Q^3$} \label{sec:ED3}
We begin with $Q^3$.
Computing the time evolution of $\norm{A^{\nu;3}Q^3}_2$ we get
\begin{align} 
\frac{1}{2}\frac{d}{dt}\norm{A^{\nu;3} Q^3}_2^2 & \leq \dot{\lambda}\norm{\abs{\grad}^{s/2}A^{\nu;3} Q^3}_2^2
-\frac{2}{t}\norm{\mathbf{1}_{t > \jap{\grad_{Y,Z}}} \tilde{A}^{\nu;3} Q^3}_2^2 - \norm{\sqrt{\frac{\partial_t w_L}{w_L}}A^{\nu;3}Q^3}_2^2  + G^\nu \nonumber \\
& \quad -2 \int A^{\nu;3} Q^3 A^{\nu;3} \partial_{YX}^t U^3 dV + 2 \int A^{\nu;3} Q^3 A^{\nu;3} \partial_{ZX}^t U^2 dV \nonumber \\   
 & \quad + \nu \int A^{\nu;3} Q^{3} A^{\nu;3} \left(\tilde{\Delta_t} Q^3\right) dv -\int A^{\nu;3} Q^3 A^{\nu;3}\left( \tilde U \cdot \grad Q^3 \right) dV \nonumber \\ 
& \quad - \int A^{\nu;3} Q^3 A^{\nu;3} \left[Q^j \partial_j^t U^3 + 2\partial_i^t U^j \partial_{ij}^t U^3  - \partial_Z^t\left(\partial_i^t U^j \partial_j^t U^i\right) \right] dV \nonumber \\
& = -\mathcal{D}Q^{\nu;3} - CK_{L}^{\nu;3} + G^{\nu} \nonumber \\ & \quad + LS3 + LP3 + \mathcal{D}_E + \mathcal{T} + NLS1 + NLS2 + NLP,  \label{ineq:AnuEvo3}
\end{align}
where we write 
\begin{align*}
\mathcal{D}_E & = \nu \int A^{\nu;3} Q^3 A^{\nu;3}\left(\tilde{\Delta_t}Q^3 - \Delta_L Q^3\right) dV, 
\end{align*} 
and
\begin{align*}
G^\nu = \alpha \int A^{\nu;3} Q^3 \min\left(1, \frac{\jap{\nabla_{Y,Z}}^2}{t^2}\right) e^{\lambda(t)\abs{\grad}^s}\jap{\grad}^\beta\jap{D(t,\partial_Y)}^{\alpha-1} \frac{D(t,\partial_Y)}{\jap{D(t,\partial_Y)}} \partial_t D(t,\partial_Y) Q^3_{\neq} dV. 
\end{align*} 
First, we need to cancel the growing term $G^\nu$ in \eqref{ineq:AnuEvo3} using part of the dissipation term $\mathcal{D}$. 
As in \cite{BGM15I} (and essentially \cite{BMV14}), 
\begin{align*} 
G^{\nu} - \nu \norm{\sqrt{-\Delta_L} A^{\nu;3}Q^3}_2^2  & \leq \nu \sum_{k \neq 0} \sum_{l} \int \left(\frac{1}{8}t^2\mathbf{1}_{t \geq 2 \abs{\eta}} - \abs{k}^2 - \abs{l}^2 - \abs{\eta-kt}^2\right)   \abs{A^{\nu;3} \widehat{Q^3_k}(\eta,l)}^2 d\eta \\ 
& \leq -\frac{\nu}{8}\norm{\sqrt{-\Delta_L}A^{\nu;3} Q^{3}_{\neq}}_2^2. 
\end{align*}
Next we see how to control the remaining linear and nonlinear contributions.
 
\subsubsection{Linear stretching term $LS3$} 
First separate into two parts (to be sub-divided further), 
\begin{align*} 
LS3 & = -2\int A^{\nu;3} Q^3 A^{\nu;3}\partial_X(\partial_Y - t\partial_X) U^3  dV - 2\int A^{\nu;3} Q^3 A^{\nu;3} \partial_X\left(\psi_y(\partial_Y - t\partial_X) + \phi_{y}\partial_{Z}\right) U^3 dV \\ 
& = LS3^0 + LS3^{C}. 
\end{align*} 
Turn first to $LS3^C$. By \eqref{ineq:L2L2L1}, \eqref{ineq:AnuiDistri}, Lemma \ref{lem:AnuLossy}, and Lemma \ref{lem:CoefCtrl},  
\begin{align} 
LS3^{C} & \lesssim \norm{\sqrt{-\Delta_L} A^{\nu;3}Q^3}_2\norm{C}_{\G^{\lambda,\beta+3\alpha+4}} \norm{A^{\nu;3}U^3}_2 \nonumber \\ 
& \lesssim \norm{\sqrt{-\Delta_L} A^{\nu;3}Q^3}_2\norm{C}_{\G^{\lambda,\beta+3\alpha+4}}\frac{1}{\jap{t}^2}\left(\norm{A^{\nu;3}Q^3}_2 + \norm{A^3Q^3}_2\right) \nonumber \\ 
& \lesssim \epsilon^{3/2}\norm{\sqrt{-\Delta_L} A^{\nu;3}Q^3}^2_2 + \frac{\epsilon^{1/2}}{\jap{t}^2}\left(\norm{A^{\nu;3}Q^3}_2 + \norm{A^3Q^3}_2\right)^2, \label{ineq:ED_LS3C}
\end{align} 
which is consistent with Proposition \ref{prop:Boot} for $\epsilon$ sufficiently small. 

For $LS3^0$ we proceed similar to the  high norm estimate in \S\ref{sec:LS30}.  
As in \eqref{eq:LS30}, we expand $\Delta_L\Delta_t^{-1}$: 
\begin{align} 
LS3^0 & = -2\int A^{\nu;3} Q^3 A^{\nu;3}\partial_X(\partial_Y - t\partial_X) \Delta_{L}^{-1} \Delta_L \Delta_t^{-1}Q^3  dV \nonumber \\
& =  -2\int A^{\nu;3} Q^3 A^{\nu;3}\partial_X(\partial_Y - t\partial_X) \Delta_{L}^{-1} \left[Q^3 - G_{yy} (\partial_Y - t\partial_X)^2 U^3 - G_{yz} \partial_Z(\partial_Y - t\partial_X)U^3 \right. \nonumber \\ & \quad\quad \left. - G_{zz} \partial_{ZZ}U^3 - \Delta_tC^1 (\partial_Y - t\partial_X) U^3 - \Delta_t C^2 \partial_Z U^3\right]  dV \nonumber \\ 
& = LS3^{0;0} + \sum_{i =1}^5 LS3^{0;Ci} . \label{def:LS30nu} 
\end{align} 
The leading order term is treated as in \cite{BGM15I}, hence we omit the details and simply state the result; for some $K > 0$, 
\begin{align*} 
LS3^{0;0} & \leq CK^{\nu;3}_{L} + \frac{\delta_\lambda}{10\jap{t}^{3/2}}\norm{\abs{\grad}^{s/2} A^{\nu;3}Q^3}_2^2 +  \frac{K}{\delta_\lambda^{\frac{1}{2s-1}}\jap{t}^{3/2}}\norm{ A^{\nu;3}Q^3}_2^2 + \frac{K}{\jap{t}^2} \norm{A^3 Q^3_{\neq}}^2_2, 
\end{align*} 
which is consistent with Proposition \ref{prop:Boot} provided $K_{ED3}$ is sufficiently large relative to $K_{H3}$ and $\delta_\lambda$.  
 
Turn to the first error term in \eqref{def:LS30nu}, $LS3^{0;C1}$, which by \eqref{ineq:AnuHiLowSep} and $\beta + 3\alpha + 6 < \gamma$ is controlled via (using also Lemma \ref{lem:CoefCtrl}), 
\begin{align} 
LS3^{0;C1} & \leq 2\norm{A^{\nu;3} Q^3}_2 \norm{A^{\nu;3} \partial_X(\partial_Y - t\partial_X) \Delta_L^{-1} \left(G_{yy}(\partial_Y - t\partial_X)^2 U^3_{\neq}\right)}_2 \nonumber \\ 
& \lesssim \frac{1}{\jap{t}^5}\norm{A^{\nu;3} Q^3_{\neq}}_{2} \norm{G_{yy}}_{\G^{\lambda,\gamma-1}} \norm{\Delta_L U^3_{\neq}}_{\G^{\lambda,\gamma-1}} + \frac{1}{\jap{t}}\norm{A^{\nu;3} Q^3}_2 \norm{A^{\nu;3} \left(G_{yy}(\partial_Y - t\partial_X)^2 U^3_{\neq}\right)}_2 \nonumber \\ 
& \lesssim \frac{\epsilon}{\jap{t}^2}\norm{A^3 Q^3_{\neq}}_2 \norm{A^3 \Delta_L U^3_{\neq}}_2 + \frac{1}{\jap{t}}\norm{A^{\nu;3} Q^3}_2 \norm{A^{\nu;3} \left(G_{yy}(\partial_Y - t\partial_X)^2 U^3_{\neq}\right)}_2. \label{ineq:LS30C1_ED}
\end{align} 
The first term is controlled via Lemma \ref{lem:SimplePEL}. To control the second term we use \eqref{ineq:AnuiDistriDecay} and Lemma \ref{lem:CoefCtrl}, 
\begin{align*} 
\frac{1}{\jap{t}}\norm{A^{\nu;3} Q^3}_2 \norm{A^{\nu;3} \left(G_{yy}(\partial_Y - t\partial_X)^2 U^3_{\neq}\right)}_2 & \lesssim \frac{1}{\jap{t}}\norm{ A^{\nu;3} Q^3}_2 \norm{C}_{\G^{\lambda,\gamma}} \norm{A^{\nu;3}(\partial_Y - t\partial_X)^2 U_{\neq}^3}_2 \\ 
& \lesssim  \epsilon\norm{A^{\nu;3}Q^3}_2 \norm{A^{\nu;3}(\partial_Y - t\partial_X)^2 U_{\neq}^3}_2. 
\end{align*}
By \eqref{ineq:PEL_CKnuIII}, this is consistent with Proposition \ref{prop:Boot} for $c_0$ sufficiently small. 
All the other $LS3^{0;Ci}$ error terms are controlled similarly and are hence omitted. 

This completes the treatment of $LS3^0$. 

\subsubsection{Linear pressure term $LP3$} \label{ineq:LP3_ED} 
Begin by separating out the contribution of the coefficients, 
 \begin{align*} 
LP3 & = 2\int A^{\nu;3} Q^3 A^{\nu;3}\partial_X \partial_Z U_{\neq}^2  dV 
 + 2\int A^{\nu;3} Q^3 A^{\nu;3} \partial_X\left( \left((\psi_z)(\partial_Y - t\partial_X) + (\phi_{z})\partial_Z\right)U_{\neq}^2\right)  dV \\ 
& = LP3^0 + LP3^C.
\end{align*}                 
As in \cite{BGM15I}, Cauchy-Schwarz and \eqref{def:wL}, 
\begin{align*}
LP3^0 & \leq \frac{1}{2\kappa}\norm{\sqrt{\frac{\partial_t w_L}{w_L}} A^{\nu;3} Q^3_{\neq}}_2^2 + \frac{1}{2\kappa} \norm{\sqrt{\frac{\partial_t w_L}{w_L}} \Delta_L A^{\nu;3} U^2}_2^2, 
\end{align*}
which is consistent with Proposition \ref{prop:Boot} for $\kappa$ sufficiently large, $c_0$ sufficiently small, and $K_{ED3} \gg K_{ED2}$  by Lemma \ref{lem:AnuLossy_CKnu}. 

The coefficient error term, $LP3^C$, can be treated in the same manner as $LS3^C$ above in \eqref{ineq:ED_LS3C} and yields similar contributions. 
Hence we omit the treatment for brevity. 
This completes the treatment of the linear pressure term $LP3$. 

\subsubsection{Nonlinear pressure and stretching} \label{sec:NLPS_Q3ED}
Due to the regularity gap $\beta + 3\alpha +12 \leq \gamma$ and \eqref{ineq:AnuiDistri},  
the presence of the coefficients from the coordinate transform will not greatly impact the treatment of these terms. 
Moreover, Lemma \ref{lem:AnuLossy} shows there is not a significant difference between $\partial_Y - t\partial_X$ and $\partial_Z$ derivatives when making many estimates.   
Hence, for simplicity we will treat all $NLS$ and $NLP$ terms as if there were no variable coefficients. 
As in \cite{BGM15I}, we will enumerate the terms as follows for $i,j \in \set{1,2,3}$ and 
$a,b \in \set{0,\neq}$
\begin{subequations}  \label{def:Q3enumnu}
\begin{align}
NLP(i,j,a,b) & = \int A^{\nu;3} Q^3 A^{\nu;3} \partial_Z^t(\partial_j^t U^i_a \partial_i^t U^j_b  ) dV \\
NLS1(j,a,b) & = -\int A^{\nu;3} Q^3 A^{\nu;3} \left( Q^j_a \partial_j^t U^3_b  \right) dV \\
NLS2(i,j,a,b) & = -2\int A^{\nu;3} Q^3 A^{\nu;3} (\partial_i^t U^j_a \partial_i^t\partial_j^t U^3_b  ) dV.
\end{align}
\end{subequations}
We will use repeatedly the inequalities 
\begin{subequations} 
\begin{align}
A^{\nu; 3} & \lesssim t A^{\nu;1} \\ 
A^{\nu; 3} & \lesssim A^{\nu;2}. 
\end{align}
\end{subequations} 

\paragraph{Treatment of $NLP(i,j,0,\neq)$ terms} 
Recalling, \eqref{def:Q3enumnu}, note that by the usual null structure, we have $j \neq 1$. By \eqref{ineq:AnuiDistri} 
\begin{align*} 
NLP(i,j,\neq,0) & \lesssim \norm{A^{\nu;3}Q^3}_2 \norm{A^{\nu;3} \jap{\partial_{Z}} \partial^t_i U^j}_2 \norm{U_0^i}_{\G^{\lambda,\beta+3\alpha+5}}.  
\end{align*} 
From Lemma \ref{lem:AnuLossy}, we see that the loss of $t$ if $i=1$ on the third factor is balanced by a gain of $t$ on the second. 
On the other hand, if $i \neq 1$ then there is no loss of $t$ on the last factor 
but a loss of $t$ on the second. 
Therefore, after Lemma \ref{lem:AnuLossy} we get
\begin{align*} 
NLP(i,j,\neq,0) & \lesssim \epsilon \norm{A^{\nu;3}Q^3}_2\left(\norm{A^{j} Q^j_{\neq}}_2 + \norm{A^{\nu;j}Q^j}_2\right),
\end{align*} 
which is consistent with Proposition \ref{prop:Boot} for $c_0$ sufficiently small. 

\paragraph{Treatment of $NLS1(j,0,\neq)$ terms}\label{sec:NLS10neq_Q3ED} 
Next turn to the treatment of the $NLS1(j,0,\neq)$ terms (recalling \eqref{def:Q3enumnu}), which by \eqref{ineq:AnuiDistri} followed by \eqref{ineq:AnuLossyII} (noting a above that when $j = 1$, the loss of $t$ from the second factor is balanced by a gain of $t$ on the third factor),  
\begin{align*} 
NLS1(j,0,\neq) & \lesssim \norm{A^{\nu;3}Q^3}_2 \norm{Q^j_0}_{\G^{\lambda,\beta + 3\alpha + 4}}  \norm{A^{\nu;3}\partial_j^t U^3}_2 \lesssim \frac{\epsilon}{\jap{t}}\norm{A^{\nu;3}Q^3}_2\left(\norm{A^{\nu;3}Q^3}_2 + \norm{A^3 Q^3} \right) 
\end{align*} 
which is consistent with Proposition \ref{prop:Boot} for $c_0$ sufficiently small. 

\paragraph{Treatment of $NLS1(j,\neq,0)$ terms} \label{sec:ED3NSL1ijneq0}
Next turn to the treatment of the $NLS1(j,\neq,0)$ terms which by \eqref{ineq:AnuiDistri} followed by \eqref{ineq:AnuLossyII} (noting that $j \neq 1$), 
\begin{align*} 
NLS1(j,\neq,0) \lesssim \norm{A^{\nu;3}Q^3}_2 \norm{A^{\nu;3} Q^j}_2  \norm{U^3_0}_{\G^{\lambda,\beta + 3\alpha + 4}} \lesssim \epsilon \norm{A^{\nu;3}Q^3}_2 \norm{A^{\nu;j}Q^j}_2, 
\end{align*} 
which is consistent with Proposition \ref{prop:Boot} for $c_0$ sufficiently small. 

\paragraph{Treatment of $NLS2(i,j,\neq,0)$ terms} 
From \eqref{def:Q3enumnu} we see that that \emph{neither} $i$ nor $j$ can be one. 
Therefore, similar to \S\ref{sec:NLS10neq_Q3ED}, we get by \eqref{ineq:AnuiDistri}, 
\begin{align*}
NLS2(i,j,\neq,0) & \lesssim \frac{\epsilon}{\jap{t}}\norm{A^{\nu;3}Q^3}_2\left(\norm{A^{\nu;j}Q^j}_2 + \norm{A^{j}Q^j_{\neq}}_2 \right)
\end{align*}
which is consistent with Proposition \ref{prop:Boot} for $c_0$ sufficiently small. 

\paragraph{Treatment of $NLS2(i,j,0,\neq)$ terms}  
Next turn to the treatment of the $NLS1(i,j,\neq,0)$ terms,  
where now notice that $i$ cannot be one but $j$ can. 
However, if $j = 1$ then we will gain a power of $t$ on $\partial_X U^3_{\neq}$ using Lemma \ref{lem:AnuLossy}. 
Therefore, it follows from \eqref{ineq:AnuiDistri} and Lemma \ref{lem:AnuLossy} that, 
\begin{align*} 
NLS2(i,j,0,\neq) & \lesssim \epsilon \norm{A^{\nu;3}Q^3}_2\left(\norm{A^{3} Q^3_{\neq}}_2 + \norm{A^{\nu;3}Q^3}_2 \right). 
\end{align*} 

\paragraph{Treatment of $NLP(i,j,\neq,\neq)$} \label{sec:NLPneqneq_Q3ED}
Notice that we will lose a power of $t$ from $A^1$ if $j$ or $i$ is one, but in this case we would lose one less power of $t$ in Lemma \ref{lem:AnuLossy} due to the presence of $X$ derivatives. 
Hence regardless of the combination of $i$ and $j$, we will gain at least one power of $t$
Therefore, from \eqref{ineq:AnuiDistriDecay},  
\begin{align*}  
NLP(i,j,\neq,\neq) & \leq \norm{A^{\nu;3}Q^3}_2 \norm{A^{\nu;3}\partial_Z^t \left( \partial_i^t U^j_{\neq} \partial_j^t U^i_{\neq}\right)}_2 \\ 
& \lesssim \frac{t^2}{\jap{\nu t^3}^\alpha} \norm{A^{\nu;3}Q^3}_2 \left(\norm{A^{\nu;3}\partial^t_{Z} \partial_i^t U^j}_2\norm{ A^{\nu;3}\partial_{j}^t U^i}_2 + \norm{A^{\nu;3}\partial_{i}^t U^j}_2\norm{A^{\nu;3}\partial_{Z}^t \partial_j^t U^i}_2\right) \\ 
& \lesssim \frac{\epsilon^2 \jap{t}}{\jap{\nu t^3}^\alpha} \norm{A^{\nu;3}Q^3}_2 \lesssim \frac{\epsilon \jap{t}}{\jap{\nu t^3}^{\alpha}}\norm{A^{\nu;3}Q^3}_2^2 + \frac{\epsilon^3 \jap{t}}{\jap{\nu t^3}^{\alpha}},
\end{align*} 
which is consistent with Proposition \ref{prop:Boot} for $\epsilon$ sufficiently small. 

\paragraph{Treatment of $NLS1(j,\neq,\neq)$} \label{sec:NLS1neqneq_Q3ED}
These terms are all treated in essentially the same manner. 
Indeed, using as usual that $j = 1$ loses a power of $t$ from $A^{\nu;1}$ but gains a power from Lemma \ref{lem:AnuLossy}, we get from \eqref{ineq:AnuiDistriDecay} and \eqref{ineq:AnuLossyII}, 
\begin{align*}
NLS1(j,\neq,\neq) & \lesssim \frac{\jap{t}^2 }{\jap{\nu t^3}^{\alpha}} \norm{A^{\nu;3} Q^3}_2\norm{A^{\nu;3} Q^j}_2 \norm{A^{\nu;3} \partial_j^t U_{\neq}^3}_2 \\ 
& \lesssim \frac{\jap{t} }{\jap{\nu t^3}^{\alpha}} \norm{A^{\nu;3} Q^3}_2\norm{A^{\nu;j} Q^j}_2\left(\norm{A^{\nu;3} Q^3}_2 + \norm{A^{3} Q^3_{\neq}}_2 \right), 
\end{align*}
which is consistent with Proposition \ref{prop:Boot} for $\epsilon$ sufficiently small.

\paragraph{Treatment of $NLS2(i,j,\neq,\neq)$} \label{sec:NLS2neqneq_Q3ED}
The treatment of $NLS2$ is essentially the same as $NLP$, using again that the losses and gains balance regardless of the combination of $i$ and $j$, we get  from \eqref{ineq:AnuiDistriDecay} and Lemma \ref{lem:AnuLossy},   
\begin{align*} 
NLS2(i,j,\neq,\neq) 
& \lesssim \frac{t^2}{\jap{\nu t^3}^\alpha} \norm{A^{\nu;3}Q^3}_2 \norm{A^{\nu;3}\partial_{i} U^j}_2 \norm{ A^{\nu;3}\partial_{ij}^t U^3}_2 \\ 
& \lesssim \frac{\jap{t}}{\jap{\nu t^3}^\alpha} \norm{A^{\nu;3}Q^3}_2 \left(\norm{A^{\nu;3} Q^3}_2 + \norm{A^{3} Q^3_{\neq}}_2 \right)\left(\norm{A^{\nu;j} Q^j}_2 + \norm{A^{j} Q^j_{\neq}}_2 \right), 
\end{align*} 
which is consistent with Proposition \ref{prop:Boot} for $\epsilon$ sufficiently small. 

\subsubsection{Transport nonlinearity} \label{sec:Trans_ED_Q3}
Divide the transport nonlinearity:
\begin{align*} 
\mathcal{T} & = -\int A^{\nu;3}Q^3 A^{\nu;3}\left(g \partial_Y Q^3_{\neq}\right) dV - \int A^{\nu;3}Q^3 A^{\nu;3}\left(\tilde U_{\neq} \cdot \grad Q^3_{0}\right) dV -\sum \int A^{\nu;3}Q^3 A^{\nu;3}\left(\tilde U_{\neq} \cdot \grad Q^3_{\neq}\right) dV \\ 
& = \mathcal{T}_0 + \mathcal{T}_{\neq 0} + \mathcal{T}_{\neq \neq}
\end{align*}   
Consider first $\mathcal{T}_0$. 
By \eqref{ineq:AnuiDistri} and $\abs{\eta} \leq \abs{\eta-kt} + \abs{kt} \leq \jap{t}\left(\abs{\eta-kt} + \abs{k}\right)$, 
\begin{align*}  
\mathcal{T}_0 & \lesssim \norm{A^{\nu;3}Q^3}_2 \norm{g}_{\G^{\lambda,\gamma}}\jap{t} \norm{\sqrt{-\Delta_L} A^{\nu;3}Q^3}_2  \lesssim \frac{\epsilon^{1/2}}{\jap{t}^{2}}\norm{A^{\nu;3}Q^3}_2^2 + \epsilon^{3/2}\norm{\sqrt{-\Delta_L} A^{\nu;3}Q^3}_2^2. 
\end{align*}
where the last line followed from both \eqref{ineq:Boot_gLow} and \eqref{ineq:Boot_Ag}. 
Hence, for $\epsilon$ and $c_0$ sufficiently small, this is consistent with Proposition \ref{prop:Boot}. 

Turn next to $\mathcal{T}_{\neq 0}$, which reads 
\begin{align*}
\mathcal{T}_{\neq 0} & = \int A^{\nu;3}Q^3 A^{\nu;3}\left(\begin{pmatrix} (1 + \psi_y) U^2_{\neq} + \psi_zU^3_{\neq} \\ (1+\phi_{z})U^3_{\neq} + \phi_{y}U_{\neq}^2 \end{pmatrix}  \cdot \begin{pmatrix} \partial_{Y} Q_{0}^3 \\ \partial_Z Q_{0}^3 \end{pmatrix}\right) dV. 
\end{align*}
By \eqref{ineq:AnuiDistri}, Lemma \ref{lem:CoefCtrl}, and Lemma \ref{lem:AnuLossy}, we have 
\begin{align*}
\mathcal{T}_{\neq 0} & \lesssim \norm{A^{\nu;3}Q^3}_2\left(\norm{A^{\nu;3}U^2}_2 + \norm{A^{\nu;3}U^3}_2\right)\norm{\grad Q^3_0}_{\G^{\lambda,\gamma}} \\ 
& \lesssim \frac{\epsilon}{\jap{t}^2} \norm{A^{\nu;3}Q^3}_2\left(\norm{A^{\nu;2}Q^2}_2 + \norm{A^{2}Q^2_{\neq}}_2 + \norm{A^{\nu;3}Q^3}_2 + \norm{A^{3}Q^3_{\neq}}_2\right), 
\end{align*} 
which is consistent with Proposition \ref{prop:Boot}.

Turn next to $\mathcal{T}_{\neq \ne}$, which is the most subtle contribution. This is written 
\begin{align*} 
\mathcal{T}_{\neq \neq} & = \int A^{\nu;3}Q^3 A^{\nu;3}\left(\begin{pmatrix}U^1_{\neq} \\ (1 + \psi_y) U^2_{\neq} + \psi_zU^3_{\neq} \\ (1+\phi_{z})U^3_{\neq} + \phi_{y}U_{\neq}^2 \end{pmatrix}  \cdot \begin{pmatrix} \partial_X Q^3_{\neq} \\ (\partial_{Y} - t\partial_X)Q_{\neq}^3 \\ \partial_Z Q_{\neq}^3 \end{pmatrix}\right) dV. 
\end{align*}  
By Cauchy-Schwarz, \eqref{ineq:AnuiDistri}, Lemma \ref{lem:CoefCtrl} and \eqref{ineq:AnuiDistriDecay}, we get
\begin{align*} 
\mathcal{T}_{\neq \neq}  & \lesssim \norm{A^{\nu;3}Q^3}_2 \frac{\jap{t}^2}{\jap{\nu t^3}^{\alpha}} \left(\norm{A^{\nu;3}U^3}_2 + \norm{A^{\nu;3}U^2}_2\right) \norm{\sqrt{-\Delta_L} A^{\nu;3} Q^3}_{2} \\
& \quad + \norm{A^{\nu;3}Q^3}_2 \frac{\jap{t}^{2}}{\jap{\nu t^3}^\alpha}\left(\norm{\jap{\grad}^{2-\beta} A^{\nu;3}U^1}_2\norm{\sqrt{-\Delta_L} A^{\nu;3} Q^3}_{2} + \norm{A^{\nu;3}U^1}_2\norm{A^{\nu;3}Q^3}_2\right); 
\end{align*} 
note the extra precision applied to the treatment of $U^1$.  
By
\begin{align}
\frac{1}{\jap{\eta,l}} A^{\nu;3}_k(\eta,l) & \approx \frac{\jap{t}}{\jap{\eta,l} \jap{\frac{t}{\jap{\eta,l}}}^{1-\delta_1} } A^{\nu;1} \lesssim \jap{t}^{\delta_1}A^{\nu;1}, \label{ineq:A3nuA1nuRelation}
\end{align}
it follows that
\begin{align*}
\mathcal{T}_{\neq \neq}  & \lesssim \norm{A^{\nu;3}Q^3}_2 \frac{\jap{t}^2}{\jap{\nu t^3}^{\alpha}} \left(\norm{A^{\nu;3}U^3}_2 + \norm{A^{\nu;2}U^2}_2\right) \norm{\sqrt{-\Delta_L} A^{\nu;3} Q^3}_{2} \\
& \quad + \norm{A^{\nu;3}Q^3}_2 \frac{\jap{t}^{2+\delta_1}}{\jap{\nu t^3}^\alpha}\left(\norm{A^{\nu;1}U^1}_2\norm{\sqrt{-\Delta_L} A^{\nu;3} Q^3}_{2} + t\norm{A^{\nu;1}U^1}_2\norm{A^{\nu;3}Q^3}_2\right). 
\end{align*}  
Applying from \eqref{ineq:AnuHiLowSep2} to the $t\norm{A^{\nu;3}Q^3}_2$ in the last factor and Lemma \ref{lem:AnuLossy} to all factors (also \eqref{ineq:Boot_ED} with \eqref{ineq:Boot_Hi}) it follows that, 
\begin{align*} 
& \lesssim \norm{A^{\nu;3}Q^3}_2 \frac{1}{\jap{\nu t^3}^{\alpha}} \left(\norm{A^{\nu;3}Q^3}_2 + \norm{A^{3}Q^3_{\neq}}_2\right)  \norm{\sqrt{-\Delta_L} A^{\nu;3} Q^3}_{2} \\
& \quad + \norm{A^{\nu;3}Q^3}_2 \frac{1}{\jap{\nu t^3}^\alpha}\left(\norm{A^{\nu;2}Q^2}_2 + \norm{A^2 Q^2_{\neq}}_2\right) \norm{\sqrt{-\Delta_L} A^{\nu;3} Q^3}_{2} \\  
& \quad + \norm{A^{\nu;3}Q^3}_2 \frac{t^{\delta_1}}{\jap{\nu t^3}^\alpha}\left(\norm{A^{\nu;1}Q^1}_2 + \norm{A^1Q^1_{\neq}}_2\right)\left(\norm{\sqrt{-\Delta_L} A^{\nu;3} Q^3}_{2} + \norm{A^{3}Q^3_{\neq}}_2\right) \\
& \lesssim \frac{\epsilon t^{\delta_1}}{\jap{\nu t^3}^{\alpha}} \norm{A^{\nu;3}Q^3}_2\left(\norm{\sqrt{-\Delta_L} A^{\nu;3} Q^3}_{2} + \norm{A^{3}Q^3_{\neq}}_2\right) \\  
& \lesssim \epsilon^{3/2}\norm{\sqrt{-\Delta_L} A^{\nu;3} Q^3}_{2}^2 + \frac{\epsilon^{1/2} t^{2\delta_1}}{\jap{\nu t^3}^{\alpha}} \norm{A^{\nu;3}Q^3}^2_2 + \frac{\epsilon^{3/2}}{\jap{\nu t^3}^{\alpha}} \norm{A^3 Q^3_{\neq}}_2^2, 
\end{align*}
which is consistent with Proposition \ref{prop:Boot} for $\epsilon$, $\delta_1$, and $c_0$ sufficiently small (also $\delta > 0$). 

\subsubsection{Dissipation error terms} \label{sec:DE_ED_Q3}
The dissipation error terms are easily absorbed by the dissipation as in \cite{BMV14,BGM15I} using \eqref{ineq:AnuiDistri} together with the regularity gap between $A^{\nu;3}$ and the coefficient control in \eqref{ineq:Boot_LowC}. 
We hence omit the treatment for brevity.  

\subsection{Enhanced dissipation of $Q^2$} 
The enhanced dissipation of $Q^2$ is deduced in a manner very similar to $Q^3$, however, since we are imposing more control on $Q^2$, some nonlinear interactions must be handled with more precision. 

Computing the time evolution of $\norm{A^{\nu;2}Q^2}_2$ we get
\begin{align} 
\frac{1}{2}\frac{d}{dt}\norm{A^{\nu;2} Q^2}_2^2 & \leq \dot{\lambda}\norm{\abs{\grad}^{s/2}A^{\nu;2} Q^2}_2^2 -\frac{1}{t}\norm{\mathbf{1}_{t > \jap{\grad_{Y,Z}}} A^{\nu;2} Q^2}_2^2 - \norm{\sqrt{\frac{\partial_t w_L}{w_L}}A^{\nu;2}Q^2}_2^2 + G^\nu  \nonumber \\ %
 & \quad + \nu \int A^{\nu;2} Q^{2} A^{\nu;2} \left(\tilde{\Delta_t} Q^2\right) dV -\int A^{\nu;2} Q^2 A^{\nu;2}\left( \tilde U \cdot \grad Q^2 \right) dV \nonumber \\ 
& \quad - \int A^{\nu;2} Q^2 A^{\nu;2} \left[\left(Q^j \partial_j^t U^2\right) + 2\partial_i^t U^j \partial_{ij}^t U^2  - \partial_Y^t\left(\partial_i^t U^j \partial_j^t U^i\right) \right] dV \nonumber \\  
& = -\mathcal{D}Q^{\nu;2} - CK_{L}^{\nu;2}  + G^{\nu} + \mathcal{D}_E + \mathcal{T} + NLS1 + NLS2 + NLP,  \label{ineq:AnuEvo2}
\end{align} 
where as in \S\ref{sec:ED3}, we write 
\begin{align*}
\mathcal{D}_E & = \nu \int A^{\nu;2} Q^2 A^{\nu;2}\left(\tilde{\Delta_t}Q^2 - \Delta_L Q^2\right) dV, 
\end{align*} 
and
\begin{align*}
G^\nu = \alpha \int A^{\nu;2} Q^2 \min\left(1, \frac{\jap{\grad_{Y,Z}}}{t}\right) e^{\lambda(t)\abs{\grad}^s}\jap{\grad}^\beta\jap{D(t,\partial_v)}^{\alpha-1} \frac{D(t,\partial_v)}{\jap{D(t,\partial_v)}} \partial_t D(t,\partial_v) Q^2_{\neq} dV. 
\end{align*} 
As in \eqref{ineq:AnuEvo3} we have 
\begin{align*} 
-\nu \norm{\sqrt{-\Delta_L}A^{\nu;2} Q^{2}_{\neq}}_2^2 + G^{\nu}& \leq -\frac{\nu}{8}\norm{\sqrt{-\Delta_L}A^{\nu;2} Q^{2}}_2^2. 
\end{align*}

\subsubsection{Nonlinear pressure and stretching}
In this section we treat $NLS1$, $NLS2$ and $NLP$. 
As in \S\ref{sec:NLPS_Q3ED}, for simplicity we will treat all $NLS$ and $NLP$ terms as if there were no variable coefficients. 
We also recall the following enumeration from \cite{BGM15I}, for $i,j \in \set{1,2,3}$ and 
$a,b \in \set{0,\neq}$
\begin{subequations}  \label{def:Q2enumnu}
\begin{align}
NLP(i,j,a,b) & = \int A^{\nu;2} Q^2 A^{\nu;2} \partial_Y^t(\partial_j^t U^i_a \partial_i^t U^j_b  ) dV \\
NLS1(j,a,b) & = -\int A^{\nu;2} Q^2 A^{\nu;2} \left( Q^j_a \partial_j^t U^2_b  \right) dV \\
NLS2(i,j,a,b) & = -2\int A^{\nu;2} Q^2 A^{\nu;2} (\partial_i^t U^j_a \partial_i^t\partial_j^t U^2_b  ) dV.
\end{align}
\end{subequations}
We will use repeatedly the inequalities 
\begin{subequations} 
\begin{align}
A^{\nu; 2} & \lesssim t^{1+\delta_1}A^{\nu;1} \\ 
A^{\nu; 2} & \lesssim tA^{\nu;3}. 
\end{align}
\end{subequations}

\paragraph{Treatment of $NLP(i,j,0,\neq)$ terms} \label{sec:NLP0neq_Q2ED}
This includes terms identified in \S\ref{sec:Toy} as requiring that $Q^2$ grow linearly at low frequencies, and we will see that we will need this in order to estimate these terms. 
By \eqref{ineq:AnuiDistri}, 
\begin{align*}
NLP(i,j,0,\neq) & \lesssim \norm{A^{\nu;2} Q^2}_2 \norm{U_0^i}_{\G^{\lambda,\beta+3\alpha + 5}} \norm{A^{\nu;2} \jap{\partial_Y^t} \partial_i U^j}_2. 
\end{align*}
With \eqref{ineq:AnuLossyED} in mind, the power of $t$ lost from the derivatives or $j=1$ together is at most two and the powers of $t$ lost from the possibility that $j=3$ is also at most an additional one (also note $j \neq 1$), so at worst we get from Lemma \ref{lem:AnuLossy} (which recovers the powers of time), 
and \eqref{ineq:AprioriU0}, 
\begin{align*}
NLP(i,j,0,\neq) & \lesssim \epsilon\norm{A^{\nu;2} Q^2}_2 \left( \norm{\sqrt{-\Delta_L} A^{\nu;j} Q^j}_2 + \norm{\sqrt{-\Delta_L} A^{j} Q^j_{\neq}}_2\right) \\ 
& \lesssim \epsilon^{1/2}\norm{A^{\nu;2} Q^2}_2^{2} + \epsilon^{3/2}\left( \norm{\sqrt{-\Delta_L} A^{\nu;j} Q^j}^2_2 + \norm{\sqrt{-\Delta_L} A^{j} Q^j_{\neq}}^2_2\right). 
\end{align*}
For $\epsilon$ sufficiently small this is consistent with Proposition \ref{prop:Boot} for times until $t \sim \epsilon^{-1/2+\delta/100}$. 
At this point we can apply \eqref{ineq:AnuHiLowSep2} to the first term and deduce 
\begin{align}
NLP(i,j,0,\neq) & \lesssim \frac{\epsilon^{1/2}}{\jap{t}^2}\norm{A^{2} Q^2_{\neq}}_2^{2} + \frac{\epsilon}{\jap{t}^2}\norm{\sqrt{-\Delta_L} A^{\nu;2} Q^2}^2_2  + \epsilon^{3/2}\left( \norm{\sqrt{-\Delta_L} A^{\nu;j} Q^j}_2 + \norm{\sqrt{-\Delta_L} A^{j} Q^j_{\neq}}_2\right) \nonumber \\ 
& \lesssim  \frac{\epsilon^{1/2}}{\jap{t}^2}\norm{A^{2} Q^2_{\neq}}_2^{2} + \epsilon^{3/2-\delta/50} \norm{\sqrt{-\Delta_L} A^{\nu;2} Q^2}^2_2 \nonumber \\ & \quad  + \epsilon^{3/2}\left( \norm{\sqrt{-\Delta_L} A^{\nu;j} Q^j}_2 + \norm{\sqrt{-\Delta_L} A^{j} Q^j_{\neq}}_2\right), \label{ineq:NLPij0neq_Q2ED}
\end{align} 
which is consistent with Proposition \ref{prop:Boot} for all time for $\epsilon$ sufficiently small.  

\paragraph{Treatment of $NLS1(j,0,\neq)$ terms} \label{sec:NLS10neq_Q2ED}
These terms are straightforward by \eqref{ineq:AnuiDistri}, \eqref{ineq:Boot_Hi}, and \eqref{ineq:AnuLossyII}; we omit the details and conclude
\begin{align*}
NLS1(j,0,\neq) & \lesssim \frac{\epsilon}{\jap{t}}\norm{A^{\nu;2}Q^2}_2\left(\norm{A^{\nu;2}Q^2}_2 + \norm{A^{2}Q^2_{\neq}}_2\right). 
\end{align*}

\paragraph{Treatment of $NLS1(j,\neq,0)$ terms} \label{sec:NLS1neq0_Q2ED}
Due to the nonlinear structure, $j \neq 1$. Hence, the worst possibility is $j = 3$, where at most one power of time is lost -- notice that this also depends on the linear growth at low frequencies of $Q^2$. Hence, this term emphasizes this important difference with \cite{BGM15I}.  
Hence, by \eqref{ineq:AnuiDistri}, \eqref{ineq:AprioriU0}, and \eqref{ineq:AnuHiLowSep2},
\begin{align*}
NLS1(j,\neq,0) & \lesssim \epsilon\left(\norm{\sqrt{-\Delta_L} A^{\nu;2}Q^2}_2 + \norm{A^{2}Q^2_{\neq}}_2\right)\norm{A^{\nu;j}Q^{j}}_2 \\ 
& \lesssim \epsilon^{3/2}\left(\norm{\sqrt{-\Delta_L} A^{\nu;2}Q^2}_2^2 + \norm{A^{2}Q^2_{\neq}}_2^2\right) + \epsilon^{1/2}\norm{A^{\nu;j}Q^{j}}_2^2.      
\end{align*}
By applying \eqref{ineq:AnuHiLowSep2} for $t \gtrsim \epsilon^{-1/2+\delta/100}$ as in \eqref{ineq:NLPij0neq_Q2ED}, this is consistent with Proposition \ref{prop:Boot} for $\epsilon$ sufficiently small. 

\paragraph{Treatment of $NLS2(i,j,0,\neq)$ terms} 
These are treated similar to the analogous $NLS1$ terms in \S\ref{sec:NLS10neq_Q2ED}, yielding the following 
\begin{align*} 
NLS2(i,j,0,\neq)  \lesssim \epsilon\norm{A^{\nu;2}Q^2}\left(\norm{A^{\nu;2} Q^2} + \norm{A^2 Q^2_{\neq}}_2\right), 
\end{align*} 
which is consistent with Proposition \ref{prop:Boot} for $c_{0}$ sufficiently small. 

\paragraph{Treatment of $NLS2(i,j,\neq,0)$ terms} 
Again, due to the nonlinear structure, $j \neq 1$ and $i \neq 1$.  
By \eqref{ineq:AnuiDistri}, 
\begin{align*} 
NLS2(i,j,\neq,0) & \lesssim \norm{A^{\nu;2}Q^2}_2 \norm{\partial_i^t A^{\nu ;2} U^j_{\neq}}_2 \norm{U_0^2}_{\G^{\lambda,\beta+3\alpha + 7}}.
\end{align*} 
The worst case is $j=3$ and $i = 2$, however, even in this case Lemma \ref{lem:AnuLossy} recovers all of the time losses due to the permitted linear growth in $Q^2$ (also applying \eqref{ineq:AprioriU0}): 
\begin{align*} 
NLS2(i,j,\neq,0) & \lesssim \epsilon\norm{A^{\nu;2}Q^2}_2\left(\norm{A^{\nu;j}Q^j}_2 + \norm{A^{j}Q_{\neq}^j}_2 \right), 
\end{align*} 
which is consistent with Proposition \ref{prop:Boot} for $c_{0}$ sufficiently small. 

\paragraph{Treatment of $NLP(i,j,\neq,\neq)$} \label{sec:NLPneqneq_Q2ED}
Turn next to the nonlinear pressure interactions of two non-zero frequencies, which requires a careful treatment. 
First, observe that the case $i = j = 2$ cancels with the $NLS2$ term. 
By \eqref{ineq:AnuiDistriDecay},
\begin{align*} 
NLP(i,j,\neq,\neq) & \lesssim \norm{A^{\nu;2}Q^2}_2 \frac{\jap{t}}{\jap{\nu t^3}^\alpha} \left(\norm{\jap{\grad}^{2-\beta}A^{\nu;2} \partial_j^t U^i_{\neq}}_2 \norm{A^{\nu;2}\partial_Y^t \partial_i^tU^j_{\neq}}_2 \right.\\
 & \quad \left. + \norm{\jap{\grad}^{2-\beta} A^{\nu;2} \partial_Y^t \partial_j^t U^i_{\neq}}_2 \norm{A^{\nu;2}\partial_i^tU^j_{\neq}}_2 + \norm{A^{\nu;2} \partial_j^t U^i_{\neq}}_2 \norm{\jap{\grad}^{2-\beta} A^{\nu;2}\partial_Y^t \partial_i^tU^j_{\neq}}_2 \right.\\
 & \quad \left. + \norm{A^{\nu;2} \partial_Y^t \partial_j^t U^i_{\neq}}_2 \norm{\jap{\grad}^{2-\beta} A^{\nu;2}\partial_i^tU^j_{\neq}}_2\right). 
\end{align*} 
Each combination of $i$ and $j$ can be treated in a rather similar manner, each time using \eqref{ineq:AnuiDistri} and Lemma \ref{lem:AnuLossy}.  
As could be expected, $NLP(1,3,\neq,\neq)$ and $NLP(3,3,\neq,\neq)$ turn out to be the hardest.
Let us focus on the case $NLP(3,3,\neq,\neq)$ and omit the easier cases for brevity. 
Note that the inverse derivatives can recover losses associated with $\partial_Z$ but not $\partial_Y - t\partial_X$. 
They will also still work when considering $\partial_Z^t = (1 + \phi_z)\partial_Z + \psi_z(\partial_Y - t\partial_X)$, since it will introduce $O(\epsilon t^2)$ powers that are absorbed using $\epsilon t^2 \jap{\nu t^3}^{-1} \lesssim 1$. Hence, we can continue to ignore the coefficients. 
By Lemma \ref{lem:AnuLossy} and \eqref{ineq:Boot_ED} there holds,  
\begin{align*}
NLP(3,3,\neq,\neq) & \lesssim  \norm{A^{\nu;2}Q^2}_2 \frac{\jap{t}^{3}}{\jap{\nu t^3}^\alpha} \left(\norm{\jap{\grad}^{2-\beta}A^{\nu;3} \partial_Z^t U^3_{\neq}}_2 \norm{A^{\nu;3}\partial_Z^t \partial_Z^tU^3_{\neq}}_2 \right.\\
 & \quad \left. + \norm{\jap{\grad}^{2-\beta} A^{\nu;3} \partial_Z^t \partial_Z^t U^3_{\neq}}_2 \norm{A^{\nu;3}\partial_Z^tU^3_{\neq}}_2\right) \\  
& \lesssim  \frac{\epsilon\jap{t}}{\jap{\nu t^3}^\alpha}\left(1 + \epsilon t^2\right)\norm{A^{\nu;2}Q^2}_2\left(\norm{A^{\nu;3} Q^3}_2 + \norm{A^{3} Q^3_{\neq}}_2\right), 
\end{align*}
which is consistent with Proposition \ref{prop:Boot} for $\epsilon$ sufficiently small.
The other terms can be treated with a simple variation or easier arguments and are hence omitted.  

\paragraph{Treatment of $NLS1(j,\neq,\neq)$} \label{sec:NLS1neqneq_Q2ED}
By \eqref{ineq:AnuiDistriDecay}, \eqref{ineq:AnuLossyII}, and \eqref{ineq:Boot_ED}, we have the following (e.g. consider the worst case of $j = 3$),  
\begin{align*}
NLS1(j,\neq,\neq) & \lesssim \frac{\jap{t}}{\jap{\nu t^3}^{\alpha}}\norm{A^{\nu;2}Q^2}_2\norm{A^{\nu;2} Q^j}_2 \norm{A^{\nu;2} \partial_j^t U^2_{\neq}}_2 \\ 
& \lesssim \frac{\epsilon \jap{t}}{\jap{\nu t^3}^{\alpha}}\norm{A^{\nu;2}Q^2}_2\left(\norm{A^{\nu;2} Q^2}_2 + \norm{A^{2} Q^2_{\neq}}_2\right),
\end{align*}
which is consistent with Proposition \ref{prop:Boot}. 

\paragraph{Treatment of $NLS2(i,j,\neq,\neq)$} \label{sec:NLS2neqneq_Q2ED}
First, note that the $i = j = 2$ term cancels with $NLP$. 
For the remaining terms we again apply \eqref{ineq:AnuiDistriDecay} to deduce 
\begin{align*} 
NLS2(i,j,\neq,\neq)  & \lesssim \norm{A^{\nu;2}Q^2}_2 \frac{\jap{t}}{\jap{\nu t^3}^\alpha} \norm{A^{\nu;2} \partial_i^t U^j_{\neq}}_2 \norm{A^{\nu;2}\partial_{ij}^tU^2_{\neq}}_2. 
\end{align*} 
The most problematic term is $j = 3$, $i = 2$; however by \eqref{ineq:AnuLossyII} and \eqref{ineq:Boot_ED},   
\begin{align*} 
NLS2(2,3,\neq,\neq) & \lesssim \frac{\epsilon \jap{t}}{\jap{\nu t^3}^\alpha} \norm{ A^{\nu;2}Q^2}_2 \left(\norm{A^{\nu;2} Q^2}_2 + \norm{A^{2} Q^2_{\neq}}_2 \right),  
\end{align*} 
which is consistent with Proposition \ref{prop:Boot} for $\epsilon$ sufficiently small. 
The other cases can be treated similarly and are hence omitted for brevity. 
This completes the treatment of all of the nonlinear pressure and stretching terms. 

\subsubsection{Transport nonlinearity} \label{sec:Trans_ED_Q2}
These terms are easier than the analogous terms in \S\ref{sec:Trans_ED_Q3}. As noted in \cite{BGM15I}, this is consistent with the observation that the so-called ``reaction'' terms are stronger in $Q^3$ than $Q^2$ (note that $Q^3$ reaction terms are included in the toy model in \S\ref{sec:Toy} but the $Q^2$ reaction terms are not; see \cite{BGM15I} for more information).
 Write the transport nonlinearity as 
\begin{align*} 
\mathcal{T} & = -\int A^{\nu;2}Q^2 A^{\nu;2}\left(g \partial_Y Q^2_{\neq}\right) dV -  \int A^{\nu;2}Q^2  A^{\nu;2} \left(\tilde U_{\neq} \cdot \grad Q^2_{0}\right) dV \\  
& \quad - \int A^{\nu;2}Q^2  A^{\nu;2} \left(\tilde U_{\neq} \cdot \grad Q^2_{\neq}\right) dV \\  
& = \mathcal{T}_0 + \mathcal{T}_{\neq 0} + \mathcal{T}_{\neq \neq}. 
\end{align*}  
As in \S\ref{sec:Trans_ED_Q3}, we have 
\begin{align*} 
\mathcal{T}_0 & \lesssim \frac{\epsilon^{1/2}}{\jap{t}^{2}}\norm{A^{\nu;2}Q^2}_2^2 + \epsilon^{3/2}\norm{\sqrt{-\Delta_L} A^{\nu;2}Q^2}_2^2.
\end{align*} 
Similarly, we can treat $\mathcal{T}_{\neq 0}$ as we did in \S\ref{sec:Trans_ED_Q3}: \eqref{ineq:AnuiDistri}, Lemma \ref{lem:CoefCtrl}, and Lemma \ref{lem:AnuLossy}, we have 
\begin{align*}
\mathcal{T}_{\neq 0} & \lesssim \norm{A^{\nu;2}Q^2}_2\left(\norm{A^{\nu;2}U^2}_2 + \norm{A^{\nu;2}U^3}_2\right)\norm{\grad Q^2_0}_{\G^{\lambda,\gamma}} \\ 
& \lesssim \frac{\epsilon}{\jap{t}^2} \norm{A^{\nu;3}Q^3}_2\left(\norm{A^{\nu;2}Q^2}_2 + \norm{A^{2}Q^2_{\neq}}_2 + \jap{t}\left(\norm{A^{\nu;3}Q^3}_2 + \norm{A^{3}Q^3_{\neq}}_2\right)\right), 
\end{align*} 
which is consistent with Proposition \ref{prop:Boot}.

For $\mathcal{T}_{\neq \neq}$, we get from \eqref{ineq:AnuiDistriDecay},
\begin{align*} 
\mathcal{T}_{\neq \neq}  & \lesssim \norm{A^{\nu;2}Q^2}_2 \frac{\jap{t}}{\jap{\nu t^3}^\alpha} \left(\norm{A^{\nu;2}U^1}_2 + \norm{A^{\nu;2}U^2}_2 + \norm{A^{\nu;2}U^3}_2\right)\norm{\sqrt{-\Delta_L} A^{\nu;2} Q^2}_{2} \\
& \lesssim \norm{A^{\nu;2}Q^2}_2 \frac{\jap{t}}{\jap{\nu t^3}^\alpha} \left(\jap{t}^{1+\delta_1}\norm{A^{\nu;1}U^1}_2 + \norm{A^{\nu;2}U^2}_2 + \jap{t}\norm{A^{\nu;3}U^3}_2\right)\norm{\sqrt{-\Delta_L} A^{\nu;2} Q^2}_{2} \\
& \lesssim \norm{A^{\nu;2}Q^2}_2 \frac{\epsilon \jap{t}^{\delta_1}}{\jap{\nu t^3}^\alpha}\norm{\sqrt{-\Delta_L} A^{\nu;2} Q^2}_{2} \\
& \lesssim \epsilon^{3/2} \norm{\sqrt{-\Delta_L} A^{\nu;2} Q^2}_{2}^2 + \frac{\epsilon^{1/2} \jap{t}^{2\delta_1}}{\jap{\nu t^3}^{2\alpha}} \norm{A^{\nu;2}Q^2}^2_2,
\end{align*} 
which completes the treatment of $\mathcal{T}_{\neq \neq}$. 

\subsubsection{Dissipation error terms}
As in \S\ref{sec:DE_ED_Q3}, these terms are treated in the same manner as the analogous terms in \cite{BMV14,BGM15I}; the details are omitted for brevity. 

\subsection{Enhanced dissipation of $Q^1$}
Computing the time evolution of $\norm{A^{\nu;1}Q^1}_2$, we get
\begin{align} 
\frac{1}{2}\frac{d}{dt} \norm{A^{\nu;1} Q^1}_2^2 & \leq \dot{\lambda}\norm{\abs{\grad}^{s/2}A^{\nu;1} Q^1}_2^2  + G^\nu -\norm{\sqrt{\frac{\partial_t w_L}{w_L}} A^{\nu;1} Q^1}_2^2 \nonumber \\ 
& \quad - \frac{t}{\jap{t}^{2}}\norm{A^{\nu;1}Q^1}_2^2 -\frac{(1+\delta_1)}{t} \norm{\mathbf{1}_{t > \jap{\grad_{Y,Z}}} A^{\nu;1} Q^1}_2^2 \nonumber \\
& \quad - \int A^{\nu;1}Q^1 A^{\nu;1} Q^2 dV -2 \int A^{\nu;1} Q^1 A^{\nu;1} \partial_{YX}^t U^1 dV \nonumber \\
 & \quad + 2 \int A^{\nu;1} Q^1 A^{\nu;1} \partial_{XX} U^2 dV  + \nu\int A^{\nu;1} Q^{1} A^{\nu;1} \left(\tilde{\Delta_t} Q^1\right) dv \nonumber \\ 
& \quad -\int A^{\nu;1} Q^1 A^{\nu;1}\left( \tilde U \cdot \grad Q^2 \right) dv \nonumber \\ 
& \quad -\int A^{\nu;1} Q^1 A^{\nu;1} \left[\left(Q^j \partial_j^t U^1\right) + 2\partial_i^t U^j \partial_{ij}^t U^1  - \partial_X\left(\partial_i^t U^j \partial_j^t U^i\right) \right] dv \nonumber \\ 
& = -\mathcal{D}Q^{\nu;1} + G^\nu - CK_{L1}^{\nu;1} - (1+\delta_1) CK_{L2}^{\nu;1} \nonumber \\ & \quad + LU + LS1 + LP1  + \mathcal{D}_E + \mathcal{T} + NLS1 + NLS2 + NLP.  \label{ineq:AnuEvo1}
\end{align} 
where $G^\nu$ is analogous to the corresponding term in \eqref{ineq:AnuEvo3}. 
As in \S\ref{sec:ED3}, $G^\nu$ is absorbed by using the dissipation. 
Note that for $i \in \set{2,3}$, 
\begin{align}
A^{\nu; 1} & \lesssim A^{\nu;i}. 
\end{align}

\subsubsection{Linear terms} 
The treatment of $LU$ and $LS1$ can be made analogous to the linear terms treated in 
\S\ref{sec:ED3} combined with the $t^{\delta_1}$ tweak introduced for the improvement of \eqref{ineq:Boot_Q1Hi2} in \S\ref{sec:LUQhi2}. 
We omit the details for brevity and conclude for some $K > 0$, 
\begin{align*} 
LU & \leq \delta_1t\jap{t}^{-2}\norm{A^{\nu;1} Q^1_{\neq}}^2_2 + \frac{\delta_\lambda}{4\delta_1 t^{3/2}}\norm{\abs{\grad}^{s/2}A^{\nu;2} Q^2_{\neq}}_2^2 + \frac{K}{\delta_1 \delta_\lambda^{\frac{1}{2s-1}} t^{3/2}}\norm{A^{\nu;2} Q^2_{\neq}}_2^2  +  \frac{K}{\delta_1 t} \norm{\mathbf{1}_{t > \jap{\grad_{Y,Z}}} A^{\nu;2}Q^2}_2^2,  \\ 
& = \delta_1CK_{L1}^{\nu;1} + \frac{1}{4\delta_1}CK_{\lambda}^{\nu;2} + \frac{K}{\delta_1}CK_{L}^{\nu;2} + \frac{K}{\delta_1 \delta_\lambda^{\frac{1}{2s-1}} t^{3/2}}\norm{A^{\nu;2} Q^2_{\neq}}_2^2, 
\end{align*}
and, 
\begin{align*} 
LS1 & \leq  (1+\delta_1)CK^{\nu;1}_{L2} + (1-\delta_1)CK_{L1}^{\nu;1} + \frac{\delta_\lambda}{10\jap{t}^{3/2}}\norm{\abs{\grad}^{s/2} A^{\nu;1}Q^1}_2^2 + \frac{K\epsilon}{\jap{t}^2}\norm{A^{1}\Delta_L U^1_{\neq}}_2^2 \\ & \quad + \frac{K}{\jap{t}^2}\norm{A^1 Q^1_{\neq}}^2_2 +  \frac{K}{\delta_\lambda^{\frac{1}{2s-1}}\jap{t}^{3/2}} \norm{ A^{\nu;1}Q^1}_2^2 + \epsilon\norm{A^{\nu;1}Q^1}_2\norm{\Delta_L A^{\nu;1}U^1_{\neq}}_2, 
\end{align*} 
which, after Lemmas \ref{lem:SimplePEL} and \ref{lem:AnuLossy}, are both consistent with Proposition \ref{prop:Boot} provided $K_{ED1}$ is chosen large relative to both $K_{ED2}$ and $K_{H1}$ (and $\delta_\lambda$, $\delta_1^{-1}$, $K$ and universal constants). 

Next consider the linear pressure term $LP1$. 
We may directly apply Lemma \ref{lem:AnuLossy} to deduce 
\begin{align*} 
LP1  \leq 2\norm{A^{\nu;1}Q^1}_2 \norm{\partial_{XX} A^{\nu;1}U^2_{\neq}}_2 &\lesssim \jap{t}^{-3} \norm{A^{\nu;1}Q^1}_2\left(\norm{A^{\nu;2}Q^2_{\neq}}_2 + \norm{A^{2}Q^2_{\neq}}_2\right) \\ 
& \lesssim \frac{1}{\jap{t}^{3}}\norm{A^{\nu;1}Q^1}^2_2 + \frac{1 + K_{ED2}}{\jap{t}^{3}}\epsilon^2, 
\end{align*} 
which is consistent with Proposition \ref{prop:Boot} via integrating factors provided $K_{ED1} \gg K_{ED2}$. 

\subsubsection{Nonlinear pressure and stretching} 
These terms are treated in essentially the same manner as in \S\ref{sec:NLPS_Q3ED}; we only briefly sketch a few terms. 
We use enumerations analogous to those employed in \eqref{def:Q2enumnu}.  

\paragraph{Treatment of $NLP(i,j,0,\neq)$ terms} \label{sec:NLP0neq_Q1ED}
Notice that in this case $j \neq 1$. 
From \eqref{ineq:AnuiDistri}, Lemma \ref{lem:AnuLossy}, and \eqref{ineq:AprioriU0}, 
\begin{align*}                                
NLP(i,j,0,\neq) & \lesssim  \norm{A^{\nu;1}Q^1}_2 \norm{A^{\nu;1} \partial_X\partial_i^t U^j}_2\norm{U_0^i}_{\G^{\lambda,\beta + 3\alpha + 5}} \lesssim  \frac{\epsilon}{\jap{t}}\norm{A^{\nu;1}Q^1}_2 \left(\norm{A^{\nu;j} Q^j}_2 + \norm{A^{j}Q^j_{\neq}}_2\right).  
\end{align*} 

\paragraph{Treatment of $NLS1(j,0,\neq)$ terms} \label{sec:NLS10neq_Q1ED} 
From \eqref{ineq:AnuiDistri}, Lemma \ref{lem:AnuLossy}, and \eqref{ineq:AprioriU0}, 
\begin{align*}
NLS1(j,0,\neq) \lesssim \norm{A^{\nu;1}Q^1}_2 \norm{Q^j_0}_{\G^{\lambda,\beta + 3\alpha + 5}}  \norm{\partial_j^t A^{\nu;1}U^1}_2  \lesssim \frac{\epsilon}{\jap{t}}\norm{A^{\nu;1}Q^1}_2\left(\norm{A^{\nu;1}Q^1}_2 + \norm{A^{1}Q^1_{\neq}}_2 \right).
\end{align*}

\paragraph{Treatment of $NLS1(j,\neq,0)$ terms} \label{sec:NLS1neq0_Q1ED} 
Note in this case that $j \neq 1$. 
From \eqref{ineq:AnuiDistri}, Lemma \ref{lem:AnuLossy}, \eqref{ineq:AnuHiLowSep2}, and \eqref{ineq:AprioriU0}, 
\begin{align*}
NLS1(j,\neq,0) &  \lesssim \norm{A^{\nu;1}Q^1}_2 \norm{A^{\nu;1} Q^j}_2\norm{U_0^1}_{\G^{\lambda,\beta + 3\alpha + 5}} \\ 
& \lesssim \epsilon^{3/2}\left(\norm{\sqrt{-\Delta_L} A^{\nu;1}Q^1}^2_2 + \norm{A^{1}Q^1_{\neq}}_2^2 \right) + \epsilon^{1/2} \norm{A^{\nu;j} Q^j}_2^2,  
\end{align*}
which suffices for $t \lesssim \epsilon^{-1/2+\delta/100}$, after which we use again \eqref{ineq:AnuHiLowSep2} to deduce 
\begin{align*}
NLS1(j,\neq,0) & \lesssim \epsilon^{3/2}\left(\norm{\sqrt{-\Delta_L} A^{\nu;1}Q^1}^2_2 + \norm{A^{1}Q^1_{\neq}}_2^2 \right) + \epsilon^{3/2-\delta/50} \left(\norm{\sqrt{-\Delta_L} A^{\nu;j}Q^j}^2_2 + \norm{A^{j}Q^j_{\neq}}_2^2 \right),  
\end{align*}
which is consistent with Proposition \ref{prop:Boot} for $\epsilon$ sufficiently small. 

\paragraph{Treatment of $NLS2(i,j,\neq,0)$ terms} \label{sec:NLS2neq0_Q1ED} 
From \eqref{ineq:AnuiDistri}, Lemma \ref{lem:AnuLossy}, and \eqref{ineq:AprioriU0}, we have 
\begin{align*} 
NLS2(i,j,\neq,0) & \lesssim \norm{A^{\nu;1}Q^1}_2 \norm{A^{\nu;1} \partial_i^t U^j_{\neq}}_2 \norm{U_0^1}_{\G^{\lambda,\beta + 3\alpha +6}} \lesssim \epsilon \norm{A^{\nu;1}Q^1}_2 \left(\norm{A^{\nu;j}Q^j}_2 + \norm{A^{j}Q^j_{\neq}}_2\right). 
\end{align*} 

\paragraph{Treatment of $NLS2(i,j,0,\neq)$ terms} \label{sec:NLS20neq_Q1ED} 
From \eqref{ineq:AnuiDistri}, Lemma \ref{lem:AnuLossy}, and \eqref{ineq:AprioriU0}. we have (noting that $i \neq 1$):
\begin{align*} 
NLS2(i,j,0,\neq) & \lesssim \norm{A^{\nu;1}Q^1}_2 \norm{A^{\nu;1} \partial_{ij}^t U^1}_2 \norm{U_0^j}_{\G^{\lambda,\beta+3\gamma+5}} \lesssim \epsilon \norm{ A^{\nu;1}Q^1}_2 \left(\norm{A^{\nu;1}Q^1}_2 + \norm{A^{1}Q^1_{\neq}}_2\right). 
\end{align*} 
Notice that we again used the structure which for $j = 1$, balances the loss of $\jap{t}$ from the third factor with a gain of $\jap{t}^{-1}$ from the second factor.  

\paragraph{Treatment of $NLP(i,j,\neq,\neq)$, $NLS1(i,j,\neq,\neq)$, and $NLS2(i,j,\neq,\neq)$}

The nonlinear terms involving two non-zero frequencies can all be treated in essentially the same manner as 
in $Q^3$ in \S\ref{sec:NLPneqneq_Q3ED}, \S\ref{sec:NLS1neqneq_Q3ED} and \S\ref{sec:NLS1neqneq_Q3ED}. 
We omit the treatments for the sake of brevity. 

\subsubsection{Transport nonlinearity}
The transport nonlinearity, $\mathcal{T}$ in \eqref{ineq:AnuEvo1}, can be treated in the same manner as the transport nonlinearity in \S\ref{sec:Trans_ED_Q3}. We omit the details for brevity.

\subsubsection{Dissipation error terms}   
The dissipation error terms can be treated in same manner as those in \S\ref{sec:DE_ED_Q3} and \cite{BMV14,BGM15I}, and hence we omit the details for brevity. This completes the enhanced dissipation estimate on $Q^1$.

\section{Sobolev estimates} \label{sec:LowNrmVel}
In this section we improve the $H^{\sigma^\prime}$ estimates in \eqref{ineq:Boot_LowFreq}, which are more straightforward than the analogous estimates proved in \cite{BGM15I} (the main challenge in \cite{BGM15I} was getting good decay properties for $t \gtrsim \nu^{-1}$, which is irrelevant here). 
As in \cite{BGM15I}, these estimates are performed in the coordinate system given by $(X,y,z)$; see \S\ref{sec:RegCont}. In Lemma \ref{lem:intermedSob}, the a priori estimates from the bootstrap hypotheses in these coordinates are given.  
The estimates are performed on \eqref{eq:u0i} and then transferred back to the $(X,Y,Z)$ coordinates. 
Indeed, as long as the $C^i$ remain small, the coordinate change is uniformly bounded in Sobolev regularity, and hence by suitably adjusting the constants in \eqref{ineq:Boot_LowFreq}, one can prove these finite regularity estimates in whichever coordinate system is most convenient (see \cite{BGM15I} for more details). 

\subsection{Improvement of \eqref{ineq:Boot_U03_Low} and \eqref{ineq:Boot_U02_Low}} \label{sec:U023Low}
These estimates are best proved together using a standard energy method. Recall the notation $u_0 = (u_0^2, u_0^3)^{T}$. 
From \eqref{eq:u0i},  
\begin{align*} 
\frac{1}{2}\frac{d}{dt} \norm{u_0}_{H^{\sigma^\prime}}^2 & = -\nu \norm{\grad u_0}_{H^{\sigma^\prime}}^2 - \int \jap{\grad}^{\sigma^\prime} u_0^i \jap{\grad}^{\sigma^\prime} \left( u_0^j \cdot \partial_j u_0^i \right) dy dz \\ & \quad - \int \jap{\grad}^{\sigma^\prime} u_0^i \jap{\grad}^{\sigma^\prime} \partial_{i} p^{NL0} dy dz + \int \jap{\grad}^{\sigma^\prime} u_0^i \jap{\grad}^{\sigma^\prime} \mathcal{F}^i dy dz \\ 
& = -\nu \norm{\grad u_0}_{H^{\sigma^\prime}}^2 + \mathcal{T} + \mathcal{P} + \mathcal{F}. 
\end{align*} 
For the transport term $\mathcal{T}$, we use integration by parts (and the divergence free condition) to introduce the following commutator:  
\begin{align*}
\mathcal{T} = \int \jap{\grad}^{\sigma^\prime} u_0^i \left(u_0 \cdot \grad \jap{\grad}^{\sigma^\prime} u_0^i - \jap{\grad}^{\sigma^\prime} \left( u_0 \cdot \grad u_0^i\right)\right) dy dz. 
\end{align*}
Treating this commutator is by now classical and, in particular, by using that for $\abs{\eta,l} \approx \abs{\xi,l^\prime}$,  
\begin{align*}
\jap{\eta,l}^{\sigma^\prime} - \jap{\xi,l^\prime}^{\sigma^\prime} \lesssim \abs{\eta-\xi,l-l^\prime} \jap{\xi,l^\prime}^{\sigma^\prime-1},  
\end{align*}
one can show that 
\begin{align*}
\mathcal{T} & \lesssim \norm{\grad u_0}_{H^{1+}} \norm{u_0^i}_{H^{\sigma^\prime}}^2 + \norm{u_0}_{H^{\sigma^\prime}} \norm{u_0^i}_{H^{\sigma^\prime}}\norm{\grad u_0^i}_{H^{1+}} \\
& \lesssim \norm{u_0}_{H^{\sigma^\prime}}\norm{u_0^i}_{H^{\sigma^\prime}}^2 \lesssim \epsilon \norm{u_0}_{H^{\sigma^\prime}}^2, 
\end{align*}
(where we also used $\sigma^\prime > 2$, by \eqref{ineq:uzAPriori}) which is consistent with Proposition \ref{prop:Boot} for $c_0$ sufficiently small. 

For the pressure term $\mathcal{P}$, we simply use the divergence free condition: 
\begin{align*}
\mathcal{P} & = -\int \jap{\grad}^{\sigma^\prime} u_0^i \jap{\grad}^{\sigma^\prime} \partial_{i} p^{NL0} dy dz  = \int \jap{\grad}^{\sigma^\prime} \partial_i u_0^i \jap{\grad}^{\sigma^\prime} p^{NL0} dy dz  = 0. 
\end{align*}

The forcing term is straightforward from \eqref{ineq:Xyzubds}, indeed it follows immediately that 
\begin{align*} 
\int \jap{\grad}^{\sigma^\prime} u_0^i \jap{\grad}^{\sigma^\prime} \mathcal{F}^i dy dz \leq \frac{\epsilon^2}{\jap{\nu t^3}^{2\alpha}} \norm{u_0}_{H^{\sigma^\prime}}. 
\end{align*} 

Hence, the improvements to \eqref{ineq:Boot_U03_Low} and \eqref{ineq:Boot_U02_Low} follow for $\epsilon$ and $c_0$ sufficiently small. 

\subsection{Improvement of \eqref{ineq:Boot_U01_Low}} 
The improvement of \eqref{ineq:Boot_U01_Low} is very similar to those of \eqref{ineq:Boot_U03_Low} and \eqref{ineq:Boot_U02_Low} with the exception of the lift-up effect term. 
Indeed, by \eqref{eq:u0i}, 
\begin{align*} 
\frac{1}{2}\frac{d}{dt} \left(\jap{t}^{-2} \norm{u_0^1}_{H^{\sigma^\prime}}^2\right) & = -\frac{t}{\jap{t}^4}  \norm{u_0^1}_{H^{\sigma^\prime}}^2 -\nu \jap{t}^{-2} \norm{\grad u_0^1}_{H^{\sigma^\prime}}^2 - \jap{t}^{-2} \int \jap{\grad}^{\sigma^\prime} u_0^1 \jap{\grad}^{\sigma^\prime} \left( u_0 \cdot \grad u_0^1\right) dy dz \nonumber \\ & \quad - \jap{t}^{-2}\int\jap{\grad}^{\sigma^\prime} u_0^1 \jap{\grad}^{\sigma^\prime}u_0^2 dy dz  + \jap{t}^{-2}\int \jap{\grad}^{\sigma^\prime} u_0^1 \jap{\grad}^{\sigma^\prime} \mathcal{F}^1 dy dz.
\end{align*}  
All the terms are treated as in \S\ref{sec:U023Low} except of course the lift up effect term. 
For this  we use \eqref{ineq:Boot_LowFreq}, 
\begin{align*}
- \jap{t}^{-2}\int\jap{\grad}^{\sigma^\prime} u_0^1 \jap{\grad}^{\sigma^\prime}u_0^2 dy dz & \leq 4 \epsilon \jap{t}^{-2} \norm{u_0^1}_{H^{\sigma^\prime}}.  
\end{align*}
From here, one applies the super-solution method used in \S\ref{sec:Q1Hi1}. We omit the details for brevity as it follows the same.

\section*{Acknowledgments}
The authors would like to thank the following people for helpful discussions: Margaret Beck, Steve Childress, Michele Coti Zelati, Bruno Eckhardt, Pierre-Emmanuel Jabin, Susan Friedlander, Yan Guo, Alex Kiselev, Nick Trefethen, Mike Shelley, Vladimir Sverak, Vlad Vicol, and Gene Wayne. 
The authors would like to especially thank Tej Ghoul for encouraging us to focus our attention on finite Reynolds number questions. 
The work of JB was in part supported by NSF Postdoctoral Fellowship in Mathematical Sciences DMS-1103765 and NSF grant DMS-1413177, the work of PG was in part supported by a Sloan fellowship and the NSF grant DMS-1101269, while the work of NM was in part supported by the NSF grant DMS-1211806. 

\appendix

\section{Fourier analysis conventions, elementary inequalities, and Gevrey spaces} \label{apx:Gev}
We take the same Fourier analysis conventions as \cite{BGM15I}; we briefly recall them here for completeness. 
For $f(x,y,z)$ in the Schwartz space (or $(X,Y,Z)$), we define the Fourier transform $\hat{f}_k(\eta,l)$ where $(k,\eta,l) \in \Integer \times \Real \times \Integer$ and the inverse Fourier transform via 
\begin{align*} 
\hat{f}_k(\eta,l) & = \frac{1}{(2\pi)^{3/2}}\int_{\Torus \times \Real \times \Torus} e^{-i x k - iy\eta - ilz} f(x,y,z) dx dy dz \\ 
f(x,y,z) & = \frac{1}{(2\pi)^{3/2}}\sum_{k,l \in \Integer} \int_{\Real} e^{i x k + iy\eta + izl} \hat{f}_k(\eta,l) d\eta. 
\end{align*} 
With these conventions: 
\begin{align*} 
\int f(x,y,z) \overline{g}(x,y,z) dx dy dz & = \sum_{k}\int \hat{f}_k(\eta,l) \overline{\hat{g}}_{k}(\eta,l) d\eta \\ 
\widehat{fg} & = \frac{1}{(2\pi)^{3/2}}\hat{f} \ast \hat{g} \\ 
(\widehat{\grad f})_k(\eta,l) & = (ik,i\eta,il)\widehat{f_k}(\eta,l). 
\end{align*}
The paraproducts defined above in \S\ref{sec:paranote} are defined using the Littlewood-Paley dyadic decomposition (see e.g. \cite{BCD11} for more details).  
Let $\psi \in C_0^\infty(\Real_+;\Real_+)$ be such that $\psi(\xi) = 1$ for $\xi \leq 1/2$ and $\psi(\xi) = 0$ for $\xi \geq 3/4$ and define $\rho(\xi) = \psi(\xi/2) - \psi(\xi)$, supported in the range $\xi \in (1/2,3/2)$. 
Then we have the partition of unity for $\xi > 0$, 
\begin{align*} 
1 = \sum_{M \in 2^\Integers} \rho(M^{-1}\xi), 
\end{align*}  
where we mean that the sum runs over the dyadic integers $M = ...,2^{-j},...,1/4,1/2,1,2,4,...,2^{j},...$
 and we define the cut-off $\rho_M(\xi) = \rho(M^{-1}\xi)$, each supported in $M/2 \leq \xi \leq 3M/2$. 
For $f \in L^2(\Torus \times \Real \times \Torus)$ we define
\begin{align*} 
f_{M} = \rho_M(\abs{\grad})f, \quad\quad f_{< M} = \sum_{K \in 2^{\Integers}: K < M} f_K, 
\end{align*}
which defines the decomposition (in the $L^2$ sense)  
\begin{align*} 
f = \sum_{M \in 2^\Integers} f_M.  
\end{align*}
There holds the almost orthogonality and the approximate projection property 
\begin{subequations} \label{ineq:LPOrthoProject}
\begin{align} 
\norm{f}^2_2 & \approx \sum_{M \in 2^{\Integers}} \norm{f_M}_2^2 \\
 \norm{f_M}_2 & \approx  \norm{(f_{M})_{\sim M}}_2, 
\end{align}
\end{subequations}
where we make use of the notation 
\begin{align*} 
f_{\sim M} = \sum_{K \in 2^{\Integer}: \frac{1}{C}M \leq K \leq CM} f_{K}, 
\end{align*}
for some constant $C$ which is independent of $M$.
Generally the exact value of $C$ which is being used is not important; what is important is that it is finite and independent of $M$. 
Similar to \eqref{ineq:LPOrthoProject} but more generally, if $f = \sum_{j} D_j$ for any $D_j$ with $\frac{1}{C}2^{j} \subset \textup{supp}\, D_j \subset C2^{j}$ it follows that 
\begin{align} 
\norm{f}^2_2 \approx_C \sum_{j \in \Integers} \norm{D_j}_2^2. \label{ineq:GeneralOrtho}
\end{align}

Recall the following two lemmas. 
\begin{lemma}
Let $f(\xi),g(\xi) \in L_\xi^2(\Real^d)$, $\jap{\xi}^\sigma h(\xi) \in L_\xi^2(\Real^d)$ and $\jap{\xi}^\sigma b(\xi) \in L_\xi^2(\Real^d)$  for $\sigma > d/2$, 
Then we have 
\begin{align} 
\norm{f \ast h}_2 & \lesssim_{\sigma, d} \norm{f}_2\norm{\jap{\cdot}^\sigma h}_2, \label{ineq:L2L1}  \\
\int \abs{f(\xi) (g \ast h)(\xi)} d\xi & \lesssim_{\sigma,d} \norm{f}_2\norm{g}_2\norm{\jap{\cdot}^\sigma h}_2 \label{ineq:L2L2L1} \\ 
\int \abs{f(\xi) (g \ast h \ast b) (\xi)} d\xi & \lesssim_{\sigma,d} \norm{f}_2\norm{g}_2\norm{\jap{\cdot}^\sigma h}_2\norm{\jap{\cdot}^\sigma b}_2. \label{ineq:L2L2L1L1}  
\end{align}
Further iterates are applied for higher order nonlinear terms in Lemma \ref{lem:ParaHighOrder} and are similar to \eqref{ineq:L2L2L1L1} but are omitted here. 
\end{lemma}

\begin{lemma}
Let $0 < s < 1$, $x,y>0$, and $K>1$. 
\begin{itemize} 
\item[(i)] There holds
\begin{align} 
\abs{x^s - y^s} \leq s \max(x^{s-1},y^{s-1})\abs{x-y}. \label{ineq:TrivDiff}
\end{align}
so that if $|x-y|<\frac{x}{K}$,
\begin{align} 
\abs{x^s - y^s} \leq \frac{s}{(K-1)^{1-s}}\abs{x-y}^s. \label{lem:scon}
\end{align} 
Note $\frac{s}{(K-1)^{1-s}} < 1$ as soon as $s^{\frac{1}{1-s}} + 1 < K$. 
\item[(ii)] There holds 
\begin{align} 
\abs{x + y}^s \leq \left(\frac{\max(x,y)}{x+y}\right)^{1-s}\left(x^s + y^s\right), \label{lem:smoretrivial}
\end{align}  
so that, if $\frac{1}{K}y \leq x \leq Ky$,
\begin{align} 
\abs{x + y}^s \leq \left(\frac{K}{1 + K}\right)^{1-s}\left(x^s + y^s\right). \label{lem:strivial}
\end{align} 
\end{itemize}
\end{lemma}
Gevrey and Sobolev regularities can be related with the following two inequalities:  
\begin{itemize}
\item[(i)] For all $x \geq 0$, $\alpha > \beta \geq 0$, $C,\delta > 0$, 
\begin{align} 
e^{Cx^{\beta}} \leq e^{C\left(\frac{C}{\delta}\right)^{\frac{\beta}{\alpha - \beta}}} e^{\delta x^{\alpha}};  \label{ineq:IncExp}
\end{align}
\item[(ii)] For all $x \geq 0$, $\alpha,\sigma,\delta > 0$, 
\begin{align} 
e^{-\delta x^{\alpha}} \lesssim \frac{1}{\delta^{\frac{\sigma}{\alpha}} \jap{x}^{\sigma}}. \label{ineq:SobExp}
\end{align}
\end{itemize}
Together these inequalities show that for $\alpha > \beta \geq 0$, $\norm{f}_{\mathcal{G}^{C,\sigma;\beta}} \lesssim_{\alpha,\beta,C,\delta,\sigma} \norm{f}_{\mathcal{G}^{\delta,0;\alpha}}$.

\section{Definition and analysis of the norms} \label{sec:def_nrm}  

\subsection{Definition and analysis of $w$} \label{sec:Defw}
As mentioned above in \S\ref{sec:Toy}, the multipliers we use are variants of those used in \cite{BM13,BMV14,BGM15I}, and we build on those constructions. 
We first begin by defining $\bar{w}(t,\eta)$, which is used to construct $w(t,\eta)$ and $w^3(t,k,\eta)$.
For $\bar{w}$ and $w$ we use the same multipliers as \cite{BGM15I}, however, we include the constructions here for completeness and also to make the explanation of $w^3(t,k,\eta)$ more natural.  

In what follows fix $k,\eta > 0$; we will see that the norms do not depend on the sign of $k$ and $\eta$. 
Further, recall the definitions in \S\ref{sec:Notation}. 
The multiplier is built backwards in time, which makes resonance counting easier. 
Let $t \in I_{k,\eta}$. Let $\bar{w}(t,\eta)$ be a non-decreasing function of time with $\bar{w}(t,\eta) = 1  $ for $t \geq  2\eta $. 
For $ k \geq 1$, we assume that $\bar{w}(t_{k-1,\eta})  $ was computed.  
To compute $\bar{w}$ on the interval $I_{k,\eta} $, we use the behavior predicted by the toy model in \eqref{def:unstablesuper}.  
For a parameter $\kappa > 1$ fixed sufficiently large depending on a universal constant determined by the proof,
for $k=1,2,3,..., E(\sqrt{\eta}) $, we define
\begin{subequations} \label{def:wNR}
\begin{align}
\bar{w}(t,\eta) &=   \Big( \frac{k^2}{\eta}   \left[ 1 + b_{k,\eta} |t-\frac{\eta}k | \right]   \Big)^{\kappa}  \bar{w} (t_{k-1,\eta}),  \quad& 
  \quad  \forall t \in  I^R_{k,\eta} =  \left[ \frac{\eta}k ,t_{k-1,\eta}  \right], \\ 
\bar{w}(t,\eta) &=   \Big(1 + a_{k,\eta} |t-\frac{\eta}k |   \Big)^{-1-\kappa}  \bar{w} \left(\frac{\eta}k\right),  \quad& 
  \quad \forall  t \in  I^L_{k,\eta} =  \left[ t_{k,\eta}  , \frac{\eta}k   \right].  
\end{align} 
\end{subequations}
The constant $b_{k,\eta}  $   is chosen to ensure that $ \frac{k^2}{\eta}   \left[ 1 + b_{k,\eta} |t_{k-1,\eta}-\frac{\eta}k | \right]  =1$, hence for $k \geq2$, we have 
\begin{align} \label{bk} 
 b_{k,\eta} = \frac{2(k-1)}{k} \left(1 - \frac{k^2}{\eta} \right)
\end{align} 
and $b_{1,\eta} = 1 - 1/\eta$. 
Similarly,  $a_{k,\eta}$ is chosen to ensure $ \frac{k^2}{\eta}\left[ 1 + a_{k,\eta} |t_{k,\eta}-\frac{\eta}k | \right] = 1$, which implies
\begin{align} \label{ak} 
 a_{k,\eta} = \frac{2(k+1)}{k} \left(1 - \frac{k^2}{\eta} \right). 
\end{align} 
Hence, we have $ \bar{w}(\frac{\eta}k) = \bar{w} (t_{k-1,\eta})  \Big( \frac{k^2}{\eta} \Big)^{\kappa}$  
and  $\bar{w} ( t_{k,\eta} ) = \bar{w} (t_{k-1,\eta})  \Big( \frac{k^2}{\eta} \Big)^{1+ 2\kappa}$. 
For earlier times $[0, t_{E(\sqrt{\eta}),\eta }] $, we take $\bar{w}$ to be constant. 
Next, we will impose additional losses in time on $\bar{w}$:  
\begin{align}
w(t,\eta) = \bar{w}(t,\eta) \exp\left[-\kappa \int_{t}^\infty \mathbf{1}_{\tau \leq 2\sqrt{\eta}} d\tau  - \kappa \int_{t}^\infty \mathbf{1}_{\sqrt{\abs{\eta}} \leq \tau \leq 2\abs{\eta}} \frac{\abs{\eta}}{\tau^2} d\tau \right]. \label{def:wextraloss}
\end{align}

Next, we define $w^3_k(t,\eta)$. Suppose $t \in I_{k,\eta}$ then, for $k^\prime \neq k$, 
\begin{subequations} \label{def:wNR3} 
\begin{align}
w^3_{k^\prime}(t,,\eta) &= \frac{\eta}{k^2\left(1 + b_{k,\eta}\abs{t-\frac{\eta}{k}}\right)}w(t,\eta)   \quad  \forall t \in  I^R_{k,\eta} =  \left[ \frac{\eta}k ,t_{k-1,\eta}  \right], \\ 
w^3_{k^\prime}(t,,\eta) &= \frac{\eta}{k^2\left(1 + a_{k,\eta}\abs{t-\frac{\eta}{k}}\right)}w(t,\eta)   \quad \forall  t \in  I^L_{k,\eta} =  \left[ t_{k,\eta}  , \frac{\eta}k   \right].  \\ 
w^3_k(t,\eta) & = w(t,\eta) \quad  \forall t \in  I_{k,\eta}, 
\end{align} 
\end{subequations}
and we take $w^3_k(t,\eta) = w(t,\eta)$ if $t\not\in I_{j,\eta}$ for any $j$.

The following lemma is essentially Lemma 3.1 in \cite{BM13} and  shows that $w(t,\eta)^{-1}$ loses some fixed radius of Gevrey-2 regularity. The proof is omitted for brevity.  
\begin{lemma} \label{lem:totalGrowthw} 
There is a constant $\mu$ (depending on $\kappa$) and a constant $p > 0$ such that for all $\abs{\eta} > 1$, we have 
\begin{align*} 
 \frac{1}{w(t,\eta)} & \leq \frac{1}{w(1,\eta)}  \sim \eta^{-p} e^{\frac{\mu}{2} \sqrt{\eta} } \\  
\frac{1}{w^3_k(t,\eta)} & \leq \frac{1}{w^3_k(1,\eta)} \sim \eta^{-p} e^{\frac{\mu}{2} \sqrt{\eta} },  
\end{align*} 
where `$\sim$' is in the sense of asymptotic expansion (up to a multiplicative constant) as $\eta \rightarrow \infty$. 
\end{lemma} 

The following lemma is from \cite{BM13}, and shows how to use the well-separation of critical times. 
\begin{lemma} \label{lem:wellsep}
Let $\xi,\eta$ be such that there exists some $K \geq 1$ with $\frac{1}{K}\abs{\xi} \leq \abs{\eta} \leq K\abs{\xi}$ and let $k,n$ be such that $t \in I_{k,\eta}$ and $t \in I_{n,\xi}$  (note that $k \approx n$).  
Then at least one of following holds:
\begin{itemize} 
\item[(a)] $k = n$ (almost same interval); 
\item[(b)] $\abs{t - \frac{\eta}{k}} \geq \frac{1}{10 K}\frac{\abs{\eta}}{k^2}$ and $\abs{t - \frac{\xi}{n}} \geq \frac{1}{10 K}\frac{\abs{\xi}}{n^2}$ (far from resonance);
\item[(c)] $\abs{\eta - \xi} \gtrsim_K \frac{\abs{\eta}}{\abs{n}}$ (well-separated). 
\end{itemize}
\end{lemma}

The next lemma tells us how to take advantage of the time derivative of $w$ and hence the $CK_w$ terms. 

\begin{lemma}[Time derivatives near the critical times] \label{lem:dtw}
If $t \leq 2\sqrt{\eta}$, then there holds 
\begin{align} 
\frac{\partial_t w(t,\eta)}{w(t,\eta)}  = \frac{\partial_t w^3_k(t,\eta)}{w^3_k(t,\eta)}  & \approx \kappa.  
\end{align}
If we instead have $t \in \I_{r,\eta}$ for some $r$, then the following holds
\begin{align} 
\frac{\partial_t w(t,\eta)}{w(t,\eta)} \approx \frac{\partial_t w^3_k(t,\eta)}{w^3_k(t,\eta)}  & \approx \frac{\kappa}{1 + \abs{\frac{\eta}{r} - t}} + \frac{\kappa \abs{\eta}}{t^2} \approx \frac{\kappa}{1 + \abs{\frac{\eta}{r} - t}} + \frac{\kappa \abs{r}}{t}  \label{dtw}
\end{align}
\end{lemma} 

The next lemma is from \cite{BGM15I} and is a variant of Lemma 3.4 in \cite{BM13}.  
It is important for estimating nonlinear terms where we need to be able to compare $CK_w$ multipliers of different frequencies. 

\begin{lemma} \label{lem:WtFreqCompare}
\begin{itemize}
\item[(i)] For $t \gtrsim 1$, and $\eta,\xi$ such that $t < 2  \min( \abs{\xi},  \abs{\eta})    $,
\begin{align} \label{dtw-xi} 
 \frac{\partial_t w(t,\eta)}{w(t,\eta)}\frac{w(t,\xi)}{\partial_t w(t,\xi)}
 \lesssim \jap{\eta - \xi}^2
\end{align}
\item[(ii)] For all $t \gtrsim 1$, and $\eta,\xi$, such that for some $K \geq 1$, $\frac{1}{K}\abs{\xi} \leq \abs{\eta} \leq K\abs{\xi}$,  
\begin{align}
\sqrt{\frac{\partial_t w(t,\xi)}{w(t,\xi)}} \lesssim_K \left[\sqrt{\frac{\partial_t w(t,\eta)}{w(t,\eta)}} + \frac{\abs{\eta}^{s/2}}{\jap{t}^{s}}\right]\jap{\eta-\xi}^2. \label{ineq:partialtw_endpt}  
\end{align}
\end{itemize}
By Lemma \ref{lem:dtw}, these hold also for $w^3$ (and we do not need to make a distinction). 
\end{lemma}

The next lemma from \cite{BGM15I} and is an easy variant of the analogous [Lemma 3.5, \cite{BM13}].  
It is of crucial importance for estimating nonlinear terms we need to be able to compare ratios.

\begin{lemma}[Ratio estimates for nonlinear interactions] \label{lem:wRat}
There exists a $K > 0$ such that for all $\eta,\xi$, 
\begin{align}
\frac{w(t,\eta)}{w(t,\xi)} & \lesssim e^{K\abs{\eta-\xi}^{1/2}}. 
\end{align} 
\end{lemma} 

Next, we want to write the analogue of Lemma \ref{lem:wRat} for $w^3$, which is somewhat trickier.
Instead of Lemma \ref{lem:wRat}, we have the following, which is analogous to [Lemma 3.6, \cite{BM13}] (although here easier due to the simpler $k$ dependence). 

\begin{lemma}\label{lem:Jswap}
There is a universal $K > 0$ such that in general we have 
\begin{align}
\frac{w^3_{k^\prime}(\eta)}{w^3_{k}(\xi)} \lesssim \frac{t}{\abs{k}+ \abs{\eta-kt}} e^{K{\mu}\abs{k-k^\prime,\eta - \xi}^{1/2}}. \label{ineq:WFreqCompRes}
\end{align}  
If any one of the following holds: ($t \not\in \I_{k,\eta}$) or ($k = k^\prime$) or ($t \in \I_{k,\eta}$, $t \not\in \I_{k,\xi}$) then we have the improved estimate  
\begin{align} 
\frac{w^3_{k^\prime}(\eta)}{w^3_{k}(\xi)} \lesssim e^{K{\mu}\abs{k-k^\prime,\eta - \xi}^{1/2}}. \label{ineq:BasicJswap} 
\end{align}
Finally if $t \in \I_{k^\prime,\xi}$ and $k \neq k^\prime$, then 
\begin{align} 
\frac{w^3_{k^\prime}(\eta)}{w^3_{k}(\xi)} \lesssim \frac{\abs{k^\prime} + \abs{\xi - k^\prime t}}{t}e^{K{\mu}\abs{k-k^\prime,\eta - \xi}^{1/2}}. \label{ineq:WFreqCompNRGain}
\end{align}
\end{lemma} 

\begin{remark} \label{rmk:GainLoss}
In the case $t \in \I_{k,\eta} \cap \I_{k,\xi}$, $k \neq k^\prime$,  the only case where \eqref{ineq:WFreqCompRes} is needed, we also have $\abs{\eta} \approx \abs{\xi}$ and from \eqref{ineq:WFreqCompRes}, the definition \eqref{def:wNR3}, and \eqref{dtw}, Lemma \ref{lem:WtFreqCompare} and \eqref{ineq:SobExp} implies that there is a $K > 0$ such that (see \cite{BM13} for more information) 
\begin{align} 
\frac{w_k^3(\eta)}{w_{k^\prime}^3(\xi)} \lesssim \frac{t}{\abs{k}}\sqrt{\frac{\partial_t w_k(t,\eta)}{w_k(t,\eta)}}\sqrt{\frac{\partial_t w_l(t,\xi)}{w_l(t,\xi)}}e^{K\mu\abs{k-l,\eta-\xi}^{1/2}}. \label{ineq:RatJ2partt}
\end{align}
\end{remark} 

\begin{remark} 
Notice the appearance of $\I_{k,\eta}$ as opposed to $I_{k,\eta}$. Each are defined in \S\ref{sec:Notation}. The use of $\I$ is to rule out the end case $t \approx \sqrt{\abs{\eta}}$, for example, we see that \eqref{ineq:BasicJswap} holds if $t \approx \sqrt{\abs{\eta}}$ even if $t \in I_{k,\eta}$ and hence inequalities like \eqref{ineq:RatJ2partt} will not be necessary. 
\end{remark}

\subsection{The design and analysis of $w_L$} \label{sec:Nmult}
We also recall the definition of the multiplier $w_L$ from \cite{BGM15I}. 
We define $w_L$ such that it solves the following: 
\begin{subequations}  \label{def:wL}
\begin{align} 
\partial_tw_L(t,k,\eta,l) & = \kappa \frac{\abs{k} \jap{l} }{k^2 + l^2 + \abs{\eta - kt}^2} w_L(t,k,\eta,l) \quad\quad t \geq 1 \\ 
w_L(1,k,\eta,l) & = 1. 
\end{align}
\end{subequations} 
Since the following holds uniformly in $k,l,\eta$: 
\begin{align} 
\int_0^\infty \frac{\abs{k} \jap{l}}{k^2 + l^2 + \abs{\eta - kt}^2} dt \approx 1, \label{ineq:unifN}
\end{align}
the multiplier $w_L$ is $O(1)$ and hence will have very little effect on most estimates. 

\section{Elliptic estimates} \label{sec:Elliptic} 
In this section, we group and discuss all of the necessary ``elliptic'' estimates on $\Delta_t^{-1}$. 
We will need the estimates from \cite{BGM15I} as well as a number of new estimates specific to the above threshold case. 

\subsection{Lossy estimates} \label{sec:Lossy}
First, recall the lossy elliptic lemma [Lemma C.1, \cite{BGM15I}]. 

\begin{lemma}[Lossy elliptic lemma] \label{lem:LossyElliptic}
Under the bootstrap hypotheses, for $c_0$ chosen sufficiently small, then for any function $h$ and $a \leq \sigma$, there holds 
\begin{align*} 
\norm{\Delta^{-1}_t  h_{\neq}}_{\G^{\lambda,a-2}} \lesssim \frac{1}{\jap{t}^2}\norm{h_{\neq}}_{\G^{\lambda,a}}. 
\end{align*} 
\end{lemma}  

We also need the enhanced dissipation lossy elliptic lemma [Lemma C.2, \cite{BGM15I}]. 

\begin{lemma}[Lossy elliptic lemma II] \label{lem:AnuLossy}
If $C$ satisfies the bootstrap assumptions~\eqref{ineq:Boot_CgHi}, then for $c_0$ sufficiently small, for any function $h$, and $\gamma^\prime = \beta + 3\alpha + 5$,  
\begin{subequations} \label{ineq:AnuLossyII} 
\begin{align} 
\norm{ A^{\nu;i} \Delta^{-1}_t h}_{2} + \norm{\partial_X A^{\nu;i} \Delta^{-1}_t h}_{2} & \lesssim \frac{1}{\jap{t}^2}\left(\norm{A^{\nu;i} \phi}_2 + \jap{t}^{-3}\norm{h_{\neq}}_{\G^{\lambda,\gamma^\prime}} \right) \\ 
\norm{ \partial_Z A^{\nu;i} \Delta^{-1}_t h}_{2} + \norm{ (\partial_Y - t \partial_X) A^{\nu;i} \Delta^{-1}_t h}_{2} & \lesssim \frac{1}{\jap{t}} \left(\norm{A^{\nu;i} h}_2 + \jap{t}^{-3}\norm{h_{\neq}}_{\G^{\lambda,\gamma^\prime}} \right) \\ 
\norm{\partial_{m}^t \partial_n^t A^{\nu;i} \Delta_t^{-1} h}_2 & \lesssim \frac{1}{\jap{t}^{b}}\left(\norm{A^{\nu;i} h}_2 + \jap{t}^{-3}\norm{h_{\neq}}_{\G^{\lambda,\gamma^\prime}} \right), 
\end{align} 
\end{subequations}
where $b = 0$ if $n,m \neq 1$, $b = 1$ if exactly one of $m$ or $n$ equals one, and $b = 2$ if $m = n = 1$. 
Moreover, 
\begin{subequations} \label{ineq:AnuLossyED} 
\begin{align} 
\norm{ A^{\nu;i} \Delta^{-1}_t h}_{2} + \norm{\partial_X A^{\nu;i} \Delta^{-1}_t h}_{2} & \lesssim \frac{1}{\jap{t}^3}\left(\norm{ \sqrt{-\Delta_L}A^{\nu;i} h}_2 + \jap{t}^{-3}\norm{h_{\neq}}_{\G^{\lambda,\gamma^\prime}} \right) \\ 
\norm{ \partial_Z A^{\nu;i} \Delta^{-1}_t h}_{2} + \norm{ (\partial_Y - t \partial_X) A^{\nu;i} \Delta^{-1}_t h}_{2} & \lesssim \frac{1}{\jap{t}^2} \left(\norm{\sqrt{-\Delta_L} A^{\nu;i} h}_2 + \jap{t}^{-3}\norm{h_{\neq}}_{\G^{\lambda,\gamma^\prime}} \right) \\  
\norm{\partial_{m}^t \partial_n^t A^{\nu;i} \Delta_t^{-1} h}_2 & \lesssim \frac{1}{\jap{t}^{1+b}}\left(\norm{\sqrt{-\Delta_L} A^{\nu;i} h}_2 + \jap{t}^{-3}\norm{h_{\neq}}_{\G^{\lambda,\gamma^\prime}} \right). 
\end{align} 
\end{subequations}
Finally, we have 
\begin{align} 
\norm{A^{\nu;i} \Delta_L \Delta^{-1}_t h}_2 & \lesssim \norm{A^{\nu;i} h}_2. \label{ineq:PEL_CKnuIII} 
\end{align}  
\end{lemma}  

Also recall the following lemma [Lemma C.3, \cite{BGM15I}]. 

\begin{lemma} [$CK^\nu_{wL}$ elliptic lemma] \label{lem:AnuLossy_CKnu}
Under the bootstrap hypotheses, for $c_0$ sufficiently small we have for any function $h$, 
\begin{align}
\norm{\sqrt{\frac{\partial_t w_L}{w_L}} A^{\nu;i} \Delta_L \Delta^{-1}_t h}_{2} & \lesssim \norm{\sqrt{\frac{\partial_t w_L}{w_L}} A^{\nu;i} h}_2. \label{ineq:PEL_CKnuII} 
\end{align}   
\end{lemma} 

\subsection{Precision lemmas} \label{sec:PEL}
As in \cite{BGM15I}, the so-called `precision elliptic lemmas' (PEL) are variations on the common theme of using $\Delta_L^{-1}$ as an approximate inverse. 
We will need those found in \cite{BGM15I} and several more as well. 

\subsubsection{Zero mode PELs}
The first PEL is essentially [Lemma C.4, \cite{BGM15I}], and puts $U_0^i$ in the high norm.  
\begin{lemma}[Zero mode PEL] \label{lem:PELbasicZero}
Under the bootstrap hypotheses, for $c_0$ and $\epsilon$ sufficiently small there holds, 
\begin{subequations} \label{ineq:AU0PEL} 
\begin{align} 
\norm{A U_0^1}_2^2 & \lesssim \jap{t}^2\norm{A^1 Q^1_0}_2^2 + \norm{U_0^1}_2^2 + \epsilon^2\jap{t}^2\norm{AC}_2^2 \\ 
\norm{\jap{\grad}^2 A^1 U_0^1}_2^2 & \lesssim \norm{A^1 Q^1_0}_2^2 + \jap{t}^{-2}\norm{U_0^1}_2^2 + \epsilon^2\norm{AC}_2^2 \\ 
\norm{A U_0^2}_2^2 & \lesssim \norm{A^2 Q^2_0}_2^2 + \norm{U_0^2}_2^2 + \epsilon^2 \norm{AC}_2^2 \\
\norm{\jap{\grad}^2 A^3 U_0^3}_2^2 & \lesssim \norm{A^3 Q^3_0}_2^2 + \norm{U_0^3}_2^2 + \epsilon^2 \norm{AC}_2^2. \label{ineq:AU03PEL} 
\end{align} 
\end{subequations} 
Moreover, we have
\begin{subequations} \label{ineq:gradAU0i}  
\begin{align} 
\norm{\grad \jap{\grad}^2 A^1 U_0^1}_2^2 & \lesssim \norm{\grad A^1 Q^1_0}_2^2 + \jap{t}^{-2}\norm{\grad U_0^1}_2^2 + \epsilon^2 \norm{AC}_2^2 \\ 
\norm{\grad A U_0^2}_2^2 & \lesssim \norm{\grad A^2 Q_0^2}_2^2 + \norm{\grad U_0^2}_2^2 + \epsilon^2\norm{\grad AC}_2^2. \label{ineq:gradAU02_PEL} \\
\norm{\grad \jap{\grad}^2 A^3 U_0^3}_2^2 & \lesssim \norm{\grad A^3 Q^3_0}_2^2 + \norm{\grad U_0^3}_2^2 + \epsilon^2 \norm{\grad AC}_2^2. 
\end{align}
\end{subequations} 
\end{lemma}

The next PEL is specific to this work and has no analogue in \cite{BGM15I}. 
This is due to the increased precision at which we need to understand the regularity of the zero mode of the velocity field in the \textbf{(2.5NS)} terms.  

\begin{lemma}[Zero mode $CK$ PEL] \label{lem:PELCKZero}
Under the bootstrap hypotheses for $t \geq 1$, for $c_0$ and $\epsilon$ sufficiently small, for $i \in \set{2,3}$, there holds 
\begin{align}
\norm{\left(\sqrt{\frac{\partial_t w}{w}}\tilde{A}^i + \frac{\abs{\grad}^{s/2}}{\jap{t}^{s}}A^i \right)\jap{\grad}^2  U_0^i}_2^2  & \lesssim \norm{\left(\sqrt{\frac{\partial_t w}{w}}\tilde{A}^i + \frac{\abs{\grad}^{s/2}}{\jap{t}^{s}}A^i \right) Q_0^i}_2^2 + \frac{1}{\jap{t}^{2s}}\norm{U_0^i}_2^2 \nonumber \\ & \quad + \epsilon^2\norm{\left(\sqrt{\frac{\partial_t w}{w}}\tilde{A} + \frac{\abs{\grad}^{s/2}}{\jap{t}^{s}}A\right) C}_2^2. \label{ineq:PELCKZero}
\end{align}
\end{lemma} 
\begin{proof} 
First observe that 
\begin{align*}
\partial_t w(t,\eta) \mathbf{1}_{t \geq 1} \mathbf{1}_{\abs{\eta} \leq 1/2} = 0. 
\end{align*}
Hence, 
\begin{align}
\norm{\left(\sqrt{\frac{\partial_t w}{w}}\tilde{A}^i + \frac{\abs{\grad}^{s/2}}{\jap{t}^{s}}A^i \right) \jap{\grad}^2 U_0^i}_2^2 & \lesssim \norm{\left(\sqrt{\frac{\partial_t w}{w}}\tilde{A}^i + \frac{\abs{\grad}^{s/2}}{\jap{t}^{s}} A^i \right) \left(\Delta_L U_0^i\right)_{\geq 1/2}}_2^2 + \frac{1}{\jap{t}^{2s}}\norm{U_0^i}_2^2. \label{ineq:lowfreqdecay} 
\end{align}
Therefore, similar to the proof of Lemma \ref{lem:PELbasicZero} (see \cite{BGM15I}), it suffices to control the higher frequencies. 
Next, write $\Delta_L U_0^3$ using the formula for $\Delta_t U_0^3$ and projecting both sides of the equation to frequencies larger than $1/2$: 
\begin{align}
\left(\Delta_L U_0^i\right)_{\geq 1/2} & = (Q_0^i)_{\geq 1/2} - \left(G_{yy}\partial_{YY}U_0^i + G_{zy}\partial_{YZ}U_0^i + G_{zz}\partial_{ZZ}U_0^i + \Delta_t C^1 \partial_{Y}U_0^i + \Delta_t C^2 \partial_{Z}U_0^i\right)_{\geq 1/2} \nonumber \\ 
& = (Q_0^i)_{\geq 1/2} + \sum_{j = 1}^5 \mathcal{E}_i.\label{def:PELCKU0i} 
\end{align}
Apply
\begin{align*}
\mathcal{M} = \left(\sqrt{\frac{\partial_t w}{w}}\tilde{A}^i + \frac{\abs{\grad}^{s/2}}{\jap{t}^{s}}A^i \right)
\end{align*}
to both sides of \eqref{def:PELCKU0i} and deduce 
\begin{align}
\norm{\mathcal{M} \left(\Delta_L U_0^i\right)_{\geq 1/2}}_2^2 \lesssim \norm{\mathcal{M} \left(Q_0^i\right)_{\geq 1/2}}_2^2 + \sum_{j = 1}^5 \norm{\mathcal{M} \mathcal{E}_i}_2^2. \label{ineq:PELCKU0iEstimate} 
\end{align} 
The error terms will be divided into pieces which will either be absorbed by the LHS of \eqref{ineq:PELCKU0iEstimate}  or will appear on the RHS of \eqref{ineq:PELCKZero}. 
The latter two error terms are the most difficult and they are also very similar, hence it suffices to treat only $\mathcal{E}_5$. 
First, expand with a paraproduct 
\begin{align}
\mathcal{M}\mathcal{E}_5 & = - \mathcal{M}\left( (\Delta_t C^2)_{Hi} (\partial_{Z}U_0^i)_{Lo} \right)_{\geq 1/2} - \mathcal{M}\left( (\Delta_t C^2)_{Lo} (\partial_{Z}U_0^i)_{Hi} \right)_{\geq 1/2} - \mathcal{M}\left( (\Delta_t C^2 \partial_{Z}U_0^i)_{\mathcal{R}} \right)_{\geq 1/2} \nonumber \\ 
& = \mathcal{M}\mathcal{E}_{5;HL} + \mathcal{M}\mathcal{E}_{5;LH} + \mathcal{M}\mathcal{E}_{5;\mathcal{R}}. \label{eq:ME5}
\end{align}
For the high-low term we use Lemma \ref{lem:ABasic} and \eqref{ineq:dtwBasic}, 
\begin{align*}
\mathcal{M}\mathcal{E}_{5;HL} & \lesssim \epsilon \sum\int \frac{\mathbf{1}_{\abs{\eta,l} \geq 1/2}}{\jap{\xi,l^\prime}^2} \left(\sqrt{\frac{\partial_t w(t,\xi)}{w(t,\xi)}}\tilde{A}(\xi,l^\prime) + \frac{\abs{\xi,l^\prime}^{s/2}}{\jap{t}^{s}}A(\xi,l^\prime) \right) \abs{\widehat{\Delta_t C^2}(\xi,l^\prime)_{Hi}} Low(\eta-\xi,l-l^\prime) d\xi. 
\end{align*}
Hence, by \eqref{ineq:quadHL} and Lemma \ref{lem:CoefCtrl} we have 
\begin{align*}
\norm{\mathcal{M}\mathcal{E}_{5;HL}}^2_2 & \lesssim \epsilon^2\norm{\left(\sqrt{\frac{\partial_t w}{w}}\tilde{A} + \frac{\abs{\grad}^{s/2}}{\jap{t}^{s}}A\right)C}_2^2, 
\end{align*}
which appears on the RHS of \eqref{ineq:PELCKZero}. 
To treat the low-high term in  \eqref{eq:ME5}, we use a similar method to deduce 
\begin{align*}
\norm{\mathcal{M}\mathcal{E}_{5;LH}}^2_2 & \lesssim c_0^2 \norm{\mathcal{M} \partial_Z U_0^i}_2 \\ 
& \lesssim c_0^2 \left(\norm{\mathcal{M} (\partial_Z U_0^i)_{\geq 1/2}}_2 + \norm{\mathcal{M}(\partial_Z U_0^i)_{\geq 1/2}}_2^2\right) \\ 
& \lesssim c_0^2 \left(\norm{\mathcal{M} \left(\Delta_L U_0^i\right)_{\geq 1/2}}_2 + \frac{1}{\jap{t}^{2s}}\norm{U_0^i}_2^2\right)
\end{align*}
where the last line followed as in \eqref{ineq:lowfreqdecay}. 
The first term is absorbed on the LHS of \eqref{ineq:PELCKU0iEstimate} whereas the second term appears on the RHS of \eqref{ineq:PELCKZero}.  
The remainder term is straightforward and can be treated in essentially the same way as the low-high term; see the proof of [Lemma 4.9 \cite{BGM15I}] for a similar argument. 
As the other error terms are essentially the same, this completes the proof of \eqref{ineq:PELCKZero}. 
\end{proof} 

\subsubsection{Non-zero mode PELs}

The next PEL is an easy variant of the analogous [Lemma C.5, \cite{BGM15I}]. The proof is a slight variation of that in \cite{BGM15I}. Here we need to deal with the large $Z$ frequencies but this is straightforward due to the inequalities derived in \S\ref{sec:basicmult} and hence the details are omitted here. 

\begin{lemma}[$CK$ PEL] \label{lem:PEL_NLP120neq}
Let $h$ be given such that $\norm{h}_{\G^{\lambda}}\lesssim \epsilon \jap{t}^{b} \jap{\nu t^3}^{-a}$ for some $a \geq 0$ and $b \geq 0$. 
Then, under the bootstrap hypotheses, for $c_0$ and $\epsilon$ sufficiently small, there holds, 
\begin{subequations} \label{ineq:PEL_NLP120neq}
\begin{align} 
\norm{\left(\sqrt{\frac{\partial_t w}{w}}\tilde{A}^{i} + \frac{\abs{\grad}^{s/2}}{\jap{t}^s}A^{i}\right)  \Delta_L \Delta_t^{-1} h_{\neq}}_2^2 
  & \lesssim \norm{\left(\sqrt{\frac{\partial_t w}{w}}\tilde{A}^{i} + \frac{\abs{\grad}^{s/2}}{\jap{t}^s}A^{i} \right) h_{\neq}}_2^2 \nonumber \\ 
 & \quad + \frac{\epsilon^2 \jap{t}^{2b-2} }{\jap{\nu t^3}^a} \norm{\left(\sqrt{\frac{\partial_t w}{w}}\tilde{A} + \frac{\abs{\grad}^{s/2}}{\jap{t}^s}A \right) C}_2^2.   \label{ineq:PEL_NLP120neq1}
\end{align} 
\end{subequations}
\end{lemma}

The next PEL is also basically  [Lemma C.6, \cite{BGM15I}] and is slightly simpler than Lemma \ref{lem:PEL_NLP120neq}. 
\begin{lemma}[Zero order PEL] \label{lem:SimplePEL} 
Let $h$ be given such that $\norm{h}_{\G^{\lambda}} \lesssim \epsilon\jap{t}^{b} \jap{\nu t^3}^{-a}$ for $a,b \geq 0$. 
Then, for $c_0$ and $\epsilon$ sufficiently small, under the bootstrap hypotheses we have for all $i \in \set{1,2,3}$,  
\begin{align}
\norm{A^{i} \Delta_L \Delta_t^{-1} h_{\neq}}_2^2 & \lesssim \norm{A^{i} h_{\neq}}_2^2 + \frac{\epsilon^2 \jap{t}^{2b-2} }{\jap{\nu t^3}^{2a}}\norm{A C}_2^2, \label{ineq:SimplePEL}
\end{align} 
\end{lemma}  
 
Finally, from [Lemma C.7, \cite{BGM15I}] is the following PEL for treating the linear pressure term $LP3$ in the $Q^3$ equation.  

\begin{lemma}[PEL for $CK_{wL}$] \label{lem:QPELpressureI}
Let $h$ be given such that $\norm{h}_{\G^{\lambda}} \lesssim \epsilon \jap{t}^b \jap{\nu t^3}^{-a}$ for $a,b \geq 0$ and suppose $C$ satisfies the bootstrap hypotheses.
Then for $c_0$ and $\epsilon$ sufficiently small, there holds
\begin{align} 
\norm{\sqrt{\frac{\partial_t w_L}{w_L}} A^{3} \Delta_{L} \Delta_t^{-1} h_{\neq}}_2^2  & \lesssim \norm{\sqrt{\frac{\partial_t w_L}{w_L}} A^{3} h_{\neq}}_2^2 + \frac{\epsilon^{2}\jap{t}^{2b-2}}{\jap{\nu t^3}^{2a}} \norm{\left(\sqrt{\frac{\partial_t w}{w}}\tilde{A} + \frac{\abs{\grad}^{s/2}}{\jap{t}^s}A\right)C}_2^2. \label{ineq:PELpressureI}
\end{align} 
\end{lemma}

The last PEL is unique to this work (it was not necessary in \cite{BGM15I}).
It is needed here to gain additional precision for times $t \gtrsim \epsilon^{-1/2}$.
It is used in, e.g. \eqref{ineq:teps12trick} above. 

\begin{lemma}[Enhanced dissipation PEL] \label{lem:PELED}
Let $h$ be given such that $\norm{h}_{\G^{\lambda}}\lesssim \epsilon \jap{t}^{b} \jap{\nu t^3}^{-a}$ for some $a \geq 0$ and $b \geq 0$. 
Then, under the bootstrap hypotheses, for $c_0$ and $\epsilon$ sufficiently small there holds
\begin{align} 
\norm{\sqrt{-\Delta_L} A^{i} \Delta_L \Delta_t^{-1} h_{\neq}}_2^2 
  & \lesssim \norm{\sqrt{-\Delta_L} A^{i} h_{\neq}}_2^2 + \frac{\epsilon^2 \jap{t}^{2b-2}}{\jap{\nu t^3}^{2a}}\norm{\grad AC}_2^2 + \frac{\epsilon^2 \jap{t}^{2b}}{\jap{\nu t^3}^{2a}}\norm{AC}_2^2.      \label{ineq:PELED}
\end{align} 
\end{lemma}
\begin{proof} 
The proof is very similar to the proof of Lemma \ref{lem:PEL_NLP120neq} (the proof of which is found in \cite{BGM15I}). 
Let us briefly sketch the argument.  
Write $P = \Delta_t^{-1}h_{\neq}$ 
\begin{align} 
\Delta_L P & = h_{\neq} - G_{yy}(\partial_Y - t\partial_X)^2 P - G_{yz}(\partial_Y - t\partial_X) \partial_Z P + G_{zz}\partial_{ZZ}P  - \Delta_tC^1 (\partial_Y - t\partial_X) P - \Delta_t C^2 \partial_Z P \nonumber \\ 
& = h_{\neq} + \sum_{i = 1}^5 \mathcal{E}_i. \label{def:PkPEL}
\end{align} 
We apply $\sqrt{-\Delta_L}A^i$ to both sides of \eqref{def:PkPEL} and estimate the terms on the RHS. 
Hence we get
\begin{align} 
\norm{\sqrt{-\Delta_L}A^i \Delta_L P}_2^2 \lesssim \norm{\sqrt{-\Delta_L}A^ih_{\neq}}_2^2 + \sum_{i = 1}^5 \norm{\sqrt{-\Delta_L}A^i \mathcal{E}_i}_2^2. \label{ineq:DeltaPk} 
\end{align}
For example, consider the first error term and expand with a paraproduct: 
\begin{align*} 
\sqrt{-\Delta_L}A^i\mathcal{E}_1 & = \sqrt{-\Delta_L}A^i\left( (G_{yy})_{Hi}(\partial_Y - t\partial_X)^2 P_{Lo}\right) + \sqrt{-\Delta_L}A^i\left( (G_{yy})_{Lo}(\partial_Y - t\partial_X)^2 P_{Hi}\right) + \mathcal{E}_{1;\mathcal{R}} \\ 
& := \mathcal{E}_{1;C} + \mathcal{E}_{1;P} + \mathcal{E}_{1;\mathcal{R}}. 
\end{align*}
By \eqref{ineq:ABasic}, \eqref{ineq:TriTriv}, \eqref{ineq:dtwBasic}, \eqref{ineq:quadHL}, and Lemma \ref{lem:CoefCtrl} it follows that
\begin{align*} 
\norm{\sqrt{-\Delta_L}A^i \mathcal{E}_{1;P}}_2^2 & \lesssim c_{0}^2\norm{\sqrt{-\Delta_L}A^i \Delta_L P}_2^2, 
\end{align*} 
which can hence be absorbed on the LHS of \eqref{ineq:DeltaPk} by choosing $c_{0}$ sufficiently small.
The remainder is treated $\mathcal{E}_{1;\mathcal{R}}$ is treated similarly. 
Consider next $\mathcal{E}_{1;C}$ for which, by the hypotheses, Lemma \ref{lem:ABasic}, and Lemma \ref{lem:LossyElliptic}, we have
\begin{align*} 
\mathcal{E}_{1;C} & \lesssim \frac{\epsilon \jap{t}^{b}}{\jap{\nu t^3}^{a}}\sum_{l} \int_\xi \abs{k,\eta-kt,l} \frac{1}{\jap{\xi,l^\prime}^2}\jap{\frac{t}{\jap{\xi,l^\prime}}}^{-1} A\abs{\widehat{G_{yy}}(\xi,l^\prime)_{Hi}} Low(k,\eta-\xi,l-l^\prime) d\xi; 
\end{align*}  
the extra $\jap{t}^2$ from $(\partial_Y - t\partial_X)^2$ was canceled by the $\Delta_t^{-1}$ in the definition of $P$ and Lemma \ref{lem:LossyElliptic}.  
It follows from \eqref{ineq:quadHL} and Lemma \ref{lem:CoefCtrl} that 
\begin{align*}
\norm{\sqrt{-\Delta_L}A^i \mathcal{E}_{1;C}}_2^2 & \lesssim \frac{\epsilon^2 \jap{t}^{2b-2}}{\jap{\nu t^3}^{2a}}\norm{\grad AC}_2^2 + \frac{\epsilon^2 \jap{t}^{2b}}{\jap{\nu t^3}^{2a}}\norm{AC}_2^2,  
\end{align*}
which suffices. This completes the treatment of $\mathcal{E}_{1}$. The error terms $\mathcal{E}_{2}$ and $\mathcal{E}_{3}$ are treated exactly the same. 
In treating the error terms $\mathcal{E}_{4}$ and $\mathcal{E}_{5}$, note that there is an extra derivative on $C^i$. 
As a result, we cannot recover a power of time from Lemma \ref{lem:ABasic} using the low-frequency growth. However, there is one less power of $t$ on $P$ and hence there is a balance and a similar proof as that used on $\mathcal{E}_1$ will adapt in a straightforward manner to the last two error terms. We omit the details for brevity.  
\end{proof}

\bibliographystyle{plain} \bibliography{eulereqns_vlad,IDnLD}

\end{document}